\documentclass[11 pt]{amsart}

\usepackage{fullpage} 
\usepackage{diagbox}

\usepackage{hyperref}
\usepackage{etex}
\usepackage[shortlabels]{enumitem}
\usepackage{amsmath}
\usepackage{amsxtra}
\usepackage{amscd}
\usepackage{amsthm}
\usepackage{amsfonts}
\usepackage{amssymb}
\usepackage{eucal}
\usepackage[all]{xy}
\usepackage{graphicx}
\usepackage{tikz}
\usetikzlibrary{cd}
\usetikzlibrary{fit, patterns}
\usepackage{mathrsfs}
\usepackage{subfiles}
\usepackage{mathpazo}
\usepackage[colorinlistoftodos, textsize=tiny]{todonotes}
\usepackage{morefloats}
\usepackage{pdfpages}
\usepackage{thm-restate}
\usepackage[utf8]{inputenc}
\usepackage[T1]{fontenc}
\usepackage{epigraph}
\usepackage{csquotes}
\usepackage{moreverb}
\usepackage[margin=1in]{geometry}
\usepackage{adjustbox}
\usepackage{accents}

\graphicspath{ {images/} }

\RequirePackage{color}
\definecolor{myred}{rgb}{0.75,0,0}
\definecolor{mygreen}{rgb}{0,0.5,0}
\definecolor{myblue}{rgb}{0,0,0.65}

\hypersetup{citecolor=blue}
\usetikzlibrary{matrix,arrows,decorations.pathmorphing}

\theoremstyle{plain}
\newtheorem{theorem}[subsubsection]{Theorem}
\newtheorem{proposition}[theorem]{Proposition}
\newtheorem{lemma}[theorem]{Lemma}
\newtheorem{corollary}[theorem]{Corollary}
\theoremstyle{definition}
\newtheorem{definition}[theorem]{Definition}
\newtheorem{remark}[theorem]{Remark}
\newtheorem{example}[theorem]{Example}

\newtheorem{question}[theorem]{Question}
\newtheorem{conjecture}[theorem]{Conjecture} 

\newtheorem{warn}[theorem]{Warning}

\newtheorem{notation}[theorem]{Notation}
\newtheorem{hypotheses}[theorem]{Hypotheses}
\theoremstyle{remark}
\numberwithin{equation}{section}
\numberwithin{theorem}{subsection}

\newcommand\nc{\newcommand}
\nc\on{\operatorname}
\nc\renc{\renewcommand}

\newcommand*{\shom}{\mathscr{H}\kern -.5pt om}
\newcommand*{\stor}{\mathscr{T}\kern -.5pt or}
\newcommand*{\sext}{\mathscr{E}\kern -.5pt xt}

\makeatletter
\providecommand\@dotsep{5}
\renewcommand{\listoftodos}[1][\@todonotes@todolistname]{%
\@starttoc{tdo}{#1}}
\makeatother

\newcommand{\customlabel}[2]{\protected@write \@auxout {}{\string \newlabel {#1}{{#2}{\thepage}{#2}{#1}{}} }\hypertarget{#1}{#2}}

\DeclareMathOperator\tor{Tor}

\renewcommand\hom{\mathrm{Hom}}
\DeclareMathOperator\coker{coker}

\DeclareMathOperator\rk{rk}

\DeclareMathOperator\spec{Spec}

\DeclareMathOperator\im{im}

\DeclareMathOperator\sym{Sym}
\renewcommand\drop{\mathrm{Drop}}

\renewcommand\sp{\mathrm{Sp}}

\DeclareMathOperator\gal{Gal}
\DeclareMathOperator\tr{tr}
\DeclareMathOperator\frob{Frob}
\DeclareMathOperator\chr{\operatorname{char}}

\newcommand\ol{\overline}

\DeclareMathOperator\ev{ev}

\DeclareMathOperator\prob{Prob}

\DeclareMathOperator\aut{Aut}
\DeclareMathOperator\gl{GL}

\DeclareMathOperator\asp{\mathrm{ASp}}
\DeclareMathOperator\ahsp{\mathrm{A}^{\operatorname{H}}\mathrm{Sp}}

\DeclareMathOperator\so{SO}

\renewcommand\o{{\rm{O}}}

\DeclareMathOperator\sel{Sel}
\DeclareMathOperator\Sel{Sel}
\DeclareMathOperator\tors{tors}

\DeclareMathOperator\quo{Quo}

\DeclareMathOperator\grp{grp}
\DeclareMathOperator\ogr{OGr}
\DeclareMathOperator\et{\acute et}

\DeclareMathOperator\colim{colim}

\renewcommand\top{\mathrm{top}}
\DeclareFontFamily{U}{wncy}{}
\DeclareFontShape{U}{wncy}{m}{n}{<->wncyr10}{}
\DeclareSymbolFont{mcy}{U}{wncy}{m}{n}
\DeclareMathSymbol{\Sha}{\mathord}{mcy}{"58}

\DeclareMathOperator{\Sym}{Sym}

\DeclareMathOperator{\Conf}{Conf}

\newcommand{\field}[1]{\mathbb{#1}}

\newcommand{\Z}{\field{Z}}

\newcommand{\F}{\field{F}}

\newcommand{\R}{\field{R}}

\newcommand{\ord}{\mbox{ord}}
\newcommand{\EE} {\mathbb{E}}

\newcommand{\ra}{\rightarrow}

\DeclareMathOperator{\id}{\mathrm{id}}

\newcommand{\tensor} {\otimes}

\newcommand{\beq}{\begin{displaymath}}
\newcommand{\eeq}{\end{displaymath}}
\newcommand{\beqn}{\begin{equation}}
\newcommand{\eeqn}{\end{equation}}

\newcommand{\conf}[3]{\operatorname{Conf}^{#1}_{#2/#3}} 
\newcommand{\st}[4]{\operatorname{St}^{#1, #2}_{#3/#4}} 
\newcommand{\stqtwist}[4]{\operatorname{QTwist}^{#1,
#2}_{\operatorname{St},#3/#4}} 
\newcommand{\openstqtwist}[4]{\operatorname{QTwist}^{#1,
#2}_{#3/#4}} 
\newcommand{\qtwist}[3]{\operatorname{QTwist}^{#1}_{#2/#3}} 
\newcommand{\rankcover}[2]{\operatorname{QTwist}^{\rk,
#1}_{#2}} 

\newcommand{\hur}[6]{\operatorname{Hur}^{#1, #2, #3, #4}_{#5/#6}}
\newcommand\inertiaindex[5]{T^{#1}_{#2, #3, #4, #5}}
\newcommand\allinertia[4]{T_{#1, #2, #3, #4}}
\newcommand\sinertiaindex[5]{S^{#1}_{#2,#3,#4,#5}}
\newcommand\sallinertia[4]{S_{#1, #2,#3,#4}}

\newcommand{\selhur}[2]{\operatorname{Hur}^{#2}_{#1}}
\newcommand{\rankselhur}[2]{\operatorname{Hur}^{#2,\rk}_{#1}}
\newcommand{\selsheaf}[1]{{\mathcal{S}e\ell}_{#1}} 
\newcommand{\selspace}[1]{\operatorname{Sel}_{#1}} 
\newcommand{\selspacemoments}[2]{\operatorname{Sel}_{#1}^{#2}} 
\newcommand{\rankselspacemoments}[2]{\operatorname{Sel}^{#2,\rk}_{#1}}
\newcommand{\logselspace}[1]{\operatorname{Sel}^{\operatorname{log}}_{#1}}
\DeclareMathOperator\surj{Surj}
\newcommand{\bklpr}[1]{\operatorname{Sel}^{\operatorname{BKLPR}}_{#1}} 
\newcommand{\paritybklpr}[2]{\operatorname{Sel}^{\operatorname{BKLPR},#2}_{#1}} 

\def\listtodoname{List of Todos}
\def\listoftodos{\@starttoc{tdo}\listtodoname}

\title[The BKLPR heuristics over function fields]{Homological stability for generalized Hurwitz spaces and Selmer groups in quadratic twist families over function fields}

\author{Jordan S. Ellenberg}
\author{Aaron Landesman}

\usepackage{microtype}
\begin{document}

\date{\today}

\begin{abstract}
	We prove a version of the Bhargava-Kane-Lenstra-Poonen-Rains heuristics for Selmer groups of quadratic twist families of abelian varieties over global function	fields. As a consequence, we derive a result towards the  ``minimalist conjecture" on Selmer ranks of abelian varieties in such families.  More precisely, we show that the probabilities predicted in these two conjectures are correct to within an error term in the size of the constant field, $q$, which goes to $0$ as $q$ grows. Two key inputs are a new homological stability theorem for a generalized version of Hurwitz spaces parameterizing covers of punctured Riemann surfaces of arbitrary genus, and an expression of average sizes of Selmer groups in terms of the number of rational points on these Hurwitz spaces over finite fields.
\end{abstract}

\maketitle
\tableofcontents

\section{Introduction}
\label{section:introduction}

For $\nu$ a positive integer and $A$ an abelian variety over a global field $K$, the $\nu$-Selmer group
of $A$, denoted $\sel_\nu(A)$, is a group which sits in an exact sequence between the mod $\nu$ Mordell-Weil group 
$A(K)/\nu A(K)$ and the $\nu$ torsion in the Tate-Shafarevich group $\Sha(A)[\nu]$. These Selmer groups, unlike the other two terms in the exact sequence, are computationally
approachable, and provide the most tractable means of obtaining information about 
the rank of $A$ and of $\Sha(A)$.

The Selmer group of an abelian variety can be thought of as a higher analogue of
the class group of a number field.  The behavior of the class group of a number
field chosen at random from a specified family is the subject of the
Cohen-Lenstra conjecture and its many subsequent generalizations. In the same way, the question ``what does the $\nu$-Selmer group of a random abelian variety look like?" is the subject of a suite of more recent conjectures.  Conjectures predicting the distribution of Selmer groups were formulated in
\cite{poonenR:random-maximal-isotropic-subspaces-and-selmer-groups}, when $\nu$
is prime, and generalized to the case of composite $\nu$ in
\cite[\S5.7]{bhargavaKLPR:modeling-the-distribution-of-ranks-selmer-groups},
see also
\cite[\S5.3.3]{fengLR:geometric-distribution-of-selmer-groups}.
We call these conjectures 
the ``BKLPR heuristics.''  
Although the above papers state their conjectures in the context of the universal family parameterizing all elliptic
curves, it is also natural to ask under what circumstances they apply to quadratic twist families (
\cite[Remark 1.9]{poonenR:random-maximal-isotropic-subspaces-and-selmer-groups}.)
Our main result is a proof of these conjectures
over function fields of arbitrary genus, up to an error term in $q$ that
approaches $0$ as $q$ grows, in the case where the family of abelian varieties is the family of quadratic twists of a fixed abelian variety.

For $\ell$ a suitably large prime,
as an immediate consequence of our main result, we obtain a version of the {\em minimalist 
conjecture} for $\ell^\infty$-Selmer ranks, which predicts that quadratic twists of a fixed elliptic curve have $\ell^\infty$-Selmer rank $0$ half the
time, $\ell^\infty$-Selmer rank $1$ half the time, and $\ell^\infty$-Selmer rank at least $2$ zero percent of the time.

The approach of this paper is similar to that of 
\cite{EllenbergVW:cohenLenstra},
which verifies
a version of the Cohen-Lenstra heuristics over genus $0$ function fields.
As in \cite{EllenbergVW:cohenLenstra},
one key input is a new homological stability theorem.  This theorem, which is purely topological in nature, is used to bound the \'{e}tale cohomology
of relevant moduli spaces, whose $\mathbb F_q$-points count elements of Selmer
groups of quadratic twists of an abelian variety.

\subsection{Main Results}
\label{subsection:main-results}

To give an indication of the nature of the results we prove in this paper, we
start with a very special case of \autoref{theorem:main-moments} below, see
\autoref{remark:beginning-explanation}. We now describe this special case informally. 
Let $\mathbb F_q$ be a finite field of odd characteristic,
$A$ be an abelian variety over the field $\F_q(t)$ with good reduction over
$\infty$, and $\ell$ an odd prime not
dividing $q$.  For any squarefree polynomial $f \in \F_q[t]$ of even degree $n$,\footnote{See \autoref{remark:ramified-quadratic-twists} for a discussion on how to
generalize this to the case that the degree, $n$, is odd.}
we denote by $A_f$ the quadratic twist of $A$ by the quadratic character of
$\F_q(t)$ associated to $f$.  
Write $\EE_n \Sel_\ell A_f$ for the average size
of the $\ell$-Selmer group of $A_f$ as $f$ ranges over squarefree 
polynomials of degree $n$ which are coprime to the bad reduction locus of $A$.
Similarly, write $\EE_{n,j} \Sel_\ell A_f$ for the same average obtained from
the base change $A / \F_{q^j}(t)$, so that the average is now over the
squarefree polynomials in $\F_{q^j}(t)$ coprime to bad reduction.  Then, the Poonen-Rains heuristics
assert that $\lim_n \EE_{n,j} \Sel_\ell A_f = \ell+1$ for all $j$.  What we
prove, subject to some modest conditions on $A$ and $\ell$, which will be
specified in \autoref{theorem:main-moments}, is that 
\beq
\lim_j \lim_n \EE_{n,j} \Sel_\ell A_f 
= \ell+1.
\eeq

We emphasize that the computation that $\lim_j \EE_{n,j} \Sel_\ell A_f =
\ell+1$, without first taking a limit in $n$,
is substantially easier, see \autoref{subsection:past-work} for more on this
issue.
The contribution of the present
paper is to understand, as in the BKLPR heuristics, what happens when $n$ goes
to infinity {\em with $j$ fixed}, or, in other words, $A$ is defined over a specific global field $\F_{q^j}(t)$.

Before getting to our most general results,
we present a second special
case of 
\autoref{theorem:main-moments} 
which is again simpler to state than \autoref{theorem:main-moments}, but still of significant interest.
Let $C$ be a smooth proper geometrically connected curve
over a finite field $\mathbb F_q$ of odd characteristic and let $U \subset C$ be a nonempty open subscheme with
nonempty complement. 
Let $\nu$ be an odd integer and $A \to U$ be a polarized abelian scheme 
with polarization of degree prime to $\nu$.
Let $\qtwist n U {\mathbb F_q}(\mathbb F_{q^j})$ denote the groupoid of
quadratic twists of $A\times_{\spec \mathbb F_q} \spec \mathbb F_{q^j}$, ramified over a
degree $n$ divisor contained in $U$ with $n$ even. (See
\autoref{notation:quadratic-twist-notation} for a precise definition.)
For $x \in \qtwist n U {\mathbb F_q}(\mathbb F_{q^j})$, we use the notation
$A_x$ to indicate the quadratic twist of $A$ corresponding to the point $x$.
We use $\bklpr \nu$ for the predicted distribution of the $\nu$-Selmer group, as
given in
\cite{bhargavaKLPR:modeling-the-distribution-of-ranks-selmer-groups}; see
\autoref{definition:bklpr-rank-selmer} for a brief definition.
The following consequence of our main result says the BKLPR heuristics hold for quadratic
twists of an elliptic curve with squarefree discriminant, up to an error that
goes to $0$ as $q$ grows.

\begin{theorem}
	\label{theorem:squarefree-discriminant-case}
	With notation as above, suppose $A$ is a nonconstant elliptic curve with
	squarefree discriminant.
	Choose $\nu$ and $q$ so that $\chr \mathbb F_q > 3$ and $\nu$ is prime to
	$6q$.  Let $H$ be a finitely generated $\Z/\nu \Z$-module.
	Then
	\begin{align*}
\prob(\bklpr \nu \simeq H) &=
		\lim_{j \to \infty} \limsup_{\substack{n \to \infty \\ n
		\hspace{.1cm} \mathrm{even}}} \prob( \sel_\nu(A_x)
	\simeq H :x \in \qtwist n U {\mathbb F_q}(\mathbb F_{q^j}) )  \\
		&=\lim_{j \to \infty} \liminf_{\substack{n \to \infty \\ n
		\hspace{.1cm} \mathrm{even}}} \prob( \sel_\nu(A_x)
	\simeq H :x \in \qtwist n U {\mathbb F_q}(\mathbb F_{q^j}) ).
	 	\end{align*}
\end{theorem}

We next state a more general theorem of which
\autoref{theorem:squarefree-discriminant-case} is a consequence, as we next
explain:
The irreducibility assumption in 
\autoref{theorem:main-finite-field}
holds in the setting of
\autoref{theorem:squarefree-discriminant-case}
by
\cite[Proposition 2.7]{zywina:inverse-orthogonal}. 
The reason we assume $\nu$ is prime to $6q$ is to ensure $A[\nu] \to U$ is tame,
which is one of the assumptions of
\autoref{theorem:main-finite-field}.
In particular, we will not need to assume $\nu$ is prime to $6q$ in future results
in this introduction.
The remaining assumptions
in \autoref{theorem:main-finite-field} also
automatically hold for any nonconstant elliptic
curve of squarefree discriminant.
For the next theorem, we use notation as defined prior to \autoref{theorem:squarefree-discriminant-case}.
\begin{theorem}
	\label{theorem:main-finite-field}
	With notation as above,
	choose an abelian scheme $A$ so that
\begin{equation}
\label{equation:multiplicative-hypothesis}
	\text{$A$ has multiplicative
		reduction with toric part of dimension $1$ over some point of
		$C$.}
\end{equation}
Choose $\nu$ so that every prime $\ell \mid \nu$ satisfies $\ell > 2 \dim A +
1$ and $A[\ell] \times_{\mathbb F_q} \overline{\mathbb F}_q$ corresponds to a irreducible sheaf of $\mathbb Z/\ell \mathbb Z$-modules on $U\times_{\mathbb F_q} \overline{\mathbb F}_q$,
$\nu$ is prime to $q$,
and $A[\nu]$ is a tame finite \'etale
cover of $U$.
Further assume that $\nu$ is relatively prime to the order of the geometric component group of the N\'eron model of
$A$ over $C$, as defined in \autoref{notation:component-group}.
	We have
	\begin{align*}
		\lim_{j \to \infty} \limsup_{\substack{n \to \infty \\ n
		\hspace{.1cm} \mathrm{even}}} \prob( \sel_\nu(A_x)
	\simeq H :x \in \qtwist n U {\mathbb F_q}(\mathbb F_{q^j}) ) &= \prob(\bklpr \nu \simeq
H),	\end{align*}
	as well as the analogous statement with $\limsup$ replaced with
	$\liminf$.
%
\end{theorem}
\autoref{theorem:main-finite-field}
is proven in
\autoref{subsubsection:proof-main-finite-field}.
We also explain in
\autoref{subsubsection:proof-special-case-prime}
how the proof of 
\autoref{theorem:main-finite-field} can be somewhat shortened in
the case that $\nu$ is prime.
\begin{remark}
	\label{remark:number-field-intro}
If we start with an abelian scheme over an affine curve over a number field $K$, one can spread it
out to an abelian scheme over an affine curve over a sufficiently small nonempty open $\spec \mathscr O
\subset \spec \mathscr
O_K$. One can then deduce a version of \autoref{theorem:main-finite-field} where one
takes a limit over prime powers with characteristic avoiding finitely many
primes, instead of restricting the characteristic to take a single fixed value,
as in \autoref{theorem:main-finite-field}. See \autoref{remark:number-field} for
more on this.
The key point is that the cohomology groups of the relevant moduli space will be
independent of the geometric point of $\spec \mathscr O$ we choose.
\end{remark}

We next include some remarks on the relation between our results, the BKLPR
heuristics, and the results of
\cite{EllenbergVW:cohenLenstra}.
\begin{remark}
	\label{remark:}
	\autoref{theorem:main-finite-field} can be thought of as a version of the conjectures of
\cite{bhargavaKLPR:modeling-the-distribution-of-ranks-selmer-groups} over global
function fields for quadratic twist families of abelian varieties.  There are two respects in which our result does not precisely say that the BKLPR conjecture holds for such families.  The first difference, and the more substantial one, is that we can't show the probabilities we analyze agree with the BKLPR heuristics exactly, but only up to an error term that shrinks as the finite field gets larger and larger.  The second difference is that BKLPR makes conjectures for $\ell^\infty$-Selmer groups,
while our results apply only to finite order Selmer groups.
It seems likely the ideas in this paper could be extended to the case of
$\ell^\infty$-Selmer groups, and we think it would be quite interesting to do so.

The relationship between the theorems of the present paper and the BKLPR
heuristics is analogous to the relationship between the results of
\cite{EllenbergVW:cohenLenstra} and the Cohen-Lenstra heuristics.
The connection between the two papers is discussed further in the next remark.
\end{remark}
\begin{remark}
\label{remark:}
We believe the version of the Cohen-Lenstra heuristics proven in
\cite{EllenbergVW:cohenLenstra} should be viewable as a degenerate case of
\autoref{theorem:main-finite-field}, where one takes $A$ to be a $1$-dimensional
torus, instead of an abelian scheme.
The torus may be viewed as a degeneration of an elliptic curve.
We note that \cite{EllenbergVW:cohenLenstra} does not directly follow from the
results presented here, but we are hopeful that a modest generalization of the work in this paper could imply both those results and ours.

\end{remark}

The next result computes the moments of Selmer groups.
To introduce some further notation, if $X$ and $Y$ are two finite abelian groups, we use $\#\surj(X, Y)$ for the number of surjections from $X$ to $Y$.
We also define $Z := C- U$.

\begin{theorem}
\label{theorem:main-moments}
With the same hypotheses on $A$ and $\nu$ as in
\autoref{theorem:main-finite-field},
\begin{align}
	\label{equation:moment-limsup}
		\lim_{j \to \infty} \limsup_{\substack{n \to \infty \\ n
		\hspace{.1cm} \mathrm{even}}}\frac{\sum_{x \in \qtwist n U
			{\mathbb F_q}(\mathbb F_{q^j})}  \# \surj
		(\sel_\nu(A_x), H)}{\sum_{x \in \qtwist n U {\mathbb F_q}(\mathbb F_{q^j})}
		1}
		&= \#\sym^2 H,
		\end{align}
	as well as the analogous statement with $\limsup$ replaced with
	$\liminf$.

If, moreover, there is some $\sigma \in Z(\mathbb F_q)$ over which $A$ has good
reduction, 
\begin{align}
	\label{equation:moment-limit-residue}
		\lim_{j \to \infty} \lim_{\substack{n \to \infty \\ n
		\hspace{.1cm} \mathrm{even}}}
				\frac{\sum_{x \in \qtwist n U
			{\mathbb F_q}(\mathbb F_{q^j})}  \# \surj
		(\sel_\nu(A_x), H)}{\sum_{x \in \qtwist n U {\mathbb F_q}(\mathbb F_{q^j})}
		1}
		&= \#\sym^2 H.
		\end{align}
\end{theorem}

\autoref{theorem:main-moments} 
is proven in \autoref{subsubsection:proof-main-moments}.

\begin{remark}
	\label{remark:error-term-moments}
	An upgraded version of \eqref{equation:moment-limsup}, bounding 
the error term as $j \to \infty$ by a constant (depending on $A$ and $H$) divided by $\sqrt{q}$ can be deduced from
the analogous error term provided in
\autoref{theorem:point-counting-computation},
following the same proof in \autoref{subsubsection:proof-main-moments}.
\end{remark}

\begin{remark}
	\label{remark:removing-sigma-hypothesis}
	The condition that there is some $\sigma \in
	Z(\mathbb F_q)$ over which $A$ has good reduction is fairly easy to
	arrange, by first passing to an extension where $C$ has a $\mathbb F_q$-point of good reduction, and then augmenting $Z$ to include that point.
	Note that this requires us to restrict the class of quadratic twists we consider to those which are unramified at the point we added to Z.

	Moreover, it seems likely that the hypothesis that there is $\sigma \in
	Z(\mathbb F_q)$ over which $A$ has good reduction can be removed.
	A viable path to doing so would involve two generalizations.
	First, we would need to carry out the whole paper in a setting where we
	require our quadratic twists be ramified at specified points in $Z$, as described
	further in
	\autoref{remark:ramified-quadratic-twists}.
	Second, 
	we would need to carry out \autoref{section:frobenius-equivariance} in the
	setting where $A[\nu]$ has inertia type $-\id$ over $\sigma$.
	If one were able to verify both these generalizations, one could then
	show the limit in $n$ exists over $\mathbb F_{q^j}$ for sufficiently
	large $j$ where there is a point
	$\sigma\in C(\mathbb F_{q^j})$ by verifying the limit exists both
	in the case of quadratic twists ramified at $\sigma$ and unramified at
	$\sigma$, and then adding the two
	resulting limits.
	These generalizations both seem quite approachable, and we believe it would be interesting to work
	this out.
\end{remark}

\begin{remark}
	\label{remark:beginning-explanation}
	As we now explain, the informal example given in the first paragraph of
\autoref{subsection:main-results} is the special case of 
of \autoref{theorem:main-moments} where
$\nu =
\ell, H = \Z / \ell \Z,$ and $C = \mathbb P^1_{\mathbb F_q}$, and $Z := C-U$ is the union of the
places of bad reduction of the abelian scheme, together with $\infty$.
We will assume $A$ has good reduction over $\infty$ so that the hypothesis
preceding \eqref{equation:moment-limit-residue} is satisfied, although, as
mentioned in \autoref{remark:removing-sigma-hypothesis}, this is likely
unnecessary.
In this case, $\#\Sym^2 H = \ell$, so the average number of surjections from the
$\ell$-Selmer group to $\Z / \ell \Z$ is $\ell$. Since the $\ell$-Selmer group
is a finite dimensional vector space $V$ over $\Z/ \ell \Z$, 
\begin{align*}
	\# V=
\#\hom(\mathbb Z/\ell \mathbb Z, V) = \#\hom(V, \mathbb Z/\ell \mathbb Z) =
\#\surj(V,\Z/\ell\Z) + 1.
\end{align*}
Thus, the average size of the $\ell$-Selmer group is $\ell+1$ as claimed. 
\end{remark}

It is well-known that bounds for average sizes (or more generally moments) of
$\nu$-Selmer groups yield interesting bounds on algebraic ranks (also known as
Mordell-Weil ranks). Moreover, 
control of algebraic ranks gets better as $\nu$ gets larger. See 
\cite[Proposition 5]{bhargavaS:average-4-selmer} 
and
\cite[p.246-247]{poonenR:random-maximal-isotropic-subspaces-and-selmer-groups}.
Since the results of the present paper allow $\nu$ to be arbitrarily large, they
are well-suited for results on algebraic ranks.  
For $A$ an abelian variety over a global field, we use $\rk_{\ell^\infty} A$ to denote
the $\ell^\infty$-Selmer rank of $A$, which means that we can write
$\sel_{\ell^\infty}(A) \simeq (\mathbb Q_\ell/\mathbb Z_\ell)^{\rk_{\ell^\infty} A}
\oplus G$, for $G$ a finite group.
The minimalist conjecture, a version of which was originally posed by Goldfeld in 1979
\cite[Conjecture B]{goldfeld:conjectures-on-elliptic-curves}, states that for
suitable families of elliptic curves, the rank takes the value $0$ half the
time and $1$ half the time.
In this direction, we will prove the following version of the minimalist
conjecture:

\begin{theorem}
\label{theorem:main-minimalist}
Suppose $A$ is an abelian scheme over $U$ satisfying
\eqref{equation:multiplicative-hypothesis}, and $\nu = \ell$ is a prime
satisfying the hypotheses of \autoref{theorem:main-finite-field}.
Then,
\begin{align*}
\lim_{j \to \infty} \limsup_{\substack{n \to \infty \\ n
		\hspace{.1cm} \mathrm{even}}} \prob(\rk_{\ell^\infty} A_x =0
		: x \in \qtwist n U {\mathbb F_q}(\mathbb F_{q^j})) &= \frac{1}{2},\\
\lim_{j \to
		\infty} \limsup_{\substack{n \to \infty \\ n
		\hspace{.1cm} \mathrm{even}}} \prob(\rk_{\ell^\infty} A_x =1: x \in \qtwist n U
		{\mathbb F_q}(\mathbb F_{q^j})) &= 	\frac{1}{2}, \\
\lim_{j \to
		\infty} \limsup_{\substack{n \to \infty \\ n
		\hspace{.1cm} \mathrm{even}}} \prob(\rk_{\ell^\infty} A_x \geq 2: x \in \qtwist n U
		{\mathbb F_q}(\mathbb F_{q^j})) &= 	0,
		\end{align*}
	as well as the analogous statements with $\limsup$ replaced with
	$\liminf$.
\end{theorem}

\autoref{theorem:main-minimalist} is proven in
\autoref{subsubsection:proof-main-minimalist}.

\begin{remark}[Versions of \autoref{theorem:main-minimalist} for algebraic and
		analytic rank]
	\label{remark:rank-versions}
	The $\ell^\infty$-Selmer rank is conjecturally independent of $\ell$ and equal to the
analytic rank and algebraic rank.
Since the Selmer rank is an upper bound for the algebraic rank, we can
immediately deduce from 
\autoref{theorem:main-minimalist} that the algebraic rank is at most $1$ with
probability $1$, as $j \to \infty$.
We can also deduce from the parity conjecture 
\cite{trihanY:the-ell-parity-conjecture}
that the parity of the analytic rank approaches equidistribution as $j \to \infty$.
If we knew that the parity of the algebraic rank approached equidistribution as $j \to
\infty$, we could prove a version of the minimalist conjecture above for
algebraic rank.
Similarly, if we knew
the analytic rank is at most $1$ with probability
$1$ as $j \to \infty$, we could deduce a version of the minimalist conjecture
for analytic rank, and also use this and known relations between analytic and
algebraic rank to deduce a version of the minimalist conjecture for algebraic
rank.
\end{remark}
%

\subsection{Overview of the proof}
\label{subsection:proof-outline}

The method of the proof has similar broad strokes to that of
\cite{EllenbergVW:cohenLenstra}. See also
\cite{randal-williams:homology-of-hurwitz-spaces} for a summary of this method.
The loose idea is to construct moduli spaces parameterizing objects associated
to the Selmer groups we want
to count.
We then count $\mathbb F_q$-points on these moduli spaces using the
Grothendieck-Lefschetz trace formula and Deligne's bounds, which relates these point counts to the
cohomology of these moduli spaces.
We bound the higher homology groups using a homological stability theorem, and
control the $0$th homology group via a big monodromy result.
Altogether, this gives us enough control on the point counts to estimate the
moments. Finally, we show that these moments determine the distribution of
Selmer groups, and that the resulting distribution agrees with the predicted one.

Nearly every aspect of this strategy turns out to be trickier in the context of
the BKLPR heuristics than it was in the context of the Cohen-Lenstra heuristics.
We next outline the additional difficulties.

\subsection{Summary of the main innovations}
\label{subsection:difficulties-of-the-proof}

\subsubsection{The connection between Selmer groups and Hurwitz stacks}
One of the main insights in this paper is 
that there is a close relation between Selmer groups and
Hurwitz stacks.
It has been well known for many years that the moduli spaces parameterizing
objects in the Cohen-Lenstra heuristics were Hurwitz stacks related to dihedral
group covers. 
However, it seems not to have been previously noticed that the moduli spaces appearing in the BKLPR heuristics are also 
closely related to Hurwitz stacks.
Indeed, in
\autoref{proposition:selmer-to-hurwitz}, we relate stacks parameterizing
$\nu$-Selmer group elements to Hurwitz stacks for the group $\on{ASp}_{2r}(\mathbb
Z/\nu \mathbb Z)$, where $\on{ASp}$ denotes the affine symplectic group, see
\autoref{definition:asp}.

\subsubsection{Homological stability over higher genus punctured curves}

A second difficulty is that
the above Hurwitz stacks do not occur over compact topological surfaces, but instead occur over
punctured surfaces, where the punctures occur at the places of bad reduction of
the abelian scheme. This necessitates that we prove a generalization
of the topological results of \cite{EllenbergVW:cohenLenstra} (which only apply to
Hurwitz stacks over the disc) to Hurwitz stacks over more general Riemann
surfaces which may be punctured and may have positive genus.

The reader familiar with \cite{EllenbergVW:cohenLenstra} may note the absence of something that plays a crucial role in that paper: a conjugacy class $c$ in $G = \on{ASp}_{2r}(\mathbb
Z/\nu \mathbb Z)$ which generates the whole group and which satisfies the
``non-splitting" condition necessary for that paper.  In fact, that role is
played in the present work by the conjugacy class in $G$ consisting of elements
whose image in the symplectic group is $-\id.$  This conjugacy class does not,
of course, generate the whole of $G$, which places us outside the context in
which the methods of \cite{EllenbergVW:cohenLenstra} directly apply.  More
precisely, a branched $G$-cover of the disc, all of whose monodromy lies in
$c$, is automatically disconnected, consisting of components whose monodromy
group is actually the smaller group generated by $c$.  But, in the generality of
the present paper, our Hurwitz spaces will be covers of a Riemann surface with
$(f+1)+n$ punctures, where the monodromies of the relevant $\sp_{2g}(\mathbb
Z/\nu \mathbb Z)$ representation around the first $f+1$ punctures and
around loops forming a basis for the homology of the surface are specified in
advance, while only the monodromies around the last $n$ punctures are required
to lie in the conjugacy class $c$.  Such a cover of a Riemann surface can
certainly be connected, i.e., have full monodromy group $G$. As we will see, it is examples precisely of this kind that will arise when we analyze the moduli stacks attached to variation of Selmer groups in quadratic twist families.

\subsubsection{Homological stability for spaces more exotic than Hurwitz stacks}
Once one deals with the above issues, one might then expect it to be possible to follow the strategy of 
\cite{EllenbergVW:cohenLenstra} to control the cohomology of these spaces, use
this to control the finite field point counts via the Grothendieck-Lefschetz
trace formula and Deligne's bounds, and finally deduce the relevant BKLPR
conjectures. However, this approach would, at best, only compute the moments of
the BKLPR distribution. It turns out that this distribution is not completely determined by its
moments, see \cite[Example 1.12]{fengLR:geometric-distribution-of-selmer-groups}.
In particular, if one restricts to elliptic curves whose Selmer rank is even, the
resulting distribution has the same moments as the full BKLPR distribution.
Therefore, at the very least, in order to show these heuristics hold, we need a
way of separating out abelian varieties of even and odd Selmer rank.
Fortunately, it turns out that there is a certain double cover of the stack of
quadratic twists which governs whether the corresponding abelian variety has
even or odd Selmer rank.  This double cover is not a Hurwitz stack; nonetheless, the new homological stability results proved in this paper are general enough to apply to such covers. In this way, we prove homological stability
results not just for Hurwitz stacks over punctured Riemann surfaces, but a more
general class of covers of configuration space on these Riemann surfaces.
We note that the degrees of our covers of configuration space grow exponentially
in the parameter $n$. A similar framework which inspired ours, 
applying to a very different different class of covers whose degrees grow
polynomially in the parameter, was developed in
\cite{randal-williamsW:homological-stability-for-automorphism-groups}.

\subsubsection{Proving the stabilization maps respect the Frobenius action}

One step of this paper whose analog does not appear in
\cite{EllenbergVW:cohenLenstra} is that we prove that the limit in $n$ exists in 
\eqref{equation:moment-limit-residue}.
To show this limit exists, the key point is to show that the homological
stabilization maps appearing in our main results respect the action of
Frobenius, and hence the traces of Frobenius on these cohomology groups are
compatible.  This is carried out in \autoref{section:frobenius-equivariance}.

A natural explanation for the equivariance would be that the stabilization map
we exhibit topologically is a lift to characteristic $0$ of a
map of schemes over $\mathbb F_q$; this appears to be too much to hope for.
Instead, we show the map is induced by a map of log schemes over $\mathbb F_q$,
which is enough to obtain Frobenius equivariance of the stabilization map.
This idea was inspired by a similar use of log schemes in
\cite[\S8]{bergstromDPW:hyperelliptic-curves-the-scanning-map}.
In that paper, log structures were used not for the purpose of showing
stabilization maps are equivariant, but instead for the purpose of showing that
the cohomology of the relevant spaces are of Tate type. 

In our setting,
significant technical care and new ideas are needed to properly construct the
stabilization maps and show they are equivariant. First, we need to carefully
construct partial compactifications of Selmer spaces. Second, we must endow 
these spaces with the correct
additional data and log structure so that the resulting map of log stacks
matches the topological stabilization map over $\mathbb C$. 

\subsubsection{Proving the stabilization maps have degree $2$}

Even once the Frobenius equivariance described above was in place, in order to
show the limit in
\eqref{equation:moment-limit-residue}
exists over all even
$n$, we needed to construct a stabilization map of degree $2$. 
If we only had a degree $d$ stabilization map,
we would only be able to show the
limit exists along $n$ lying a given residue class modulo $d$.
Previously, as far as we are aware, the general belief of the community seems to
have been that the degree of the
stabilization map was rather large.
However, by using recent work of Wood, we are able to show in \autoref{subsection:u-degree} that there is a stabilization map of
degree $2$, and so the limit over all even $n$ exists on the nose.

\subsubsection{Working with symplectically self-dual sheaves}
Another crucial point is that throughout we work not with $\nu$-torsion in an
abelian scheme, but in the more general setting of symplectically self-dual
sheaves. This idea is also prominent in many works of Katz, such as
\cite{katz:twisted-l-functions-and-monodromy}. Working in this level of
generality is crucial for us, as our topological results only apply in
characteristic $0$, so if we start with an abelian scheme in positive
characteristic, we need some way of lifting it to characteristic $0$ in a way
compatible with our hypotheses. While we are quite unsure whether this is
possible for
abelian schemes, it is not too difficult for symplectically self-dual sheaves.

We now explain why we are able to get away with working with symplectically
self-dual sheaves, in place of abelian schemes.
Under the assumptions of 
\autoref{theorem:main-finite-field}, the $\sel_\nu(A)$ only depends on $A[\nu]$.
Namely, if $C, A,\nu,$ and $q$ are as in
\autoref{theorem:main-finite-field},
$\sel_\nu(A) \simeq H^1(C, \mathscr A[\nu])$, for $\mathscr A$ the N\'eron model
of $A$ over $C$. 
(A similar isomorphism holds in the number field case, see
\cite[Proposition
5.4(c)]{cesnavicius:selmer-groups-as-flat-cohomology-groups}.)
Hence, $\sel_\nu(A)$ is determined just from the group scheme
$A[\nu]$ because $\mathscr A[\nu] = j_* A[\nu]$ for $j: U \to C$ the open
inclusion. Therefore, we are free to forget that we started with an abelian scheme, so
long as we remember this symplectically self-dual \'etale sheaf $A[\nu]$.

\subsubsection{Difficulties related to $g > 0$, BKLPR moments, and monodromy}
There are several further subtleties, and we now briefly summarize a couple of
them.
First, unlike the case of genus $0$, in higher genus, there may be many
quadratic twists with the same ramification divisor.
Second, 
for $\nu$ a general composite integer,
the moments of the BKLPR
distribution do not seem to be computed in the existing literature.
We note that when $\nu$ is prime, and more generally
when $H$ is a free $\mathbb Z/\nu \mathbb Z$-module, these moments were computed
in 
\cite[Theorem
5.10]{bhargavaKLPR:modeling-the-distribution-of-ranks-selmer-groups}.
We compute the moments of the BKLPR distribution for general composite $\nu$ in
\autoref{proposition:bklpr-moments}.

Third, we need to compute the relevant
monodromy groups. This too requires additional technical work, where we draw
great inspiration from works of Katz
\cite{katz:twisted-l-functions-and-monodromy} and Hall
\cite{hall:bigMonodromySympletic}, relying on the theory of middle
convolution.

\subsection{Discussion of equidistribution of parity of rank}

We next include a number of remarks relating to our main results and
equidistribution of the parity of rank.
The following example gives a case where the parity of rank is not
equidistributed, and shows that some version of our assumption
\eqref{equation:multiplicative-hypothesis}
is necessary.

\begin{remark}
\label{remark:katz-legendre-so}
Some version of the assumption \eqref{equation:multiplicative-hypothesis} in \autoref{theorem:main-finite-field}
is necessary.
Indeed, without \eqref{equation:multiplicative-hypothesis}, it is possible that every quadratic twist
corresponding to a point of
$\qtwist n U {\mathbb F_q}(\mathbb F_{q^j})$ has Selmer rank of a fixed parity. Hence, quadratic twists of such a
curve do not satisfy the minimalist conjecture.
A specific example is given by the elliptic curve $y^2 =
\lambda(\lambda-1)x(x-1)(x-\lambda)$,  over $\mathbb F_q(\lambda)$, where $q$ is
a prime which is $1 \bmod 4$.
This is a variant of the Legendre family.
Indeed, in \cite[8.6.7]{katz:twisted-l-functions-and-monodromy},
it is shown the relevant arithmetic monodromy group we define in
\autoref{definition:monodromy} is contained in the special orthogonal group.
(We can also see the geometric monodromy is contained in the special orthogonal group using
the methods of this paper, since one can use
\autoref{lemma:action-on-rank-cover} to show all generators of the
fundamental group of configuration space map to the special orthogonal group.)
In this case, the proof of \autoref{lemma:image-of-rank-cover-has-selmer-parity}
shows that for all but finitely many primes $\ell$, the $\ell$-Selmer group of
every quadratic twist unramified over the places of bad reduction has
even Selmer rank.
Note here that assumption \eqref{equation:multiplicative-hypothesis} of \autoref{theorem:main-finite-field} is not
satisfied as each of the three places of bad reduction of the elliptic curve
$y^2 = \lambda(\lambda-1)x(x-1)(x-\lambda)$,
given by $\lambda = 0,\lambda=1,$ and $\lambda=\infty$, has 
additive reduction.
For some further related examples, also see
\cite{rizzo:on-the-variation,
	rizzo:average-root-numbers,
	rizzo:average-root-numbers-nonconstant}.
\end{remark}


\begin{remark}
	\label{remark:}
	Under the assumptions of \autoref{theorem:main-finite-field}, the parity of
	the rank of Selmer groups in the quadratic twist families we consider
	is equidistributed.
	The proportion of the time the rank takes a given parity in the number field setting has been the object
	of much study, see for example
	\cite[Conjecture 7.12]{klagsburnMR:disparity-in-selmer-ranks}.
	We believe it would be quite interesting to understand better understand the
	relation between the number field and function field perspectives on
	this question.
\end{remark}

In the example considered in \autoref{remark:katz-legendre-so}, for sufficiently large $q$, the
proportion of quadratic twists with Selmer rank $\geq 2$ becomes arbitrarily close to
$0$.
We wonder whether this continues to hold even in the absence of
\eqref{equation:multiplicative-hypothesis}:
\begin{question}
\label{question:}
Suppose $A$ is any abelian scheme over $U$, for $U$ an affine curve over
$\mathbb F_{q}$. What conditions do we need on $A$ so that the proportion of quadratic twists of $A
\times_{\spec \mathbb F_q} \mathbb F_{q^j}$
with (Selmer) rank $\geq 2$ tend to $0$ as $j$ grows, even in the absence of
\eqref{equation:multiplicative-hypothesis}?
\end{question}
We conjecture that an irreducibility
condition on the Galois representation associated to $A$ will suffice.
More specifically make the following
conjecture, many cases of which are suggested by \autoref{theorem:main-minimalist}.
We say a quadratic twist is unramified at a real place if the corresponding
double cover has two real places over that real place, and is 
ramified at a real place if the double has a complex
place over that real place.
\begin{conjecture}
	\label{conjecture:twist-proportion}
Let $K$ be any global field of characteristic not $2$ and $A$ any abelian
variety of dimension $r$ over $K$.
\begin{equation}
	\label{equation:big-monodromy-assumption}
\begin{aligned}
	\text{
		Suppose that for some prime $\ell$, $\ell \neq \chr(K)$,
	the
	identity component of the Zariski}
	\\
	\text{closure of
		$\on{im}(\on{Gal}(\overline K/K) \to \on{GL}(H^1(A_{\overline K},
		\overline{\mathbb Q}_\ell(1))))$ acts
	irreducibly on $H^1(A_{\overline K}, \overline{\mathbb Q}_\ell(1))$.}
\end{aligned}
\end{equation}
Specify divisors $D_{\on{unram}}, D_{\on{ram}}$ whose union contains all places of bad
reduction of $A$ and all real places.
The set of quadratic twists of $A$ unramified over $D_{\on{unram}}$ and ramified over
$D_{\on{ram}}$
have ranks distributed according to one of the following three possibilities:
\begin{enumerate}
	\item $0\%$ rank $>1$, $50\%$ rank $0$, $50\%$ rank $1$,
	\item $0\%$ rank $>1$, $100\%$ rank $0$, $0\%$ rank $1$,
	\item $0\%$ rank $>1$, $0\%$ rank $0$, $100\%$ rank $1$.
\end{enumerate}
\end{conjecture}
We next explain some of our motivation for the above conjecture, especially the
hypothesis
\eqref{equation:big-monodromy-assumption}.

\begin{remark}
	\label{remark:}
	Note that some sort of assumption of the flavor 
	of \eqref{equation:big-monodromy-assumption}
	is necessary in
	\autoref{conjecture:twist-proportion}, since if $A = E^r$, for $r > 1$ and
$E$ a generic elliptic curve, we would expect the rank to be $0$ half the
time and $r$ half the time.

The reason that we believe \eqref{equation:big-monodromy-assumption}
should be sufficient comes from the
big monodromy result of Katz,
\cite[Proposition 5.4.3]{katz:twisted-l-functions-and-monodromy}.
This essentially says that if, in the function field setting,
for $A$ an abelian scheme over $U$ and geometric point $\spec \overline{\mathbb
F}_q \simeq \overline{x} \to U_{\overline {\mathbb F}_q}$,
$H^1(A_{\overline{\mathbb
F}_q} \times_{U_{\overline{\mathbb
F}_q}}  {\overline x}, \overline{\mathbb Q}_\ell(1))$
corresponds to an irreducible representation 
of $\pi_1(U_{\overline{\mathbb F}_q}, \overline x)$
for some $\ell \neq \chr(K)$,
a certain related monodromy group should be big, i.e., contain the special
orthogonal group. 
It seems to us this should imply that the 
geometric
monodromy representation considered in \autoref{definition:monodromy} for $\nu =
\ell$ has index
at most $4$ in the orthogonal group $\bmod \ell$. 
We conjecture that in this case the BKLPR
conjectures hold, with the possible caveat that the rank may have a fixed parity
if the monodromy group is contained in the special orthogonal group.
It is not immediately clear how to best generalize the condition 
that 
$H^1(A_{\overline{\mathbb
F}_q} \times_{U_{\overline{\mathbb
F}_q}}  {\overline x}, \overline{\mathbb Q}_\ell(1))$
is irreducible
to the number field setting, but it seems that 
\eqref{equation:big-monodromy-assumption}
should imply it, and so
\eqref{equation:big-monodromy-assumption}
seems a reasonable sufficient criterion.
\end{remark}

\begin{remark}
	\label{remark:ramified-quadratic-twists}
	Throughout this paper, we work with the space of quadratic twists
	parameterizing double covers whose ramification locus does not
	intersect
	the discriminant locus. 
	As a variant, we could work
	with the space of finite double covers whose ramification locus contains a specified divisor $R$
	(where $R$ may intersect the discriminant locus) but the ramification locus of
	the cover does not meet the discriminant locus outside of $R$.

	Assuming there is a place of multiplicative reduction with toric part of
	codimension $1$ outside of $R$,
	and replacing the space of quadratic twists in our main theorems with
	the above 
	variant, we believe the conclusions of
	\autoref{theorem:main-finite-field},
	\autoref{theorem:main-moments},
	and
	\autoref{theorem:main-minimalist}
	should still hold.

	In fact, we believe one can make a more precise version of \autoref{conjecture:twist-proportion}
	that predicts which of the three cases we are in based on local data
	associated to the abelian variety, similarly to the case of elliptic
	curves which is closely related to \cite[Proposition
	7.9]{klagsburnMR:disparity-in-selmer-ranks}. 
	We believe this generalization
	would lead to a version of
	\cite[Conjecture 7.12]{klagsburnMR:disparity-in-selmer-ranks} for global
	arbitrary fields.

	It would be quite interesting to work the above claims out precisely.
\end{remark}

\subsection{Discussion on the presence of limsup and liminf}
We conclude our remarks with comments pertaining to the presence of the $\limsup$ and
$\liminf$.

\begin{remark}
\label{remark:}
Previously, it was not even known that the $\limsup$ and $\liminf$ appearing in
\eqref{equation:moment-limsup} of \autoref{theorem:main-moments} even existed,
nor that the limit in $n$ appearing in \eqref{equation:moment-limit-residue}
existed, let alone what their limiting
value as $j \to \infty$ was. The fact that these limits exist is an important part of these
theorems.
We also note that if one only cares about verifying the existence of
the $\limsup$ and $\liminf$,
without computing the value after taking a further
limit in $j$, 
one does not need the full force of
our big monodromy results culminating in
\autoref{proposition:geometric-moments}, which are what enables us to compute these values precisely.
Instead, one may use
\autoref{theorem:evw-stability} and
	\autoref{lemma:finite-generation-mp} to obtain an ineffective bound on
	the relevant number of irreducible components.
\end{remark}


\begin{remark}
The reason we cannot propagate this existence of the limit in $n$ of
\eqref{equation:moment-limit-residue} to our other main results such
as \autoref{theorem:main-finite-field} (which only has a $\limsup$ and a
$\liminf$) is that we do not know how to rule out
the possibility that the moments grow too quickly to determine a distribution
for any fixed value of $q$.
%

Even more ambitiously, one might want to know what these limits in $n$ actually are, and in
particular whether they agree with the BKLPR heuristics.  For this, one would
likely want to know not only that the \'etale cohomology groups stabilize as
Frobenius modules up to Tate twist, but what Frobenius module they stabilize to. 
For the moment, this appears to be a substantially harder problem.
See also 
\autoref{remark:mapping-class-group-action-on-cohomology} and
\autoref{remark:0-cohomology}.
\end{remark}

\subsection{Past work}
\label{subsection:past-work}

As mentioned above, two guiding sets of conjectures in number theory are the
Cohen-Lenstra heuristics and the BKLPR heuristics.
Focusing on the latter over number fields, very little is known.
Over $\mathbb Q$, 
work by \cite{Heath:theSizeOf-i,Heath:theSizeOf-ii,swinnerton-dyer:the-effect-of-twisting,kane:on-the-ranks-of-the-2-selmer-groups-of-twists}
led to a determination of the distribution of $2$-Selmer groups in quadratic
twist families of elliptic curves. Building on this, Smith described the
distribution of
$2^\infty$-Selmer groups of elliptic curves over $\mathbb Q$ in
\cite[Theorem 1.5]{smith:the-distribution-of-selmer-groups-1}.  
Smith is able to use this to deduce the minimalist conjecture in many
quadratic twist families over
$\mathbb Q$ \cite[Theorem
1.2]{smith:the-distribution-of-selmer-groups-1}. The reason for this
deduction is that Smith's work, like ours, but unlike the previous papers cited in this paragraph, provides distributional information about $\nu$-Selmer groups with $\nu$ arbitrarily large. 
These results for quadratic twist families over number fields nearly exclusively deal with
$2$-power Selmer groups.  Our results are in some sense disjoint, applying only to 
$\nu$-Selmer groups for $\nu$ odd.

There is also some work toward understanding $3$-isogeny Selmer groups in
quadratic twist families (\cite{bhargavaKLS:3-isogeny}.)
However, the above results are only for $3$-Selmer groups, and only when the
pertinent curves possess unexpected isogenies.
As far as we are aware, our work provides the first results toward describing the distribution of odd
order Selmer groups in quadratic twist families when there are no unexpected
isogenies.

There is also a growing literature about variation of Selmer groups in the universal family parameterizing all elliptic curves.  For this family, Bhargava and Shankar
computed the average size of the $\nu$-Selmer group for $\nu \leq 5$
\cite{bhargava-shankar:binary-quartic-forms-having-bounded-invariants, bhargavaS:ternary, bhargavaS:average-4-selmer, bhargavaS:average-5-selmer},
and Bhargava-Shankar-Swaminathan computed the second moment of $2$-Selmer
groups \cite{bhargavaSS:the-second-moment}.

Over function fields, much more is known if one permits taking a limit in
the finite field order $q$ {\em before} any limit in log-height is taken. (Here, the log-height of a quadratic twist refers to the degree of its ramification locus.)
In the context of the Cohen-Lenstra heuristics, \cite{Achter:cohenQuadratic}
established a large $q$ limit version of the Cohen-Lenstra heuristics, where he
took a $q$ limit before letting the log-height grow.

In the context of the BKLPR heuristics, some results were also known when one
takes a large $q$ limit prior to large log-height limit:
The average size of certain Selmer
groups in quadratic twist families were computed in
\cite{parkW:average-selmer-rank-in-quadratic-twist-families}.
In the context of the universal family,
\cite{landesman:geometric-average-selmer} computed the average size of Selmer
groups, and the full BKLPR distribution was computed in
\cite{fengLR:geometric-distribution-of-selmer-groups}.

Closer to the present work are results in which one takes a limit in log-height first, with $q$ fixed, and only then lets $q$ increase.
De Jong \cite{de-jong:counting-elliptic-surfaces-over-finite-fields}
computed the average size of $3$-Selmer groups over $\mathbb F_q(t)$ in the
universal family of elliptic curves.
H{\`{\^o}}, L\^e H\`ung, and Ng{\^o}
\cite{ho-lehung-ngo:average-size-of-2-selmber-groups}
compute the average size of $2$-Selmer groups over function fields for the
universal family, while
\cite{achenjang:the-average-size-of-2-selmer}
carries out a similar program in all characteristics, including characteristic
$2$.
We note that these results both have the same flavor as our main results, in that they only arrive at 
the predicted value after first taking a large log-height limit, and then taking a
large $q$ limit.
Another more recent result of Thorne
\cite{thorne:on-the-average-number-of-2-selmer}
calculates the average size of $2$-Selmer groups in a family of elliptic curves
with $2$ marked points over genus $0$ function fields, and, interestingly, this result does not require taking
a large $q$ limit at the end.
We also note that 
\cite[Theorem 2.2.5]{ho-lehung-ngo:average-size-of-2-selmber-groups}
does not require taking a large $q$ limit if one restricts to elliptic curves
with squarefree discriminant.

Since the work \cite{EllenbergVW:cohenLenstra} proved a homological
stability result for Hurwitz stacks, there has also been further activity in
this topological direction.
The homological stability results of 
\cite{EllenbergVW:cohenLenstra} have been employed in a number of arithmetic papers, such
as in
\cite{lipnowskiST:cohen-lenstra-roots-of-unity},
\cite{lipnowskyT:cohen-lenstra-for-etale-group-schemes}, and
\cite{ellenbergLS:nonvanishing-of-hyperelliptic-zeta-functions}.
However, few papers have further developed the
homological stability techniques.
Some notable examples where these techniques were developed further include
\cite{ellenbergTW:fox-neuwirth-fuks-cells},
proving a version of Malle's conjecture,
a polynomial version of homological stability in
\cite{bianchiM:polynomial-stability},
a verification that stability in \cite{EllenbergVW:cohenLenstra}
holds with period $1$ instead of with period $\deg U$ in 
\cite{davisS:the-hilbert-polynomial-of-quandles},
and a bound on the ranks of homology groups for Hurwitz spaces
associated to punctured genus $0$ surfaces in
\cite{hoang:fox-neuwirth-cells-resultant}.
Finally, \cite{bergstromDPW:hyperelliptic-curves-the-scanning-map} and
\cite{millerPPR:uniform-twisted-homological-stability}
used
homological stability techniques to approach a 
conjecture on moments of quadratic L-functions, and were able to not only show
the relevant cohomology groups stabilize, but even compute their limiting values.

\subsection{Outline}

The structure of the paper is as follows.
We suggest the reader consult \autoref{figure:proof-schematic} for a schematic
depiction of the main ingredients in the proof.
In \autoref{section:background}
we review background on orthogonal groups, the BKLPR heuristics, and Hurwitz stacks.
Next, we continue to the topological part of our paper.
In \autoref{section:arc-complex}, we set up a general notion of coefficient systems (which include 
Hurwitz stacks over the complex numbers as a special case)
to which the arc complex spectral sequence applies.
This is the context in which we  prove our main homological stability results in
\autoref{section:homological-stability}.
We next continue to the more algebraic part of the paper, beginning with
\autoref{section:selmer-space},
where we construct Selmer stacks which parameterize Selmer elements on quadratic
twists of our abelian scheme.
In \autoref{section:hurwitz-and-selmer}, we show that the above constructed
Selmer stacks can be identified with Hurwitz stacks over the complex numbers.
In order to compute the $0$th homology of these spaces, we prove a big monodromy
result in 
\autoref{section:big-monodromy}.
We verify our homological stability results apply to these Selmer stacks, as
well as to certain double covers, which control the parity of the $\ell^\infty$-Selmer rank of the
quadratic twists of our abelian scheme, in 
\autoref{section:rank-cover}.
Having controlled the cohomologies of the spaces we care about, we conclude our
main results by combining the above with some slightly more analytic
computations.
In \autoref{section:moments}, we compute the moments related to Selmer stacks,
as well as fiber products of these with the above mentioned double cover.
In \autoref{section:determining-distribution}, we show these moments determine the
distribution, obtaining our main result, \autoref{theorem:main-finite-field}.
In 
\autoref{section:frobenius-equivariance}, we use logarithmic
geometry to prove that the
stabilization maps on cohomology are equivariant for the action of Frobenius, up
to twist, which allows us to show that a limit as $n \to \infty$ exists in
\eqref{equation:moment-limit-residue}, instead of only knowing that the $\liminf$ and
$\limsup$ exist as in \eqref{equation:moment-limsup}.
Finally, in \autoref{section:nc}, we use logarithmic geometry to prove that configuration spaces and Hurwitz spaces
have normal crossings compactifications. This is a crucial ingredient for us to
be able to transfer cohomology between the complex numbers and finite fields.

\begin{figure}
\large
	\centering
\adjustbox{scale=.75,center}{
\tikz[overlay]{
	\filldraw[fill=yellow!50,draw=red!50!yellow] (0,-4.5) rectangle
	(12.2,-.5);
\node[draw] at (2,-1) {Topology};
	\filldraw[fill=green!50,draw=red!50!yellow] (0,1.5) rectangle (16,4.5);
	\node[draw] at (2,3.8) {Big monodromy};
	\filldraw[fill=blue!50,draw=red!50!yellow] (16.5,1.5) rectangle (23,3.5);
	\node[draw] at (21,2.5) {Probability};
	\filldraw[fill=red!50,draw=red!50!yellow] (12.6,-.9) rectangle
	(19.2,-3);
	\node[draw] at (15,-1.5) {Logarithmic Geometry};
	\filldraw[fill=orange!50,draw=red!50!yellow] (13.1,-3.2) rectangle
	(22,-4.5);
	\node[draw] at (19,-3.9) {Combinatorial Group Theory};
}

   \begin{tikzcd}[column sep = 2.3em]
	& \text{Lem.}~\ref{lemma:selmer-identification} \ar{r} & 
  \text{Lem.}~\ref{lemma:determinant-monodromy} \ar{rd} &
  \text{Lem.}~\ref{lemma:kernel-and-bklpr-distributions} \ar{d} &&& \\
  \text{Prop.}~\ref{proposition:rank-description}
  \ar{r} &
  \underbrace{\text{Thm.}~\ref{theorem:big-monodromy-mod-ell}}_{\text{\cite[Middle
  convolution]{hall:bigMonodromySympletic,katz:twisted-l-functions-and-monodromy}}}
  \ar{ur} \ar{r} &
  \text{Prop.}~\ref{proposition:big-monodromy-composite} \ar{r} &
  \text{Prop.}~\ref{proposition:geometric-moments} \ar{rd} &
  \text{Prop.}~\ref{proposition:general-moments-to-distribution}\ar{rd}&&
 \text{Lem.}~\ref{lemma:lifting-to-char-0} \ar{d}
  \\
	\qquad & \text{Lem.}~\ref{lemma:torsor-description} \ar{r} &
	\text{Cor.}~\ref{corollary:selmer-to-hurwitz-moments} \ar{r} &
	\text{Lem.}~\ref{lemma:moments-cohomology} \ar{r}
	&
	\text{Thm.}~\ref{theorem:point-counting-computation} \ar{r}\ar{rrd} &
	\text{Thm.}~\ref{theorem:main-dvr} \ar{r} &
	\text{Thm.}~\ref{theorem:main-finite-field} \\
\qquad & \text{Thm.}~\ref{theorem:evw-stability} \ar{r}\ar{rdd} &
\text{Cor.}~\ref{corollary:cohomology-bound}\ar[swap]{ur}{\text{\cite{EllenbergVW:cohenLenstra}}}
	&
	&
	&&
	\text{Thm.}~\ref{theorem:main-moments}
  \\
	\text{Prop.}~\ref{proposition:arc-complex-spectral-sequence}\ar{r} &
	\text{Thm.}~\ref{theorem:1-controlled-bound} \ar{ur}\ar{dr} & 
	&
	\text{Lem.}~\ref{lemma:factor-through-trivial-log} \ar{r}
	&
	\text{Thm.}~\ref{theorem:frobenius-equivariance} \ar{uu}
	& \\
	\qquad & \text{Lem.}~\ref{lemma:degree-bound-for-km0} \ar{uur} \ar{r} &
	\text{Thm.}~\ref{theorem:central-u-implies-cohomology-bound} \ar{urr}
	&
	\text{Prop.}~\ref{proposition:degree-order} \ar{ur}
	&
	&&
 	  \end{tikzcd}
}
\caption{
A diagram depicting the structure of the proof of
the main result,
\autoref{theorem:main-finite-field}.}
\label{figure:proof-schematic}
\end{figure}
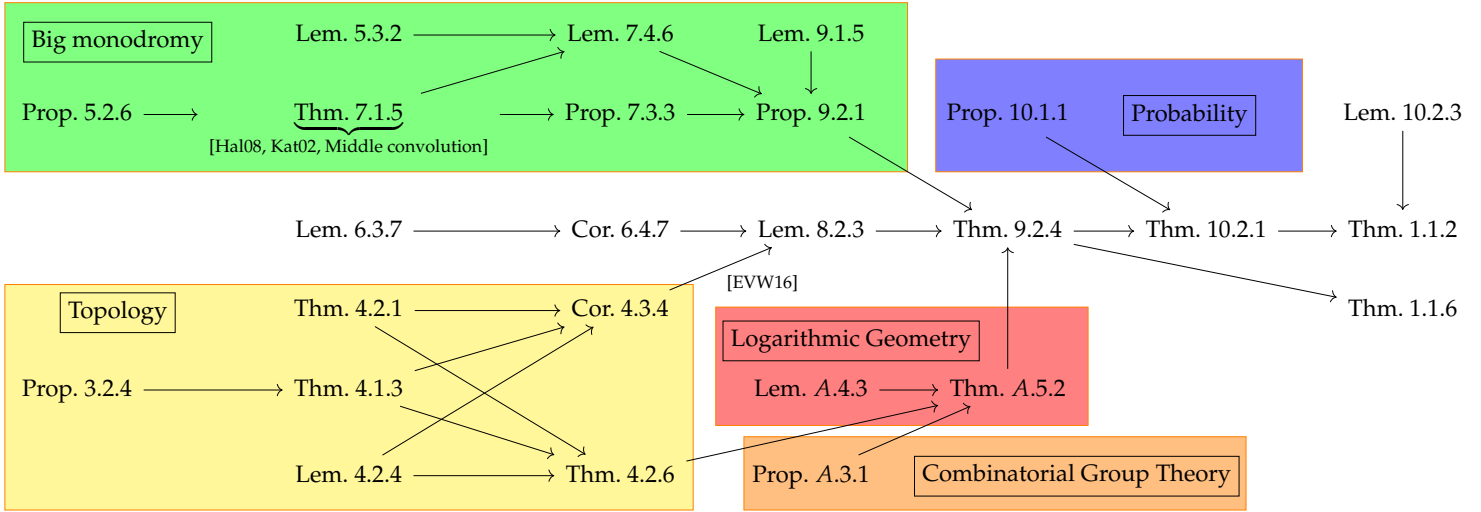

\subsection{Notation}
For the reader's convenience, in \autoref{table:notation}
we collect some notation introduced throughout the paper.
\begin{figure}
	\centering
\scalebox{0.7}{
	\begin{tabular}{|l|l|l|}
		\hline
		Notation & Description & Location defined \\\hline
		$\nu$ & Odd integer indexing the Selmer group $\sel_\nu(A)$ &
		\autoref{notation:omega}  \\ \hline
		$D_Q$ & The Dickson invariant map associated to a quadratic form
		$Q$ & \autoref{notation:omega}  \\ \hline
		$\Omega(Q)$ & The subgroup of the special orthogonal group which is the kernel of
		the spinor norm & \autoref{notation:omega} \\ \hline
		$\bklpr \nu$ & The BKLPR predicted distribution of $\nu$-Selmer
		groups & \autoref{definition:bklpr-rank-selmer}  \\ \hline
		$B$ & Base scheme & \autoref{notation:curve-notation}  \\ \hline
		$C$ & Smooth proper curve over $B$ & \autoref{notation:curve-notation}  \\ \hline
		$Z$ & Divisor in $C$ of degree $f+1$ which twists are unramified
		along & \autoref{notation:curve-notation}  \\ \hline
		$\conf n U B$ & Configuration space of degree $n$ divisors in
		$U$  & \autoref{notation:curve-notation}  \\ \hline
		$\hur G n Z S C B$ & Hurwitz space parameterizing $G$-covers
	with monodromy in $S$ & \autoref{definition:fixed-hur} \\ \hline
			$\Sigma^b_{g,f}$ & Topological surface of genus $g$ with
			$b$ boundary components and $f$ punctures &
			\autoref{notation:surface-braid-groups}  \\ \hline
			$X^{\oplus n} \oplus A_{g,f}$ & $n$ copies of a marked
			cylinder glued onto $\Sigma^1_{g,f}$ &
			\autoref{notation:surface-braid-groups}  \\ \hline
			$B^n_{g,f}$ & The surface braid group $\pi_1(\Conf^n_{X^{\oplus n} \oplus
			A_{g,f}})$ &
			\autoref{notation:surface-braid-groups}  \\ \hline
		$H_{\inertiaindex n G c g f}$.
 & The $n$th vector space from a coefficient system corresponding to a Hurwitz space &  
 \autoref{example:hurwitz-coefficient-system}
 \\ \hline
		$R^V$ & Ring of connected components associated to a coefficient
		system over $\Sigma^1_{0,0}$ & \autoref{definition:k-complex} \\ \hline
		$\mathcal K(M)$ & $K$-complex associated a graded $R^V$-module
		& \autoref{definition:k-complex}  \\ \hline
		$M^{V,F}_p$ & $\oplus_{n \geq 0} H_p(B^n_{g,f}, F_n)$ &
		\autoref{notation:homology-coefficient-system} \\ \hline
		$\mathscr F$ & A tame, symplectically self-dual lcc sheaf of free
		$\mathbb Z/\nu \mathbb Z$-modules on $U$&
		\autoref{notation:quadratic-twist-notation}  \\ \hline
		$\qtwist n U B$ & A Hurwitz space parameterizing quadratic
		twists & \autoref{notation:quadratic-twist-notation}  \\ \hline
		$\mathscr F^n_B$ & The universal degree $n$ quadratic twist of $\mathscr F$ & \autoref{notation:quadratic-twist-notation} \\ \hline
		$\selsheaf{\mathscr F^n_B}$ & The Selmer sheaf, which
		parameterizes torsors for quadratic twists of $\mathscr F$ of log-height
		$n$ & \autoref{definition:selmer-sheaf}  \\ \hline
		$\selspace {\mathscr F^n_B}$ & The Selmer stack, which is the
		finite \'etale cover of $\qtwist n U B$ corresponding to
		the Selmer sheaf & \autoref{definition:selmer-sheaf}  \\ \hline
		$C_x, U_x, \mathscr F_x, A_x$ & The fiber of the relevant object over $x
		\in \qtwist n U B$ & \autoref{notation:x-points} \\ \hline
		$\Phi_{A'}$ & Component group of the abelian scheme $A'$ over
		$U'$ & \autoref{notation:component-group} \\ \hline
		$\drop_x(\mathscr F)$ & Corank of the invariants of the inertia
		of $\mathscr F$ at $x$ & \autoref{definition:drop} \\ \hline
		$\asp_{2r}$ & The affine symplectic group &
		\autoref{definition:asp} \\ \hline
		$\ahsp_{2r}$ & The $H$ moment of the affine symplectic group &
	\autoref{notation:asp-generalization} \\ \hline
	$\selhur H {\mathscr F^n_B}$ & A certain Hurwitz space which is geometrically
	isomorphic to the Selmer
	sheaf & \autoref{notation:sel-hur} \\ \hline
	$V_{\mathscr F^n_B}$ & The vector space corresponding to a geometric
	fiber of the Selmer sheaf & \autoref{definition:monodromy} \\ \hline
	$\rho_{\mathscr F^n_B}$ & The monodromy representation associated to the
	Selmer sheaf & \autoref{definition:monodromy}  \\ \hline
	$X_{A[\nu]^n_{\mathbb F_q}}$ & Probability distribution on $\nu$-Selmer groups
	of quadratic twists of $A$ over $\mathbb F_q$ &
	\autoref{definition:actual-selmer-distribution} \\ \hline
	$X^i_{A[\nu]^n_{\mathbb F_q}}$ & Probability distribution on $\nu$-Selmer groups
	of quadratic twists of $A$ over $\mathbb F_q$ with fixed parity of rank &
	\autoref{definition:actual-selmer-distribution} \\ \hline
	$\sinertiaindex n {\mathscr F}{H} g f$ & Coefficient system associated to
	$H$-moment of the $\nu$-Selmer sheaf &
	\autoref{example:specified-hurwitz-coefficient-system} \\ \hline
	$\sinertiaindex n {\underline{\mathscr F}}{H} 0 0$ & The coefficient
	system for $\Sigma^1_{0,0}$ which $\sinertiaindex n {\mathscr F}{H} g f$
	lies over &
	\autoref{example:specified-hurwitz-coefficient-system} \\ \hline
	$S^{n,\on{rk}}_{\mathscr F, g, f}$ & The coefficient
	system associated to the rank double cover
	&
	\autoref{example:rank-coefficient-system} \\ \hline
		$\mathcal N$ & Finite abelian $\mathbb Z/\nu \mathbb Z$-modules & \autoref{definition:actual-selmer-distribution} \\ \hline
	$\mathcal N^i$ & The subset of objects of $N$ of the form $(\mathbb Z/\nu \mathbb Z)^i \times G^2$ & \autoref{definition:actual-selmer-distribution} \\ \hline
\end{tabular}
}
	\vspace{.5cm}
	\caption{Some notation introduced in the paper.}
	\label{table:notation}
\end{figure}

\subsection{Acknowledgements}
We thank Craig Westerland for numerous helpful and detailed discussions which
were invaluable in pinning down some of the trickiest
topological inputs to this paper.
We also thank Dori Bejleri for suggesting the idea to prove the stabilization
maps respect the Frobenius action, for extensive help with technical aspects of log
geometry.
We thank several anonymous referees for numerous detailed and helpful comments.
We thank Chris Hall for a meticulously close reading and numerous
helpful discussions.
Thanks to Eric Rains for many helpful exchanges, especially relating to the
BKLPR heuristics and Vasiu's lifting results.
We thank Melanie Wood for a number of useful conversations relating to
determining the distribution from the moments.
Thanks additionally to Levent Alp\"{o}ge and Bjorn Poonen for help understanding the
possible structures of the Tate-Shafarevich group.
We also thank Sun Woo Park for a close reading and for numerous
detailed and helpful comments.
We further thank
Dan Abramovich,
Niven Achenjang,
Andrea Bianchi,
Kevin Chang,
Qile Chen,
Chantal David,
Tony Feng,
Jeremy Hahn,
David Harbater,
Anh Trong Nam Hoang,
Hyun Jong Kim,
Ben Knudsen,
Michael Kural,
Jef Laga,
Peter Landesman,
Eric Larson,
Robert Lemke Oliver,
Ishan Levy,
Siyan Daniel Li-Huerta,
Daniel Litt,
Zhao Yu Ma,
Davesh Maulik,
Barry Mazur,
Jeremy Miller,
Samouil Molcho,
Martin Olsson,
Dan Petersen,
Andy Putman,
Oscar Randal-Williams,
Zev Rosengarten,
Will Sawin,
Mark Shusterman,
Alex Smith,
Salim Tayou,
Ravi Vakil,
and
David Yang.
This work also owes a large intellectual debt to a number of others including
work of Chris Hall, work of Nick Katz,
and work of Oscar Randal-Williams and Nathalie Wahl.  
Open AI GPT-5.1 was used to help find typos in the manuscript.
JE was supported by the National Science
Foundation 
under Award No.
DMS 2301386, and AL was supported by the National Science
Foundation 
under Award No.
DMS 2102955
and award No. 2449797.

\section{Background}
\label{section:background}

We now review some background on orthogonal groups in
\autoref{subsection:orthogonal-groups}, background on the BKLPR heuristics in
\autoref{subsection:bklpr-review}, and background on Hurwitz stacks in \autoref{subsection:hurwitz-stacks-background}.
The one new part of this section is \autoref{subsection:computing-moments},
where we compute the moments of the BKLPR distribution.

\subsection{Orthogonal groups}
\label{subsection:orthogonal-groups}

We now define some notation we will use relating to orthogonal groups.
Throughout, we will be working over base rings $R$ with $2$ invertible on $R$,
and so we will freely pass between quadratic spaces and spaces with a
bilinear pairing. For some additional detail and further references, we refer
the reader to \cite[\S3.2]{fengLR:geometric-distribution-of-selmer-groups}
whose material in turn was largely drawn from \cite[Appendix
C]{conrad:reductive-group-schemes}.
\begin{notation}
	\label{notation:omega}
	Let $R= \mathbb Z/\nu \mathbb Z$, for some $\nu$ with $\gcd(\nu,2) = 1$.
	Let $V$ be a free
	$R$-module of rank at least $3$ with a bilinear pairing $B: V \times V \to R$.
	Let $Q: V \to R,$ defined by $Q(v) := B(v,v)$ denote the associated quadratic form.
	We assume throughout that $Q$ is {\em nondegenerate}, meaning that the
	quadric associated to $Q$ is smooth, or equivalently $Q$ is nondegenerate
	modulo every prime $\ell \mid \nu$.
	We let $\o(Q)$ denote the associated orthogonal group preserving $Q$.
	There is a Dickson invariant map $D_Q: \o(Q) \to \prod_{\ell \mid \nu \text{
	prime}} \mathbb Z/2\mathbb Z$ by sending an element to $0$ in coordinate $\ell$ if its
	determinant $\bmod \ell$ is $1$ and sending it to $1$ if its determinant
	$\bmod \ell$ is $-1$. 
	There is also a {\em $+1$-spinor norm map} $\on{sp}_Q^+: \o(Q) \to
	H^1(\on{Spec} R, \mu_2) \simeq R^\times/(R^\times)^2 \simeq \prod_{\ell
		\mid \nu \text{
	prime}} \mathbb Z/2\mathbb Z$,
	where the map in cohomology is induced by the boundary map associated to
	the exact sequence of algebraic groups $\mu_2 \to \on{Pin}(Q) \to
	\o(Q)$.
	The $-1$-spinor norm, $\on{sp}_Q^-: \o(Q) \to
	\prod_{\ell \mid \nu \text{
	prime}} \mathbb Z/2\mathbb Z$,
	is the composition of $\on{sp}_Q^+$ with the identification $\o(Q)
	\simeq \o(-Q)$, see 
	\cite[Remark C.4.9, Remark C.5.4, and
	p.348]{conrad:reductive-group-schemes}.
	In particular, if $r_v$ is the reflection about the vector $v$,
	$\on{sp}_Q^-(r_v) = [-Q(v)]$, where $[x]$ denotes the square class of
	$x$, viewed as an element of $\prod_{\ell \mid \nu \text{
	prime}} \mathbb Z/2\mathbb Z$.

	We define $\Omega(Q) := \ker D_Q \cap \ker \on{sp}_Q^- \subset \o(Q)$.
	In particular, since $\nu$ is odd, $\Omega(Q) \subset \o(Q)$ has index $4^{\omega(\nu)}$,
	where $\omega(\nu)$ denotes
	the number of primes dividing $\nu$.
\end{notation}
\begin{remark}
	\label{remark:}
	It turns out that the map $D_Q \times \on{sp}_Q^-: \o(Q) \to \prod_{\ell
		\mid \nu \text{
	prime}} (\mathbb Z/2 \mathbb Z \times \mathbb Z/2\mathbb Z)$ can be identified with the
	abelianization of $\o(Q)$, assuming $Q$ is nondegenerate and has rank more than $2$.
\end{remark}

The following lemma will be useful throughout the paper, and connects the Dickson invariant to the dimension of the
$1$-eigenspace of an element of the orthogonal group. 
We will see that the latter is related
to Selmer groups via \autoref{lemma:selmer-identification}.

\begin{lemma}
	\label{lemma:1-eigenvalue-parity}
	Let $(V, Q)$ be a quadratic space over a field and $g \in \o(Q)$.
	We have
	\begin{align*}
	\dim \ker(g -\id) \bmod 2 \equiv \rk V - D_Q(g).
	\end{align*}
\end{lemma}
\begin{proof}
	Note that since we are working over a field, the Dickson invariant
	$D_Q(g)$ is valued in $\mathbb Z/2 \mathbb Z$, and not a product of
	copies of $\mathbb Z/2 \mathbb Z$.
	It follows from 
	\cite[p. 160]{taylor:the-geometry-of-the-classical-groups}, 
	that $\dim \im (g - \id) = D_Q(g) \bmod 2$.
	We find
		\begin{equation}
	\begin{aligned}
		\label{equation:kernel-to-determinant}
	\dim \ker(g -\id) \bmod 2 &\equiv \rk V -
	\dim \im (g - \id) \bmod 2 \\
	&\equiv\rk V- D_Q(g) \bmod 2,
	\end{aligned}
\end{equation}
using the exact sequence relating the kernel and image of
	$g - \id : V \to V$.
\end{proof}

\subsection{Review of the BKLPR distribution}
\label{subsection:bklpr-review}

We now give a quick review of the predicted distribution for $\nu$-Selmer groups
given in \cite{bhargavaKLPR:modeling-the-distribution-of-ranks-selmer-groups}.
We also suggest the reader consult
\cite[\S5.3]{fengLR:geometric-distribution-of-selmer-groups} for a slightly more
detailed description of this distribution, geared to the context in which we will use it
in this paper.

\subsubsection{The $\ell^\infty$-Selmer distribution from BKLPR conditioned on rank}
\label{subsubsection:conditioned-bklpr}

Let $\ell$ be a prime. For non-negative integers $m,r$ with $m - r \in 2 \Z_{\geq 0}$, let $A$ be drawn
randomly from the Haar probability measure on the set of \emph{alternating} $m
\times m$-matrices over $\Z_{\ell}$ having rank $m-r$. 
Let
$\mathscr{T}_{m,r,\ell}$ be the distribution of $(\coker A)_{\tors}$, the
torsion in $\coker A$. According
to \cite[Theorem
1.10]{bhargavaKLPR:modeling-the-distribution-of-ranks-selmer-groups}, as $m
\rightarrow \infty$ through integers with $m-r \in 2 \Z_{\geq 0}$, the
distributions $\mathscr{T}_{m,r, \ell}$ converge to a limit $\mathscr{T}_{r,\ell}$. 

\subsubsection{The BKLPR $\nu$-Selmer distribution}
\label{subsubsection:bklpr-n-selmer}
We next review the model for $\nu$-Selmer elements described at the beginning of \cite[\S 5.7]{bhargavaKLPR:modeling-the-distribution-of-ranks-selmer-groups}.
Let $\mathscr T_{r,\ell}$ denote the random variable defined on isomorphism classes of finite abelian $\ell$ groups (notated $\mathscr T_r$ in \cite{bhargavaKLPR:modeling-the-distribution-of-ranks-selmer-groups})
defined in \cite[Theorem 1.6]{bhargavaKLPR:modeling-the-distribution-of-ranks-selmer-groups} and reviewed in \autoref{subsubsection:conditioned-bklpr}.
For $G$ an abelian group, we let $G[\nu]$ denote the $\nu$ torsion of $G$.
For $\nu \in \mathbb Z_{\geq 1}$ with prime factorization $\nu = \prod_{\ell
\mid \nu} \ell^{a_\ell}$, 
define a distribution $\mathscr T_{r,\mathbb Z/\nu\mathbb Z}$ on finitely
generated $\mathbb Z/\nu\mathbb Z$-modules
by choosing a collection of abelian groups $\{ T_\ell\}_{\ell \mid \nu}$, 
with $T_\ell$ drawn from $\mathscr T_{r,\ell}$, and defining the probability
$\mathscr T_{r,\mathbb Z/\nu\mathbb Z} = G$ to be the probability that
$\oplus_{\ell \mid \nu} T_\ell[\nu] \simeq G$.

Given the above predicted distribution for the $\nu$-Selmer group of abelian
varieties of rank $r$,
the heuristic that $50\%$ of abelian varieties have rank $0$ and $50\%$ have rank $1$ leads to
the predicted joint distribution of the $\nu$-Selmer group and rank given in
\autoref{definition:bklpr-rank-selmer}.
We use $\mathscr T_{1, \mathbb Z/\nu\mathbb Z} \oplus \mathbb Z/\nu\mathbb Z$ as notation for
the random variable so that the probability $\mathscr T_{1, \mathbb Z/\nu\mathbb Z} \oplus
\mathbb Z/\nu\mathbb Z \simeq G \oplus \mathbb Z/\nu\mathbb Z$ is equal to the
probability that $\mathscr T_{1, \mathbb Z/\nu\mathbb Z} \simeq G$.

\begin{definition}
	\label{definition:bklpr-rank-selmer}
	Let $\mathcal N$ denote the set of finite $\mathbb Z/\nu \mathbb Z$-modules.
 	Let $\bklpr \nu : \mathcal N \to \mathbb R_{\geq 0}$ denote the probability
distribution defined by
	\begin{align*}
		\bklpr \nu :=	\frac{1}{2} \mathscr T_{0, \mathbb Z/\nu\mathbb Z} +
\frac{1}{2} \left( \mathscr T_{1, \mathbb Z/\nu\mathbb Z} \oplus \mathbb
Z/\nu\mathbb Z \right).
\end{align*}
For $i \in \{0,1\}$ let 
$\paritybklpr \nu i: \mathcal N \to \mathbb R_{\geq 0}$
denote the distribution $\bklpr \nu$ conditioning on $\rk \bklpr \nu \bmod \ell
\equiv i \bmod 2,$
for any $\ell \mid \nu$.
In particular, 
$\paritybklpr \nu 0$ is the distribution 
$\mathscr T_{0, \mathbb Z/\nu\mathbb Z}$
while
$\paritybklpr \nu 1$ is the distribution 
$\mathscr T_{1, \mathbb Z/\nu\mathbb Z} \oplus \mathbb Z/\nu \mathbb Z$.
\end{definition}
\begin{remark}
	\label{remark:}
	Note that $\paritybklpr \nu i$ is independent of $\ell \mid \nu$ 
	as follows from the definition of $\paritybklpr \nu i$,
	\autoref{definition:bklpr-rank-selmer},
	so the definition of $\paritybklpr \nu i$ is independent of the choice
	of $\ell \mid
	\nu$.
\end{remark}
\begin{remark}
	\label{remark:}
	We note that there was a slight error in \cite[Definition
	5.12]{fengLR:geometric-distribution-of-selmer-groups}.
	There, when $r = 1$, the distribution should have been given by
	$\frac{1}{2} \left( \mathscr T_{1, \mathbb Z/\nu\mathbb Z} \oplus \mathbb
Z/\nu\mathbb Z \right)$ and not 
$\frac{1}{2} \left( \mathscr T_{1, \mathbb Z/\nu\mathbb Z} \right)$ as written there. The
latter models $\Sha[\nu]$ as opposed to $\sel_{\nu}$.
\end{remark}

\subsection{Computing the moments of $\nu$-Selmer groups}
\label{subsection:computing-moments}

We next compute moments of the BKLPR distribution.
For a distribution $X$ valued in finite abelian groups, we use the $H$-moment of
$X$ as terminology for the expected number of surjections or homomorphisms 
$X \to H$. Knowing the expected number of homomorphisms for all $H$ is
equivalent to knowing the expected number of surjections for all $H$ by an
inclusion exclusion argument.

The computation of the moments below in the case that $H \simeq (\mathbb Z/\ell^j \mathbb Z)^m$ was explained in
	\cite[Theorem 5.10 and Remark
	5.11]{bhargavaKLPR:modeling-the-distribution-of-ranks-selmer-groups}.
	Surprisingly, the general case appears to be missing from the literature.
	We follow a similar method of proof to
	\cite[Theorem
	5.10]{bhargavaKLPR:modeling-the-distribution-of-ranks-selmer-groups},
	though it is somewhat more involved.
\begin{proposition}
	\label{proposition:bklpr-moments}
	We have
	\begin{align*}
\# \sym^2 H &= \mathbb E(\# \surj(\bklpr \nu, H)) \\
&= \mathbb E(\# \surj(\paritybklpr \nu 0,
H)) \\
&=\mathbb E(\# \surj(\paritybklpr \nu 1, H)).
	\end{align*}
\end{proposition}
\begin{proof}
	We first reduce to the case that $\nu = \ell^j$, for $\ell$ prime and $j
	\geq 1$.
	First, if $H_\ell$ is the Sylow $\ell$ subgroup of $H$, we have
	$\sym^2 H = \prod_{\ell \mid \nu}\sym^2 H_\ell$.
	Using the universal property of products, we also have that for any
	abelian group $A$, $\hom(A, H) = \prod_{\ell \mid \nu} \hom(A, H_\ell)$.
	Hence, we may
	assume that $\nu = \ell^j$.
	Instead of counting surjections, we can dually count injections from $H$
	to any of the above three distributions.
	
	Now, write $H \simeq \oplus_{i=1}^m \mathbb Z/\ell^{\lambda_i} \mathbb Z$, so
	that $H$ is determined by a partition $\lambda = (\lambda_1, \ldots,
	\lambda_m)$. Let $\lambda'$ denote the partition conjugate to $\lambda$
	so that $\lambda'_i$ is the number of copies of $\mathbb Z/\ell^i
	\mathbb Z$ appearing in $H$.
	We first consider the case of computing injections $H \to \bklpr
	{\ell^j}$.
	The number of injective homomorphisms $H \to \bklpr {\ell^j}$
	can be expressed as the limit as $n \to \infty$ of the number of injections $H \to Z \cap W$ where $Z,
	W \in \ogr_n(\mathbb Z/\ell^j \mathbb Z)$, for $\ogr_n$ the orthogonal
	Grassmannian parameterizing $n$-dimensional maximal isotropic subspaces
	in the rank $2n$ quadratic space with the split quadratic form
	$\sum_{i=1}^n x_i x_{i+n}$.
	(This uses an alternate description of the BKLPR distribution from the
	one we gave in \autoref{subsubsection:bklpr-n-selmer}, given in
	\cite[\S1.2 and 1.3]{bhargavaKLPR:modeling-the-distribution-of-ranks-selmer-groups};
	see also \cite[\S5.3.1]{fengLR:geometric-distribution-of-selmer-groups} for a
	summary.)
	For fixed $n$, we can express this as the number of injective
	homomorphisms $h: H \to W$ times the probability that a uniformly random $Z$ contains
	$\im (h)$.
	We can compute both of these numbers by inductively computing the answer
	on $\ell^t$ torsion for each $t \leq j$. 

	First, we compute the number of injective homomorphisms $h: H \to W$.
	In the case $t = 1$, so $\ell^t = \ell$, this was shown in the proof of \cite[Theorem
	5.10]{bhargavaKLPR:modeling-the-distribution-of-ranks-selmer-groups}
	to be $(\ell^n)^{\lambda'_1} \prod_{i=0}^{\lambda'_1-1} (1- \ell^{i-n})$.
	In general, a map $H \to W$ is injective if and only if $H[\ell] \to
	W$ is injective, so the number of injective maps $H[\ell^t] \to
	W$ lifting a given map $H[\ell^{t-1}] \to W$ for
	$t \geq 2$ is $(\ell^n)^{\lambda'_t}$. 
	Recall we defined $m$ by $H \simeq \oplus_{i=1}^m \mathbb
	Z/\ell^{\lambda_i} \mathbb Z$.
	Then, the total number of
	injective maps $H \to W$ is
	\begin{align}
		\label{equation:injective-maps}
		\ell^{n \cdot \sum_t \lambda'_t} \cdot \prod_{i=0}^{m-1}
		(1-\ell^{i-n}).
	\end{align}
	
	Next, we compute the probability that $Z$ contains $\im(h),$ for $h: H
	\to W$ an injective homomorphism.
	First, the chance that $Z$ contains $\im (H[\ell])$ was computed in 
\cite[Theorem 5.10]{bhargavaKLPR:modeling-the-distribution-of-ranks-selmer-groups}
and it is 
\begin{align*}
	\frac{\# \ogr_{n-m}(\mathbb Z/\ell \mathbb Z)}{\# \ogr_n(\mathbb
Z/\ell \mathbb Z)} = \ell^{\frac{n(n-1)}{2} - \frac{(n-m)(n-m-1)}{2}}
\prod_{i=n-m}^{n-1} (1+\ell^{-i}).
\end{align*}
Let $V$ denote the quadratic space we are working in.
Suppose we have fixed the image $Z/\ell^{t-1} Z \subset V/\ell^{t-1} V$ containing $h(H
[\ell^{t-1}])$. We next compute the chance that $Z/\ell^{t} Z$ contains the
image of $h(H [\ell^t])$ in $V/\ell^{t}V$.
Since $\ogr$ is smooth of dimension $\frac{n(n-1)}{2}$, there are
$\ell^{\frac{n(n-1)}{2}}$ lifts of $\ell^{t-1}Z$ to $\ell^{t} Z$.
The number of these containing $\im h(H [\ell^t])$ can be identified with lifts
of a maximal isotropic subspace of dimension $n - \lambda'_t$, since an
isotropic subspace of $W$ containing a rank $m$ isotropic space $T$ can be
identified with an isotropic subspace of the rank $m - n$ space $T^\perp/T$.
There are $\ell^{\frac{(n-\lambda'_t)(n-\lambda'_t-1)}{2}}$ such subspaces.
Hence, the chance $Z/\ell^{t} Z$ contains the image of $h(H [\ell^t])$ is
$\ell^{\frac{(n-\lambda'_t)(n-\lambda'_t-1)}{2} - \frac{n(n-1)}{2}} =
\ell^{\frac{(\lambda'_t)^2 - 2n\lambda'_t + \lambda'_t}{2}}$.
Multiplying these probabilities over all values of $t$ up to $j$, the chance $Z$
contains $h(H)$ is
\begin{align}
	\label{equation:chance-of-containment}
	\ell^{\sum_{t=1}^j {\frac{(\lambda'_t)^2 - 2n\lambda'_t +
	\lambda'_t}{2}}} \prod_{i=n-m}^{n-1} (1+ \ell^{-i}).
\end{align}

Therefore, the moment we are seeking is the product of
\eqref{equation:injective-maps} with \eqref{equation:chance-of-containment},
which gives
\begin{align*}
			&\ell^{n \cdot \sum_t \lambda'_t} \cdot \prod_{i=0}^{m-1}
			(1-\ell^{i-n}) \cdot \ell^{\sum_{t=1}^j {\frac{(\lambda'_t)^2 - 2n\lambda'_t +
	\lambda'_t}{2}}} \prod_{i=n-m}^{n-1} (1+ \ell^{-i})
	\\
	&= \ell^{\sum_{t=1}^j {\frac{(\lambda'_t)^2 + \lambda'_t}{2}}}
	\prod_{i=0}^{m-1}(1-\ell^{i-n})\prod_{i=n-m}^{n-1} (1+ \ell^{-i}).
\end{align*}
As $n \to \infty$, this approaches 
$\ell^{\sum_{t=1}^j {\frac{(\lambda'_t)^2 + \lambda'_t}{2}}}$.
A standard argument shows this agrees with $\# \sym^2 H$.
For example, the analogous computation of the size of $\wedge^2 H$ in place of
$\sym^2 H$ was carried
out in \cite[\S2.4]{wood:the-distribution-of-sandpile-groups}.
	
	The cases of $\paritybklpr \nu 1$ and $\paritybklpr \nu 0$ follow similarly by
	only taking one of the components of the orthogonal Grassmannian, as
	also explained in \cite[Remark
	5.11]{bhargavaKLPR:modeling-the-distribution-of-ranks-selmer-groups}.
\end{proof}

\subsection{Background on Hurwitz stacks}
\label{subsection:hurwitz-stacks-background}

In this subsection, we give a precise definition of the Hurwitz stacks we will
be working with.
Throughout the paper, we will employ the following notation.

\begin{notation}
	\label{notation:curve-notation}
Let $B$ be a base scheme.
Let $C \to B$ be a relative curve, which is smooth and proper of genus $g$ with
geometrically connected fibers.
Let $Z \subset C$ be a divisor, with $Z$
finite \'etale over $B$ of degree $f+1$, for $f \geq 0$.
Let $U := C - Z$. The situation is summarized in the following diagram:
\begin{equation}
	\label{equation:}
	\begin{tikzcd} 
		Z \ar {r}  & C \ar {d} & \ar{l} U := C - Z \\
		& B.
\end{tikzcd}\end{equation}
Let $n \geq 0$ be an integer.
Let $\sym^n_{C/B}$ denote the relative $n$th symmetric power of the curve $C$
over $B$.
Define $\conf n U B \subset \sym^n_{C/B}$ to be the open subscheme parameterizing
effective
divisors on $C$ which are finite \'etale of degree $n$ over $B$ and disjoint from $Z$.
Let $\mathscr C^n_B := C \times_B \conf n U B \to \conf n U B$ denote the universal curve, which has a
universal degree $n$ divisor $\mathscr D^n_B \subset \mathscr C^n_B$ whose fiber over a
point $[D] \in \conf n U B$ is $D \subset U$.
Let $\mathscr U^n_B := \mathscr C^n_B - \mathscr D^n_B - (\mathscr C^n_B \times_C
Z)$ and let $j: \mathscr
U^n_B \subset \mathscr
C^n_B$ denote the open inclusion.
This setup is pictured in the next diagram:
\begin{equation}
	\label{equation:}
	\begin{tikzcd} 
		\mathscr D^n_B \ar{r} \ar{rd} & \mathscr C^n_B \ar{d} &
		\ar{l}{j}\ar{ld} \mathscr
		U^n_B := \mathscr C^n_B - \mathscr D^n_B - (\mathscr C^n_B \times_C
Z)\\
\sym^n_{C/B}  & \conf n U B \ar{l} \ar {d}  \\
		& B.
\end{tikzcd}\end{equation}
\end{notation}

\begin{definition}
	\label{definition:fixed-hur}
	Keeping notation from
\autoref{notation:curve-notation},
	suppose $B$ is a scheme and $G$ is a finite group with $\# G$ invertible on $B$
	with chosen geometric point ${\overline{b}} \in B$.
	Suppose $\mathcal S \subset \hom(\pi_1(\Sigma_{g,n+f+1}), G)$ is a $G$ conjugation
	invariant subset preserved by the action
	of $\pi_1(\conf n {U_{\overline{b}}} {\overline{b}})$, acting on the first $n$ points.
	Define
	$\hur G n Z {\mathcal S} C B$
	to be the stack over $B$ whose functor of points is defined as
	follows: For $T$ a $B$-scheme, 
	$\hur G n Z {\mathcal S} C B(T)$
	is the groupoid
	\begin{align*}
		\left( D, i: D \to C_T, X, h: X \to C_T  \right)
	\end{align*}
	satisfying the following conditions:
		\begin{enumerate}
			\item $D$ is a finite \'etale cover of $T$ of degree $n$.
	        \item $i$ is a closed immersion $i: D \subset C_T$ which is
			disjoint from $Z_T \subset C_T$.
		\item $X$ is a smooth proper relative curve over $T,$ not
		necessarily having geometrically
		connected fibers.
		\item $h: X \to C_T$ is a finite locally free Galois $G$-cover,
			(meaning that $G$ acts simply transitively on the
			geometric generic fiber of $h$,) which is
			\'etale away from $Z_T \cup i(D) \subset C_T$.
		\item Let ${\overline{t}} \to T$ be a geometric point.
			Let $\overline{\eta}$
		denote the geometric generic point of $(C_T)_{\overline{t}}$. Then the
			representation $\rho:
		\pi_1((U_T)_{\overline{t}}-i(D_{\overline{t}}), \overline{\eta}) \to G$ 
		afforded by $h$,
		under the identification of
		$\hom(\pi_1((U_T)_{\overline{t}}-i(D_{\overline{t}}),
		\overline{\eta}),G)$
and $\hom(\pi_1(\Sigma_{g,n+f+1}), G)$
corresponds to an element of $\mathcal S$.
	\item The morphisms between two points $(D_i, i_i, X_i, h_i)$ for $i \in
		\{1,2\}$ are given by $(\phi_D, \psi_X)$ where $\phi_D: D_1
		\simeq D_2$ is an isomorphism so that $i_2 \circ \phi_D = i_2$
		and $\psi_X:X_1 \simeq X_2$ is an isomorphism such $h_2 \circ
		\psi_X = h_1$ and $\psi_X
		= g^{-1} \psi_X g$ for every $g \in G$.
	\end{enumerate}
\end{definition}
\begin{remark}
	\label{remark:hurwitz-stack-algebraic}
The above Hurwitz stacks are algebraic by
\cite[Theorem 1.4.1]{abramovichV:compactifying-the-space-of-stable-maps}.
Specifically, one can construct these Hurwitz stacks as an open substack of the quotient stack
$[{\mathcal K}_{g,n}([C/G],Z,1)/S_n]$, where 
${\mathcal K}_{g,n}([C/G],Z,1)$ is defined in \autoref{notation:stable-maps}.
\end{remark}
\begin{remark}
	\label{remark:center-free-implies-scheme}
	When $G$ is center free, the Hurwitz stacks
	parameterizing connected covers are indeed schemes, see \cite[Theorem 4]{wewers:thesis}.
	However, we will consider Hurwitz stacks parameterizing disconnected
	covers, and, in this case, it is possible that those components may be
	stacks which are not schemes, even when $G$ is center free. This will
	actually occur in the cases we investigate in this paper.
\end{remark}

The following pointed Hurwitz stack, which is a variant of the Hurwitz stack
defined above, will be useful in
connecting Hurwitz stacks to Hurwitz spaces over the complex numbers, described
in terms of tuples of monodromy elements. See \autoref{warning:no-quotient}.
We learned about the following slick construction from
\cite{chang:hurwitz-spaces-nichols-algebras}.
\begin{definition}
	\label{definition:pointed-hurwitz-space}
	With notation as in \autoref{notation:curve-notation},
	suppose there is a section $\sigma: B \to C$ with image contained in
	$Z$.
	Fix an integer $w$ and first define $\mathscr C_{(\sigma,w)}$ to be the root stack of order $w$ along
	$\sigma$, as defined in \cite[Definition
	2.2.4]{cadman:using-stacks-to-impose-tangency}. The fiber of this root
	stack over $\sigma$ is the stack quotient 
$[\left(\spec_B \mathscr O_B[x]/(x^w)\right)/ \mu_w]$ of the relative spectrum $\spec_B \mathscr
O_B[x]/(x^w)$ by $\mu_w$. Let $\widetilde{\sigma} : B \to \mathscr
	C_{(\sigma,w)}$ denote the section over $\sigma$ corresponding to map $B
	\to [\left(\spec_B \mathscr
	O_B[x]/(x^w) \right)/ \mu_w]$ given by the trivial $\mu_w$ torsor over
	$B$, $\mu_{w} \to
	B$, and the $\mu_w$ equivariant map $\mu_{w} \to B \to \spec_B \mathscr
	O_B[x]/(x^w)$.

	Define the {\em $w$-pointed Hurwitz stack}, $\left(\hur G n {\sigma \subset Z} {\mathcal
	S} C B \right)^w$, to be the stack
	whose groupoid of $T$-points is a setoid parameterizing data of the form
	\begin{align*}
		\left( D, 
			h': X \to (\mathscr
		C_{(\sigma,w)})_T, t: T \to X
	\times_{h', (\mathscr C_{(\sigma,w)})_T, \widetilde{\sigma}_T} T, i: D \to C_T, X, h: X \to C_T  \right),
	\end{align*}
	where $D, i,X,$ and $h$ are as defined in
	\autoref{definition:fixed-hur}. We also assume the order of inertia of $h$
	along $\sigma$ is $w$ and define
	$\widetilde{\sigma}_T$ to be the base change of the section
	$\widetilde{\sigma}$ defined above to $T$.
	We also impose the condition that
	$h'$ is a finite locally free $G$-cover, \'etale over
	$\widetilde{\sigma}$, such that the composition of $h': X \to \mathscr
		(C_{(\sigma,w)})_T$ with the coarse
		space map $(\mathscr C_{(\sigma,w)})_T \to C_T$ is $h$,
		and $t: T \to X \times_{(\mathscr C_{(\sigma,w)})_T,
		\widetilde{\sigma}_T} T$
		is a section of $h'$ over $\widetilde{\sigma}$.

		In general, we define the {\em pointed Hurwitz stack} as $\hur G n {\sigma \subset Z} {\mathcal
	S} C B := \coprod_{w \geq 1} \left(\hur G n {\sigma \subset Z} {\mathcal
	S} C B \right)^w.$
\end{definition}

\begin{remark}
	\label{remark:hurwitz-as-quotient}
	It will be useful for later to note that
	there is a $G$ action on 
	$\hur G n {\sigma \subset Z} {\mathcal S} C B$ obtained by sending $t$
	to $g \circ t$, for $g: X \to X$ the automorphism corresponding to $g
	\in G$.
	By construction, the stack quotient $[\hur G n {\sigma \subset Z}
	{\mathcal S} C B/G]$ is $\hur G n {Z} {\mathcal S} C B$.
\end{remark}

Although we will not need the next remark it what follows, it may comfort
the reader who is less familiar with stacks.

\begin{remark}
	\label{remark:}
	In fact,
	$\hur G n {\sigma \subset Z} {\mathcal S} C B$ is a scheme. One
	may verify this by proving it is a finite \'etale cover of $\conf n U
	B$.
\end{remark}

We will see later that the complex points of Hurwitz stacks as in
\autoref{definition:pointed-hurwitz-space}
admit a purely combinatorial description arising from actions of braid groups on finite sets.  We turn to the relevant topology now.

\section{The arc complex spectral sequence}
\label{section:arc-complex}

In this section, we set up the spectral sequence which will relate 
various finite index subgroups of surface braid groups corresponding to Hurwitz spaces
and allow induction arguments to take place.  As usual in arguments of this kind, the decisive fact is the high degree of connectivity of a certain complex, provided to us in this case by a theorem of Hatcher and Wahl.
In \autoref{subsection:coefficient-systems} we define the basic objects, called {\em coefficient systems}, we will
work with associated to surfaces. 
In \autoref{section:homological-stability}, we will show these coefficient systems have nice homological
stability properties. In \autoref{subsection:spectral-sequence} we set up the spectral sequence coming from the arc complex for these coefficient systems.

\subsection{Defining coefficient systems}
\label{subsection:coefficient-systems}

In this subsection, we define coefficient systems, which correspond to a
certain kind of compatible sequence of local systems on the unordered
configuration space of $n$ points on a topological surface with $1$ boundary
component, as $n$ varies. Later, we will show these have desirable
homological stability properties.
We are strongly guided here by the setup in 
\cite{randal-williamsW:homological-stability-for-automorphism-groups}.
We note that there are many different notions related to coefficient systems
in the literature appearing before 
\cite{randal-williamsW:homological-stability-for-automorphism-groups}, see, for
example,
\cite{dwyer:twisted-homological-stability},
\cite{vanderkallen:homology-stability-for-general-linear-groups},
\cite{ivanov:on-the-homology-stability-for-teichmuller-modular-forms}, and
\cite[\S8]{galatiusRW:homological-stability-for-moduli-spaces-i}.

In order to define coefficient systems, which will be our basic objects guiding
our study of homological stability, we begin by introducing some notation for
surface braid groups.

\begin{figure}
\includegraphics[scale=.5]{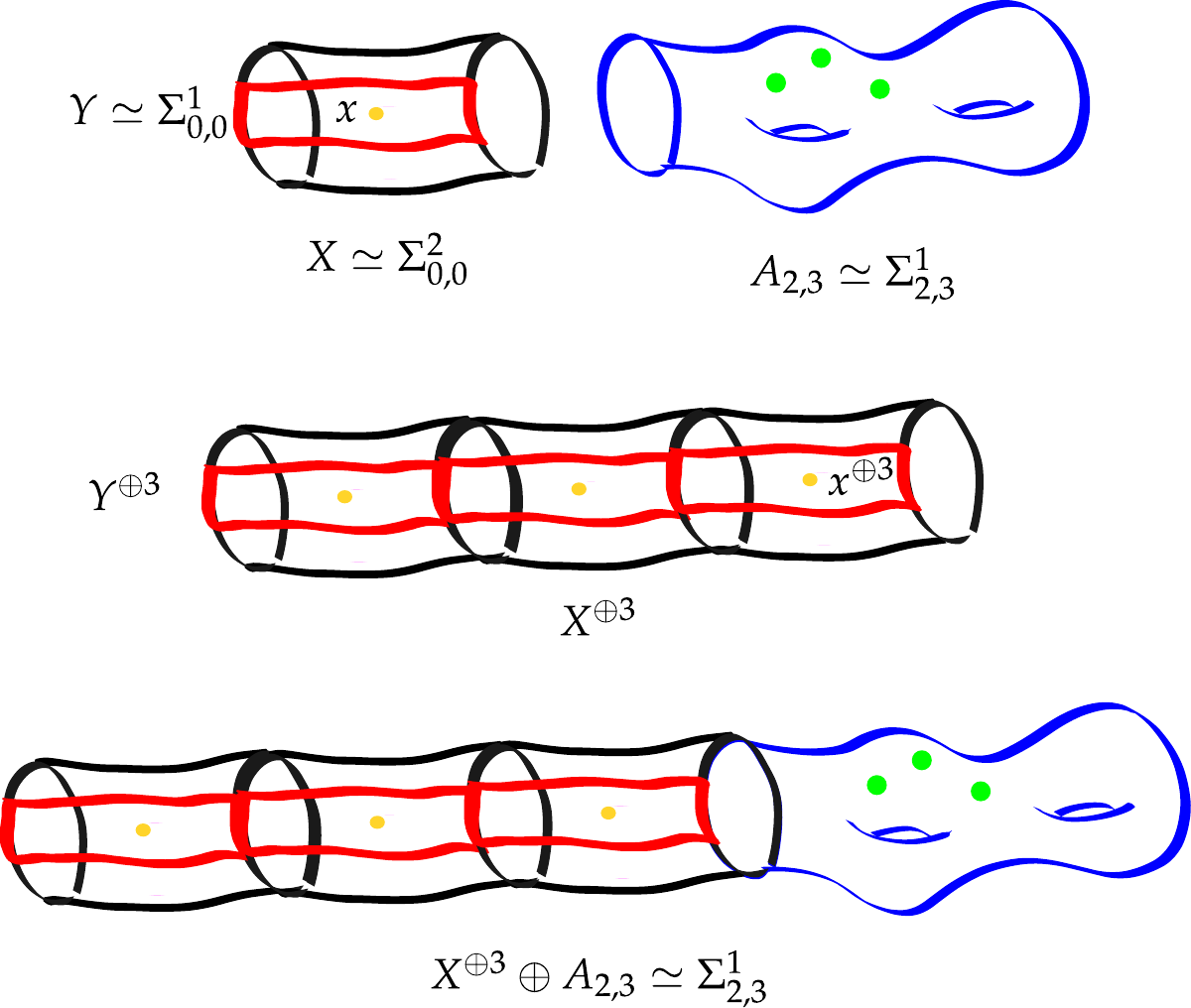}
\caption{
	The blue surface with green punctures is a picture of $A_{2,3} \simeq \Sigma^{1}_{2,3}$ and
	the black surface is $X \simeq \Sigma^2_{0,0}$. The yellow circles
	correspond to the point $x$, the red rectangles are the subsurface $Y$
	with $x \in Y \subset X$. We also depict $X^{\oplus 3}$ and $X^{\oplus 3}
	\oplus A_{2,3}$.}
\label{figure:surface-gluing}
\end{figure}

\begin{notation}
\label{notation:surface-braid-groups}
Let $\Sigma^b_{g,f}$ denote a genus $g$ topological surface with $b$ boundary components and
$f$ punctures.
For $W$ a topological space, we use $\Conf^n_{W}$ for the configuration space
parameterizing tuples of $n$ unordered distinct points on $W$.
Let $A_{g,f} := \Sigma^1_{g,f}$, let $X := \Sigma^2_{0,0}$, and let $x$ be a
point in the interior of $X$. If we think of $X$ as $\R/\Z \times [0,1]$, we may
place $x$ at $(0,1/2)$.  With this same identification, we denote by $Y$ the
rectangle $[-1/4,1/4] \times [0,1]$, thought of as a subspace of $X$. See \autoref{figure:surface-gluing}.

For $n > 0$, define the surface $X^{\oplus n} \oplus A_{g,f}$, which is homeomorphic to
$\Sigma^1_{g,f}$, inductively by gluing the first boundary component of $X$ along
a chosen isomorphism to the boundary component of $X^{\oplus{n-1}} \oplus
A_{g,f}$. 
We suggest the reader consult \autoref{figure:surface-gluing} for a
visualization.
We denote by $x^{\oplus n}$ the $n$-element subset of $X^{\oplus n}
\oplus A_{g,f}$ obtained as the union of the copy of the point $x$ in each of the $n$
copies of $X$.
We also let $X^{\oplus n}$ denote the complement of the interior of $A_{g,f}$ in
$X^{\oplus n} \oplus A_{g,f}$ and we let $Y^{\oplus n} \subset X^{\oplus n}$
denote the subsurface of $X^{\oplus n}$ covered by the $n$ copies of $Y \subset
X$. Again, see \autoref{figure:surface-gluing} for a visualization.

Now, let $B^n_{g,f} := \pi_1(\Conf^n_{X^{\oplus n} \oplus A_{g,f}}, x^{\oplus
n})$ denote the surface braid group.  The natural map
\beq
Y^{\oplus i} \coprod (X^{\oplus n-i} \oplus A_{g,f}) \to X^{\oplus i} \coprod (X^{\oplus n-i} \oplus A_{g,f}) \to 
X^{\oplus n} \oplus A_{g,f}
\eeq
induces a map $\Conf^i_{Y^{\oplus i}} \times \Conf^{n-i}_{X^{\oplus n-i} \oplus
A_{g,f}} \to \Conf^n_{X^{\oplus n} \oplus A_{g,f}}$ which sends $x^{\oplus i} \coprod x^{\oplus n-i}$ to $x^{\oplus n}$.  We note that $Y^{\oplus i}$ is homeomorphic to a disc embedded in $X^{\oplus i}$, so the fundamental group of the configuration space $\Conf^i_{Y^{\oplus i}}$ is just the usual Artin braid group on $i$ strands.  We thus get a map of fundamental groups
\beq
\pi_1(\Conf^i_{Y^{\oplus i}}, x^{\oplus i}) \times \pi_1(\Conf^{n-i}_{X^{\oplus n-i} \oplus A_{g,f}}, x^{\oplus n-i}) \to \pi_1(\Conf^n_{X^{\oplus n} \oplus A_{g,f}},x^{\oplus n})
\eeq
or, in shorter terms, $B^i_{0,0} \times B^{n-i}_{g,f} \to B^n_{g,f}$.
\end{notation}

\begin{remark}
    By means of the homeomorphism between $X^{\oplus n} \oplus A_{g,f}$ and $\Sigma_{g,f}^1$, we may think of $B^n_{g,f}$ as the usual surface braid group on $n$ strands in a genus $g$ surface with $f$ punctures and a boundary component.  We have chosen to define $B^n_{g,f}$ in this more specific way because it will help us keep track of the maps between braid groups we will need to invoke.
\end{remark}
\begin{remark}
	\label{remark:}
	The reason for us introducing $Y$ in
	\autoref{notation:surface-braid-groups}, instead of just using $X$, is to
	obtain an inclusion $B^i_{0,0} = B^i_{0,0} \times \{1\} \subset
	B^i_{0,0} \times B^{n-i}_{g,f} \to B^n_{g,f}$, which gives an inclusion
	from a braid group for a surface with $1$ boundary component instead of
	from a surface with two boundary components.
	In fact, $X$ is not really needed for our arguments, and would be
	slightly more efficient not to introduce $X$ at all and to simply work with $Y$, as is essentially done in 
	\cite[\S5.6.1]{randal-williamsW:homological-stability-for-automorphism-groups}.
	Primarily for psychological and aesthetic reasons to reflect the way we
	were thinking about the monoidal operations, we prefer to view $Y$ as
	sitting inside of $X$.
	The key point of our homological stability results is that we will view
	certain systems of representations of $B^n_{g,f}$ as modules-like
	objects for
	systems of representations of $B^n_{0,0}$, and in order to define the
	module structure, the inclusion $B^n_{0,0} \to B^n_{g,f}$ is essential.
\end{remark}

We next define coefficient systems.
Our definition of coefficient systems is inspired by \cite[Definition
4.1]{randal-williamsW:homological-stability-for-automorphism-groups}, though it is not exactly the same.

\begin{definition}
\label{definition:coefficient-system}
For $k$ a field, {\em a coefficient system for $\Sigma^1_{0,0}$} is a
sequence of $k$ vector spaces $(V_n)_{n \geq 0}$ with actions $B^n_{0,0} \times
V_n \to V_n$
so that $V_0 := k$, $V_n :=
V_1^{\otimes n}$, 
and so that the $B^n_{0,0}$ action on $V_n$
satisfies the following condition.
For any $0 \leq i \leq n$, the diagram
\begin{equation}
\label{equation:coefficient-system-compatibility}
\begin{tikzcd} 
(B^i_{0,0} \times B^{n-i}_{0,0}) \times V_i \otimes V_{n-i}  \ar {r} \ar {d} &
V_i \otimes V_{n-i}  \ar {d} \\
B^n_{0,0} \times V_n \ar {r} & V_n
\end{tikzcd}\end{equation}
commutes, with maps described as follows:
the right vertical map is induced by the isomorphism
coming from the definition of $V_n$, 
the left vertical map is induced by this isomorphism together with the inclusion
$B^i_{0,0} \times B^{n-i}_{0,0} \to B^n_{0,0}$ described in 
\autoref{notation:surface-braid-groups},
and the horizontal maps are induced by the
given actions of $B^j_{0,0}$ on $V_j$.
\end{definition}

\begin{remark}
\label{remark:braid-structure}
If $(V_n)_{n \geq 0}$ is a coefficient system, then $V_1$ naturally has
the structure of a braided vector space coming from the action of a specified generator of
$B^2_{0,0}\simeq \mathbb Z$ on
$V_2 = V_1 \otimes V_1$. 
For any braided vector space $V$, the tensor powers $V^{\tensor n}$ acquire
actions of $B^n_{0,0}$ satisfying
\eqref{equation:coefficient-system-compatibility}.  So, the
definition of coefficient system for $\Sigma^1_{0,0}$ is equivalent to that of a braided vector space.

We chose to set up \autoref{definition:coefficient-system} as we did so that its
structure is analogous to that of coefficient systems for higher genus surfaces,
which we define next.
\end{remark}

\begin{definition}
\label{definition:coefficient-system-over}
Fix a field $k$ and let $V$ be a fixed coefficient system for $\Sigma^1_{0,0}$.
For $g, f \geq 0$, {\em a coefficient system for $\Sigma^1_{g,f}$ over $V$}
is a sequence of $k$ vector spaces $(F_n)_{n \geq 0}$ with actions
$B^n_{g,f} \times F_n \to F_n$
so that $F_n := V_n \otimes
F_0$ and the $B^n_{g,f}$ action on $F_n$ satisfies the following condition.
For any $0 \leq i \leq n$,
the diagram
\begin{equation}
\label{equation:coefficient-over-compatibility}
\begin{tikzcd} 
(B^i_{0,0} \times B^{n-i}_{g,f}) \times V_i \otimes F_{n-i}  \ar {r} \ar {d} &
V_i \otimes F_{n-i}  \ar {d} \\
B^n_{g,f} \times F_n \ar {r} & F_n
\end{tikzcd}\end{equation}
commutes, with maps described as follows:
the right vertical map is an equality
coming from the definition of $F_n$,
the left vertical map is induced by the above equality and the inclusion
$B^i_{0,0} \times B^{n-i}_{g,f} \to B^n_{g,f}$ described in 
\autoref{notation:surface-braid-groups}, and the horizontal maps are induced by the
given actions of $B^j_{0,0}$ on $V_j$
and $B^j_{g,f}$ on $F_j$.
\end{definition}

\begin{remark}
\label{remark:}
It is natural to think of coefficient systems (over $V$) as a compatible sequence of local
systems on $\Conf^n_{\Sigma^1_{g,f}}$. The compatibility condition amounts to
commutativity of the diagram \eqref{equation:coefficient-over-compatibility}.
\end{remark}

\begin{remark}
Just as a braided vector space is determined by a finite amount of linear
algebraic data (an endomorphism of $V_1^{\tensor 2}$ satisfying a certain
identity) it would be interesting to define a coefficient system for
$\Sigma^1_{g,f}$ over $V$ in a similar way, in the spirit of the definitions
introduced by Hoang in \cite[\S3]{hoang:fox-neuwirth-cells-resultant}.
\end{remark}

We next describe coefficient systems related to Hurwitz spaces, which come from
maps from
$\pi_1(A_{g,f})$ to a finite group.

\begin{example}
	\label{example:hurwitz-coefficient-system}
	Fix $g, f \geq 0$.
    Let $G$ be a finite group and $c$ a conjugacy-closed subset of $G$,
    and use notation as in \autoref{notation:surface-braid-groups}.
    Choose a basepoint $p_{g,f}$ on $A_{g,f}$.
    Note that $B^n_{g,f}$ acts on 
	$\pi_1(X^{\oplus n} \oplus A_{g,f} -
    x^{\oplus n}, p_{g,f})$ and hence on 
	$\hom(\pi_1(X^{\oplus n} \oplus A_{g,f} -
	x^{\oplus n}, p_{g,f}), G)$.
	Since $\pi_1(X^{\oplus n} \oplus A_{g,f} -
	x^{\oplus n}, p_{g,f})$ is a free group on $n + 2g + f$ generators, we
	can identify the above set of homomorphisms with $G^{n+2g+f}$.
    Choose subsets
    $\inertiaindex n G c g f \subset \hom(\pi_1(X^{\oplus n} \oplus A_{g,f} -
    x^{\oplus n}, p_{g,f}), G)$
(under the identification
$\hom(\pi_1(X^{\oplus n} \oplus A_{g,f} -
x^{\oplus n}, p_{g,f}), G) \simeq G^{n+2g+f}$)
    so that $\inertiaindex n G c g f = c \times \inertiaindex {n-1} G c g f$ and
    $\inertiaindex n G c g f$ 
    is closed
    under the action of $B^n_{g,f}$ on $\hom(\pi_1(X^{\oplus n} \oplus A_{g,f}-x^{\oplus n}, p_{g,f}),G)$.
    Write $H_{\inertiaindex n G c g f}$ for
    the vector space freely spanned over $k$ by the subset
    $\inertiaindex n G c g f \subset \hom(\pi_1(X^{\oplus n} \oplus A_{g,f} -
    x^{\oplus n}, p_{g,f}),G)$.
   
    Specializing the above to the case $g = f = 0$,
    the action of $B^n_{0,0}$ on 
	$\inertiaindex n G c 0 0$
     induces an action of $B^n_{0,0}$ on $H_{\inertiaindex n G c 0 0}$.
    We denote by $H_{\allinertia G c 0 0}$ the coefficient system $V$
    for $\Sigma^1_{0,0}$
    given by 
    $V_n = H_{\inertiaindex n G c 0 0}$.
    This corresponds to the usual action of the Artin braid group on Nielsen
    tuples that underlies the classical combinatorial description of Hurwitz
    spaces of covers of the disc.
    Further, we denote by $H_{\allinertia G c g f}$ 
	the coefficient system $F$ for $\Sigma^1_{g,f}$
    over $H_{\allinertia G c 0 0}$ given by $F_n := H_{\inertiaindex n G c g f}$.
\end{example}
\begin{warn}
	\label{warning:no-quotient}
	We note that the cover of configuration space afforded by the coefficient system $H_{\allinertia G c g f}$ 
	in \autoref{example:hurwitz-coefficient-system}
	is not exactly the same thing as the space of complex points of the
	Hurwitz stack defined in \autoref{definition:fixed-hur}, but rather it
	is the complex points of the pointed Hurwitz space from
	\autoref{definition:pointed-hurwitz-space}.
	The sets
	$\inertiaindex n G c g f$ carry an action of $G$ by conjugation
	and, via \autoref{remark:hurwitz-as-quotient},
	the Hurwitz stack as in 
	\autoref{definition:fixed-hur} is the quotient
	of the cover afforded by $\inertiaindex n G c g f$ by this $G$-action. 

	The reason we work more with this quotient is that it is easier to
	access
	from the point of view of moduli theory in algebraic geometry, while the unquotiented version is more suitable for the topological arguments we will make over the next several sections.  This is easiest to see in the case $(g,f) = (0,0)$, where an element of $\inertiaindex n G c g f$ is an $n$-tuple of elements of $c$.  Then the concatenation operation $c^m \times c^n \ra c^{m+n}$ plays a key role in our arguments; but there is no well-defined concatenation on $c^m / G \times c^n / G$. 
\end{warn}

\begin{example}
	\label{example:trivial-coefficient-system}
	Take $V$ to be the coefficient system for $\Sigma^1_{0,0}$ with $V_i =
	k$ and the trivial action for all $i$.
	We call $V$ the {\em trivial coefficient system} for $\Sigma^1_{0,0}$.
	Let $F_0$ be a vector space. Then
	$F_i := F_0$ defines a coefficient system where the action of
	$B^n_{g,f}$ on $F_0$ is trivial.
\end{example}

\begin{remark}
	A different but related notion of ``coefficient system'' is 
    considered in
    \cite{randal-williamsW:homological-stability-for-automorphism-groups}.
    Their precise definition doesn't concern us here, but the coefficient
    systems they consider, which they call {\em
    finite degree}, 
    have the property that their sequence of finite-dimensional vector spaces $V = \{V_n\}_{n \in \mathbb Z_{\geq
    0}}$, $\dim V_n$ is
    eventually polynomial in $n$.  
    The coefficient system $H_{\allinertia G c g f}$ considered above, by
    contrast, have $\dim H_{\inertiaindex n G c g f}$ growing exponentially in $n$. More
    precisely, the dimension grows proportionally to $|c|^{n}$.  
\end{remark}

\subsection{The spectral sequence}
\label{subsection:spectral-sequence}

Our next main result is \autoref{proposition:arc-complex-spectral-sequence},
which sets up a spectral sequence coming from the arc complex.
In order to describe this, we first describe the $\mathcal K$ complex associated to a
module.

\begin{definition}
	\label{definition:k-complex}
	Let $V$ be a coefficient system for $\Sigma^1_{0,0}$.
	Let $R^V = \oplus_{n \geq 0} H_0(B^n_{0,0}, V_n)$,
	which has the
	structure of a graded ring induced by the isomorphisms $V_s \otimes V_r
	\to
	V_{r+s}$.
	Let $M$ be a graded $R^V$-module and let $\{M\}_n$ denote the $n$th
	graded part of $M$.
	Let
	$\mathcal K(M)$ denote the complex of graded $R^V$-modules whose $q$th term is $\mathcal
	K(M)_q := V_q \otimes
	M[q]$. That is, $\mathcal K(M)$ is given by
	\begin{align*}
		\cdots \to V_n \otimes M[n] \to \cdots \to V_1 \otimes
		M[1] \to M[0]	
	\end{align*}
	where $M[i]$ denotes the shift by grading $i$ so that 
	$\{M[i]\}_n=\{M\}_{n-i}$. Here we treat $V_i$ as living in grading $i$.	

	To define the differential, we next introduce some notation.
	Using $\tau$ to denote the braiding automorphism of $V_1 \otimes V_1$
	from \autoref{remark:braid-structure}, for $1 \leq i < n$, we let 
	$\tau^n_{i}: V_1^{\otimes
	n} \to V_1^{\otimes n}$ denote the automorphism $\tau^n_i :=
	\id^{\otimes i-1} \otimes \tau \otimes \id^{\otimes n - i-1}$, which
	applies $\tau$ to the $i$ and $i+1$ factors.
	For $1 \leq i \leq j \leq n$, we define $\tau^n_{i,j} := \tau_{j-1}^n
	\cdots \tau_{i+1}^n \tau_i^n$. 
	(Intuitively, $\tau^n_{i,j}$ corresponds to the element of $B^n_{0,0}$
	pulling the $i$th strand past strands $i+1, \ldots, j$, to position $j$.)
	So, in particular, $\tau^n_i = \tau^n_{i,
	i+1}$ and $\id = \tau^n_{i,i}.$
	We use $\mu_n : V_1 \otimes
	\{M\}_n \to \{M\}_{n+1}$ to denote the multiplication map coming from the
	structure of $M$ as a $R^V$-module.
	As mentioned above, we use $\{M\}_n$ to denote the $n$th graded piece
	of a graded module $M$, 
	and then the differential on $\mathcal K(M)$ is given by
	\begin{equation}
	\begin{aligned}
	\label{equation:k-differential}
	\{\mathcal K(M)_{q+1}\}_{n} & \rightarrow \{\mathcal K(M)_q\}_n \\
	(v_0 \otimes \cdots \otimes v_q) \otimes m & \mapsto \sum_{i=0}^q (-1)^i (\id^{\otimes q}
	\otimes \mu_{n-q-1}) \left(\tau_{i+1,q+1}^{q+1}(v_0 \otimes \cdots \otimes v_q) \otimes m\right).
	\end{aligned}
	\end{equation}
\end{definition}

The main case of \autoref{definition:k-complex} we will be interested in is when
our module for
$R^V$ is of the form $M_p^{V,F}$, which we now define.

\begin{notation}
\label{notation:homology-coefficient-system}
Given a coefficient system $V$ for $\Sigma^1_{0,0}$
and
a coefficient system $F$ for $\Sigma^1_{g,f}$ over $V$,
define $M_p^{V,F} := \oplus_{n\geq 0} H_p(B^n_{g,f}, F_n)$, where
here the homology denotes group homology.
\end{notation}

In the case our coefficient system is of the form $M_p^{V,F}$, we next
describe the map
$\mu_n$ concretely as well as the $R^V$-module structure on $M_p^{V,F}$.

\begin{remark}
\label{remark:mp-module-structure}
In the case we take our module for $R^V$ in
\autoref{definition:k-complex} to be $M_p^{V,F}$ from
\autoref{notation:homology-coefficient-system},
	we can describe the map $\mu_n : V_1 \otimes
	\{M_p^{V,F}\}_n \to \{M_p^{V,F}\}_{n+1}$ concretely as follows.
	The inclusion $B^n_{g,f} \to B^{n+1}_{g,f}$ from
	\autoref{notation:surface-braid-groups} coming from the inclusion
	$X^{\oplus n} \oplus A_{g,f} \to X^{\oplus n + 1} \oplus A_{g,f}$
	induces a cross product map 
	\begin{align*}
		V_1 \otimes H_p(B^n_{g,f}, F_n) &\simeq H_0(B^1_{0,0}, V_1) \otimes
		H_p(B^n_{g,f}, F_n)
		\\
		&\to 
		H_0(B^1_{0,0} \times B^n_{g,f}, V_1 \otimes F_n) \\
		&\to
		H_p(B^{n+1}_{g,f}, V_1 \otimes F_n) \\
		&\simeq H_p(B^{n+1}_{g,f},
		F_{n+1}).
	\end{align*}
	The third map above is the map induced on homology by the homomorphism
	$B^1_{0,0} \times B^n_{g,f} \to B^{n+1}_{g,f}$ as defined at the end of
	\autoref{notation:surface-braid-groups}, where we take $n+1$ here in
	place of $n$ there and $1$ here in place of $i$ there.
	More generally, for $n \geq m$, the inclusions
	$B^m_{0,0} \times B^{n-m}_{g,f} \to B^n_{g,f}$ from 
	\autoref{notation:surface-braid-groups} give $M_p^{V,F}$ the structure
	of a $R^V$-module via the cup product map
	\begin{align*}
		H_0(B^i_{0,0}, V_i) \otimes H_p(B^n_{g,f}, F_n) \to H_p(B^{n+i}_{g,f}, V_i\otimes F_n) \simeq H_p(B^{n+1}_{g,f}, F_{n+i}).
	\end{align*}
\end{remark}

We now describe the spectral sequence coming from the arc complex. For a
picture of the $E^2$ page of this spectral sequence, see
\autoref{figure:spectral-sequence-picture}.
(We include the picture only in a later section as we believe it is helpful
to see it side by
side the proof of \autoref{theorem:1-controlled-bound}.)

\begin{proposition}
	\label{proposition:arc-complex-spectral-sequence}
	Fix a non-negative integer $n$, corresponding to a choice of grading.
	Let $V$ be a coefficient system for $\Sigma^1_{0,0}$ and let
	$F$ be a coefficient system for $\Sigma^1_{g,f}$ over $V$.
	There is a homological spectral sequence of $k$-modules $E^1_{q,p}$ converging to $0$
	in dimensions $q + p \leq n - 1$,
where the $p$th row $(E^1_{\ast,p}, d_1)$ is isomorphic to the $n$th graded
piece of $\mathcal K(M_p^{V,F})$. That is, $E^1_{q,p}$ is the $n$th graded piece of
$\mathcal K(M_p^{V,F})_q$ for $p,q \geq 0.$
\end{proposition}
\begin{proof}
The proof is a fairly immediate generalization of \cite[Proposition
5.1]{EllenbergVW:cohenLenstra}. We now fill in some of the details.
One minor difference is that we opt to use an augmented version of the arc
complex so that the spectral sequence converges to $0$, instead of $\mathcal K(M_p^{V,F})$
as in \cite[Proposition 5.1]{EllenbergVW:cohenLenstra}, and we will also use
\cite[Lemma 5.21]{randal-williamsW:homological-stability-for-automorphism-groups} to
demonstrate high connectivity of a relevant complex.

We begin by defining a combinatorial version of the arc complex, which we denote
$\mathbb A(g,f,n)$.
The spectral sequence will be obtained from filtering this complex by the
dimension of its simplices.
For $-1 \leq q \leq n-1$, let $L_q \subset B^n_{g,f}$ denote the subgroup $L_q
\simeq B^{n-q-1}_{g,f}$ obtained via the inclusion
$B^{n-q-1}_{g,f} \subset B^n_{g,f}$ coming from
\autoref{notation:surface-braid-groups}.
If $q = n-1$, $L_q$ is the trivial group. Define $\mathbb A(g,f,n)_q :=
B^n_{g,f}/L_q$ as a $B^n_{g,f}$ set.
Define the faces of the $q$-simplex $bL_q$ by $\partial_i(bL_q) = b s_{q,i}
L_{q-1}$ for $0 \leq i \leq q$, where $s_{q,i} = \sigma_{i+1} \sigma_{i+2}
\cdots \sigma_q$ and $\sigma_i$ denotes an elementary transformation moving the
$i$th point counterclockwise around the $i+1$st point in $\pi_1(\Conf^n_{\Sigma^1_{g,f}}) \simeq
B^n_{g,f}$.
Here, $s_{q,q} = 1$.
An identical computation to \cite[Proposition 5.3]{EllenbergVW:cohenLenstra}
shows $\partial_i \partial_j = \partial_{j-1} \partial_i$ for $i < j$, implying
$\mathbb A(g,f,n)$ is a semi-simplicial set.
It will follow from \cite[Lemma 5.18 and Lemma 5.21]{randal-williamsW:homological-stability-for-automorphism-groups}
that the part of this semi-simplicial set in non-negative degrees,
$\mathbb A(g,f,n)_{q \geq 0}$,
is $(n-2)$-connected.
(Recall $X$ is {\em $n$-connected} if $X$ is non-empty and $\pi_i(X) = 0$ for $i \leq n$.)
This will then imply the homology of the augmented semi-simplicial set 
$\mathbb A(g,f,n)$ vanishes in degree at most $n-2$.
For the reader not familiar with
\cite{randal-williamsW:homological-stability-for-automorphism-groups}, we
next explain 
why $\mathbb A(g,f,n)_{q \geq 0}$,
is $(n-2)$-connected in more detail.

%
To verify $(n-2)$-connectedness
of
$\mathbb A(g,f,n)_{q \geq 0}$,
we next define another semi-simplicial set, which we denote by
$W_n(A,X)$;
here $A = (\Sigma^1_{g,f},J), X =
(S^1 \times [0,1], I)$, for $J$ and $I$ two parameterized intervals on the
boundaries of $\Sigma^1_{g,f}$ and $S^1 \times[0,1]$.
We will explain why $W_n(A,X)$ is equivalent to 
$\mathbb A(g,f,n)_{q \geq 0}$ and why it is $(n-2)$-connected.
As in
\cite[\S5.6]{randal-williamsW:homological-stability-for-automorphism-groups},
we let $\mathcal M_2$ denote the groupoid of {\em decorated surfaces}
$(S,I)$, where $S$ is a compact connected surface with at least one boundary
component and $I: [-1,1] \to \partial S$ is a parameterized interval on its
boundary. Morphisms in $\mathcal M_2$ are diffeomorphisms restricting to the
identity on a neighborhood of $I$.
Boundary connected sum induces a monoidal product $\oplus$ on $\mathcal M_2$, as
explained in 
\cite[\S5.6.1]{randal-williamsW:homological-stability-for-automorphism-groups}.
As in 
\cite[\S1.1]{randal-williamsW:homological-stability-for-automorphism-groups},
let $U\mathcal M_2$ denote the category with the same objects as $\mathcal M_2$,
and with morphisms
given by $\hom_{U\mathcal M_2}(A,B) := \colim_{Y \in \mathcal M_2} \hom_{\mathcal
M_2}(Y \oplus A, B)$.
As in \cite[Definition
2.1]{randal-williamsW:homological-stability-for-automorphism-groups},
we then define the semi-simplicial $W_n(A,X)$ whose $p$-simplices are given by
$W_n(A,X)_p := \hom_{U\mathcal M_2}(X^{\oplus p+1}, X^{\oplus n} \oplus A)$ with
face maps $d_i : \hom_{U\mathcal M_2}(X^{\oplus p+1}, X^{\oplus n} \oplus A) \to
\hom_{U\mathcal M_2}(X^{\oplus p}, X^{\oplus n} \oplus A)$ obtained by
precomposing with the inclusion $X^{\oplus i} \oplus \iota_X \oplus X^{\oplus
p-i}$, for $\iota_X : 0 \to X$ the unit.
Let $v_q \in W_n(A,X)_q$ denote the $q$-simplex corresponding to the map $X^{\oplus q+1} \oplus
\iota_X^{\oplus n - q -1} \oplus \iota_A$.
There is a natural $B^n_{g,f}$-equivariant map
$\mathbb A(g,f,n)_{q \geq 0} \to W_n(A,X)$ induced by sending
$b \mapsto b \cdot v_q$.
Let us now explain why this map is an equivalence of semi-simplicial sets.
By \cite[Lemma
5.18]{randal-williamsW:homological-stability-for-automorphism-groups},
the category $U\mathcal M_2$ defined there is ``locally
homogeneous'' at $(A,X)$. This means the two axioms LH1 and LH2
from 
\cite[Definition
1.2]{randal-williamsW:homological-stability-for-automorphism-groups} are
satisfied;
LH1 guarantees transitivity of
$B^n_{g,f}$ on
$q$-simplices of
$W_n(A,X)$
and LH2 implies the stabilizer in $B^n_{g,f}$ of any $q$-simplex in
$W_n(A,X)$
is
$B^{n-q-1}_{g,f}$.
(See also steps (b) and (c) in the proof of Theorem 3.1 \cite[p.
563]{randal-williamsW:homological-stability-for-automorphism-groups} for further
explanation.)
The transitivity of the action and identification of stabilizers together imply
the map
$\mathbb A(g,f,n)_{q \geq 0} \to W_n(A,X)$
is an equivalence.
Finally, it is shown in 
\cite[Lemma
5.21]{randal-williamsW:homological-stability-for-automorphism-groups} and the
paragraph following it that $W_n(A,X)$ is $(n-2)$-connected, and therefore 
$\mathbb A(g,f,n)_{q \geq 0}$ is as well.

We next describe the claimed spectral sequence.
We will write $k\{\mathbb A(g,f,n)\}$ to denote the free vector space on
the simplices of $\mathbb A(g,f,n)$, which we view both as a semi-simplicial
$k$-module and as a $B^n_{g,f}$ representation.
Since, as we explained above, the non-augmented semi-simplicial set
$\mathbb A(g,f,n)_{q \geq 0}$ is $(n-2)$-connected, we obtain that
$H_p(B^n_{g,f}, k\{\mathbb A (g,f,n)\}\otimes F_n ) = 0$ for
$p \leq n-2$.
We can also identify $H_p( B^n_{g,f}, \left(k\{\mathbb A(g,f,n)_q\} \otimes F_n \right) )$ with
$\{\mathcal K(M_p^{V,F})_{q+1}\}_n$ via the isomorphisms 
\begin{align*}
H_p\left(B^n_{g,f}, k\{\mathbb A(g,f,n)_q\} \otimes F_n \right) &\simeq H_p(L_q,
F_n) \\
&\simeq H_p(B^{n-q-1}_{g,f} , V_1^{\otimes q+1} \otimes F_{n-q-1}) \\
&\simeq 
V_1^{\otimes q+1} \otimes H_p(B^{n-q-1}_{g,f} , F_{n-q-1}) \\
&\simeq \{\mathcal K(M_p^{V,F})_{q+1}\}_n,
\end{align*}
where the first equivalence follows from a version of Shapiro's lemma and the third equivalence uses that $B^{n-q-1}_{g,f}$ acts on
$V_1^{\otimes q+1} \otimes F_{n-q-1}$ through its action on $F_{n-q-1}$.
Filtering $\left(k\{\mathbb A(g,f,n)\} \otimes F_n \right)$ 
(viewed as a vector space with $B^n_{g,f}$ action)
by the
simplicial structure on $\mathbb A(g,f,n)$, we obtain a spectral sequence
\begin{align}
\label{equation:arc-spectral-sequence}
E^1_{q,p} := H_p(B^n_{g,f}, k\{\mathbb A(g,f,n)_q \} \otimes F_n ) \implies
H_{p+q}\left( B^n_{g,f}, k\{\mathbb A(g,f,n)\} \otimes F_n
\right).
\end{align}
Since $\mathbb A(g,f,n)_{q \geq 0}$ is $(n-2)$-connected, 
$\mathbb A(g,f,n)$ has trivial homology in degrees $\leq n-2$ and hence 
the right hand side of \eqref{equation:arc-spectral-sequence} vanishes for $p +
q \leq n - 2$.
Analogously to \cite[Lemma 5.4]{EllenbergVW:cohenLenstra}, one may verify that
the differential $d_1 : E^1_{q,p} \to E^1_{q-1,p}$ is identified with the $n$th
graded part of the
differential $\mathcal K(M_p^{V,F})_{q+1} \to \mathcal K(M_p^{V,F})_{q}$ as in
\eqref{equation:k-differential}.
The spectral sequence we have now constructed has bounds $-1 \leq q \leq n-1$.
We now have spectral sequence $E^1_{q,p} = K(\mathcal K(M_p^{V,F})_{q+1})$
but we want one with
$E^1_{p,q} = K(\mathcal K(M_p^{V,F})_{q})$.
We can obtain the latter by 
replacing $q$ by $q-1$.
This implies
$-1 \leq q-1 \leq n$, or equivalently 
$0 \leq q \leq n$, and yields vanishing in
degrees $p + (q-1) \leq n-2$, or equivalently $p + q \leq n - 1$.
This gives
desired spectral
sequence, as in the statement.
\end{proof}

\section{Deducing homological stability results for coefficient systems}
\label{section:homological-stability}

In this section, we prove that certain types of coefficient systems have nice
homological stability properties, following closely ideas from
\cite{EllenbergVW:cohenLenstra}.
In \autoref{subsection:1-controlled} we give a general formulation of this
stability property.
In \autoref{subsection:sufficient-condition} we show that finitely generated
modules for coefficient systems with a suitable central element satisfy this
stability property.
Finally, in \autoref{subsection:exponential-bound} we put together all the topological material developed in this section and the previous one to arrive at an exponential bound
on the cohomology of these coefficient systems.  

For the reader primarily interested in our application to Selmer groups, 
only two results from this section are needed in future parts.
First, \autoref{corollary:cohomology-bound} will be used as a central ingredient in the proof of
\autoref{lemma:moments-cohomology}.
Second, \autoref{theorem:central-u-implies-cohomology-bound} will be used in the
proof of
\autoref{theorem:frobenius-equivariance}
to show the trace of Frobenius on the cohomology stabilizes.

\subsection{Homological stability for $1$-controlled coefficient systems}
\label{subsection:1-controlled}

We next prove the main homological stability result of this paper in
\autoref{theorem:1-controlled-bound}, using the arc
complex spectral sequence from the previous section.
To set things up in a general context, we define the notion of a $1$-controlled
coefficient system.
Recall that for $M$ a graded $R^V$ module, we use $\{M\}_n$ to denote the $n$th
graded part of $M$. We define the {\em degree of $M$}, which we notate by $\deg M$, to be
the supremum of all $n$ such that $\{M\}_n \neq 0$.
By definition of $\mathcal K(M_p^{V,F})$, the grading on $M$ induces a grading on 
$\mathcal K(M_p^{V,F})$, which descends to a grading on 
	$H_i(\mathcal K(M_p^{V,F}))$.
The idea is that modules for $1$-controlled coefficient systems have degrees of
their $i$th homologies controlled in terms of degrees of their $0$th and $1$st
homologies.
For $R$ a graded ring, we say an element is {\em homogeneous} if it lies in a
single degree of the grading of $R$.

\begin{definition}
	\label{definition:1-controlled}
	Recall that we denote
	\begin{equation}
		\label{equation:r-v-definition}
R^V :=\oplus_{n \geq 0} H_0(B^n_{0,0}, V_n).
\end{equation}
Note that there are isomorphisms
$V_{m+n} \simeq V_m \otimes V_n$, coming from the assumption that $V_n \simeq
(V_1)^{\otimes n}$.
These isomorphisms, together with the monoidal structure of $\Sigma^1_{0,0}$,
supplies $R^V$ with the structure of a graded ring supported in nonnegative gradings. 
	We assume a homogeneous element $\mathbb U \in R^V$ exists so that
	left multiplication by $\mathbb U$
	induces
	a map $\mathbb U: R^V \to R^V$. 
	We also assume $\mathbb U$ is non-invertible, which means that either $\mathbb U$
	has positive grading or $\mathbb U = 0$ (but $\mathbb U$ cannot be a
	non-zero element in grading $0$).
	We fix such a $\mathbb U$. 
	A coefficient system $V$ for $\Sigma^1_{0,0}$ is {\em
	$1$-controlled} (with respect to the operator $\mathbb U$)
	if $\deg H_0(\mathcal K(R^V))$ and $\deg H_1(\mathcal K(R^V))$ are
	finite and
there exists a constant $A_0(V) \geq 1$ 
so that for any left $R^V$-module $M$, 
	the following two properties hold:
	\begin{enumerate}
		\item We have
\begin{align*}
\deg H_i(\mathcal
K(M)) \leq \max(\deg H_0(\mathcal
K(M)),\deg H_1(\mathcal
K(M))) +A_0(V)i.
\end{align*}
\item The map induced by left multiplication by $\mathbb U$, denoted $\mathbb U: M \to M$,
	induces an isomorphism $\{M\}_n \simeq \{M\}_{n + \deg \mathbb U}$
	for 
\begin{align*}
n \geq  \max(\deg H_0(\mathcal K(M)),\deg H_1(\mathcal
K(M))) + A_0(V).
\end{align*}
	\end{enumerate}
\end{definition}

Next, we give a crucial example of a $1$-controlled coefficient system.
\begin{example}
	\label{example:hurwitz-is-1-controlled}
	Let $G$ be a group and $c \subset G$ be a conjugacy class in $G$.
	We will assume $(G, c)$ is
	non-splitting in the sense of \cite[Definition
	3.1]{EllenbergVW:cohenLenstra},
	meaning that $c$ generates $G$ and for every subgroup $H \subset G$, $H \cap
	c$ 
	either consists of a single conjugacy class in $H$ or is empty.
	Let $V := H_{\allinertia G c 0 0}$, as defined in
	\autoref{example:hurwitz-coefficient-system}.
	Note that the ring 
	$R^V$ 
	is called $R$ in the paper \cite{EllenbergVW:cohenLenstra}.
	As described in \cite[\S3.3]{EllenbergVW:cohenLenstra}, the ring $R^V$
	is generated in degree $1$ by elements of the form $r_g$ (corresponding
	to right multiplication by $g$) for $g \in c$.
	Consider the map $\mathbb U :=\sum_{g \in c} r_g^{\on{ord}(g)}$,
	where $\on{ord}(g)$ denotes the order of $g\in G$.
	We will show in \autoref{proposition:degree-order},
	that $\ker \mathbb U, \coker \mathbb U: R^V \to R^V$ both have finite
	degree.
		As we will show later, $R^V$ is $1$-controlled with respect to the
	operator $\mathbb U$ by
	\autoref{theorem:evw-stability}.\footnote{It was already shown in \cite[Lemma
		3.5]{EllenbergVW:cohenLenstra}
		that there is some value of $D$ such that, for 
		$\mathbb U_D := \sum_{g \in c} r_g^{D\on{ord}(g)}$, both 
$\ker \mathbb U_D, \coker \mathbb U_D$ have finite
	degree.
		It would have been fine for the purposes of this article, except
		for proving \autoref{theorem:frobenius-equivariance}, to use this
	operator $\mathbb U_D$ in place of $\mathbb U$.
	It is shown in 
	\cite[Theorem
	4.2]{EllenbergVW:cohenLenstra} that $R^V$ is $1$-controlled with
	respect to $\mathbb U_D$ for some $D$. 
	However, the proof given there only uses that $\mathbb U_D$ has kernel and
	cokernel of finite degree. Thus, we could also have referred to the proof of
\cite[Theorem 4.2]{EllenbergVW:cohenLenstra} in place of 
\autoref{theorem:evw-stability} for showing that $R^V$ is $1$-controlled with
respect to $\mathbb U$.}
\end{example}

The proof of this next result closely follows the proof of
\cite[Theorem 6.1]{EllenbergVW:cohenLenstra}. Indeed, it corresponds to an
abstraction of the first part of the proof of \cite[Theorem 6.1]{EllenbergVW:cohenLenstra}
to the setting of $1$-controlled coefficient systems.

\begin{theorem}
	\label{theorem:1-controlled-bound}
	Suppose $V$ is a $1$-controlled coefficient system for $\Sigma^1_{0,0}$
	and $F$ is a coefficient system for $\Sigma^1_{g,f}$ over $V$.
	Using notation as in \autoref{notation:homology-coefficient-system},
	assume moreover that 
	$\deg H_0(\mathcal K(M_0^{V,F}))$ and $\deg H_1(\mathcal
K(M_0^{V,F}))$ are finite.
Then, there exist constants $I(V),J(V, F)$ depending on $V$ and $F$ but not
	on $n$ or $p$
	so that
	$\mathbb U$ restricts to an isomorphism $\{M_p^{V,F}\}_n \to \{M_p^{V,F}\}_{n+\deg
	\mathbb U}$ whenever $n > I(V)p + J(V, F)$.
\end{theorem}

\begin{figure}
\includegraphics[scale=.5]{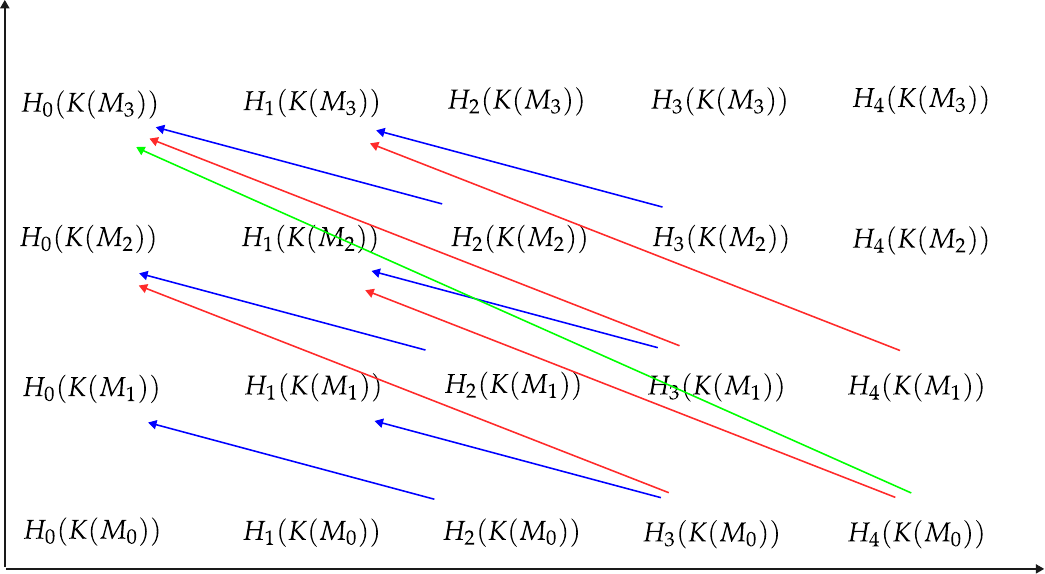}
\caption{We depict the spectral sequence coming from the arc complex.
	To avoid clutter in the picture, we write $\mathcal K(M_i)$ where we
	should write $\{\mathcal K(M_i^{V,F})\}_n$.
The entries here start on the $E^2$ page, so $E^2_{q,p} = H_q(\{\mathcal
K(M_p^{V,F})\}_n)$. The blue arrows depict the differentials on the $E^2$ page, the
red arrows depict the differentials on the $E^3$ page, and the green arrow
depicts a differential on the $E^4$ page.}
\label{figure:spectral-sequence-picture}
\end{figure}

\begin{proof}
Let $A_0(V)$ be the constant associated to $V$ from
\autoref{definition:1-controlled}.
By way of induction on $p$, we will prove there exists a nonnegative constant $C(F)$,
independent of $p,q,$ and $n$, so that
\begin{align}
\label{equation:ss-induction}
\deg H_q(\mathcal K(M_p^{V,F})) \leq C(F) + A_0(V)(3p+q).
\end{align}
for all $q \geq 0$.
Once we establish \eqref{equation:ss-induction}, we will obtain the result because, plugging in the cases
$q = 0$ and $q =1$, we get
\begin{align*}
\deg H_0(\mathcal K(M_p^{V,F})) &\leq C(F) + A_0(V)(3p) \\
\deg H_1(\mathcal K(M_p^{V,F})) &\leq C(F) + A_0(V)(3p+1).
\end{align*}
Hence, by \autoref{definition:1-controlled}(2), we find $\mathbb U$ restricts to an
isomorphism
$\{M_p^{V,F}\}_n \to \{M_p^{V,F}\}_{n+\deg \mathbb U}$
whenever 
\begin{align*}
n > C(F) + 2A_0(V) + 3A_0(V) p,
\end{align*}
and we can then take the constant $I(V) := 3A_0(V)$ and $J(V, F) :=
C(F) + 2A_0(V)$.

We first verify \eqref{equation:ss-induction}
for $p = 0$.
Indeed, let 
\begin{align*}
C(F) := \max( \deg H_0(\mathcal K(M_0^{V,F})),\deg H_1(\mathcal
K(M_0^{V,F})),1).
\end{align*}
By \autoref{definition:1-controlled}(1), we have
$\deg H_q(\mathcal K(M_0^{V,F})) \leq C(F) + A_0(V)q$. This amounts to
\eqref{equation:ss-induction} for the case $p = 0$.

We next assume the result holds for $p < P$, and aim to show it holds
for $P$.
It suffices to show 
\begin{equation}
\label{equation:P-homology-bound}
\begin{aligned}
\deg H_0(\mathcal K(M_P^{V,F})) &\leq C(F) + 3A_0(V)P \\
\deg H_1(\mathcal K(M_P^{V,F})) &\leq C(F) + 3A_0(V)P,
\end{aligned}
\end{equation}
as then
\autoref{definition:1-controlled}(1), implies 
\begin{align*}
	\deg H_q(\mathcal K(M_P^{V,F})) \leq C(F)+3A_0(V)P + A_0(V) q = C(F)+A_0(V) (3P + q),
\end{align*}
which is the inductive claim we
wished to prove.

We conclude by proving \eqref{equation:P-homology-bound}.
From \autoref{proposition:arc-complex-spectral-sequence},
we can identify $E^2_{q,p} \simeq H_q(\{\mathcal K(M_p^{V,F})\}_n)$.
Therefore, it is enough to show
$E^2_{0,P} = E^2_{1,P} = 0$
in degree $n > C(F) + 3A_0(V)P$.
The differential coming into $E^{2+i}_{q,P}$ comes from $E^{2+i}_{q+2+i,P-1-i}$,
see \autoref{figure:spectral-sequence-picture}.
By our inductive hypothesis, these vanish in degree $n > C(F) + A_0(V)(3(P-1-i)
+ (q+2+i))$. For the remainder of the proof, we use $q$ to denote either $0$ or $1$.
We can bound 
\begin{align*}
C(F) + A_0(V)(3(P-1-i) + (q+2+i)) &= C(F) + A_0(V)(3P-3 + 2 + q-2i) \\
&\leq
C(F) + A_0(V)(3P-1 + q) \\
&\leq C(F) + A_0(V)(3P).
\end{align*}
Hence, once the degree $n$ satisfies $C(F) + A_0(V)(3P) < n$, we find
$E^2_{q,P} = E^\infty_{q,P}$.
Finally, $E^\infty_{q,P} = 0$ so long as $P + q \leq n-1$, for $n$ the degree,
by \autoref{proposition:arc-complex-spectral-sequence}.
Once we verify $P + q \leq n - 1$ and $C(F) + 3A_0(V)P \leq n - 1$, we will
conclude
$E^2_{q,P} = 0$.
In particular, 
since 
we have assumed $C(F) \geq 1$, and $A_0(V) \geq 1$ holds by 
\autoref{definition:1-controlled},
we find $P + q \leq C(F) + 3A_0(V)P$, and so
\eqref{equation:P-homology-bound} holds.
\end{proof}

\subsection{A sufficient condition for homological stability}
\label{subsection:sufficient-condition}

We next set out to show that a wide variety of $V$ and $F$ satisfy the
hypotheses of 
\autoref{theorem:1-controlled-bound}.
We establish this in \autoref{theorem:central-u-implies-cohomology-bound}.
Before getting to this, we start by proving a generalization of
\cite[Theorem 4.2]{EllenbergVW:cohenLenstra}, given in
\autoref{theorem:evw-stability}.
For the purposes of our applications to the Poonen-Rains conjectures,
we do not need to prove 
\autoref{theorem:evw-stability},
as we will only apply 
\autoref{theorem:evw-stability} in the $V = k\{c\}$ for certain specific
conjugacy classes $c \subset G$ associated to Hurwitz spaces; in the relevant cases it was already shown in \cite[Theorem 4.2]{EllenbergVW:cohenLenstra}
that such $V$ are $1$-controlled.
However, we 
include this generalization as we believe it may be useful for approaching similar
homological stability problems in the future.

To start, we give a sufficient criterion
for a ring to be $1$-controlled in terms of a central operator $\mathbb U \in R^V$.
The following is the above mentioned generalization of
\cite[Theorem 4.2]{EllenbergVW:cohenLenstra}.

\begin{theorem}
\label{theorem:evw-stability}
Suppose $V$ is a coefficient system for $\Sigma^1_{0,0}$ 
and define $R^V$
as in \eqref{equation:r-v-definition}.
Suppose $\mathbb U \in R^V$ is a homogeneous non-invertible central element such that
$\deg \ker \mathbb U$ and $\deg \coker \mathbb U$ are both finite.
Then, $V$ is $1$-controlled.
\end{theorem}
\begin{proof}
This is essentially proved in \cite[Theorem 4.2]{EllenbergVW:cohenLenstra}.
While technically the ring $R$ used there is for a specific $V$ of the form
$k\{c\}$, the proof
generalizes to the case stated here, as we now explain.
Throughout the proof
of \cite[Theorem 4.2]{EllenbergVW:cohenLenstra},
one may replace $k\{c\}$ with $V_1$, and, for $M$ an $R^V$-module, one may then use our
definition of $\mathcal K(M)$ from \autoref{definition:k-complex}
in place of the definition in \cite[\S4.1]{EllenbergVW:cohenLenstra}.
The two parts of the proof of
\cite[Theorem 4.2]{EllenbergVW:cohenLenstra}
whose generalization requires some thought are
the content of \cite[p. 755]{EllenbergVW:cohenLenstra}, where one wishes to
establish the bound
$\deg \on{Tor}^1_{R^V}(k, M) \leq \deg H_1(\mathcal K(M))$,
as well as
\cite[Lemma 4.11]{EllenbergVW:cohenLenstra}.
Both of these refer to specific elements of the ring $R$ in
\cite{EllenbergVW:cohenLenstra}, which is related to
Hurwitz stacks.

The only step of \cite[p. 755]{EllenbergVW:cohenLenstra}
where one cannot easily replace elements of $k\{c\}$ with elements of $V_1$ is
in the third to last paragraph.
To explain why this still holds, let $\alpha: V_1 \otimes_k M[1] \to \{R^V\}_{>0}
\otimes_{R^V} M$
denote the map sending $v \otimes m \mapsto [v] \cdot m$,
where $[v]$ denotes the class of $v$ in $\{R^V\}_1 = H_0(B^1_{0,0}, V_1) \simeq V_1$, and
$[v] \cdot m$ denotes the multiplication using that $M$ is an $R^V$-module.
For $x \in V_n$, we similarly use $[x] \cdot m$ to denote the product of the
class of $x$ in $\{R^V\}_n$ with $m$.
To establish the third to last paragraph of
\cite[p. 755]{EllenbergVW:cohenLenstra}, we wish to verify that the composite map
\begin{align*}
V_1^{\otimes 2} \otimes_k M[2] \xrightarrow{d} V_1 \otimes_k M[1]
\xrightarrow{\alpha} \{R^V\}_{>0} \otimes_{R^V} M
\end{align*}
vanishes. 
For $v \otimes w \in V_1^{\otimes 2}$, if $\tau: V_1^{\otimes 2} \to
V_1^{\otimes 2}$ denotes the isomorphism giving $V_1$ the structure of a braided
vector space, corresponding to a generator of $B^2_{0,0}$, we obtain that
$(\alpha \circ d)(v \otimes w \otimes m) = [v \otimes w] \cdot m -
[\tau(v\otimes w)] \cdot m$. This is equal to $0$ because
$[v \otimes w] = [\tau(v \otimes w)]$ as elements of $\{R^V\}_2 = H_0(B^2_{0,0},
V_2)$: Indeed, a generator of $B^2_{0,0}$ acts via $\tau$ on $V_2 \simeq V_1^{\otimes
2}$, so taking coinvariants via $H_0$ identifies 
$[v \otimes w]$ and $[\tau(v \otimes w)]$.

To conclude, it remains to prove the analog of \cite[Lemma 4.11]{EllenbergVW:cohenLenstra},
which we do in \autoref{lemma:multiplication-null-homotopic}.
\end{proof}

\begin{lemma}
\label{lemma:multiplication-null-homotopic}
For $V$ a coefficient system, the action of $\{R^V\}_{>0}$ on $H_q(\mathcal K(R^V))$ is $0$.
\end{lemma}
\begin{proof}
We generalize the proof of the analogous statement given in \cite[Lemma
4.11]{EllenbergVW:cohenLenstra}.
Since $R^V$ is generated in degree $1$, it is enough to prove
that for every $v \in V_1$, right multiplication by $[v]$ nullhomotopic. 
Start with some element $v_1 \otimes \cdots \otimes v_q \otimes s \in \{\mathcal
K(R^V)_q\}_{n+q} =
V_1^{\otimes q} \otimes H_0(B^{n}_{0,0}, V_n)$.
For any $v \in V_1$, define the linear operator 
\begin{align*}
S_v:\mathcal K(R^V)_q & \rightarrow \mathcal K(R^V)_{q+1} \\
v_1 \otimes \cdots \otimes v_q \otimes \widetilde{s} & \mapsto
\overline{(\tau^{q+n+1}_{1,q+n+1})^{-1}(v_1 \otimes \cdots \otimes v_q \otimes \widetilde{s}
\otimes v)},
\end{align*}
with notation as follows:
we use
$\widetilde{s}$ to denote a lift of $s$ from $H_0(B^n_{0,0}, V_n)$ to $V_n$,
for $x \in V_1^{\otimes q+n+1},$ we use $\overline{x}$ for the image in
$V_1^{\otimes q+1} \otimes H_0(B^n_{0,0}, V_n)$,
and we use notation as in \autoref{definition:k-complex} for
$(\tau^{q+n+1}_{1,q+n+1})^{-1}$, which corresponds to pulling $v$ from the last factor
of the tensor product to the first factor.
First, we need to verify this map is independent of the choice of lift
$\widetilde{s}$ of $s$.
If we chose a different lift $\widetilde{s}'$, we can write
$\widetilde{s}' = \sigma \widetilde{s}$ for some $\sigma \in B^n_{0,0}$.
Writing $\sigma$ as a product of generators, we may assume
$\sigma = (\tau^n_i)^{-1}$. Now, for $n \leq m$ and $i \leq m - n$, define
$\iota_{n,m,i}: B^n_{0,0} \to B^m_{0,0}$ as the inclusion sending $n$
strands of $B^n_{0,0}$ to strands in the range $[i+1, \ldots, i+n]$.
More formally, this can be realized in terms of
\autoref{notation:surface-braid-groups} as the inclusion 
\begin{align*}
B^n_{0,0} \to 
B^i_{0,0} \times B^n_{0,0} \times B^{m-i-n}_{0,0} \to B^{i}_{0,0} \times
B^{m-i}_{0,0} \to B^m_{0,0},
\end{align*}
where the first map is the inclusion to the second component, the second map is
the product of $B^i_{0,0}$ with the map of braid groups associated to the
inclusion $X^{\oplus n} \coprod X^{\oplus m-i-n} \oplus A_{0,0} \to 
X^{\oplus m-i} \oplus A_{0,0}$, and the third map is the map of braid groups
associated to the inclusion
$X^{\oplus i} \coprod X^{\oplus m-i} \oplus A_{0,0} \to 
X^{\oplus m} \oplus A_{0,0}$.
In order to show $S_v$ is well defined, we first observe
\begin{equation}
\begin{aligned}
	\label{equation:braid-identity}
	(\tau^{q+n+1}_{1,q+n+1})^{-1} \iota_{n,q+n+1,q}((\tau^n_i)^{-1})
	&= (\tau^{q+n+1}_1)^{-1} \cdots (\tau^{q+n+1}_{n+q})^{-1}
	(\tau^{q+n+1}_{q+i})^{-1}
	\\
	&= (\tau^{q+n+1}_{q+i+1})^{-1} (\tau^{q+n+1}_1)^{-1} \cdots (\tau^{q+n+1}_{n+q})^{-1}
	\\
	&= \iota_{n,q+n+1,q+1} ((\tau^n_{i})^{-1})
	(\tau^{q+n+1}_{1,q+n+1})^{-1}.
\end{aligned}
\end{equation}
The first and last equalities in 
\eqref{equation:braid-identity} are purely definitional. The second equality
expresses the relation in $B^{n+q+1}_{0,0}$ that if one first pulls strand
$q+i+1$ to position $q+i$ and then pulls strand $q+n+1$ to position $1$, this is
the same as first pulling strand $q+n+1$ to position $1$ and then pulling strand
$q+i+2$ to position $q+i+1$.
The well definedness of $S_v$ follows from 
\eqref{equation:braid-identity}
applied to 
$v_1 \otimes \cdots \otimes v_q \otimes \widetilde{s} \otimes v$, as we now
explain.
Indeed,
\eqref{equation:braid-identity} shows $S_v$ sends 
$v_1 \otimes \cdots \otimes v_q \otimes \widetilde{s}' \otimes v =
\iota_{n,q+n+1,q}((\tau^n_i)^{-1})
(v_1 \otimes \cdots \otimes v_q \otimes \widetilde{s} \otimes v)$
to the same element of $K(R^V)_{q+1}$ that it sends
$v_1 \otimes \cdots \otimes v_q \otimes \widetilde{s} \otimes v$ to, since it shows their images
in $V_1^{\otimes q + 1 + n}$ are related by 
$\iota_{n,q+n+1,q+1} ((\tau^n_{i})^{-1})$.

Since $R^V$ is generated in degree $1$, it is enough to prove
right multiplication by $[v]$ nullhomotopic. 
Having shown that $S_v$ is well defined, 
we now compute
\begin{align*}
(S_v d + d S_v)(v_1 \otimes \cdots \otimes v_q \otimes s)
&= (\id^{\otimes q} \otimes \mu_n) \left( \overline{\tau^{q+n+1}_{1,q+1} 
(\tau^{q+n+1}_{1,q+n+1})^{-1}(v_1 \otimes \cdots \otimes v_q \otimes \widetilde{s}
\otimes v)}\right) \\
&= (\id^{\otimes q} \otimes \mu_n) \overline{(\tau^{q+n+1}_{q+1,q+n+1})^{-1}(v_1 \otimes \cdots \otimes v_q \otimes \widetilde{s}
\otimes v)} \\
&= (v_1 \otimes \cdots \otimes v_q) \otimes \mu_n
[(\tau^{n+1}_{1,n+1})^{-1}(\widetilde{s} \otimes v)] \\
&= v_1 \otimes \cdots \otimes v_q \otimes (s \cdot [v]),
\end{align*}
which shows right multiplication by $[v]$ is nullhomotopic.
\end{proof}

We next observe that $R^V$ is noetherian.
A similar argument in the context of Hurwitz stacks was given in
\cite[Proposition 3.31]{davisS:the-hilbert-polynomial-of-quandles}
and also
\cite[Lemma 3.3]{bianchiM:polynomial-stability}.

\begin{lemma}
\label{lemma:noetherian}
Let $V$ be a coefficient system for $\Sigma^1_{0,0}$ with each $V_i$ finite-dimensional.
Suppose 
$R^V$
has some homogeneous non-invertible $\mathbb U \in R^V$
so that 
$\deg \coker \mathbb U$ is finite.
Then $R^V$ is noetherian.
\end{lemma}
\begin{proof}
Note that $R^V$ is not commutative. However, we claim $R^V$ is a finite module
over a commutative finitely generated ring, hence noetherian.
Let $R_{\mathbb U} \subset R^V$ denote the commutative subring generated by $\mathbb U$ over $k$.
We claim $R^V$ is a finite module over $R_{\mathbb U}$.
We will in fact show that $R^V$ is generated over $R_{\mathbb U}$ by all elements of degree at most
$\deg \coker \mathbb U$. Since each $V_i$ is finite dimensional, this will imply that $R^V$ is
finitely generated over $R_{\mathbb U}$.
To prove our claim, by induction on the homogeneous degree of an element, it is enough to show that any homogeneous element 
$r \in R^V$ with $\deg r \geq \deg \coker \mathbb U$ can be written in the form $s + \mathbb Ut$ for
$\deg s \leq \deg \coker \mathbb U$ and $\deg t < \deg r$.
Indeed, consider the image $\overline{r} \in R^V/\mathbb UR^V$. Because $R^V/\mathbb UR^V = \coker \mathbb U$
has finite degree, there is some element $s \in R^V$ of degree at most $\deg
\coker \mathbb U$ so that $r-s= 0 \in R^V/\mathbb UR^V$. This implies $r - s = \mathbb Ut$ for some $t \in
R^V$. 
Since $\mathbb U$ is non-invertible, it either has positive degree or $\mathbb U
= 0$. Therefore, we can write $r = s + \mathbb Ut$ with $\deg t < \deg r$, using
that $\mathbb U$ is not a unit in degree $0$,
and $\deg s \leq \deg \coker \mathbb U$.
\end{proof}

Using noetherianness of $R^V$, we can also prove the other hypotheses of
\autoref{theorem:1-controlled-bound} hold for finitely generated $R^V$-modules.

\begin{lemma}
\label{lemma:degree-bound-for-km0}
Let $V$ be a coefficient system for $\Sigma^1_{0,0}$ with each $V_i$ finite
dimensional.
Suppose 
$R^V$
has some homogeneous non-invertible central $\mathbb U \in R^V$
so that $\deg \coker \mathbb U$ is finite.
Then, if $N$ is a finitely generated module over $R^V$, both $H_0(\mathcal
K(N))$ and $H_1(\mathcal K(N))$ have finite degree.
\end{lemma}
\begin{proof}
First, since $R^V$ is generated in degree $1$, $H_0(\mathcal K(N)) = N/\im(V_1
\otimes N \to N) = N/ \oplus_{n > 0} \{R^V\}_n N$, and this quotient is supported in
the degrees of generators of $N$ over $R^V$. Therefore, $N$ is generated in degree at most $d$ if and only if $\deg
H_0(\mathcal K(N)) \leq d$.

Next, we show $\deg H_1(\mathcal K(N))$ is finite.
Since $\mathcal K(N) = \mathcal K(R^V) \otimes_{R^V} N$, there is a spectral
sequence
$\tor_i^{R^V}(H_j(\mathcal K(R^V)), N) \implies H_{i+j}(\mathcal
K(N))$.
By the low degree terms exact sequence coming from the spectral sequence, in
order to bound $\deg H_1(\mathcal K(N))$, it is enough to bound 
$\deg \tor_0^{R^V}(H_1(\mathcal K(R^V)), N)$ and
$\deg \tor_1^{R^V}(H_0(\mathcal K(R^V)), N)$.
By \autoref{theorem:evw-stability}, 
$H_0(\mathcal K(R^V))$ and $H_1(\mathcal K(R^V))$ have finite degree.
We will conclude by showing more generally that for $P$ any $R^V$-module of finite
degree, 
$\deg \tor_i^{R^V}(P, N)$ has finite degree.
By noetherianness of $R^V$, as established in \autoref{lemma:noetherian}, we may choose a free resolution of the finite $R^V$-module $N$ of the form $\cdots \to S_2 \to S_1 \to N$ where each term $S_i$ is a finite free
$R^V$-module, hence of finite degree. 
Since we are assuming $P$ has finite degree,
applying $P \otimes_{R^V}$ to this
resolution and taking homology shows that $\tor_i^{R^V}(P, N)$ also has finite
degree, for any fixed value of $i$.
\end{proof}

We next show that in the case $N = M_0^{V,F}$, the finite generation hypothesis
of \autoref{lemma:degree-bound-for-km0}
is automatic.
\begin{lemma}
	\label{lemma:finite-generation-mp}
	Suppose $V$ is a coefficient system for $\Sigma^1_{0,0}$.
	If $F$ is a coefficient system for $\Sigma^1_{g,f}$ over $V$ and $F_0$
	is finite dimensional, then $M_0^{V,F}$ is
finitely generated as a $R^V$-module.
\end{lemma}
\begin{proof}
	We may view $M_0^{V,F}$ as an $R^V$-module via
	\autoref{remark:mp-module-structure}.
	Via the inclusion $B^n_{0,0} \to B^n_{g,f}$ from
	\autoref{notation:surface-braid-groups}, there is a surjection
	$H_0(B^n_{0,0}, F_n) \to H_0(B^n_{g,f}, F_n)$.
	We therefore obtain a surjection of graded modules 
	\begin{align*}
	\oplus_{n \geq 0}
	H_0(B^n_{0,0}, F_n) \to \oplus_{n \geq 0}
	H_0(B^n_{g,f}, F_n) = M_0^{V,F}.
	\end{align*}
	Hence, it is enough to show 
	$\oplus_{n \geq 0} H_0(B^n_{0,0}, F_n)$ is finitely generated as an
	$R^V$-module. 
	Indeed, since $B^n_{0,0}$ acts trivially on $F_0$,
	\begin{align*}
		\oplus_{n \geq 0} H_0(B^n_{0,0}, F_n) \simeq \left (\oplus_{n \geq 0}
		H_0(B^n_{0,0}, V_n) \right) \otimes F_0 = R^V \otimes F_0,
	\end{align*}
and so the
	desired finite generation holds because $F_0$ is a finite dimensional vector
	space.
\end{proof}

Combining our work above, we obtain that if we have coefficient systems $V$ and
$F$,
and $R^V$ has a central homogeneous non-invertible element with finite degree kernel and cokernel, then
\autoref{theorem:1-controlled-bound} applies.

\begin{theorem}
\label{theorem:central-u-implies-cohomology-bound}
Suppose $V$ is a coefficient system for $\Sigma^1_{0,0}$ and $\mathbb U \in R^V$
is a homogeneous non-invertible central element such that $\deg \ker \mathbb U$ and
$\deg \coker \mathbb U$ are both finite.
Suppose $F$ is a coefficient system for $\Sigma^1_{g,f}$ over $V$.
Assume each $V_i$ and each $F_i$ are finite-dimensional.
Then, there exist constants $I(V)$ and
$J(V, F)$ independent of $p$ and $n$ so that $\mathbb U$ induces an isomorphism
$\{M_p^{V,F}\}_n \to \{M_p^{V,F}\}_{n+\deg \mathbb U}$ whenever $n > I(V)p +
J(V, F)$.
\end{theorem}
\begin{proof}
This follows from \autoref{theorem:1-controlled-bound}, once we verify its
hypotheses.
We find $V$ is $1$-controlled by \autoref{theorem:evw-stability}.
From \autoref{lemma:finite-generation-mp}, 
$M_0^{V,F}$ is finitely generated as an $R^V$-module.
By \autoref{lemma:degree-bound-for-km0}, it follows that $H_0(\mathcal K(M_0^{V,F}))$
and
$H_1(\mathcal K(M_0^{V,F}))$ both have finite degree.
\end{proof}

\begin{remark}
	\label{remark:}
	Via private communication with Oscar Randal-Williams,
	it seems likely that one may be able to prove
	\autoref{theorem:central-u-implies-cohomology-bound} using a setup
	similar to that in \cite{randal-williams:homology-of-hurwitz-spaces}.
	However, this is by no means obvious, and we believe it would be very
	interesting to work out the details.
	In particular, one of the trickiest parts to generalize is
	\cite[Proposition 8.1]{randal-williams:homology-of-hurwitz-spaces} where
	it is used that the bar construction $B(k,A,A) = k$. In our setting we need to instead
	analyze $B(k,A,M)$, for a suitable value of $M$ in
	place of $A$.

	Several years later, an anonymous referee suggested that perhaps some
	more recent results such as \cite[Theorem
	3.1.4]{landesmanL:homological-stability-for-hurwitz} could be useful for
	providing a simplification of the proof of 
	\autoref{theorem:central-u-implies-cohomology-bound}.
	It wasn't immediately obvious to us how to apply this, but Ishan Levy
	explained to us what seems to be a promising path to proving this,
	including how one might approach the above mentioned computation of
	$B(k,A,M)$ using factorization homology.
	Unfortunately, even if all the details were to work out, it appears that verifying the
	conditions of 
	\cite[Theorem 3.1.4]{landesmanL:homological-stability-for-hurwitz} in
	the setting of 
\autoref{theorem:central-u-implies-cohomology-bound} would likely take a similar
amount of space to our current proof of 
\autoref{theorem:central-u-implies-cohomology-bound}.
Since it also would require significant additional background on the part of the reader,
we have not pursued this further.

We do note one important special case of 
\autoref{theorem:central-u-implies-cohomology-bound} 
which can be deduced from recent work, which appeared many years after this
paper first appeared.
If one is only interested in the case that $V$ is a coefficient
system coming from a conjugacy class in a group, (or more generally, a rack,) and the coefficient system $F$ corresponds to a bijective Hurwitz module in the
sense of \cite[Definition 2.1.2]{landesmanL:stable-homology-of-hurwitz-modules}, one can deduce a version of 
\autoref{theorem:central-u-implies-cohomology-bound} from
\cite[Theorem 3.1.4]{landesmanL:homological-stability-for-hurwitz}, whose
conditions can be verified using 
\cite[Lemma 3.2.8]{landesmanL:homological-stability-for-hurwitz} and
\cite[Proposition 4.0.2]{landesmanL:stable-homology-of-hurwitz-modules}.
We note that the above special case will not suffice for the purposes of this paper, as we
will need to consider coefficient systems which are more general than bijective
Hurwitz modules, such as in \autoref{example:rank-coefficient-system} and
\autoref{example:fiber-product-with-rank-coefficients}.
\end{remark}

\subsection{An exponential bound on the cohomology}
\label{subsection:exponential-bound}

Our main application of the above homological stability results to the BKLPR
heuristics comes from the bound on cohomology in
\autoref{proposition:cohomology-bound}, and the corresponding consequence
\autoref{corollary:cohomology-bound}.
There are two inputs. The first is our above homological stability
results. The other is a bound on the CW structure of configuration space.

We now give this second bound, which nearly appears in
\cite[\S4.2]{bianchiS:homology-of-configuration-spaces} in the case that $f =
0$. We now give the straightforward generalization to the case of arbitrary $f$.
We will be brief here, but 
encourage the reader to consult \cite[\S4.2]{bianchiS:homology-of-configuration-spaces}
for further details.
We thank Andrea Bianchi for suggesting the following approach.

\begin{lemma}
\label{lemma:cell-bound}
For $g, f, n \geq 0$,
the space $\Conf^n_{\Sigma^{1}_{g,f}}$, parameterizing $n$ unordered points in
the interior of $\Sigma^{1}_{g,f}$,
has $1$-point compactification with a cell decomposition
possessing at most $2^{2g+f+n}+1$ cells, one of which is the $0$-cell
of the added point in the $1$-point compactification.
\end{lemma}

\begin{figure}
\includegraphics[scale=.8]{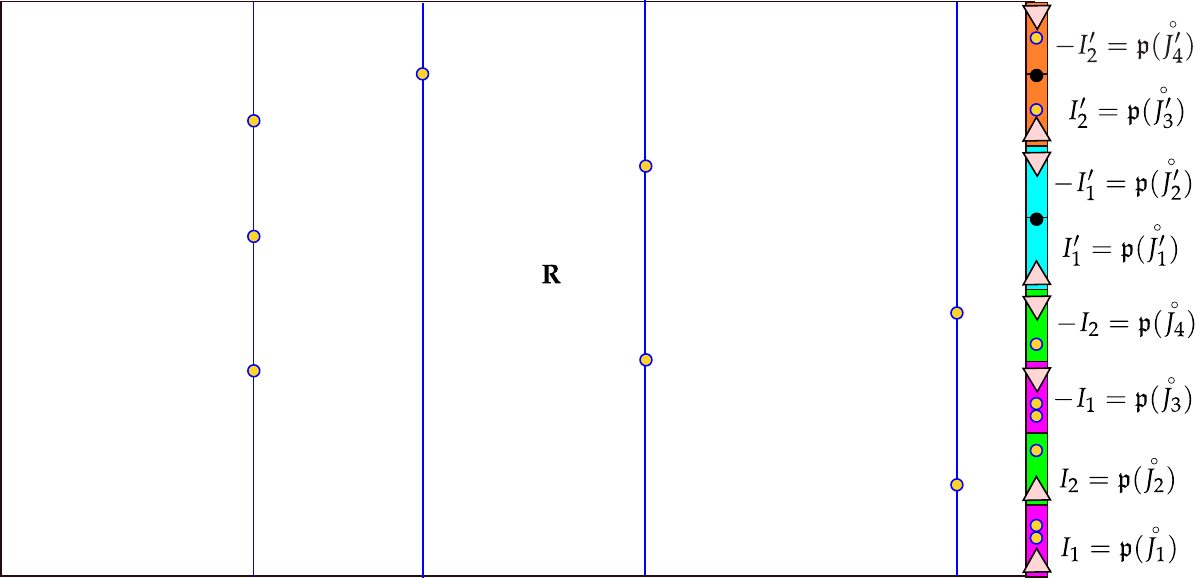}
\caption{This picture depicts a cell in the configuration space
$\Conf^{12}{\Sigma^1_{1,2}}$. 
The boundary component
consists of the union of the upper, left, and lower edges.
The arrows indicate the orientations of the segments of the edges. 
Note that the segments of the same color are glued to each other with the
orientations indicated, so there are only $4$ distinct points represented by the yellow dots on the right boundary despite
the fact that there are $8$ yellow dots on the right boundary in the picture.
The two black
dots indicate the two punctures comprising $W$. The yellow dots indicate the $12$ points in
configuration space.
The cell is labeled by the $12$-tuple
$\mathfrak t = ((3,1,2,2), (2,1),(0,1))$
with $b = 4$.
}
\label{figure:configuration-cell}
\end{figure}

\begin{proof}
	The idea is to generalize the construction of 
	\cite[\S4.2]{bianchiS:homology-of-configuration-spaces}
	to the case that $f > 0$ as follows. We modify their setup so that the
	right edge of their rectangle $\mathbf R$
	includes the intervals $I_1, I_2, -I_1, -I_2, I_3, \ldots, I_{2g},
	-I_{2g-1}, -I_{2g}$, as in the case $f = 0$,
	and then additionally includes the intervals $I'_1, -I'_1, I'_2, -I'_2,
	\ldots, I'_f, -I'_f$ from bottom to top, see
	\autoref{figure:configuration-cell}.

	We now spell this out in some more detail, reviewing the notation of
	\cite[\S4.2]{bianchiS:homology-of-configuration-spaces}.
	First, we describe $\Sigma^1_{g,f}$ as a quotient in a particular way,
	which will be useful for describing a cellular structure on the one
	point compactification of its configuration space.
	Let $\mathbf R :=[0,2] \times [0,1]$ be a rectangle.
	Decompose the side $\{2\} \times [0,1]$ into $4g + 2f$ consecutive
	intervals of equal length $J_1, \ldots, J_{4g}, J'_1, \ldots, J'_{2f}$
	ordered and oriented with increasing second coordinate, as in
	\autoref{figure:configuration-cell}.
	Let $W$ be the set of the $f$ points consisting of the larger endpoint of $J'_{2i+1}$ for $0 \leq i \leq f - 1$.
	Let $\mathbf R - W$ denote the punctured rectangle where we remove
	$W$.
	Let $\mathcal M$ denote the quotient of $\mathbf R -W$ obtained by
	identifying $J_{4i+1}$ with $J_{4i+3}$, $J_{4i+2}$ with $J_{4i+4}$, and
	$J'_{2j+1}$ with $J'_{2j+2}$
	via their unique orientation reversing isometry
	for $0 \leq i \leq g-1$ and $0 \leq j \leq f - 1$ .
	Let $\mathfrak p : \mathbf R- W \to \mathcal M$ denote the quotient map.
	Then, $\mathcal M$ is homeomorphic to $\Sigma^1_{g,f}$.

	We next give a 
	description of the cellular structure of $\mathcal M$.
	This cellular structure is not a CW complex structure, but rather a more
	general open cell structure, which is just a tool to describe the CW
	cell structure on the one point compactification of configuration
	spaces on $\mathcal M$.
	Throughout, for $X$ a topological manifold with boundary, we will use $\accentset{\circ}{X}$ to denote the interior of $X$.
	\begin{enumerate}
	\item 	The space $\mathcal M$ has a single $0$ cell $p_0$, which is the image of any of the
	endpoints of the $J_i$, and is also identified with the larger endpoint of 
	$J'_{2j+2}$.
	\item
	The space $\mathcal M$ has $2g+f+1$ one-cells, described as follows.
	There are the $1$-cells $I_{2i+j}$, where $I_{2i+j} :=\mathfrak
	p\left(\accentset{\circ}{J}_{4i+j}\right)$ with $0 \leq i \leq g - 1$
	and $j \in\{1,2\}$. There are the $1$-cells $I'_i = \mathfrak
p\left(\accentset{\circ}{J}'_{2i+1}\right)$ for $0 \leq i \leq
	f-1$. Finally, there is $I = \mathfrak p(\partial \mathbf R -
	\{2\} \times [0,1])$.
	\item 
	Finally, $\mathcal M$ has one $2$-cell which is 
	$\mathfrak p(\accentset{\circ}{\mathbf R}).$
	\end{enumerate}
	For $1 \leq i \leq g-1$ and $1 \leq j \leq 2$, we let $\iota_{2i+j}: (0,1) \to \mathcal M$ denote the
	composition of $\mathfrak p$ with the linear map 
	sending $(0,1) \to \accentset{\circ}{J}_{4i+j}$.
	We let
	$\iota'_i: (0,1) \to \mathcal M$ denote the composition of $\mathfrak p$
with the linear map sending $(0,1) \to \accentset{\circ}{J}'_{2i+1}$ for
	$0 \leq i \leq f -1$.
	(This notation differs from that of
	\cite[\S4.2]{bianchiS:homology-of-configuration-spaces}, but it is
	slightly more
	convenient for our purposes.)

	We next introduce notation to define the cells in the CW complex we will
	construct.
	For $n \geq 0$, an $n$-tuple, which we denote by $\mathfrak t$,
	consists of
	\begin{enumerate}
	\item an integer $b \geq 0$
	\item a sequence $\underline{P} = (P_1, \ldots, P_b)$ of positive
	integers
	\item a sequence $\mathfrak v = (v_1, \ldots, v_{2g})$ of non-negative
	integers
	\item a sequence $\mathfrak w = (w_1, \ldots, w_f)$ of non-negative
	integers
	\end{enumerate}
	such that $P_1 + \cdots + P_b + v_1 + \cdots v_{2g} +w_1 + \cdots +w_f =
	n$. The above data will index ways to split up $n$ points, representing a
	point of $\Conf^n_{\Sigma^1_{g,f}}$, into different cells of
	$\mathcal M$.

	We next define the cells determining a CW structure for the one point
	compactification of $\Conf^n_{\Sigma^1_{g,f}}$.
	We write $\mathfrak t = (b, \underline{P}, \mathfrak v, \mathfrak w)$
	and use the notation for our surface $\mathcal M$ described above.
	For $\mathfrak t$ an $n$-tuple, let $e_{\mathfrak t}$ denote the subset
	of $[S] \in \Conf^n_{\Sigma^1_{g,f}}$ (which we recall parameterizes points in
	the interior of $\Sigma^1_{g,f}$)
	which satisfies the following conditions.
	\begin{enumerate}
	\item For $1 \leq i \leq 2g$, $v_i$ points lie on $I_i$.
	\item For $1 \leq i \leq f$, $w_i$ points lie in $I'_i$.
	\item There are exactly $b$ real numbers $0 < x_1 < \cdots < x_b < 2$
	such that $S$ admits at least one point in $\accentset{\circ}{\mathbf
	R}$ having $x_i$ as a coordinate.
	\item For all $1 \leq i \leq b$, exactly $P_i$ points of $S$ which lie
	in $\accentset{\circ}{\mathbf R}$ have first coordinate equal to $x_i$.
	\end{enumerate}
	Each $[S] \in \Conf^n_{\Sigma^1_{g,f}}$ lies in a unique subspace
	$e_{\mathfrak t}$.
	Given an $n$-tuple $\mathfrak t$, the space $e_{\mathfrak t}$ is
	homeomorphic to an open disc. Let $d(\mathfrak t)$ denote the dimension
	of this disc.
	Let $\Delta^k$ denote the standard $k$-dimensional simplex.
	Define $\Delta^\mathfrak t := \Delta^b \times \prod_{i=1}^b \Delta^{P_i}
	\times \prod_{i=1}^{2g} \Delta^{v_i} \times \prod_{i=1}^f \Delta^{w_i}$.
	Using 
	$\Conf^{d(\mathfrak t)}_{\Sigma^1_{g,f}} \cup \{\infty\}$
	to denote the $1$-point compactification,
	for $\mathfrak t$ an $n$-tuple,
	define the map $\Phi^{\mathfrak t}$ given in simplicial coordinates by
	\begin{align*}
	&\Phi^{\mathfrak t} : \Delta^{\mathfrak t} \to
	\Conf^{d(\mathfrak t)}_{\Sigma^1_{g,f}} \cup \{\infty\} \\
	&\left( (z_i)_{1 \leq i \leq b}, (s^{(i)}_j)_{1 \leq i \leq b, 1 \leq j
	\leq P_i }, (t^{(i)}_j)_{1 \leq i \leq 2g, 1 \leq j \leq v_i},
	(r^{(i)}_j)_{1 \leq i \leq f, 1 \leq j \leq w_i}
	\right) \\
	&\mapsto
	\left[ \mathfrak p(2z_j, s_j^{(i)}) : 1 \leq i \leq b, 1 \leq j \leq P_i \right]
	\cdot
	\left[ \iota_i(t_j^{(i)}) : 1 \leq i \leq 2g, 1 \leq j \leq v_i \right]
	 \\
	& \qquad
	\cdot\left[ \iota'_i(r_j^{(i)}) : 1 \leq i \leq f, 1 \leq j \leq w_i \right],
	\end{align*}
	where $\cdot$ denotes the superposition product.
	The map above should be interpreted as taking the value $\infty$
	whenever the image does not lie in the configuration space of points in
	$\mathcal M$; that is, whenever the configuration obtained by applying
	$\Phi^{\mathfrak t}$ to some input would result in two points occupy the same position, or
	a point lying on the boundary or a puncture, the map $\Phi^{\mathfrak t}$ instead
	takes value $\infty$.
	The map $\Phi^{\mathfrak t}$ restricts to a homeomorphism sending the
$\accentset{\circ}{\Delta}^{\mathfrak t} \to e_{\mathfrak t}$ and the boundary
	$\partial \Delta^{\mathfrak t}$
	to the union of $\{\infty\}$ and some of the subspaces $e_{\mathfrak t'}$ where
	$d(\mathfrak t') < d(\mathfrak t)$.

	Let us briefly explain why 
	$d(\mathfrak t') < d(\mathfrak t)$ for the subspaces $e_{\mathfrak t'}$
	described above.
	The rough idea is that each codimension-$1$ face $e_{\mathfrak t'}$ of $\Delta^{\mathfrak
	t}$ not sent to $\infty$ corresponds to decreasing the value of $b$ by
	$1$ (either by colliding two of the columns in the interior of $\mathcal
		M$ or by colliding the rightmost column in the interior with the
	right side of the boundary). This means the number of horizontal degrees of freedom decreases by
	$1$ while the number of vertical degrees of freedom stays the same, and
	so, for this value of $\mathfrak t'$, $d(\mathfrak t')  = d(\mathfrak t) - 1$. Similar reasoning applies to higher codimension
	cells $e_{\mathfrak t'}$ as well.
	
	As in \cite[Proposition 4.4]{bianchiS:homology-of-configuration-spaces},
	one may verify the $e_{\mathfrak t}$ together with $\infty$ form a cell
	decomposition for the one point compactification of
	$\Conf^n_{\Sigma^1_{g,f}}$.
	In contrast to the cell structure on $\mathcal M$ described earlier,
	this cell decomposition of the one point compactification of 
	$\Conf^n_{\Sigma^1_{g,f}}$ is actually a CW complex structure.

	Finally, we bound the number of cells in this structure by
	$2^{n+2g + f}$.
	Note that the number of cells is $1$ more than the
	number of $n$-tuples $\mathfrak t$. (We add $1$ to account for the
	$0$-cell at $\infty$.)
	Such an $n$-tuple can equivalently be described by a choice of $b$, and a
	collection of non-negative integers $P_1-1, \ldots, P_b - 1, v_1,
	\ldots, v_{2g}, w_1, \ldots, w_f$ summing to $n-b$.
	By ``stars and bars,'' such collections of integers are in bijection with
	subsets of $\{1, \ldots, (n-b) + (b+2g+f)\} = \{1, \ldots,  n + 2g + f\}$ of size 
	$b + 2g + f$.
	Varying over different possible values of $b$ yields that the total
	number of cells is equal to the number of subsets of $\{1, \ldots,
	n+2g+f\}$ of size at least $2g + f$.
	This is at most the number of subsets of 
	$\{1, \ldots,n+2g+f\}$, which is $2^{n+2g + f}$. Adding $1$ for the
	$0$-cell at $\infty$ gives that there are at most $2^{n+2g+f}+1$ cells, as
	we wished to show.
\end{proof}

As an easy consequence of the above bound on the number of cells, we obtain the
following bound on homology.
\begin{lemma}
\label{lemma:exponential-bound-for-all-cohomology}
Suppose $V$ is a coefficient system for $\Sigma^1_{0,0}$ and $F$ is a
coefficient system for $\Sigma^1_{g,f}$ over $V$.
If each $V_i$ and $F_i$ is finite-dimensional,
$\dim H_i(B^n_{g,f}, F_n) \leq 2^{2g+f+n} \cdot \dim F_n.$
\end{lemma}
\begin{proof}
Since 
$B^n_{g,f} \simeq \pi_1(\Conf^n_{\Sigma^1_{g,f}})$,
the representation $F_n$ of $B^n_{g,f}$ corresponds to a local system
$\mathbb F_n$ on $\Conf^n_{\Sigma^1_{g,f}}$
If 
$\Conf^n_{\Sigma^1_{g,f}} \cup \{\infty\}$ denotes the $1$-point
compactification and
$j : \Conf^n_{\Sigma^1_{g,f}} \to \Conf^n_{\Sigma^1_{g,f}} \cup \{\infty\}$,
denotes the inclusion,
we have an isomorphism between the compactly supported cohomology and sheaf
cohomology of the extension by $0$ sheaf
\begin{align}
\label{equation:compactly-supported-cohomology}
H^i_{\on{c}}(\Conf^n_{\Sigma^1_{g,f}}, \mathbb F_n) \simeq H^i(
(\Conf^n_{\Sigma^1_{g,f}} \cup \infty), j_! \mathbb F_n).
\end{align}
We will now bound the dimension of the latter cohomology group.
We will use the $\on{CW}$ cell structure on 
$\Conf^n_{\Sigma^1_{g,f}} \cup \infty$ from \autoref{lemma:cell-bound}
which has at most $2^{2g+f+n}+1$ cells.
The cellular cochain complex which computes the $i$th cohomology group
\eqref{equation:compactly-supported-cohomology} has dimension less than
$\rk \mathbb F_n \cdot 2^{2g+f+n} = \dim F_n \cdot 2^{2g+f+n}$.
(We note that we multiply only by $2^{2g+f+n}$ and not $2^{2g+f+n}+1$ since one
	of the cells is the $0$-cell at $\infty$, and the stalk of $j_! \mathbb
F_n$ there vanishes.)
It follows from Poincar\'e duality that 
\[\dim H_{2n-i}(\Conf^n_{\Sigma^1_{g,f}}, \mathbb F_n) =
\dim H^{i}_{\on{c}}(\Conf^n_{\Sigma^1_{g,f}}, \mathbb F_n) \leq \dim F_n \cdot
2^{2g+f+n}. \qedhere
\]
\end{proof}

Combining our homological stability results with the above bounds on homology
gives the following bound on cohomology.
We switch from homology to cohomology because later we will apply these results to bound
the cohomology of certain spaces related to Selmer groups.
For the following, we continue to use notation from
\autoref{notation:homology-coefficient-system}.

\begin{proposition}
	\label{proposition:cohomology-bound}
	Suppose $V$ is a $1$-controlled coefficient system for
	$\Sigma^1_{0,0}$ and $F$ is a coefficient system for $\Sigma^1_{g,f}$
	over $V$.
	Assume moreover that 
	$\deg H_0(\mathcal K(M_0^{V,F}))$ is finite, $\deg H_1(\mathcal
	K(M_0^{V,F}))$ is finite, and each $V_i$ and $F_i$ is
	finite-dimensional.
	Then, there is a constant $K$ depending on $g, f$, and the coefficient
	systems $V$ and $F$, but not on the subscript $n$ or the index $i$ so that
	\begin{align}
	\label{equation:cohomology-bound}
	\dim H^i(B^n_{g,f}, F_n) \leq K^{i+1}
	\end{align}
	for all $i, n$. 
\end{proposition}
\begin{proof}
	Since the dimensions of the $F_i$
	\eqref{equation:cohomology-bound}
	are finite, and we are working with representations over a field, it follows from the universal coefficient theorem
	that $\dim H^i(B^n_{g,f}, F_n) = \dim H_i(B^n_{g,f},F_n)$.
	Hence, it is enough to bound 
	$\dim H_i(B^n_{g,f}, F_n) \leq K^{i+1}.$
	Recall we use $\mathbb U \in R^V$ to denote a central homogeneous
	non-invertible element.
	By \autoref{theorem:1-controlled-bound}, there are constants $I(V)$ and
	$J(V, F)$ so that whenever $n > I(V)i + J(V, F),$ 
	$H_i(B^n_{g,f},F_n) \simeq H_i(B^{n+\deg \mathbb U}_{g,f}, F_{n+\deg
	\mathbb U}).$
	Therefore, applying this repeatedly, it is enough to show
	$\dim H_i(B^n_{g,f}, F_n) \leq K^{i+1}$ for any $n \leq I(V)i +
	J(V, F) + \deg \mathbb U$.
	By \autoref{lemma:exponential-bound-for-all-cohomology},
	$\dim H_i(B^n_{g,f}, F_n) \leq 2^{2g+f+n} \cdot \dim F_n$.
	Hence, we only need to produce some constant $K$ so that 
	\begin{align*}
	2^{2g+f+I(V)i +
	J(V, F) + \deg \mathbb U}\cdot \dim F_{I(V)i +J(V, F) + \deg \mathbb U} \leq K^{i+1}.
	\end{align*}
	We may assume $\dim V_1 > 0$, as otherwise $R^V = k$ and the statement
	is trivial.
	Because $F_n \simeq V_1^{\otimes n} \otimes F_0$,
	\begin{align*}
		&2^{2g+f+I(V)i + J(V, F) + \deg \mathbb U}\cdot \dim F_{I(V)i
		+J(V, F) + \deg \mathbb U} \\
		&= 2^{2g+f+J(V, F)+\deg \mathbb U}\cdot 2^{I(V)i} \cdot (\dim
		V_1)^{I(V)i +J(V, F) + \deg \mathbb U}
	\cdot \dim F_0 \\
	&\leq (2 \dim V_1)^{I(V) i} \cdot (2 \dim V_1)^{2g+f+J(V, F)+\deg \mathbb U} \dim
	F_0.
	\end{align*}
	The claim then follows by taking 
	\[	K > \max((2 \dim V_1)^{I(V)}, (2 \dim V_1)^{2g+f+J(V, F)+\deg \mathbb U} \dim F_0).
	\qedhere
	\]
\end{proof}

We now reformulate the above in a slightly more convenient form for our
applications.

\begin{corollary}
\label{corollary:cohomology-bound}
	Let $\ell'$ be a prime.
Suppose $V$ is a coefficient system for
	$\Sigma^1_{0,0}$ and $F$ is a coefficient system for $\Sigma^1_{g,f}$
	over $V$ with each $V_i$ and each $F_i$ finite-dimensional.
	Assume that there is a central homogeneous non-invertible element $\mathbb U \in R^V$
	such that $\deg \ker \mathbb U$ and $\deg \coker \mathbb U$ are both finite.
	Suppose $F_n$ corresponds to a local system $\mathbb F_n$ on
	$\Conf^n_{\Sigma^1_{g,f}}$ via the identification
	$\pi_1(\Conf^n_{\Sigma^1_{g,f}}) \simeq B^n_{g,f}$
	with
	$\mathbb F_n = \pi_* (\mathbb Z/\ell' \mathbb Z)$
	for $\pi: W_n \to \Conf^n_{\Sigma^1_{g,f}}$ some finite (unramified) covering space.
	(Above, we consider $\mathbb Z/\ell' \mathbb Z$ as the constant sheaf on
	$W_n$, and so $\mathbb F_n$ is the local system on configuration
	space corresponding to the action of $B^n_{g,f}$ on the free
	$(\mathbb Z/\ell' \mathbb Z)$-module spanned by the fiber of
	$\pi$ over a specified basepoint.)
	Then, there is a constant $K$ depending on $g,f$, the coefficient systems $V$
	and $F$,
but not on the subscript $n$ or index $i$ so that
	\begin{align*}
	\dim H^i(W_n, \mathbb Z/\ell' \mathbb Z) \leq K^{i+1}
	\end{align*}
	for all $i, n$. 
\end{corollary}
\begin{proof}
This is an immediate consequence of
\autoref{proposition:cohomology-bound}, upon identifying group cohomology for a
finite-index subgroup with cohomology of the corresponding finite covering space,
once we verify that $V$ is $1$-controlled and $\deg H_1(\mathcal K(M_0^{V,F}))$
and $\deg H_0(\mathcal K(M_0^{V,F}))$ are finite.
We have that $V$ is $1$-controlled by \autoref{theorem:evw-stability}.
From \autoref{lemma:finite-generation-mp}, we find that 
$M_0^{V,F}$ is finitely generated as an $R^V$-module.
By \autoref{lemma:degree-bound-for-km0}, we find $H_0(\mathcal K(M_0^{V,F}))$
and
$H_1(\mathcal K(M_0^{V,F}))$ both have finite degree.
\end{proof}

\section{The Selmer stack and its basic properties}
\label{section:selmer-space}

In this section, we set up the Selmer stack, which is a finite cover of
the stack of quadratic twists of an abelian variety that parameterizes
pairs of a quadratic twist and a Selmer element for that quadratic twist.
We first define the Selmer stack in
\autoref{subsection:selmer-space-definition}.
In \autoref{subsection:selmer-stack-properties} we prove basic properties of the
Selmer stack, such as the fact that it is a finite \'etale cover of the stack of
quadratic twists.
Since the definition given in
\autoref{subsection:selmer-space-definition} is not obviously connected to Selmer
groups, in \autoref{subsection:selmer-stack-sizes}
we relate the Selmer stack to Selmer groups.
Variants of the Selmer stack for the universal family were studied in
\cite{landesman:geometric-average-selmer} and
\cite{fengLR:geometric-distribution-of-selmer-groups},
and many of the proofs in this section follow ideas from those articles.

\subsection{Definition of the Selmer stack}
\label{subsection:selmer-space-definition}

We now set up notation to define the Selmer stack.

\begin{definition}
	\label{definition:symplectic-sheaf}
	Let $X$ be a Deligne-Mumford stack and $\nu$ a positive integer. A locally constant constructible
	sheaf of free $\mathbb Z/\nu \mathbb Z$-modules $\mathscr F$ on $X$ is {\em
	symplectically self-dual} if there is an isomorphism
	$\mathscr F
	\simeq \mathscr F^\vee(1) := \hom(\mathscr F, \mu_\nu)$ so that the
	resulting pairing $\mathscr F \otimes \mathscr F \to \mu_\nu$ factors
	through $\mathscr F \otimes \mathscr F \to \wedge^2 \mathscr F \to
	\mu_\nu$.
\end{definition}
\begin{remark}
	\label{remark:}
	Sometimes, a symplectically self-dual sheaf is called a {\em weight $1$} symplectically self-dual sheaf. 
	Since 
this is the only kind of symplectically self-dual sheaf we will encounter in our
paper, so we omit the ``weight $1$'' adjective.
All symplectically self-dual sheaves we encounter will be assumed lcc sheaves of
free $\mathbb Z/\nu \mathbb Z$-modules.
\end{remark}

\begin{example}
	\label{example:symplectic-sheaf}
	An important example of a symplectically self-dual sheaf for
	us will be $A[\nu]$, where $A \to U$ is an
	abelian scheme as in \autoref{notation:curve-notation} with a
	polarization of degree prime to $\nu$, for $\nu$ invertible on $B$.
\end{example}

\begin{notation}
	\label{notation:quadratic-twist-notation}
	Keep notation for $B, C, Z, U, n, f$ as in
	\autoref{notation:curve-notation}.
	Let $\mathscr F$ be a tame symplectically self-dual sheaf of $\mathbb
	Z/\nu \mathbb Z$-modules on $U$.

In order to define a Hurwitz stack for the group $\mathbb Z/2\mathbb Z$, let
$\mathcal S
\subset \hom(\pi_1(\Sigma_{g,n + f+1}) , \mathbb Z/2 \mathbb Z)$ denote the subset
sending loops around the geometric points in the degree $f+1$ divisor $Z$ to the trivial element of $\mathbb
Z/2 \mathbb Z$ and loops around the $n$ marked points (corresponding to
geometric points of the divisor $D$) to the nontrivial element of 
$\mathbb Z/2\mathbb Z$. 
(Since $\mathbb Z/2 \mathbb Z$ is abelian, this Hurwitz stack is a $\mathbb Z/2
\mathbb Z$ gerbe over its coarse space.)
Define $\qtwist n U B$ to be $\hur {\mathbb Z/2 \mathbb Z} n Z {\mathcal S} C B$.

We will assume throughout $n$ is even, as otherwise there are no such covers by
Riemann-Hurwitz.
Informally, $\qtwist n U B$ is a moduli space for finite double covers of $C$ ramified over a
degree $n$ divisor $D$, disjoint from $Z$.
Let $h: \mathscr U^n_B \times_{\conf n U B} \qtwist n U B \to \mathscr U^n_B
\to U$
denote the composite projection
and let $\lambda: \mathscr C^n_B \times_{\conf n U B} \qtwist n U B \to \qtwist n U B$
denote the universal proper curve.
The universal open curve $\mathscr U^n_B \times_{\conf n U B} \qtwist n U B$
possesses a natural finite \'etale double cover $t: \mathscr X^{n}_B \to \mathscr
U^n_B\times_{\conf n U B} \qtwist n U B$. 
The extension of this double cover to the proper universal curve
is branched precisely along the boundary
divisor $\mathscr D^n_B \times_{\conf n U B} \qtwist n U B$ (but not along the preimage
of $Z$).

Define $\mathscr F^n_B := t_* t^* h^* \mathscr F/ h^* \mathscr F$.
This is a sheaf on
$\mathscr U^n_B\times_{\conf n U B} \qtwist n U B$
whose fiber over $x := [(D,\phi: \pi_1(C - D) \to \mathbb Z/2\mathbb Z)] \in
\qtwist n U B$ is
a sheaf on
$\mathscr U^n_B\times_{\conf n U B} x \subset U$ which is the
quadratic twist of $\mathscr F$ over $U$ along the finite \'etale double cover corresponding to the
surjection $\phi$, which is branched over $D$.
\end{notation}

With the above notation in hand, we are now prepared to define the Selmer stack.
Recall that any \'etale sheaf on a scheme is represented by an algebraic space
\cite[V, Theorem 1.5]{Milne:etaleBook}.
Similarly, any \'etale sheaf on an algebraic stack is represented by a relative algebraic
space, which is another algebraic stack. We use this in the following definition.

\begin{definition}
	\label{definition:selmer-sheaf}
	Maintain notation as in  \autoref{notation:quadratic-twist-notation} and
	let $\nu$ be a positive integer.
	We assume $2 \nu$ is invertible on $B$.
	As in \autoref{notation:quadratic-twist-notation}, we have a
	symplectically self-dual sheaf $\mathscr F$ on $U$, which we are
	assuming is an lcc sheaf of free $\mathbb Z/\nu \mathbb Z$-modules.
	This gives rise to
	a symplectically self dual sheaf $\mathscr F^n_B$ on $\mathscr U^n_B$ and maps
\begin{align*}
	\mathscr U^n_B \xrightarrow{j} \mathscr
	C^n_B \xrightarrow{\lambda} \qtwist n U B.
\end{align*}
	Define the {\em Selmer sheaf of log-height $n$ associated to $\mathscr F$
	over $B$}
to be $\selsheaf {\mathscr F^n_B} := R^1 \lambda_* \left(
	j_* \mathscr F^n_B \right)$.
	The {\em Selmer stack},
	$\selspace {\mathscr F^n_B}$,
	is the algebraic stack representing this \'etale
	sheaf.
\end{definition}
\begin{remark}
	\label{remark:}
	For odd $\nu$, the Selmer stack is never a scheme because $\qtwist n U
	B$ is a $\mathbb Z/2 \mathbb Z$ gerbe over a scheme, and $\selspace
	{\mathscr F^n_B}$ is an odd degree cover of $\qtwist n U B$.
	Fortunately, since this is a gerbe, its stackiness is rather mild. 
	This will pose some technical, yet overcomable, obstacles.
\end{remark}

We next give a couple examples of types of symplectically self-dual sheaves
coming from abelian varieties, which will be important for our applications to
the BKLPR heuristics.

\begin{example}
	\label{example:abelian-scheme-selmer-sheaf}
	Suppose $p: A \to U$ is a polarized abelian scheme with polarization of
	degree prime to $\nu$ over $B$, as in \autoref{example:symplectic-sheaf}.
	Take $\mathscr F := A[\nu]$. Note $A[\nu] \simeq A^\vee[\nu] \simeq R^1 p_* \mu_\nu$,
	since the polarization has degree prime to $\nu$.
	Then, the Weil pairing gives $A[\nu]$ the structure of a
	symplectically self-dual sheaf on $U$. 
	Further, with notation as in
	\autoref{notation:quadratic-twist-notation},
	$A[\nu]^n_B$ defines a sheaf on $\mathscr U^n_B$. An important example
	of a Selmer sheaf for us will be $\selsheaf {A[\nu]^n_B} = R^1 \lambda_* \left(
	j_* A[\nu]^n_B \right)$.
\end{example}
\begin{example}
	\label{example:abelian-special-fiber}
	A slightly more general setup than
	\autoref{example:abelian-scheme-selmer-sheaf} is the following. Suppose
	we are in the setting of
	\autoref{notation:quadratic-twist-notation}, and 
	$b \in B$ is a closed point. Suppose we are given $\mathscr F$ a
	symplectically self-dual sheaf over $C$ so that the fiber $\mathscr F_b$
	over $U_b$ defines a sheaf which is of the form $A_b[\nu]$ for $p: A \to
	U_b$ a polarized abelian scheme with polarization degree prime to $\nu$.
	Then we obtain a Selmer sheaf $\mathscr F^n_B$ over $\mathscr U^n_B$ so
	that $\mathscr F^n_b \simeq \selsheaf {A[\nu]^n}$.
	The difference between this and
	\autoref{example:abelian-scheme-selmer-sheaf} is that we may not
	have any abelian scheme over $U$ restricting to $A$ over $U_b$.
\end{example}
\begin{remark}
	\label{remark:}
In fact, the \autoref{example:abelian-special-fiber} will be the
	setting we work in to prove our main result
	\autoref{theorem:main-finite-field} because
	it is relatively easy to lift symplectically self-dual sheaves from the
	closed point of a DVR to the whole DVR, as we explain in
	\autoref{lemma:lifting-to-char-0}. However, we are unsure whether it is
	possible to lift abelian schemes in our setting.
\end{remark}

We conclude this subsection with some notation recording data associated to a quadratic twist,
which we will use throughout the paper.
\begin{notation}
	\label{notation:x-points}
	With notation as in \autoref{notation:quadratic-twist-notation},
	for $x \in \qtwist n U B$ a point or geometric point, let
	$y$ denote the image of $x$ under the map $\qtwist n U B \to \conf n U
	B$.
	We use $C_x$ to
	denote the fiber of 
	$\xi: \mathscr C^n_B \to \conf n U B$
	over $y$, $U_x$ to denote the fiber of $\xi \circ
	j$ over $y$, 
	and we use $\mathscr F_x$ to denote the fiber of $\mathscr F^n_B$ over
	the point $x$.

	Assume we are further in the setup of
	\autoref{example:abelian-scheme-selmer-sheaf}
	or \autoref{example:abelian-special-fiber}
	and $x \in \qtwist n {U_b} b$.
	We use $A_x$ to denote the fiber of the abelian scheme $t_* t^* h^* A/h^* A$
	over $x$,
	where $t^*$ and $h^*$ denote the pullback along $t$ and $h$, and $t_*$ denotes the Weil
	restriction along $t$.
	Note
	that $A_x$ is an abelian scheme over $U_x$. We use $\mathscr A_x$ to
	denote the N\'eron model over $C_x$ of $A_x \to U_x$.
	We let $D_x \subset C_x - U_x$ denote the divisor associated to $y$, the
	image of $x$ under the projection $\qtwist n U B \to \conf n U B$.
\end{notation}

\subsection{Basic properties of the Selmer stack}
\label{subsection:selmer-stack-properties}

We next develop some basic properties of the Selmer stack.
The next lemma shows the Selmer sheaf commutes with base change.
The proof is similar to
\cite[Lemma 2.6]{fengLR:geometric-distribution-of-selmer-groups}, though 
some additional technical difficulties come up related to working over the
space of quadratic twists, instead of the universal family.
\begin{lemma}
	\label{lemma:selmer-lcc}
	Use notation as in \autoref{notation:quadratic-twist-notation}.
	In particular, $\mathscr F$ is a tame symplectically self-dual sheaf of
	$\mathbb Z/\nu \mathbb Z$-modules.
	Suppose $2\nu$ invertible on $B$.
	Then, the sheaf $\selsheaf {\mathscr F^n_B}$ is locally constant constructible
	and its formation commutes with base change on $\qtwist n U B$.
	Further, for $\overline \lambda := \lambda \circ j$, both
$R^i \overline \lambda_* \left( \mathscr F^n_B \right)$ and
$R^i \overline \lambda_! \left( \mathscr F^n_B \right)$ are locally
constant constructible for all $i \geq 0$ and their formation commutes with base change on $\qtwist n U B$.
\end{lemma}
\begin{proof}
	In order to prove the result, we first set some notation.
	We have a natural map $\phi: R^1 \overline \lambda_! \mathscr F^n_B \to
	\selsheaf {\mathscr F^n_B}$ obtained from the map $j_! \mathscr F^n_B
	\to j_* \mathscr F^n_B$
	and the definition $R^1 \overline \lambda_! (\mathscr F^n_B) := R^1
	\lambda_*
	\left( j_! \mathscr F^n_B \right)$
	\cite[I.8.6]{FreitagK:lectures-etale}.
	Similarly, we have a map $\psi: \selsheaf {\mathscr F^n_B} \to R^1
	\overline \lambda_* \mathscr F^n_B$ obtained from the composition of functors spectral sequence
	for $\lambda \circ j$. 
	Note that $\psi$ is injective by the Leray spectral
	sequence.

	Our first goal is to show $\selsheaf {\mathscr F^n_B}$ is the image of
	$\psi \circ \phi$.
	Since $\psi$ is injective, it only remains to show $\phi$ is surjective.
	Because $\chi: j_! \mathscr F^n_B \to j_* \mathscr F^n_B$ is an isomorphism over $\mathscr U^n_B$, $\coker \chi$ is
	supported on $\mathscr D^n_B$, which is finite over $\conf n U B$, we
	find $R^1 \lambda_*(\coker \chi) = 0.$ This implies $\phi$ is surjective and
	so $\selsheaf {\mathscr F^n_B}$ is a constructible sheaf.

	We conclude by showing $R^1 \overline \lambda_! \mathscr F^n_B$ and
	$R^1 \overline \lambda_* \left( \mathscr F^n_B \right)$ are both locally
	constant constructible, and their formation commutes with base change.
	This will imply $\selspace {\mathscr F^n_B}$ is locally
	constant constructible and its formation commutes with base change, as
	it is the image of the map
	$\psi \circ \phi: R^1 \overline \lambda_! \mathscr F^n_B
	\to R^1 \overline \lambda_* \left( \mathscr F^n_B \right)$.

	We first show $R^i \overline \lambda_! \mathscr F^n_B$ is locally constant
	constructible in the case that $\nu$ is prime. Note that its formation commutes with base change by
	proper base change for any $\nu$.
	Using
	\cite[Corollaire 2.1.2 and Remarque
	2.1.3]{laumon:semi-coninuity-du-conducteur-de-swan},
	it is enough to show the Swan conductor of $\mathscr F^n_B$ is
	constant. 
	As in 
	\cite[Remarque 2.1.3]{laumon:semi-coninuity-du-conducteur-de-swan},
	the Swan conductor over a point $[D] \in \conf n U B$ is a sum of local contributions, one for each
	geometric point of $D$ and one for each geometric point of $Z$ over the
	image of $D$ in $B$. At each geometric point of $D$, because we are
	taking a quadratic twist along $D$, the ramification index is $2$, and
	hence the ramification is tame, since $2$ is invertible on $B$.
		We are also assuming the ramification along points of $Z$ is tame for
	$\mathscr F$. This is identified with the corresponding ramification for
	$\mathscr F^n_B$ along points of $Z$, and hence this is tame as well.
	Therefore, the Swan conductor vanishes identically.

	Next, we show $R^i \overline \lambda_! \mathscr F^n_B$ is locally
	constant constructible for every positive
	integer $\nu$ as in the statement of the lemma, using the case that $\nu$ is prime, settled above.
	As an initial step, we may reduce to the case $\nu = \ell^t$ is a prime power by
	observing that if $\nu$ has prime factorization $\nu = \prod
	\ell^{t_\ell}$ then
	$\mu_\nu \simeq \oplus \mu_\ell^{t_\ell}$.
	Now, suppose $\nu = \ell^t$ is a prime power, and inductively assume we have proven
	$R^i \overline \lambda_! \mathscr F^n_B[\ell^{t-1}]$ is locally constant
	constructible for all $i$.
	Since $\nu = \ell^t$ and $\mathscr F^n_B$ is a locally constant
	constructible sheaf of free $\mathbb Z/\nu \mathbb Z$-modules, we have an exact sequence
	\begin{equation}
		\label{equation:}
		\begin{tikzcd}
		\nonumber
			0 \ar {r} &  \mathscr F^n_B[\ell^{t-1}] \ar {r} &
			\mathscr F^n_B \ar {r} & \mathscr F^n_B[\ell] \ar {r}
			& 0.
	\end{tikzcd}\end{equation}
	Applying $R \overline{\lambda}_!$ to the above sequence, we get a long exact
	sequence on cohomology
	\begin{equation}
		\label{equation:}
		\begin{tikzcd}[column sep=small]
		\nonumber
			R^{i-1} \overline \lambda_! \mathscr F^n_B[\ell] \ar {r} &   
			R^i \overline \lambda_! \mathscr F^n_B[\ell^{t-1}] \ar {r} & 
			R^i \overline \lambda_! \mathscr F^n_B \ar {r} &
			R^i \overline \lambda_! \mathscr F^n_B[\ell]
			\ar {r}
			&
			R^{i+1} \overline \lambda_! \mathscr F^n_B[\ell^{t-1}].
	\end{tikzcd}\end{equation}
	Since all but the middle term are locally constant constructible
	by our inductive assumption, it follows that $R^i \overline \lambda_!
	(\mathscr F^n_B[\ell^t])$
	is also locally constant constructible by
\cite[\href{https://stacks.math.columbia.edu/tag/093U}{Tag
093U}]{stacks-project}.

	We conclude by showing $R^1 \overline \lambda_* \left( \mathscr F^n_B \right)$ is locally constant
	constructible and its formation commutes with base change. 
	Since $R^i \overline \lambda_! \mathscr F^n_B$ is locally constant
	constructible, it follows from Poincar\'e duality \cite[Theorem
	4.8]{verdier:a-duality-theorem}
	and the isomorphism coming from the polarization of degree prime to
	$\nu$ that
	\begin{align*}
		\shom \left( R^{-i} \overline \lambda_! \left(\mathscr F^n_B\right),\mu_\nu \right)
		\xleftarrow{\simeq} R^{i+2} \overline \lambda_* R \shom
		(\mathscr F^n_B,
	\mu_\nu) \simeq R^{i+2} \overline \lambda_* (\mathscr F^n_B).
	\end{align*}
	Taking $i = -2 + s$ gives $(R^{2-s} \overline \lambda_! (\mathscr F^n_B))^\vee(1)
	\simeq R^s \overline \lambda_* \left( \mathscr F^n_B \right)$. Since we have seen 
$(R^{2-s} \overline \lambda_! (\mathscr F^n_B))^\vee$ is locally constant constructible
and its formation commutes with base change,
the same holds for 
$R^s \overline \lambda_* \left( \mathscr F^n_B \right)$.
\end{proof}

\begin{notation}
	\label{notation:component-group}
	Let $k$ be a field and let $C$ be a smooth proper
	geometrically connected curve over $k$ of genus $g$, with $U' \subset C$
	an open subscheme.
	Let $A'$ an abelian scheme over
	$U'$ with N\'eron model $\mathscr A' \to C$. 
	Let $\mathscr A'^0$ denote the identity component of the N\'eron model
	$\mathscr A'$,
	meaning that $\mathscr A'^0$ is the open subscheme of $\mathscr A'$ so that each fiber is nonempty, connected, and
	contains the identity section, see \cite[p. 154]{BoschLR:Neron}.
	Let $\Phi_{A'} := \left(\mathscr
	A'/\mathscr A'^0\right)(k)$ denote the {\em component group} of
the N\'eron model of $A'$.
We use 
$\Phi_{A'_{\overline k}} =
\left(\mathscr
A'_{\overline k}/\mathscr A'^0_{\overline k}\right)(\overline{k})$ to denote the
{\em geometric component group}.
\end{notation}

The following proof is quite similar to \cite[Lemma
3.21]{landesman:geometric-average-selmer}.
We thank Tony Feng for suggesting the idea that appeared there for bootstrap from
the prime case to the general case, which we reuse
here.
In the next lemma,
note that since we are working over an algebraically closed field, the component
group is the same as the geometric component group.
\begin{lemma}
	\label{lemma:free-h1}
	Let $k$ be an algebraically closed field, let $C$ be a smooth proper
	geometrically connected curve over $k$ of genus $g$. 
	Let $\mathscr F'$ be a symplectically self-dual lcc sheaf of free
	$\mathbb Z/\nu\mathbb Z$-modules on an open $j: U'
	\subset C$.
	Suppose that 	
	\begin{enumerate}
		\item 
			for each prime $\ell$ with $w = \ord_\ell(\nu)$, and
			$0 \leq t \leq w$, the
	multiplication by $\ell^t$ map $j_* \mathscr F'[\ell^w] \to
	j_* \mathscr F'[\ell^{w-t}]$ is surjective.
\item $j_*\mathscr F'(C) = 0$.
	\end{enumerate}
	Then 
	$H^1(C, j_* \mathscr F'[\nu])$ is a free $\mathbb
	Z/\nu\mathbb Z$-module.
	In the case $j_* \mathscr F'$ is of the form of $A'[\nu]$, for $A' \to
	U'$ an abelian scheme,
	hypothesis $(1)$ above is satisfied if
	the geometric component group
	$\Phi_{A'}$ has order prime to $\nu$.
\end{lemma}
\begin{proof}
	Using the Chinese remainder theorem, we can reduce to the case that
	$\nu = \ell^w$ is a prime power.
	Suppose $H^1(C, j_*\mathscr F'[\ell]) \simeq (\mathbb Z/\ell \mathbb Z)^r.$
	We will show by induction on $w$ that $H^1(C, j_*\mathscr F'[\ell^w]) \simeq
	(\mathbb Z/\ell^w \mathbb Z)^r.$

	For $0 \leq t \leq w$
	we claim there is an exact sequence
	\begin{equation}
		\label{equation:neron-torsion-exact-sequence}
		\begin{tikzcd}
			0 \ar {r} & j_*\mathscr F'[\ell^t] \ar {r} & j_*\mathscr
			F'[\ell^w]
			\ar {r} & j_*\mathscr F'[\ell^{w-t}] \ar {r}
			& 0.
	\end{tikzcd}\end{equation}
	This is left exact because the analogous sequence for $\mathscr F'$ in place of
	$j_* \mathscr F'$ is left exact. 
	This sequence is right exact by assumption (1) from the statement
	of the lemma.

	We now prove the final clause of the statement of the lemma:
	In the case $\mathscr F' \simeq A'[\nu]$,
	the cokernel of the map
$j_* \mathscr F'[\ell^w] \to j_* \mathscr F'[\ell^{w-t}]$
	is identified with $\Phi_{A'}/\ell^t
	\Phi_{A'}$. 
	This is trivial by assumption as $\ell^t \mid \nu$.
	Therefore, in this case, $(1)$ holds.
	
	We next claim $H^0(C, j_*\mathscr F'[\ell^t]) = H^2(C, j_*\mathscr F'[\ell^t]) = 0$.
	The former holds by assumption $(2)$. By \cite[V Proposition
	2.2(b)]{Milne:etaleBook} and the polarization
	$(\mathscr F'[\ell^t])^\vee (1) \simeq \mathscr F'[\ell^t]$, we find
	\begin{align*}
		H^2(C, j_* \mathscr F'[\ell^t]) \simeq H^0\left(C, j_* \left(
				\left( \mathscr F'[\ell^t] \right)^\vee
		(1) \right)\right)^\vee \simeq H^0(C, j_* \mathscr F'[\ell^t])^\vee \simeq H^0(U,
		\mathscr F'[\ell^t])^\vee = 0.
	\end{align*}
	
	The long exact sequence associated to
	\eqref{equation:neron-torsion-exact-sequence}
	and the vanishing of the $0$th and $2$nd cohomology above implies we
	obtain an exact sequence
	\begin{equation}
		\label{equation:neron-cohomology-exact-sequence}
		\begin{tikzcd}
			0 \ar {r} & H^1(C, j_*\mathscr F'[\ell^t]) \ar {r}{\alpha^t} & H^1(C,
			j_* \mathscr F'[\ell^w]) \ar {r}{\beta^t} & H^1(C, j_*\mathscr
			F'[\ell^{w-t}]) \ar {r}
			& 0.
	\end{tikzcd}\end{equation}
	Induction on $w$ implies $\# H^1(C,j_*\mathscr F'[\ell^w]) = \ell^{wr}$ and we
	wish to show $H^1(C,j_*\mathscr F'[\ell^w])$ is free of rank $r$.
	By the structure theorem for finite abelian groups, it suffices to show
	the kernel of multiplication by $\ell^{w-1}$ on $H^1(C,j_*\mathscr F'[\ell^w])$
	has order $\ell^{(w-1)r}$.
	The multiplication by $\ell^{w-1}$ map factors as $H^1(C,j_*\mathscr F'[\ell^w])
	\xrightarrow{\beta^{w-1}} H^1(C,j_*\mathscr F'[\ell])
	\xrightarrow{\alpha^1} H^1(C,j_*\mathscr
	F'[\ell^w])$.
	We know from \eqref{equation:neron-torsion-exact-sequence} that
	$\alpha^1$ is injective so 
	\begin{align*}
		\ker (\times \ell^{w-1}) = \ker(\beta^{w-1}
		\circ \alpha^1) = \ker \beta^{w-1} = H^1(C, j_*\mathscr F'[\ell^{w-1}]),
	\end{align*}
	which has size $\ell^{(w-1)r}$, as we wished to show.
\end{proof}

We next aim to compute a formula for the rank of the Selmer sheaf, in favorable situations,
in \autoref{proposition:rank-description}.
First, we introduce notation needed to state that formula.

\begin{definition}
	\label{definition:drop}
	Suppose $\nu$ is a prime number.
	Let $\mathscr F$ 
be a locally constant constructible sheaf 
	of free
	$\mathbb Z/\nu \mathbb Z$-modules on an open $U'
	\subset C$ of a smooth proper geometrically connected curve $C$. Let $\overline{\eta}$ be the geometric
	generic point of $C$. For any point $x \in C - U'$, there is an
	associated action of the inertia group $I_x$ at $x$ on the geometric
	generic fiber $\mathscr F_{\overline\eta}$ of $\mathscr F$.
	This action is well defined up to conjugacy. We use $\drop_x(\mathscr F)$ to denote
	the corank of the invariants of $I_x$, i.e., 
	$\drop_x(\mathscr F) := \rk \mathscr F - \rk \mathscr F_x^{I_x}$.
	In general, if $\nu$ is not necessarily a prime number, for each prime
$\ell \mid \nu$ we use $\drop_{x,\ell}(\mathscr F) := \drop_x(\mathscr F[\ell])$, and if $\drop_{x,\ell}(\mathscr F)$ is independent of $\ell$, we
denote this common value simply by $\drop_x(\mathscr F)$.
Whenever we use the notation $\drop_x(\mathscr F)$ in the case $\nu$ has multiple
prime divisors, we are implicitly claiming it is independent of the prime
divisor.
\end{definition}

\begin{example}
	\label{example:drop}
	If $\nu$ is prime, and $\mathscr F \simeq A[\nu]$, then for any $x \in C -
	U$, $\drop_x(\mathscr F) = 0$ if and only if inertia acts trivially at $x$, i.e.,
	$A[\nu]$ extends over the point $x$.
	If $A$ is a relative elliptic curve and the order of the geometric component group
	of the N\'eron model of $A$ at $x$ is prime to $\nu$, then
	$\drop_x(\mathscr F) = 1$ whenever $A$ has multiplicative
	reduction at $x$ and $\drop_x(\mathscr F) = 2$ whenever $A$ has
	additive reduction at $x$.
\end{example}

\begin{proposition}
	\label{proposition:rank-description}
	Maintain notation as in \autoref{notation:quadratic-twist-notation}, so,
	in particular,
	$\mathscr F$ is a tame symplectically self-dual lcc sheaf of free
	$\mathbb Z/\nu \mathbb Z$-modules.
	Suppose $\nu$ is odd and $n > 0$.
	Assume that $B = {\overline{b}}$ is the spectrum of an algebraically
	closed field. Assume that 
	\begin{enumerate}
		\item[(1)] for each prime $\ell \mid \nu$, each integer $w$ with
			$\ell^w \mid \nu$, and each integer
			$t \leq w$, the multiplication by $\ell^t$ map $j_* \mathscr F[\ell^w] \to
	j_* \mathscr F[\ell^{w-t}]$ is surjective
	\item[(3)] the sheaf $\mathscr F[\ell]$  is irreducible for each prime
		$\ell \mid \nu$.
	\end{enumerate}
	Assume $2 \nu$ is invertible on $B$.
	For each $x \in \qtwist n U B$, 
	consider the following three properties.
\begin{enumerate}
	\item[(1')] for each prime $\ell \mid \nu$ with $\ell^w \mid \nu$ and
			$t \leq w$, the multiplication by $\ell^t$ map $j_* \mathscr F_x[\ell^w] \to
	j_* \mathscr F_x[\ell^{w-t}]$ is surjective
\item[(2')] $j_* \mathscr F_x(C_x) = 0$
\item[(3')] the sheaf $\mathscr F_x[\ell]$  is irreducible for each prime $\ell
\mid \nu$.
\end{enumerate}
Then, $(2')$ always holds, $(1')$ holds if $(1)$ holds, and $(3')$ holds if
$(3)$ holds.

Moreover, assuming $(1)$ and $(3)$, the map $\pi: \selspace {\mathscr F^n_B} \to \qtwist n U B$ is finite \'etale,
	representing a locally constant constructible sheaf of rank 
	$(2g - 2 + n) \cdot 2r + \sum_{x \in Z(B)} \drop_x(\mathscr F)$
	free $\mathbb Z/\nu \mathbb Z$-modules, whose formation commutes with
	base change.
\end{proposition}
\begin{proof}
	First, observe that by \autoref{lemma:selmer-lcc},
	$\pi: \selspace {\mathscr F^n_B} \to \qtwist n U B$ is finite \'etale,
	corresponding to a locally constant sheaf of $\mathbb Z/\nu \mathbb Z$-modules, and its formation commutes with base change on $\qtwist n U B$.

	We now verify that condition $(1')$ hold for quadratic
	twists $\mathscr F_x$ of $\mathscr F$, 
	ramified over a divisor $D_x$ disjoint from $Z_x$, using condition $(1)$.
	If $\mathscr F$ corresponds to a representation of $\pi_1(U_x-D_x)$, the quadratic twist corresponds to tensoring this
	representation with an order $2$ character, whose local inertia at any
	point outside of $D_x$ is trivial.
	Surjectivity of the map from $(1')$ can only fail at points $p \in D_x \cup Z_x$.
	If $p \in Z_x$, since surjectivity can be verified locally, surjectivity
	for $j_* \mathscr F_x$ at $p$ follows from the corresponding
	surjectivity for $j_* \mathscr F$ at $p$.
	If $p \in D_x$, 
	the stalk of $j_* \mathscr F_x[\ell^{w-t}]$ is trivial, as it is identified with the
	invariants of multiplication by $-1$, which is trivial,
	and so surjectivity at such points is automatic.

	Next, we check $(2')$ holds, just using $n > 0$.
	We wish to show
	$H^0(C_x, \mathscr F_x) = 0$.
	Thinking of $\mathscr F_x$ as a representation of $\pi_1(U_x - D_x)$,
	a section corresponds to an invariant vector.
	However, since $n > 0$, local inertia at a point of $D_x$ acts by $-1$,
	and so there are no invariant vectors.

	Third, we show $(3')$ holds for $\mathscr F_x$, assuming $(3)$ holds for
	$\mathscr F$. Note that the quadratic twist of the sheaf $\mathscr
	F$ is obtained by tensoring the corresponding representation of
	$\pi_1(U)$ with a character. This preserves irreducibility.
	
	We next show this $\pi$ corresponds to a sheaf of free $\mathbb Z/\nu \mathbb
	Z$-modules.
	We may check this at any point of $\qtwist n U B$ since the formation
	of $\selspace {\mathscr F^n_B}$ commutes with base change on $\qtwist n U B$ by
	\autoref{lemma:selmer-lcc}.
	It follows that over a geometric point of $\qtwist n U B$,
	the hypotheses $(1)$ and $(2)$ of \autoref{lemma:free-h1}, which follow
	from $(1')$ and $(2')$ in the statement of this proposition, are satisfied for any
	quadratic twist of $\mathscr F$. 
	Therefore, $ \selspace {\mathscr F^n_B}$ 
	corresponds to a sheaf of free $\mathbb Z/\nu \mathbb Z$-modules
	by \autoref{lemma:free-h1}.
	
	Finally, we compute the rank of this sheaf.
	Since we have shown  $\mathscr F$ is an irreducible $\mathbb Z/\nu \mathbb Z$ 
	locally constant constructible sheaf on $\mathscr U^n_B$, we can compute
	the formula for its rank after reduction modulo any prime $\ell \mid
	\nu$, and hence assume that $\nu$ is prime.
	The formula for the rank is given in 
	\cite[Lemma 5.1.3]{katz:twisted-l-functions-and-monodromy}.
	Technically, the argument is given there for lisse $\overline{\mathbb
	Q}_\ell$ sheaves, but the same computation applies to $\mathbb Z/\ell
	\mathbb Z$
	sheaves.
	In particular, with the above assumptions, if $B = \on{Spec} k$, for $k$ an
	algebraically closed field,
	$\selsheaf {\mathscr F^n_B}$ has rank $(2g - 2 + n) \cdot 2r + \sum_{x \in Z}
	\drop_x(\mathscr F)$.
\end{proof}

\subsection{Connecting points of the Selmer stack and Selmer groups}
\label{subsection:selmer-stack-sizes}

The next two lemmas connect the Selmer stack to the sizes of Selmer groups and
their proofs are quite similar to \cite[Proposition
3.23]{landesman:geometric-average-selmer} and \cite[Corollary
3.24]{landesman:geometric-average-selmer} respectively.
\begin{lemma}
	\label{lemma:h1-to-selmer-sheaf}
	Retaining notation from \autoref{notation:quadratic-twist-notation} and
	\autoref{notation:x-points}, suppose $n > 0$, 
	$2 \nu$ is invertible on $B$, and
	let $\pi: \selspace {\mathscr F^n_B} \to \qtwist n U B$ denote the structure
	map.
	Suppose $\mathscr F[\ell]$ is irreducible for each prime $\ell \mid
	\nu$.
	Then for $x \in \qtwist n U B(\mathbb F_q)$,
	\begin{align*}
	H^1(C_x, \mathscr F_x) \simeq \left( \pi^{-1}(x) \right) (\mathbb
	F_q).
	\end{align*}
	Note that the right hand $\left( \pi^{-1}(x) \right) (\mathbb
	F_q)$ acquires the structure of an abelian group as the $\mathbb F_q$-points of a
	locally constant constructible sheaf.
	\end{lemma}
\begin{proof}
	Using \autoref{lemma:selmer-lcc}, we know the formation of the Selmer
	sheaf commutes with base change, and hence for $\overline x$ a geometric
	point over $x$, the geometric fiber of $\selspace {\mathscr F^n_B}$ over
	$\overline x$ is identified with
	\begin{align*}
		R^1 \lambda_* ( j_* \mathscr F_{\overline{x}}) &\simeq
		H^1(C_{\overline x}, j_*\mathscr F_{\overline x}).
	\end{align*}
	To distinguish between \'etale and group cohomology, we use
	$H^i_{\grp}$ denote group cohomology and $H^i_{\et}$ to denote \'etale cohomology.
	Let $G_x := \on{Aut}(C_{\ol{x}}/C_{x})$. 
	The $\mathbb F_q$-points of $\pi^{-1}(x)$ are the $G_x$ invariants of
	$H^1_{\et}(C_{\ol{x}}, j_* \mathscr F_{\overline{x}})$. 
	That is, 
	$\pi^{-1}(x) (\mathbb F_q) = H_{\grp}^0(G_x, H^1_{\et}( C_{\overline
	{x}}, j_*\mathscr F_{\overline{x}}))$.
	
	We relate this group to $H^1(C_x, j_*\mathscr F_{x})$ using the Leray spectral sequence
	\begin{equation}
		\label{equation:selmer-to-points}
\begin{tikzpicture}[baseline= (a).base]
\node[scale=.84] (a) at (0,0){
	\begin{tikzcd}[column sep = tiny]
		0 \ar {r} & H^1_{\grp}( G_x, H^0_{\et}(C_{\overline x}, j_*\mathscr
		F_{\overline{x}})) \ar {r} & H^1_{\et}( C_x, j_*\mathscr
		F_x) \ar {r}{\theta} & H^0_{\grp}( G_x, H^1_{\et}(
		C_{\overline x}, j_*\mathscr F_{\overline{x}}))  \ar {r} &
		H^2_{\grp}( G_x, H^0_{\et}(C_{\overline x}, j_*\mathscr
		F_{\overline{x}})).
	\end{tikzcd}};
\end{tikzpicture}
\end{equation}
When $n > 0$,
we want to show $\theta$ is an isomorphism, so it suffices to 
show
$H^0_{\et}(C_{\overline x}, j_*\mathscr F_{\overline{x}}) = 0$.
This holds using \autoref{proposition:rank-description}(3').
\end{proof}

\begin{lemma}
	\label{lemma:selmer-identification}
	With the same assumptions as in \autoref{lemma:h1-to-selmer-sheaf},
	let $x \in \qtwist n U B(\mathbb F_q)$,
	and use $\sel_\nu(A_x)$ to denote the $\nu$-Selmer group of the generic
	fiber of $A_x$ over $U_x$.
	We have
	\begin{align*}
		\sel_\nu(A_x) \simeq \pi^{-1}(x)(\mathbb F_q).
	\end{align*}
\end{lemma}
\begin{proof}
	Using \autoref{lemma:selmer-lcc}, we know the geometric component group
	$\Phi_{A_{\overline x}}$ has
	order prime to $\nu$.
	As we are also assuming $q$ is prime to $\nu$, it follows from \cite[Proposition
	5.4(c)]{cesnavicius:selmer-groups-as-flat-cohomology-groups},
	$\sel_\nu(A_{x}) \simeq H^1_{\on{fppf}}(C_x, \mathscr A_x[\nu])$.
	Upon identifying fppf cohomology with \'etale cohomology 
	\cite[Th\'eor\`eme 11.7 $1^{\circ}$]{grothendieck:brauer-iii} and 
	combining this with \autoref{lemma:h1-to-selmer-sheaf}, we obtain the result.
\end{proof}

\section{Identifying Selmer elements via Hurwitz stacks}
\label{section:hurwitz-and-selmer}

Throughout this section, we'll work over 
the complex numbers
$B = \on{Spec} \mathbb C$.
One of the main new ideas in this article is that Selmer elements can actually
be parameterized by a Hurwitz stack.  The reason for doing this is that the
topological methods of the first part of the paper can, as in
\cite{EllenbergVW:cohenLenstra}, be used to control the number of $\F_q$-points
on certain Hurwitz stacks.  Using the identification between Selmer stacks and
Hurwitz stacks, we will thus be able to count $\F_q$-points on Selmer stacks.
These counts underlie our main theorems.

We produce an isomorphism between the Selmer stack, base changed to the complex
numbers, and a
certain Hurwitz stack parameterizing $\asp_{2r}(\mathbb
Z/\nu\mathbb Z)$-covers of our base curve $C$, over $B = \spec \mathbb C$.
This is shown in \autoref{proposition:selmer-to-hurwitz}.
Before jumping into the details, we describe the idea of this isomorphism
in
\autoref{subsection:idea-of-hurwitz-to-selmer}.
Continuing to the proof,
we give a monodromy theoretic description of torsion sheaves in
\autoref{subsection:elliptic-torsion-monodromy},
and give a monodromy theoretic description of torsors for torsion sheaves in
\autoref{subsection:torsors-monodromy}.
Finally, we identify the Selmer stack with certain Hurwitz stacks in
\autoref{subsection:selmer-and-hurwitz}.

\subsection{Idea of the isomorphism}
\label{subsection:idea-of-hurwitz-to-selmer}
We will now describe the idea of the proof in the context of torsion in abelian
varieties, though below the proof is carried
out in the more general context of symplectically self-dual sheaves.
The basic idea is that $\nu$-Selmer elements for an abelian variety $A'$ over
$U'$ of relative dimension $r$ with
N\'eron model $j_* A'$ over $C$ 
correspond to torsors for $j_* A'[\nu]$.
We can identify $j_* A'[\nu]$ with a
$\on{Sp}_{2r}(\mathbb Z/\nu\mathbb Z)$-Galois cover of $C$ via its Galois representation. 
We can then identify torsors for $j_* A'[\nu]$ 
as $\asp_{2r}(\mathbb Z/\nu\mathbb Z)$-covers of $C$, see
\autoref{definition:asp}.
This roughly corresponds to the fact that a torsor for $j_* A'[\nu]$ can
translate the monodromy of $j_* A'[\nu]$ by
an element of a geometric fiber
of $j_* A'[\nu]$, which can be identified with $(\mathbb Z/\nu\mathbb Z)^{2r} =
\ker\left(\asp_{2r}(\mathbb Z/\nu\mathbb Z)\to \sp_{2r}(\mathbb Z/\nu\mathbb
Z)\right)$.
The bulk of this section amounts to working out the precise conditions on the
monodromy of these Hurwitz stacks.

\subsection{Symplectically self-dual sheaves in terms of monodromy}
\label{subsection:elliptic-torsion-monodromy}

Recall that throughout this section, we are working over 
$B = \on{Spec} \mathbb C$.
As in \autoref{notation:curve-notation}, we begin with a smooth projective
connected $C$ curve over $\spec \mathbb C$, and a nonempty open subscheme $U
\subset C$. For $D \subset U$ a divisor,
we work with a symplectically self-dual sheaf $\mathscr F'$ over $U - D$ of rank
$2r$. 
A useful example to keep in mind will be when we are in the setting of
\autoref{example:abelian-special-fiber}
and there is an abelian scheme $A' \to U - D$ and $\mathscr F =
A'[\nu]$.
The main application will occur when $\mathscr F'$ 
is a quadratic twist of a sheaf $\mathscr F$, ramified over $D$.

We now describe $\mathscr F'$ in terms of its monodromy.
Fix a basepoint $p \in U-D$ and choose an identification $\mathscr F'[\nu]|_p \simeq \left(
	\mathbb Z/\nu
\mathbb Z\right)^{2r} $.
Because the fundamental group 
$\pi_1^{\top}(U-D, p)$ 
acts linearly on $\mathscr F'|_p$, we obtain a map $\pi_1^{\top}(U-D, p) \ra
\gl(\mathscr F'|_p)$.
Because the sheaf is symplectically self-dual, and we are working over $\mathbb
C$ where the cyclotomic character acts trivially,
this representation factors through
$\on{Sp}(\mathscr F'|_p)$.
In other words, we obtain a monodromy representation
\begin{align}
\label{equation:curve-monodromy}
	\rho_{\mathscr F'} &: \pi_1^{\top}(U-D, p) \ra
	\on{Sp}(\mathscr F'|_p) \simeq \on{Sp}_{2r}(\mathbb Z/\nu\mathbb Z).	
\end{align}
For convenience of notation, label
the points of $Z$ by $s_1, \ldots, s_{f+1}$.
As in \autoref{figure:genus-g-surface-with-loops},
we can draw oriented loops $\alpha_1, \ldots, \alpha_g, \beta_1, \ldots, \beta_g,
\gamma_1, \ldots,\gamma_n, \delta_1, \ldots, \delta_{f+1}$ based at $p$ which pairwise intersect only at $p$
so that 
\begin{enumerate}
\item $\alpha_1, \ldots, \alpha_g, \beta_1, \ldots, \beta_g$ forms a basis for
$H_1(C, \mathbb Z)$,
\item $\gamma_i$ is a loop winding once around $p_i$ corresponding to the local
inertia at $p_i$, where $p_1, \ldots, p_{n}$
are the $n$ points in $D$, and
\item $\delta_i$ is a loop winding once around $s_i$
corresponding to the local inertia at $s_i$.
\end{enumerate}
The above loops form generators of $\pi_1^{\top}(U-D, p)$ and satisfy the single relation 
\begin{align}
	\label{equation:fundamental-group-relation}
	(\alpha_1 \beta_1
	 \alpha_1^{-1} \beta_1^{-1}) \cdots (\alpha_g \beta_g \alpha_g^{-1} \beta_g^{-1})\gamma_1 \cdots \gamma_n\delta_1 \cdots \delta_{f+1}= \id.
\end{align}
Since $\mathscr F'$ is a $\mathbb Z/\nu\mathbb Z$ local system on $U-D$, the monodromy
representation $\rho_{\mathscr F'}$ determines $\mathscr F'$. 
\begin{figure}
\includegraphics[scale=.5]{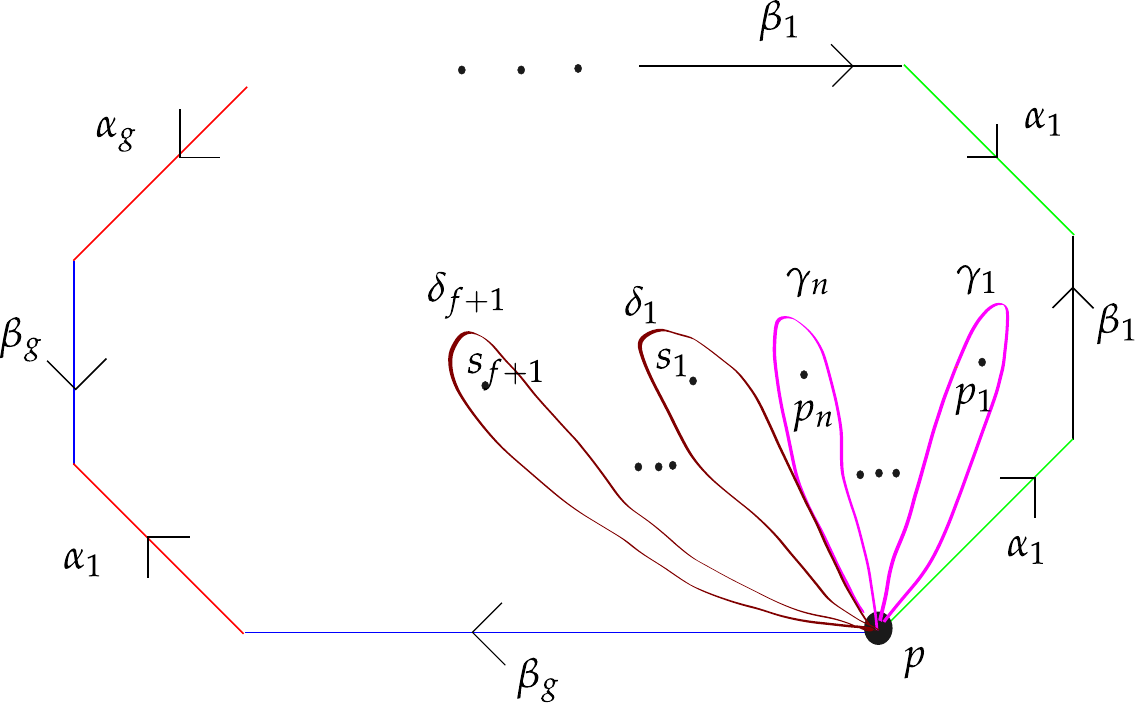}
\caption{This picture depicts a genus $g$ surface $X$ with $f+1$ punctures and
	an
$n$-point configuration. That is it corresponds to a point in $\Conf^n_{X}$. 
It includes the moving points $p_1, \ldots, p_n$, surrounded by loops
$\gamma_1,\ldots, \gamma_n$, the fixed punctures $s_1, \ldots, s_{f+1}$ surrounded by loops
$\delta_1,\ldots, \delta_{f+1}$, and
the standard generators for homology of the compact surface
$\alpha_1, \beta_1, \ldots, \alpha_g, \beta_g$.
}
\label{figure:genus-g-surface-with-loops}
\end{figure}

\subsection{Torsors for symplectically self-dual sheaves in terms of monodromy}
\label{subsection:torsors-monodromy}

The next result we are aiming toward is \autoref{lemma:torsor-description},
which gives a description of $j_* \mathscr F'$ torsors.

We retain
notation from \autoref{subsection:elliptic-torsion-monodromy}.
For $D \subset U$ a divisor, we use $j: U-D \to C$ to denote the inclusion.
As a first observation, we show that any torsor for $j_* \mathscr F'$ over $C$ is
determined by its restriction to $U-D$.
\begin{lemma}
\label{lemma:torsor-restriction}
The restriction map $H^1(C, j_* \mathscr F') \to H^1(U-D, \mathscr F')$ is injective.
Its image consists of those torsors $[\mathscr S] \in H^1(U-D, \mathscr F')$ such that
for each $q \in D \cup Z$, there is some sufficiently small complex analytic open neighborhood $C \supset W \ni q$ such that $\mathscr S|_{W-q}$ is the restriction
of a $j_*\mathscr F'|_{W}$ torsor to $W-q$.
\end{lemma}
\begin{proof}
Using the comparison between \'etale and complex analytic sheaf cohomology
\cite[Expos\'e XI, Th\'eor\'eme 4.4(iii)]{SGA4}
it is equivalent to prove the statement above in either the \'etale or analytic
topology, so we work in the analytic topology.
The spectral sequence associated to the
composition of the inclusion $j: U-D \to C$ and the global sections functor
yields an injection $H^1(C, j_*\mathscr F')
\hookrightarrow H^1(U-D, \mathscr F')$.
We may describe elements of $H^1(U - D, \mathscr F')$ as torsors in the complex analytic topology for
$\mathscr F'$.
The condition that a torsor $[\mathscr S] \in H^1(U-D, \mathscr F')$ lies in the
image of $H^1(C, j_*\mathscr F') \to H^1(U - D, \mathscr F')$
is precisely the condition that it extends to an $j_* \mathscr F'$ torsor over a
sufficiently small neighborhood of each point $q \in D \cup Z$.
\end{proof}

Recall our goal is to give a monodromy theoretic description of $j_*\mathscr F'$ torsors.
Using \autoref{lemma:torsor-restriction}, 
we can describe $j_*\mathscr F'$ torsors as $\mathscr F'$ torsors which extend
over a small neighborhood of each $p_i$ and $s_i$.
We next describe $\mathscr F'$ torsors, and then, in
\autoref{lemma:monodromy-for-torsor}, give the condition that such a torsor
extends over $D \cup Z$.
First, we introduce notation used to describe the monodromy representation
parameterizing $\mathscr F'$ torsors.

\begin{definition}
	\label{definition:asp}
	The {\em affine symplectic group} is $\asp_{2r}(\mathbb Z/\nu\mathbb
	Z):= \left( \mathbb Z/\nu\mathbb Z \right)^{2r} \rtimes
	\sp_{2r}(\mathbb Z/\nu\mathbb Z),$ where
	the action of $\sp_{2r}(\mathbb Z/\nu\mathbb Z)$ on $\left( \mathbb Z/\nu \mathbb Z
\right)^{2r}$ is via the standard action of matrices on their underlying free
rank $\mathbb Z/\nu\mathbb Z$-module of rank $2r$.
\end{definition}
\begin{remark}
	\label{remark:}
	By definition, $\asp_{2r}(\mathbb Z/\nu\mathbb Z)$ sits in an exact sequence
\begin{equation}
	\label{equation:asp-sequence}
	\begin{tikzcd}
		0 \ar {r} & \left( \mathbb Z/\nu \mathbb Z \right)^{2r} \ar
		{r}{\iota} & \asp_{2r}(\mathbb Z/\nu\mathbb Z) \ar {r}{\Pi} &
		\sp_{2r}(\mathbb Z/\nu \mathbb Z) \ar {r} & 0 
\end{tikzcd}\end{equation}
with inclusion map $\iota$ and quotient map $\Pi$.
With this presentation, $\asp_{2r}(\mathbb Z/\nu\mathbb Z)$
can be explicitly described as those matrices of the form
\begin{equation}
\scalebox{.95}{\parbox{.5\linewidth}{
		\nonumber
		\vspace{-.6cm}
\begin{align}
	\label{equation:mtor-matrix}
	\asp_{2r}(\mathbb Z/\nu\mathbb Z) \simeq
\left\{
	\begin{pmatrix}
		M & v \\
		0 &1
	\end{pmatrix}
	\in \gl_{2r+1}(\mathbb Z/\nu\mathbb Z) : 
	M \in \sp_{2r}(\mathbb Z/\nu\mathbb Z),
	v \in \left( \mathbb Z/\nu\mathbb Z \right)^{2r}
\right\}.
\end{align}
}
}
\end{equation}
\end{remark}
\begin{notation}
	\label{notation:asp-generalization}
Suppose
$H$ is a finite $\mathbb Z/\nu \mathbb Z$-module.
Define
$\ahsp_{2r}(\mathbb Z/\nu \mathbb Z) := (H \otimes_{\mathbb Z/\nu \mathbb Z} (\mathbb Z/\nu \mathbb
Z)^{2r}) \rtimes \sp_{2r}(\mathbb Z/\nu \mathbb Z)$, with 
$\sp_{2r}(\mathbb Z/\nu \mathbb Z)$
acting on $(\mathbb Z/\nu \mathbb
Z)^{2r}$ via the standard representation.
Concretely, if $H$ is of the form $H \simeq
\prod_{i=1}^m \mathbb Z/\nu_i \mathbb Z$,
then $\ahsp_{2r}(\mathbb Z/\nu \mathbb Z) \simeq \left(\prod_{i=1}^m \left( \mathbb
Z/\nu_i \mathbb Z \right)^{2r} \right) \rtimes \sp_{2r}(\mathbb Z/\nu \mathbb Z)
$.
In particular, 
$\ahsp_{2r}(\mathbb Z/\nu \mathbb Z)$
sits in a split exact sequence
\begin{equation}
	\label{equation:asp-m-sequence}
	\begin{tikzcd}
		0 \ar {r} & \prod_{i=1}^m \left( \mathbb Z/\nu_i \mathbb Z \right)^{2r}
		\ar
		{r}{\iota} & \ahsp_{2r}(\mathbb Z/\nu\mathbb Z) \ar {r}{\Pi} &
		\sp_{2r}(\mathbb Z/\nu \mathbb Z) \ar {r} & 0.
\end{tikzcd}\end{equation}
\end{notation}

We next describe the condition for a torsor for $\mathscr F'$ to extend over a
puncture, in terms of monodromy.
By \autoref{subsection:elliptic-torsion-monodromy}, $\mathscr F'$  
can be described in terms of $\rho_{\mathscr F'}$, which has target $\on{Sp}_{2r}(\mathbb Z/\nu\mathbb Z)$.
A torsor $\mathscr S$ for $\mathscr F'$
can be described in terms of $\mathscr F'$ together with the additional data of
transition functions lying in $\left( \mathbb Z/\nu\mathbb Z \right)^{2r}$.
In total, 
$\mathscr S$ can be described in terms of a monodromy representation
\begin{align*}
	\rho_ {\mathscr S} : \pi_1^{\top}(U-D, p) \to \asp_{2r}(\mathbb Z/\nu\mathbb Z).
\end{align*}

More formally, as pointed out by a referee, we can express the above description of $\mathscr F'$ torsors as
follows:
Suppose $G$ is a group acting on another group $H$ and $X$ is a space.
If $F$ is a $G$-torsor on $X$, then $H_F := (H\times F)/G \to X$ is a group
object over $X$, which is locally isomorphic to $H$. There is an equivalence of
categories between $G \ltimes H$-torsors and pairs $(F,S)$ where $F$ is a
$G$-torsor and $S$ is an $H_F$-torsor. We then apply this in our setting in the
case $G = \on{Sp}_{2r}(\mathbb Z/\nu\mathbb Z)$, $H = \left( \mathbb
Z/\nu\mathbb Z \right)^{2r},$ $X = U - D$, and $H_F = \mathscr F'$ to obtain an
equivalence between $\asp_{2r}(\mathbb Z/\nu\mathbb Z)$-torsors over $U-D$ restricting to
$\rho_{\mathscr F'}$ and torsors for $\mathscr F'$.

A composition of loops in $\pi_1^{\top}(U-D, p)$ maps under
$\rho_{\mathscr S}$ to the product of their corresponding matrices, 
viewed as elements of $\gl_{2r+1}(\mathbb Z/\nu\mathbb Z)$ via \eqref{equation:mtor-matrix}.
\begin{remark}
	\label{remark:reduction-mtor}
	By construction, for $\Pi$ as defined in \eqref{equation:asp-sequence},
$\Pi \circ \rho_{\mathscr S} = \rho_{\mathscr F'}$.
\end{remark}
We now describe the condition that a $\mathscr F'$ torsor extends to a $j_*\mathscr F'$ torsor.  We note, first of all, that by \autoref{lemma:torsor-restriction}, we know that this condition only depends on the restriction of $\rho_{\mathscr S}$ to local inertia groups.  Since these inertia groups are procyclic, this amounts to specifying some subset of $\asp_{2r}(\mathbb Z/\nu\mathbb Z)$, necessarily closed under conjugacy, in which the local monodromy groups are constrained to lie.  In the following proposition, we work out what these constraints look like in explicit matrix form.

\begin{lemma}
	\label{lemma:monodromy-for-torsor}
	With notation as in
	\autoref{subsection:elliptic-torsion-monodromy}, let $j: U-D \to C$
	denote the
	inclusion.
	Suppose $q \in Z \cup D$ with $\eta$ a small loop around $q$ whose image
	under $\rho_{\mathscr F'}$ corresponds to the local inertia at $q$.
	Let $d:=\drop_{q}(\mathscr F')$ so that, after choosing a suitable
	basis $\mathscr F'_p \simeq (\mathbb Z/\nu \mathbb Z)^{2r}$,
	we may write $\rho_{\mathscr F'}(\eta)$ in the form
	\begin{align*}
		\begin{pmatrix}
			M_1 & M_2 \\
			0 & \id_{2r-d}.
		\end{pmatrix}
	\end{align*}
	Under the identification of $\asp_{2r}(\mathbb Z/\nu\mathbb Z)$ as in \eqref{equation:mtor-matrix},
	we can extend a $\mathscr F'$ torsor $\mathscr S$ to an $j_*\mathscr
	F'$ torsor in some complex analytic neighborhood $W$ of $q$ if and only if
\begin{align}
	\label{equation:gamma-i-form}
	\rho_{\mathscr S}(\eta) =
	\begin{pmatrix}
		M_1 & M_2 & * \\
		0 & \id_{2r-d} & 0\\
		0 &0& 1
	\end{pmatrix}
\end{align}
for some vector $* \in (\mathbb Z/\nu\mathbb Z)^d$.
Stated more intrinsically, we can extend $\mathscr S$ to a $\mathscr F'$ torsor
if and only if the vector $v$ in \eqref{equation:mtor-matrix} lies in $\im (1 -
\rho_{\mathscr F'}(\eta))$
\end{lemma}
\begin{proof}
	First, \autoref{remark:reduction-mtor}
	shows all entries of 
	the matrix in \eqref{equation:gamma-i-form} 
	are necessary and sufficient for $\mathscr S$ to extend to a $j_*\mathscr F'$ torsor	
	except the first $2r$ entries of the last column, accounting for the $*$
	and the $0$.
	
	Choose a simply connected neighborhood $W$ of $q$ and fix a basepoint $p
	\in W$.
	To conclude the proof, we will show the claimed entries in the last
	column of \eqref{equation:gamma-i-form} from rows $d+1$ to $2r$ are $0$
	if and only if $\mathscr S|_{W-q}$ extends to a $j_*\mathscr F'$ torsor
	over $W$.
	We start by assuming the torsor extends, and aim to show the entries mentioned
	above are $0$.
	Note that we can identify $(\mathbb Z/\nu\mathbb Z)^{2r-d} |_{W} \subset
	j_* \mathscr F'|_{W}$ as a $\mathbb Z/\nu\mathbb Z$ subsheaf
	which restricts to 
	$\on{Span}(e_{d+1}, \ldots, e_{2r}) \subset \left( \mathbb Z/\nu\mathbb Z
	\right)^{2r} \simeq \mathscr F'|_q$ as the inertia invariants.
	Therefore, any $j_*\mathscr F'|_{W}$ torsor $\mathscr T$ has a distinguished
	$(\mathbb Z/\nu\mathbb Z)^{d}$ subtorsor, which is given as $\ker (1 -
		\rho_{\mathscr F'}(\eta))$.
	Since $W$ is simply connected, this $(\mathbb Z/\nu\mathbb Z)^{2r-d}$
	torsor is trivial, which implies that the local inertia at $q$ acts trivially on 
	$\on{Span}(e_{d+1}, \ldots, e_{2r}) \subset \left( \mathbb Z/\nu\mathbb Z
	\right)^{2r} \simeq j_*\mathscr F'|_q$, and hence there is a $0$ in
	\eqref{equation:gamma-i-form} as claimed.

	Conversely, suppose there is a $0$ in rows $d+1$ to $2r$ of the last
	column of \eqref{equation:gamma-i-form}.
	We will conclude by showing the torsor extends over $W$.
	We
	obtain a section of $\mathscr S$ over $W-q$ corresponding to each
	element of $(\mathbb Z/\nu \mathbb Z)^{2r-d}$,
	and hence a subsheaf $(\mathbb Z/\nu\mathbb Z)^{2r-d}|_{W - q} \subset
	\mathscr S|_{W - q}$.
	By gluing $(\mathbb Z/\nu\mathbb Z)^{2r-d}|_{W}$ to $\mathscr
	S|_{W - q}$ along $(\mathbb Z/\nu\mathbb Z)^{2r-d}|_{W - q}$, we obtain
	an $j_* \mathscr F'$ torsor $\mathscr T$, which is the desired extension of $\mathscr S$.
	\end{proof}

We can now describe $j_*\mathscr F'$ torsors in terms of monodromy data.
\begin{lemma}
	\label{lemma:torsor-description}
	With notation as in
	\autoref{subsection:elliptic-torsion-monodromy},
let $\mathscr F$ be an irreducible symplectically self-dual sheaf on $U$.
	Suppose $n > 0$.
	Fix some quadratic twist $\mathscr F'$ of $\mathscr F$, ramified
	along a degree $n$ divisor $D$, 
	in the sense that $\mathscr F'$ is some fiber of $\mathscr F^n_B$,
	so that we
	obtain a corresponding monodromy representation $\rho_{\mathscr F'}$. 
	Suppose $\rho_{\mathscr F'}$ satisfies the hypotheses $(1)$ and $(3)$ of
\autoref{proposition:rank-description}. 
	There are precisely $\nu^{(2g - 2 + n) \cdot 2r + \sum_{x \in Z} \drop_x(\mathscr F)}$ isomorphism classes of torsors for $j_*\mathscr F'$, which can be described in terms of monodromy data
	by specifying a representation $\rho_{\mathscr S} : \pi_1(U-D, p) \to
	\asp_{2r}(\mathbb Z/\nu\mathbb Z)$ up to $\asp_{2r}(\mathbb
	Z/\nu\mathbb Z)$ conjugacy,
	satisfying the following conditions:
	\begin{enumerate}
	\item The image of $\gamma_i$ under $\rho_{\mathscr S}$ is of the form
		\eqref{equation:mtor-matrix} with $M =-\id$.
	\item If $\drop_{s_i}(\mathscr F') = d_i$, the image of $\delta_i$ under $\rho_{\mathscr S}$ is 
		conjugate to a matrix of the form
		\eqref{equation:gamma-i-form}, where we take $(q,d)$ there to
		be $(s_i, d_i)$ here.
	\item We have $\Pi \circ \rho_{\mathscr S} = \rho_{\mathscr F'}$.
\end{enumerate}
Let $j: U - D \to C$ denote the inclusion.
As mentioned above, we consider two torsors $\mathscr T$ and $\mathscr T'$ 
equivalent if there is some $v \in \left( \mathbb Z/\nu\mathbb Z \right)^{2r}$
so that $\rho_{j^*\mathscr T}(\nabla) = \iota(v) \left( \rho_{j^*\mathscr
T'}(\nabla)\right) \iota(v)^{-1}$ for every $\nabla \in \{\alpha_1, \ldots,
\alpha_g, \beta_1, \ldots, \beta_g, \gamma_1, \ldots, \gamma_{n}, \delta_1,
\ldots, \delta_{f+1}\}$,
with $\iota$ as in \eqref{equation:asp-sequence}.
\end{lemma}
\begin{proof}
	Using \autoref{lemma:torsor-restriction}, we can describe torsors for
	$j_* \mathscr F'$ as torsors for $\mathscr F'$ which extend over a neighborhood of each $p_i \in D$.
	By \autoref{remark:reduction-mtor}, condition $(3)$ precisely corresponds to the condition that the
	$\on{Sp}_{2r}(\mathbb Z/\nu \mathbb Z)$ local system
	corresponding to $\mathscr S$ on $U - D$ is that
	associated to $\mathscr F'$, and hence $\mathscr S|_U$ is a $\mathscr F'$
	torsor.
	By \autoref{lemma:monodromy-for-torsor},
	an $\mathscr F'$ torsor extend to a $j_* \mathscr F'$ torsor over $p_1, \ldots, p_n$, if and only condition $(1)$
	holds, and extends over $s_1, \ldots, s_{f+1}$ if and only if condition
	$(2)$ holds.
	We consider the representations up to conjugacy, as this corresponds to
	a change of isomorphism $\mathscr F'|_p \simeq \left( \mathbb Z/\nu\mathbb
	Z \right)^{2r}$, and
	expresses the usual condition for two torsors to be equivalent.
%

	To conclude, we wish to see that there are $\nu^{(2g - 2 + n) \cdot 2r + \sum_{x \in Z} \drop_x(\mathscr F)}$
 isomorphism classes of torsors specified by the above data.
	Indeed, we see there are $\nu^{2r}$ possible values $\rho_{\mathscr
	S}$ can take on each of the loops $\alpha_1, \ldots,
	\alpha_g, \beta_1, \ldots, \beta_g$
	in order to satisfy $(3)$.
	For each $\gamma_i$, there are $\nu^{\drop_{p_i}(\mathscr F')} = \nu^{2r}$ possible
	values of $\rho_{\mathscr S}$, because $\Pi(\rho_{\mathscr S}(\gamma_i))
	= -\id_{2r}$.
	For each $\delta_i$, there are $\nu^{\drop_{s_i}(\mathscr F')} = \nu^{\drop_{s_i}(\mathscr F)}$ possible
	values of $\rho_{\mathscr S}$.
	We additionally must impose the condition that
	$\prod_{i=1}^g [\alpha_i, \beta_i] \prod_{i=1}^n \gamma_i
	\prod_{i=1}^{f+1} \delta_i = \id$,
	from the relation \eqref{equation:fundamental-group-relation}
	defining the fundamental group, and that we consider
	these torsors up to conjugacy.
	Before imposing these two conditions, there are $\nu^{(2g + n) \cdot 2r + \sum_{x \in Z} \drop_x(\mathscr F)}$
	possible tuples of matrices. The first condition imposes $\nu^{2r}$
	independent constraints on the matrices. 
	Further, the conjugation action
	always identifies $\nu^{2r}$ elements since the representation is center
	free, using that it is irreducible and that $\asp_{2r}(\mathbb Z/\nu
	\mathbb Z) \subset \on{GL}_{2r+1}(\mathbb Z/\nu \mathbb Z)$ contains no
	scalars, other than $\id$.
	Altogether, this yields $\nu^{(2g -2 + n) \cdot 2r + \sum_{x \in Z} \drop_x(\mathscr F)}$
	such torsors.
\end{proof}

\subsection{Identifying Selmer stacks with Hurwitz stacks}
\label{subsection:selmer-and-hurwitz}

We will use the above description of torsors to identify the Selmer stack with a
certain Hurwitz stack in \autoref{proposition:selmer-to-hurwitz}. We next define
that Hurwitz stack.

\begin{notation}
	\label{notation:sel-hur}
	Let $B = \spec \mathbb C$.
Given a symplectically self-dual sheaf $\mathscr F$ over $U$ as in
\autoref{notation:quadratic-twist-notation}, and fixing values of $\nu$ and $n$,
we now use the notation
$\selhur {\mathscr F^n_B} H$ to indicate the stack
$\hur G n Z {\mathcal S} C B$
as in \autoref{definition:fixed-hur},
for $n, Z, C, B$ as in 
\autoref{notation:quadratic-twist-notation}
and $G,\mathcal S$ as we define next.
Let $\nu_1, \ldots, \nu_m \mid \nu$ and write $H \simeq
\prod_{i=1}^m \mathbb Z/\nu_i \mathbb Z$.
Take $G := \ahsp_{2r}(\mathbb Z/\nu \mathbb Z)$.
Take $\mathcal S$ to be the orbit under the conjugation action of $G$ of the
following subset of $\phi \in \hom(\pi_1(\Sigma_{g,n+f+1}),G)$.
Any such $\phi$ sends a small loop $\gamma_i$ around one of the $n$ points
$p_i$ for $1 \leq i \leq n$,
to an element $g \in G$ so
that 
$\Pi(g) = -\id$, for $\Pi$ as defined in \eqref{equation:asp-m-sequence}.
If $\alpha_1, \ldots, \alpha_g, \beta_1, \ldots, \beta_g \subset \Sigma_{g,f+1} \subset \Sigma_g$ are a fixed set of simple closed curves forming a standard generating set for the first homology of $\Sigma_g$,
we require that $\Pi(\phi(\alpha_i)) \in \{\pm a_i\}$,
$\Pi(\phi(\beta_j)) \in \{\pm b_j\}$,
where $a_i := \rho_{\mathscr F}(\alpha_i)$ and $b_j := \rho_{\mathscr
F}(\beta_j)$.
The local inertia around $s_i$, the $i$th puncture among the $f+1$ punctures,
maps to $(M_i,v_i)$, where $M_i$ is the given local inertia for $\mathscr
F$ and $v_i \in \im (M_i - \id)$.
\end{notation}

\begin{remark}
\label{remark:}
The condition in \autoref{notation:sel-hur} that the $\alpha_i$ and $\beta_j$
map to $\pm a_i$ and $\pm b_j$ under $\Pi \circ \phi$ may seem to depend on choices of the $\alpha_i$ and
$\beta_j$, but it can be expressed independently of these choices 
as follows: if $\zeta : \sp_{2r}(\mathbb Z/\nu \mathbb Z) \to \sp_{2r}(\mathbb
Z/\nu \mathbb Z)/\{\pm 1\}$ is the quotient map,
$\zeta \circ \Pi \circ \phi = \zeta \circ \rho_{\mathscr F}$.
\end{remark}

In order to show the construction in \autoref{notation:sel-hur} gives a Hurwitz
stack as in \autoref{definition:fixed-hur}, we need to show the set $\mathcal S$ is
invariant under the action of $\pi_1(\conf n U B)$.
We now verify this.
\begin{lemma}
	\label{lemma:hurwitz-space-invariant-under-braid-group}
	The set $\mathcal S$ from \autoref{notation:sel-hur} is a subset of
$\hom(\pi_1(\Sigma_{g,n+f+1}),G)$
which is invariant under the action of $\pi_1(\conf n U B)$.
\end{lemma}
\begin{proof}
	Throughout this proof, it may help the reader to refer to \autoref{remark:explicit-mcg-action}, which
gives an explicit description of the action of $\pi_1(\conf n U B)$.
	Recall we use $\gamma_i$ for the loop giving inertia around $p_i$ for $1
	\leq i \leq n$ and $\delta_i$ for the loop giving inertia around $s_i$,
	$1 \leq i \leq f+1$.
	First, to show the image of $\gamma_i$ is preserved by the 
	$\pi_1(\conf n U B)$ action,
	note that $-\id$ is preserved by this action. Therefore, the condition that $\Pi(g)
	= -\id$ is preserved by the action as well. Hence, the condition that
	$\gamma_i$ has monodromy $g$ with $\Pi(g)
	= -\id$ is preserved by the action of $\pi_1(\conf n U B)$.
	The condition on the $\alpha_i$ and $\beta_j$ is invariant as passing
	one of the $n$ points across $\alpha_i$ or $\beta_j$ has the effect of
	negating $\Pi(\phi(\alpha_i))$ or $\Pi(\phi(\alpha_i))$, since
	$\Pi(\gamma_t)	= -\id$.
	As for the loops $\delta_i$, since the
	loops $\gamma_i$ have inertia $g$ with $\Pi(g) =
	-\id$, which lies in the center of $G$, 
	the matrices $M_i$ defined in
	\autoref{notation:sel-hur}
	are preserved by conjugation under
	$-\id$. Therefore, the $1$-eigenspace $\ker(1-M_i)$ is preserved by
	conjugation under $- \id$, and so the same holds for $\im (1-M_i)$.
	Thus, the set of such homomorphisms to $G$ is indeed preserved by the
	action of $\pi_1(\conf n U B)$.
\end{proof}

\begin{hypotheses}
\label{hypotheses:theta-map}
Suppose $n > 0$, $B = \spec \mathbb C$, and
$\mathscr F$ is an irreducible symplectically self-dual lcc sheaf
of free $\mathbb Z/\nu \mathbb Z$-modules
which satisfies the hypotheses $(1)$ and $(3)$ of
\autoref{proposition:rank-description}. 
There is a map 
\begin{align*}
	\theta : \selspace {\mathscr F^n_B}
	\to
\selhur {\mathscr F^n_B} {\mathbb Z/\nu \mathbb Z}
\end{align*}
obtained via the bijection of \autoref{proposition:selmer-to-hurwitz},
which sends a
torsor to the corresponding $\asp_{2r}(\mathbb Z/\nu \mathbb Z)$-cover for some quadratic twist $\mathscr
F'$ of $\mathscr F$.
\end{hypotheses}

\begin{proposition}
	\label{proposition:selmer-to-hurwitz}
	With hypotheses as in \autoref{hypotheses:theta-map},
	for $n > 0$, the map $\theta$, defined over $B = \spec \mathbb C$, is
	an isomorphism. 
\end{proposition}
\begin{proof}
	Note that the projection
	$\selhur {\mathscr F^n_B} {\mathbb Z/\nu \mathbb Z} \to \qtwist n U B$
	sends a point of $\selhur {\mathscr F^n_B} {\mathbb Z/\nu \mathbb
	Z}$, thought of as an $\asp_{2r}(\mathbb Z/\nu \mathbb Z)$-cover, to the
	corresponding $\sp_{2r}(\mathbb Z/\nu \mathbb Z)$-cover.
	The projection 
	$\selspace {\mathscr F^n_B} \to \qtwist n U B$ sends a torsor $\mathscr
	T$ for some quadratic twist $\mathscr F'$ to the corresponding $\mathscr
	F'$.
	Both $\selhur {\mathscr F^n_B} {\mathbb Z/\nu \mathbb Z}$ and
	$\selspace {\mathscr F^n_B}$ are finite \'etale covers of $\qtwist n U B$,
	and by \autoref{lemma:torsor-description}, $\theta$ defines a bijection
	between geometric points
	over points of $\qtwist n U B$, corresponding to a
	chosen degree $n$ quadratic twist $\mathscr F'$ of $\mathscr \mathscr F$.
	In order to show $\theta$ is an isomorphism, it is enough to show the
	bijection between two finite \'etale covers of $\qtwist n U B$ defines a homeomorphism. Indeed, we may verify this claim
	locally on $\qtwist n U B$, 
	in which case is enough to verify it on sufficiently small analytic open
	covers of $\qtwist n U B$. We can choose a small open neighborhood of
	some geometric point $[\mathscr
	F'] \in \qtwist n U B$,
	corresponding to varying the points $p_i$, along with the corresponding
	double cover, in a small, pairwise disjoint open analytic discs
	of $C$. Since the bijection of \autoref{lemma:torsor-description} is
	compatible with such variation in the points $p_i$, we obtain the
	desired isomorphism.
\end{proof}

\begin{warn}
	\label{warning:selmer-not-hurwitz-over-finite-fields}
	The Selmer stack $\selspace {\mathscr F^n_{\on{Spec} \mathbb F_q}}$ over
	$\mathbb F_q$ will {\em not} in general be isomorphic to
	the Hurwitz stack of $\asp_{2r}(\mathbb Z/\nu\mathbb Z)$-covers we are
	considering.
	Rather, they will be twists of each other, and the Hurwitz stack only
	becomes isomorphic over
	$\overline{\mathbb F}_q$.
	The reason for this is that the monodromy representation associated to
	$\mathscr F$ may fail to be 
	contained in $\on{Sp}_{2r}(\mathbb Z/\nu\mathbb Z)$, and
	in general it will only be contained in $\on{GSp}_{2r}(\mathbb
	Z/\nu\mathbb Z)$,
	the general symplectic group. However, once one ensures all roots of
	unity lie in the base field, this issue goes away.
\end{warn}


Computing the average size of a Selmer group in a quadratic twist family will
come down to counting $\F_q$-rational points on a Selmer stack.  We will
want to compute not only averages, but higher moments.  This will require
counting points on {\em fiber products} of Selmer stacks.  But, as the following
corollary shows, these stacks are isomorphic, making them amenable to the methods of this paper.

\begin{corollary}
	\label{corollary:selmer-to-hurwitz-moments}
	With hypotheses as in \autoref{hypotheses:theta-map},
	let $H \simeq \prod_{i=1}^m \mathbb Z/\nu_i \mathbb Z$.
	The map $\theta$, defined over $B = \spec \mathbb C$, induces an
	isomorphism
	\begin{align*}
	\theta^m: \selspace {\mathscr F^n_B[\nu_1]} \times_{\qtwist n U B}
	\cdots \times_{\qtwist n U B}\selspace {\mathscr F^n_B[\nu_m]}
	\to \selhur {\mathscr F^n_B} H.
	\end{align*}
\end{corollary}
\begin{proof}
	It follows from the definition of 
	$\selhur {\mathscr F^n_B} H$ as in \autoref{notation:sel-hur}
	that
	\begin{align*}
		\selhur {\mathscr F^n_B} H \simeq \selhur {\mathscr
		F^n_B[\nu_1]} {\mathbb Z/\nu \mathbb Z} \times_{\qtwist n U B}
		\cdots \times_{\qtwist n U B} \selhur {\mathscr
		F^n_B[\nu_m]}
	{\mathbb Z/\nu \mathbb Z}.
	\end{align*}
	The map $\theta$ from
	\autoref{proposition:selmer-to-hurwitz}
	also induces isomorphisms
	$\selhur{\mathscr F^n_B[\nu_i]}{\mathbb Z/\nu_i \mathbb Z} \to 
	\selspace {\mathscr F^n_B[\nu_i]}$.
	For $\nu_i \mid \nu$, we also have
$\selhur{\mathscr F^n_B[\nu_i]}{\mathbb Z/\nu_i \mathbb Z} \simeq
\selhur{\mathscr F^n_B[\nu_i]}{\mathbb
	Z/\nu \mathbb Z}$ from the definition.
	The result follows from
	\autoref{proposition:selmer-to-hurwitz}
	by taking appropriate fiber products of isomorphisms over
	$\qtwist n U B$.
\end{proof}

\section{Computing the monodromy of Hurwitz stacks}
\label{section:big-monodromy}

In this section, we compute the image of the monodromy representation related
to Selmer stacks. This will be used later to determine their connected
components.
We first control the monodromy when $\nu$ is prime in
\autoref{subsection:prime-monodromy}.
We then control the monodromy for prime power $\nu$ in
\autoref{subsection:prime-power-lemmas-for-big-monodromy}
and for composite $\nu$ in
\autoref{subsection:general-composite-lemmas-for-big-monodromy}.
The above shows that the monodromy is sufficiently large, but does not determine
it exactly.
We will, however, precisely describe the image of the Dickson invariant map in
\autoref{subsection:determinant-image}.

\subsection{Computing the monodromy when $\nu$ is a prime}
\label{subsection:prime-monodromy}

We first consider the case $\nu$ is prime.
The main result in this case is \autoref{theorem:big-monodromy-mod-ell},
which is a generalization of 
\cite[Theorem 6.3]{hall:bigMonodromySympletic}
from the case that we have an elliptic curve over a genus $0$ base to the case
of a general symplectically self-dual sheaf over a base curve of genus $g$.
We begin with a definition of the monodromy representation for general odd
$\nu$.

\begin{definition}
	\label{definition:monodromy}
	With notation as in \autoref{notation:quadratic-twist-notation},
	suppose $B$ is integral, 
	$\nu$ is odd,
	and $2 \nu$ is invertible on $B$.
	Choose a basepoint $x \in \qtwist n U B$.
	Let $V_{\mathscr F^n_B} := R^1 \lambda_* \left( j_* \mathscr F^n_B \right)_x$.
	The Selmer sheaf is a finite \'etale cover of $\qtwist n U B$ by
	\autoref{lemma:selmer-lcc}
	and so
	induces a monodromy representation 
	$\rho_{\mathscr F^n_B}: \pi_1(\qtwist n U B) \to
	\on{Aut}(V_{\mathscr F^n_B})$.
	For any geometric point ${\overline{b}} \to B$, we also obtain a geometric monodromy
	representation
	$\rho_{\mathscr F_{\overline{b}}^n}: \pi_1(\qtwist n {U_{\overline{b}}} {\overline{b}}) \to
	\on{Aut}(V_{\mathscr F^n_{\overline{b}}})$.
\end{definition}
\begin{warn}
\label{warning:}
	Note that $\rho_{\mathscr F^n_B}$ is a representation of the fundamental group of
	configuration space, while we use $\rho_{\mathscr F'}$ very differently
	in \eqref{equation:curve-monodromy}
	for a
	representation of the fundamental group of the curve $U-D$ itself.
\end{warn}

\begin{remark}
		\label{remark:orthogonal-containment}
		Using that $\gcd(\nu, 2) = 1$, there is a nondegenerate 
		pairing on $V_{\mathscr F^n_B}$.
		The pairing is obtained as the composition
		\begin{align*}
		H^1(C, j_* (\mathscr F^n_B)_x) \times H^1(C, j_* (\mathscr F^n_B)_x)
		&\to H^2(C, \wedge^2 (j_* \mathscr F^n_B)_x ) \\
		&\to H^2(C, j_* (\wedge^2
	\mathscr F^n_B)_x )\\
	& \to H^2(C, j_* \mu_\nu ) \\
	&\to \mathbb Z/\nu \mathbb Z
		\end{align*}
	using Poincar\'e duality \cite[V Proposition
	2.2(b)]{Milne:etaleBook}, which is preserved by this monodromy
representation. 
	The pairing above is symmetric because Poincar\'e duality on curves is
antisymmetric and the pairing on $j_*(\mathscr F^n_B)_x$ is antisymmetric, coming
	from the assumption that $\mathscr F$ is symplectically self-dual.
Let $Q_{\mathscr F^n_B}$ denote the associated quadratic form.
Then, $\rho_{\mathscr F_B^n}$ factors through the orthogonal group
$\on{O}(Q_{\mathscr F^n_B})$ associated to the above symmetric bilinear pairing.
\end{remark}

	We now set some assumptions, which will serve as our hypotheses going
	forward.
\begin{hypotheses}
	\label{hypotheses:big-monodromy-assumptions}
	Suppose $\nu$ is an odd integer and $r \in \mathbb Z_{>0}$ so that every prime $\ell \mid \nu$
	satisfies $\ell > 2r + 1$.
	Suppose we have a rank $2r$, tame, symplectically self-dual lcc sheaf
	of free $\mathbb Z/\nu \mathbb Z$-modules, $\mathscr F$,
	over $U \subset C$, 	a nonempty proper open in a smooth proper curve $C$ with geometrically connected fibers over an integral affine base $B$.
	Suppose $Z := C - U$ is nonempty and finite \'etale over $B$.
	Assume further $2 \nu$ is invertible on $B$.
	Fix a geometric point ${\overline b} \to B$.
	We assume there is some point $x \in C_{\overline b}$ at which
	$\drop_x(\mathscr F_{\overline b}[\ell]) = 1$ for every prime $\ell \mid \nu$.
	Also suppose
	$\mathscr F_{\overline b}[\ell]$ is irreducible
	for each $\ell \mid \nu$, and that 
	the map $j_* \mathscr F_{\overline b}[\ell^w] \to j_* \mathscr F_{\overline b}[\ell^{w-t}]$ is
	surjective for each prime $\ell \mid \nu$ such that $\ell^w \mid \nu$, and
	$w \geq t$, as in
	hypotheses $(1)$ and $(3)$ of
	\autoref{proposition:rank-description}.
Let $f+1 := \deg (C - U)$ and let $n$ be a positive even integer.

Note that if we are in the situation of
	\autoref{example:abelian-special-fiber}, hypothesis 
\autoref{proposition:rank-description}(1) 
	in the case $\mathscr F_{\overline b}=A[\nu]$ is
satisfied whenever the geometric component group $\Phi_{A_{\overline b}}$ has order prime to $\nu$, by \autoref{lemma:free-h1}.
	If we additionally assume 
	$A_{\overline b}$ has
	multiplicative reduction at some point of $U_{\overline b}$, with toric part of
	dimension $1$,
	then 
$\drop_x(\mathscr F_{\overline b}[\ell]) = 1$ for every prime $\ell \mid \nu$.
\end{hypotheses}

\begin{theorem}[Generalization of ~\protect{\cite[Theorem 6.3]{hall:bigMonodromySympletic}}]
	\label{theorem:big-monodromy-mod-ell}
	Suppose $\nu = \ell > 2r + 1$ is prime.
	Choose a geometric basepoint $x \in \qtwist n U B$ over a geometric
	point ${\overline b} \to
	B$.
	We next recall
	our assumptions from
	\autoref{hypotheses:big-monodromy-assumptions}:
	we assume $2\nu$ is invertible on the integral affine base $B$
	and $\mathscr F_{\overline b}$ is a rank $2r$ irreducible lcc symplectically self-dual sheaf. We assume there is some point $x \in C_{\overline b}$ at which
	$\drop_x(\mathscr F_{\overline b}) = 1$, and $\mathscr F_{\overline b}$ satisfies hypotheses
	\autoref{proposition:rank-description}(1) and (3).

	For $n$ an even integer satisfying
	\begin{align*}
		n > \max\left(2g,\frac{2(2r+1)(f+1) - \sum_{y \in D_x({\overline b})}
	\drop_y(\mathscr F)}{2r} - (2g-2) \right),
	\end{align*}
		the geometric monodromy representation $\rho_{\mathscr
		F_{\overline b}^n} : \pi_1(\qtwist n {U_{\overline b}} {\overline b}) \to
		\on{Aut}(V_{\mathscr F^n_{\overline b}})$
		has $\im(\rho_{\mathscr F^n_{\overline b}})$
		 of index at most $2$
		in 
		$\on{O}(Q_{\mathscr F^n_{\overline b}})$,
		for $Q_{\mathscr F^n_{\overline b}}$ as in
		\autoref{remark:orthogonal-containment}. Moreover,
		$\im(\rho_{\mathscr F^n_{\overline b}}) \neq \on{SO}(Q_{\mathscr F^n_{\overline b}})$.
\end{theorem}
\begin{proof}[Proof Sketch]
	A fair portion of this proof is essentially explained in 
	\cite[Theorem 6.3]{hall:bigMonodromySympletic}, see also 
	\cite[Theorem 3.4]{zywina:inverse-orthogonal} for an explicit version
	and 
	\cite[\S6.6]{hall:bigMonodromySympletic}
	for the generalization to $r > 1$.
	We now briefly outline the details needed in the generalization.
	For the purposes of the proof, we may assume that $B = {\overline b}$.
	Since $n > 2g$, 
	by \cite[Theorem 2.2.6]{katz:twisted-l-functions-and-monodromy}, 
	there is a map $h:
C_{x} \to \mathbb P^1$ of degree $n$ 
which is simply
branched, the branch locus of $h$ is disjoint from $h((Z \cup
D)_{x})$, $h$ separates points of $(Z \cup
D)_{x}$, 
and precisely one point $\delta \in D_{x}$ maps to $\infty \in \mathbb P^1$.
Let $\on{br}(h)$ denote the branch locus of $h$.
Take $W \subset \mathbb P^1$ to be the complement of $\on{br}(h) \cup h(Z
	\cup D)$. Note that $\infty \notin W$ by assumption.
	Then, 
	one can show
	as in \cite[Theorem 5.4.1]{katz:twisted-l-functions-and-monodromy}
	that there is a map $\phi: W \to \qtwist n U {\overline b}$
	which we now describe.

	In order to specify a finite double cover of $C \times_{\overline b} W$, it is equivalent
	to specify a rank $1$ locally constant constructible $\mathbb Z/\ell
	\mathbb Z$ sheaf on an open of $C \times_{\overline b} W$ whose monodromy is trivialized by that
	double cover.
	Let $\mathscr F'$ denote the quadratic twist of $\mathscr F$
	corresponding to our chosen geometric basepoint $x \in \qtwist n U B$.
	Then, $\mathscr F' = \mathscr F \otimes \mathbb V$, for
	$\mathbb V$ the rank $1$ locally constant constructible sheaf on 
	$U - D$ given by
	$t_* (\mathbb Z/\ell \mathbb Z) / (\mathbb Z/\ell \mathbb Z)$, for $t: X
	\to U$ the finite \'etale double cover associated to $x$.
	We will now find a family of locally constant constructible sheaves
	(corresponding to quadratic twists) over $W$ whose fiber over $0 \in W$ is
	$\mathbb V$.
	To this end, let $\chi$ denote the rank $1$ locally constant
	constructible sheaf on $\mathbb G_m :=
	\mathbb A^1 - \{0\}$
	corresponding to the double cover $\mathbb G_m \to \mathbb G_m$ via
	multiplication by $2$.
	There is a map $\alpha' : \mathbb A^1 \times \mathbb A^1 - \Delta \to \mathbb G_m$
	given by $(x,y) \mapsto x - y$.
	Consider the map $(h,\id): C \times \mathbb P^1 \to \mathbb P^1 \times
	\mathbb P^1$ and let $Y := (h,\id)^{-1}(W \times W -
	\Delta)$.
Let $\alpha$ denote the composition $Y \xrightarrow{(h,\id)} \mathbb A^1 \times \mathbb A^1 - \Delta
	\xrightarrow{\alpha'} \mathbb G_m$ and let $\mathbb W := \alpha^* \chi$.
	Let 
	$\pi_2 : Y \to \mathbb A^1$ denote the second projection.
	Take $\mathbb V' := \mathbb W|_{h^{-1}(W-0) \times 0} \otimes \mathbb V|_{h^{-1}(W-0)}$,
	viewed as a sheaf on $h^{-1}(W-0) \subset C$.
	Then $(\mathbb V' \otimes \mathbb W^\vee)|_{h^{-1}(W-0)}$ recovers
	$\mathbb V|_{h^{-1}(W-0)}$. 
	Now, the locally constant constructible sheaf $\pi_2^* \mathbb V' \otimes \mathbb W^\vee$ determines
	a locally constant constructible sheaf on $Y$. The above identifies the
	fiber of this over the point $0$ with a restriction of $\mathbb V$.
	Since both $\mathbb V'$ and $\mathbb W$ correspond to representations
	with image $\mathbb Z/2 \mathbb Z$, the same is true of  
	$\pi_2^* \mathbb V' \otimes \mathbb W^\vee$, and hence this sheaf
	corresponds to a finite \'etale double cover of $Y$.
	Overall, this gives a double cover of $C \times W$, ramified
	along a degree $n$ divisor. This divisor is \'etale and disjoint from
	$Z$ over $C\times W$, and hence
	yields a map $\phi:W \to
	\qtwist n U {\overline b}$, by the universal property of $\qtwist n U {\overline b}$ as a
	moduli stack of finite double covers branched over a divisor disjoint from $Z$.
	The sheaf $\phi^*\selsheaf {\mathscr F^n_{\overline b}}$ may also be
	viewed as the 
middle convolution $\on{MC}_\chi((h_* \mathscr F')|_W)$. (See \cite[Proposition
	5.3.7]{katz:twisted-l-functions-and-monodromy} for an analogous
statement in the $\ell$-adic setting.)


	Since
	$\phi^*\selsheaf {\mathscr F^n_{\overline b}}$
	is the middle convolution 
	$\on{MC}_\chi(
(h_* \mathscr F')|_W)$
	of the irreducible sheaf $(h_* \mathscr F')|_W$, 
	we obtain that $\phi^*\selsheaf {\mathscr F^n_{\overline b}}$ is
	irreducible.
	Here we are using that the middle convolution of an irreducible sheaf is
	irreducible.
	This holds because middle convolution is invertible, and hence sends
	irreducible objects to irreducible objects. A proof is given in
	\cite[Theorem 3.3.3(2d)]{katz2016rigid} for $\overline{\mathbb Q}_\ell$
	sheaves, but the same proof works for sheaves of $\mathbb Z/\ell \mathbb Z$-modules.
	(See also \cite[Corollary
		1.6.4]{dettweiler:on-the-middle-convolution-of-local-systems}
	for a proof in the characteristic $0$ setting.)

	We may moreover compute the monodromy of $\phi^* \selsheaf
	{\mathscr F^n_{\overline b}}$
	at the geometric points of $\mathbb A^1 - W$.
	At branch points of $h$, the monodromy is unipotent via the calculation 
	done in \cite[Proposition
	5.4.1]{katz:twisted-l-functions-and-monodromy}.
	At the other geometric points of $\mathbb A^1 - W$ the calculation is the
	same as in the proof of \cite[Theorem 6.3 and Lemma 6.5]{hall:bigMonodromySympletic}.
	In particular, at each of the geometric points of $h(D)$, the monodromy
	is also unipotent.
	This is also explained in \cite[Proposition
	5.4.1, p. 99, last 3 lines]{katz:twisted-l-functions-and-monodromy}, where it is also shown
	that $\drop_y(\phi^* \selsheaf {\mathscr F^n_{\overline b}}) \leq 2r$ at all such
	geometric points $y \in \mathbb A^1 - W$.

	We conclude by verifying the three hypotheses of 
	\cite[Theorem 3.1]{hall:bigMonodromySympletic}, whose conclusion implies
	the statement of the theorem we are proving.
	Note that the monodromy of the sheaf $\phi^* \selsheaf {\mathscr F^n_{\overline b}}$ is generated by
	the inertia around $\on{br}(h), h(Z_{x}),$ and $h(D_{x} - \delta)$.

	We need to verify hypotheses $(i),(ii),$ and $(iii)$ 
	\cite[Theorem 3.1]{hall:bigMonodromySympletic}, as well as show the
	image of monodromy contains a reflection and an isotropic shear, in the
	language of \cite[p. 185]{hall:bigMonodromySympletic}.
	We claim the local monodromy around a point of $h(Z_{x}) \subset W$ over
	which $A_x$ has toric
	part of codimension $1$ acts as a reflection, while the local monodromy
	around a point of $h(D_x)$ acts as an isotropic shear.
	These claims are proven in the case of elliptic curves in
	\cite[Lemma 6.5]{hall:bigMonodromySympletic} and the proof for higher
	dimensional abelian varieties is analogous.

	In order to verify $(i)$, take the value labeled $r$ in \cite[Theorem
	3.1]{hall:bigMonodromySympletic} to be what we are calling $2r = 2
	(\dim A - \dim U_{\overline b})$. Maintaining our notation, we have seen above that 
	the images of inertia
	around the above mentioned geometric points $y \in S := \mathbb A^1 - W$ generate an irreducible
	representation, and satisfy $\drop_y(\phi^* \selsheaf
	{\mathscr F^n_{\overline b}}) \leq 2(\dim A - \dim U_{\overline b})$.
	This verifies \cite[Theorem 3.1(i)]{hall:bigMonodromySympletic}.

	Taking $S_0 \subset S$ to be the subset of the $f+1$ geometric points over $h(Z)$,
	we find $2(2r+1)(\# S_0({\overline b})) \leq \dim V$ by rearranging the assumption that
	\begin{align*}
		n > \frac{2(2r+1)(f+1) - \sum_{y \in Z({\overline b})}
	\drop_y(\mathscr F)}{2r} - (2g-2),
	\end{align*}
	using our computation for the dimension
of $V$ from \autoref{proposition:rank-description}.
	This verifies \cite[Theorem 3.1(ii)]{hall:bigMonodromySympletic}.

	Finally, every $\gamma \in S - S_0$ has unipotent monodromy, as we
	showed above.
	Hence, every $\gamma \in S- S_0$
	has order a power of $\ell$, so has order prime to $(2r+1)!$ whenever
	$\ell > 2r +1$. This verifies \cite[Theorem
	3.1(iii)]{hall:bigMonodromySympletic}.
	Applying \cite[Theorem 3.1]{hall:bigMonodromySympletic} gives result.
\end{proof}

\subsection{Computing the monodromy for prime-power $\nu$}
\label{subsection:prime-power-lemmas-for-big-monodromy}

Our next goal is to generalize \autoref{theorem:big-monodromy-mod-ell} to prime
power $\nu$, and then to general composite $\nu$.
We next prove \autoref{proposition:prime-power-surjection},
which will imply that if we have big monodromy $\bmod \ell$, we also have big
monodromy $\bmod \ell^j$ for any integer $j > 0$.

\begin{definition}
	\label{definition:}
	Suppose $k \geq 2$ and $Q$ is a quadratic form over $\mathbb Z/\ell^k \mathbb Z$.
The Lie algebra $\mathfrak{so}(Q)(\mathbb F_\ell)$
is by definition $\ker(\so(Q)(\mathbb Z/\ell^2 \mathbb Z) \to \so(Q)(\mathbb
Z/\ell \mathbb Z))$.
\end{definition}

Recall our notation $\Omega(Q) = 
\ker D_Q \cap \ker \on{sp}_Q^- \subset \o(Q)$ from \autoref{notation:omega}.
We thank Eric Rains for help with the following proof.

\begin{proposition}
	\label{proposition:prime-power-surjection}
	Let $s \geq 3$ and $\ell \geq 5$ a prime.
	Let $(V, Q)$ be a non-degenerate
	quadratic space of rank $s$ over $\mathbb Z/\ell\mathbb Z$.
	Suppose $G \subset \Omega(Q)(\mathbb Z/\ell^j\mathbb Z)$ is a subgroup so that the composition
$G \to \Omega(Q)(\mathbb Z/\ell^j\mathbb Z) \to \Omega(Q)(\mathbb Z/\ell\mathbb Z)$ is surjective.
Then, $G = \Omega(Q)(\mathbb Z/\ell^j\mathbb Z)$.
\end{proposition}
\begin{proof}
	This is a special case of \cite[Theorem 1.3(a)]{vasiu2003surjectivity}.  Since there are a few mistakes in other parts of that theorem statement (though not in the part relevant to the proposition we're proving) we spell out a few more details here.
	The argument proceeds as indicated in the second to last
	paragraph of \cite[p. 327]{vasiu2003surjectivity}.
	First, as in \cite[Lemma 4.1.2]{vasiu2003surjectivity} we can reduce to the
	case $j = 2$.
	To deal with the case $j = 2$, it is enough to show $G$ meets the
	Lie algebra $\mathfrak{so}(Q)(\mathbb F_\ell)$ nontrivially, as argued
	in \cite[4.4.1]{vasiu2003surjectivity}.
	Finally, in \cite[Theorem 4.5]{vasiu2003surjectivity} it is shown that
	$G$ meets the Lie algebra nontrivially.
\end{proof}

\subsection{Bootstrapping to general composite $\nu$}
\label{subsection:general-composite-lemmas-for-big-monodromy}

We next collect a few lemmas to bootstrap from showing there is big monodromy
modulo prime powers, to showing there is big monodromy modulo composite
integers.
The main result is \autoref{proposition:big-monodromy-composite}.
The general strategy will be to apply Goursat's lemma.
A key input in Goursat's lemma is to understand which simple groups appear as
subquotients of orthogonal groups.
As a first step, using \autoref{proposition:prime-power-surjection}, we can prove $\Omega(Q)(\mathbb Z/\nu\mathbb Z)$ is perfect.
\begin{lemma}
	\label{lemma:perfection-omega}
	For $s \geq 3$, $\nu$ a positive integer, and $(V, Q)$ a 
	non-degenerate quadratic space of rank $s$ over $\mathbb Z/\nu\mathbb Z$,
	$\Omega(Q)(\mathbb Z/\nu\mathbb Z)$ is perfect.
	That is, $\Omega(Q)(\mathbb Z/\nu\mathbb Z)$ is its own commutator.
\end{lemma}
\begin{proof}
	Write $\nu = \prod_{i=1}^t \ell_i^{a_i}$ for $\ell_i$ pairwise distinct primes.
	Note $\Omega(Q)(\mathbb Z/\ell_i\mathbb Z)$ is perfect as shown in \cite[p. 73, lines 2-7]{wilson:the-finite-simple-groups}.
	Then, since the commutator subgroup 
	\begin{align*}
	\left[ \Omega(Q)(\mathbb Z/\ell_i^{a_i}\mathbb Z), \Omega(Q)(\mathbb Z/\ell_i^{a_i}\mathbb Z) \right] \subset \Omega(Q)(\mathbb Z/\ell_i^{a_i}\mathbb Z)
	\end{align*}
	is a subgroup of $\Omega(Q)(\mathbb Z/\ell_i^{a_i}\mathbb Z)$
	surjecting onto $\Omega(Q)(\mathbb Z/\ell_i\mathbb Z)$, it must be all of $\Omega(Q)(\mathbb Z/\ell_i^{a_i}\mathbb Z)$
	by \autoref{proposition:prime-power-surjection}.
	Finally, as commutators commute with products, and
	$\Omega(Q)(\mathbb Z/\nu\mathbb Z) = \prod_{i=1}^t \Omega(Q)(\mathbb Z/\ell_i^{a_i}\mathbb Z)$, it follows that
	$\Omega(Q)(\mathbb Z/\nu\mathbb Z)$ is its own commutator.
\end{proof}
The next result relates monodromy for prime power $\nu$ to 
monodromy for general composite $\nu$.
\begin{proposition}
	\label{proposition:goursat-spin-groups}
	Let $s \geq 5$. Let $(V, Q)$ be a 
	non-degenerate
	quadratic space of rank $s$ over $\mathbb Z/\nu\mathbb Z$.
	Suppose $G \subset \Omega(Q)(\mathbb Z/\nu\mathbb Z)$ is a subgroup so
	that for each prime $\ell \mid \nu$, the composition
	$G \to \Omega(Q)(\mathbb Z/\nu\mathbb Z) \to \Omega(Q)(\mathbb Z/\ell\mathbb Z)$ is surjective.
	Then, $G = \Omega(Q)(\mathbb Z/\nu\mathbb Z)$.
\end{proposition}
\begin{proof}
	We have already proven this in the case $\nu$ is a prime power in \autoref{proposition:prime-power-surjection}.
	It now remains to deal with general composite $\nu$. 
	
	To this end, write $\nu = \prod_{i=1}^t \ell_i^{a_i}$, for $\ell_i$ pairwise
	distinct primes.
	The proposition follows from an application of Goursat's lemma, as we now explain.
	We will show that the groups $\Omega(Q)(\mathbb Z/\ell_i^{a_i}\mathbb Z)$ for $1 \leq i \leq t$
	satisfy the following two properties:
	$(1)$ they have trivial abelianization and $(2)$ they have no finite non-abelian simple quotients in common.
These two facts verify the hypotheses of Goursat's lemma as stated in \cite[Proposition 2.5]{Greicius:surjective},
	which implies that $G =\prod_{i=1}^t \Omega(Q)\left( \mathbb Z/\ell_i^{a_i}\mathbb Z \right) = \Omega(Q)(\mathbb Z/\nu\mathbb Z)$.
	
	It remains to verify $(1)$ and $(2)$.
	Observe that $(1)$ follows from \autoref{lemma:perfection-omega}.
	To conclude our proof, we only need to check $(2)$: that
	the groups $\Omega(Q)(\mathbb Z/\ell_i^{a_i}\mathbb Z)$ for $1 \leq i \leq t$ have no finite non-abelian simple quotients in common.
	For $G'$ a group, let $\quo(G')$ denote the set of finite simple non-abelian
	quotients of $G'$. 
	To prove $(2)$, it suffices
	to show
	$\quo(\Omega(Q)(\mathbb Z/\ell_i^{a_i}\mathbb Z)) = \left\{ \mathbb P\Omega(Q)(\mathbb Z/\ell_i\mathbb Z) \right\}.$
	Note that the latter group is indeed simple by \cite[3.7.3 and
	3.8.2]{wilson:the-finite-simple-groups}, using that $s \geq 5$.

	So, we now check 
	$\quo(\Omega(Q)(\mathbb Z/\ell_i^{a_i}\mathbb Z)) = \left\{ \mathbb P\Omega(Q)(\mathbb Z/\ell_i\mathbb Z) \right\}.$
	Since every finite simple quotient appears as some Jordan Holder factor,
	it suffices to check the all simple Jordan Holder factors of $\Omega(Q)(\mathbb Z/\ell_i^{a_i}\mathbb Z)$ are
	contained in $\{\mathbb P\Omega(Q)(\mathbb Z/\ell_i\mathbb Z), \mathbb Z/\ell_i\mathbb Z, \mathbb Z/2\mathbb Z\}.$
	To see this, consider the surjections
	$\Omega(Q)(\mathbb Z/\ell_i^{a_i}\mathbb Z) \to \Omega(Q)(\mathbb
	Z/\ell_i^{a_i-1}\mathbb Z) \to \cdots \to \Omega(Q)(\mathbb Z/\ell_i^2\mathbb Z) \to \Omega(Q)(\mathbb Z/\ell_i\mathbb Z) \to \left\{ \id \right\}$.
	From these surjections, we obtain an associated filtration.
	The Jordan Holder factors associated to any refinement
	of this filtration
	will all lie in $\{\mathbb P\Omega(Q)(\mathbb Z/\ell_i\mathbb Z), \mathbb Z/\ell_i\mathbb Z, \mathbb Z/2\mathbb Z\}$
	since the kernels of all maps but the last are products of $\mathbb Z/\ell_i\mathbb Z$.
\end{proof}

\begin{proposition}
	\label{proposition:big-monodromy-composite}
	Keep assumptions as in \autoref{hypotheses:big-monodromy-assumptions}.
	Suppose ${\overline b} \to B$ is a geometric point.
	If 
	\begin{align}
	\label{equation:bound-on-n}
	n > \max\left( 2, 2g,\frac{2(2r+1)(f+1) - \sum_{y \in D_x({\overline b})}
		\drop_y(\mathscr F)}{2r} - (2g-2)\right),
	\end{align}
		then the geometric monodromy representation $\rho_{\mathscr
		F_{\overline b}^n}: \pi_1(\qtwist n {U_{\overline b}} {\overline b}) \to
		\on{Aut}(V_{\mathscr F_{\overline b}^n})$
		satisfies $\Omega(Q_{\mathscr F_{\overline b}^n}) \subset \im(\rho_{\mathscr
		F_{\overline b}^n})
		\subset \on{O}(Q_{\mathscr F_{\overline b}^n})$
		and 
		$\im(\rho_{\mathscr F^n_{\overline b}}) \not \subset \on{SO}(Q_{\mathscr F^n_{\overline b}})$.
	\end{proposition}
\begin{proof}
	We have seen in \autoref{remark:orthogonal-containment} that 
	$\im(\rho_{\mathscr F_{\overline b}^n}) \subset \on{O}(Q_{\mathscr F_{\overline b}^n})$ holds.
	By \autoref{theorem:big-monodromy-mod-ell}, we know 
	$\Omega(Q_{\mathscr F[\ell]^n}) \subset \im(\rho_{\mathscr F_{\overline b}[\ell]^n})$ for each prime
	$\ell \mid \nu$. It follows from
	\autoref{proposition:goursat-spin-groups} that 
	$\Omega(Q_{\mathscr F^n_{\overline b}}) \subset \im(\rho_{\mathscr F_{\overline b}^n})$.
	Note that since $n > 2$, the formula for the rank of $V_{\mathscr
	F^n_{\overline b}}$ from \autoref{proposition:rank-description} shows it is at least
	$5$, so the hypotheses of \autoref{proposition:goursat-spin-groups} are
	satisfied.
	From \autoref{theorem:big-monodromy-mod-ell}, we also find
	that $\im(\rho_{\mathscr F^n_{\overline b}}) \not \subset \on{SO}(Q_{\mathscr F^n_{\overline b}})$.
\end{proof}

\subsection{Understanding the image of the Dickson invariant map}
\label{subsection:determinant-image}

Having shown that the image of monodromy is close to the orthogonal group, so in
particular contains $\Omega(Q_{\mathscr F^n_B})$, its
failure to equal the orthogonal group can be understood in terms of the spinor
norm and Dickson invariant. The spinor norm will not have much effect on the
distribution of Selmer elements, but the Dickson invariant will have a huge effect,
and is closely connected to the parity of the rank of $A$ in the case $\mathscr
F_b \simeq A[\nu],$ for $A \to U$ an abelian scheme as in
\autoref{example:abelian-special-fiber}. In the remainder of this
section, specifically \autoref{lemma:determinant-monodromy},
we precisely determine the image of the Dickson invariant, under the
arithmetic monodromy representation $\rho_{\mathscr F^n_b}$.

Our strategy for determining the arithmetic monodromy will be to use
equidistribution of Frobenius elements, and compute images
of Frobenius elements by relating them to Selmer groups.
The following notation for the distribution of Selmer groups will make it
convenient to express the types of Selmer groups which appear.

\begin{definition}
	\label{definition:actual-selmer-distribution}
	Keep assumptions as in \autoref{notation:quadratic-twist-notation} and
	\autoref{example:abelian-special-fiber}, and assume that $B$ is the
	spectrum of a local ring so that $b \in B$ is the unique closed point and
	has residue field contained in $\mathbb F_q$.
	In particular, $\mathscr F_b \simeq A[\nu]$ for $A \to U_b$ a polarized abelian
	scheme with polarization degree prime to $\nu$.

	Let $\mathcal N$ denote the set of isomorphism classes of finite $\mathbb Z/\nu \mathbb Z$-modules.
	Let $X_{A[\nu]^n_{\mathbb F_q}}$ denote the probability distribution on
	$\mathcal N$ defined by
\begin{align*}
	\prob \left(X_{A[\nu]^n_{\mathbb F_q}} = H \right) = \frac{\# \{x \in
			\qtwist n {U_b} b(\mathbb
	F_q) : \sel_\nu(A_x) \simeq H \}}{\# \qtwist n {U_b} b(\mathbb F_q)}.
\end{align*}
Here, as usual, point counts of stacks are weighted inversely proportional to
the isotropy group at that point.
For $i \in \{0,1\}$, let $\mathcal N^i\subset \mathcal N$ denote the subset of
$\mathcal
N$ of those $H$ so that there exists some $\mathbb Z/\nu \mathbb Z$-module $G$ such
that $H \simeq (\mathbb Z/\nu \mathbb Z)^i \times G^2$.
Given $H \in \mathcal N^i$, define
\begin{align*}
	\prob\left(X^{i}_{A[\nu]^n_{\mathbb F_q}} = H \right) = \frac{\# \{x \in
			\qtwist n {U_b} b(\mathbb
	F_q) : \sel_\nu(A_x) \simeq H \}}{\# \{\qtwist n {U_b} b(\mathbb F_q) :
\sel_\nu(A_x) \in \mathcal N^i \}}.
\end{align*}
\end{definition}

The next two lemmas give the key constraint on Tate-Shafarevich groups and Selmer groups we will use to
determine the image of the Dickson invariant.
It is one of the few places in this paper that the arithmetic of abelian
varieties comes crucially into play.

\begin{lemma}
	\label{lemma:sha-almost-square}
Let $\nu$ be an odd positive
	integer.
	Let $K$ be the function field of a curve over a finite field, and let
	$A$ be an abelian variety over $K$ with a polarization of degree prime
	to $\nu$.
	Then, there is a finite $\mathbb Z/\nu \mathbb Z$-module $G$ so that either
	$\Sha(A)[\nu] \simeq G^2$ or $\Sha(A)[\nu] \simeq G^2 \oplus \mathbb Z/\nu
	\mathbb Z$.
\end{lemma}
\begin{remark}
	\label{remark:}
	If we assume the BSD conjecture, $\Sha(A)$ will be finite and then the
	assumptions that the polarization has degree prime to $\nu$ and $\nu$ is
	odd will imply
	$\Sha(A)[\nu]$ has square order.
\end{remark}
\begin{remark}
	The condition that the polarization has degree prime to $\nu$ is
	important here: 
	In general, even when the Tate-Shafarevich group is known to
	be finite, it can fail to be a square or twice a square, see
	\cite[p. 278, Theorem
	1.4]{cremonaLQR:modular-curves-and-abelian-varieties}.
\end{remark}

\begin{proof}
To approach this, we first review some general facts about the structure of the
Tate-Shafarevich group.
We can write 
$\Sha(A)[\ell^\infty] \simeq (\mathbb
Q_\ell/\mathbb Z_\ell)^{r_\ell} \oplus K_\ell$,
where $K_\ell$ is a finite group and $r_\ell$ is the rank of 
$\Sha(A)[\ell^\infty]$. Note that the BSD conjecture would imply
$r_\ell = 0$, but we will not use this.

We next claim that $\oplus_{\ell \mid \nu} K_\ell \simeq G_{\on{nd}}^2$, for
some finite $\mathbb Z/\nu \mathbb Z$-module $G_{\on{nd}}$.
Indeed, let $\Sha(A)[\nu]_{\on{nd}}$ denote the non-divisible part of $\Sha(A)[\nu]$.
Then, $\Sha(A)[\nu]_{\on{nd}}$ has a nondegenerate pairing, by
\cite[Theorem 3.2]{tate:duality-theorems}, which is antisymmetric by
\cite[Theorem 1]{flach:a-generalization-of-the-cassels-tate-pairing}.
Since $\nu$ is odd, any finite $\mathbb Z/\nu \mathbb Z$-module with an
nondegenerate antisymmetric pairing is a square, so there is some $\mathbb Z/\nu
\mathbb Z$-module $G_{\on{nd}}$ with $\Sha(A)[\nu]_{\on{nd}} \simeq
G_{\on{nd}}^2$.

We now conclude the proof.
By
\cite[Corollary 1.0.3]{trihanY:the-ell-parity-conjecture},
$r_\ell$ has parity
independent of $\ell$. 
Write $\nu = \prod_{\ell \mid \nu} \ell^{a_\ell}$, and take
$G = G_{\on{nd}} \oplus \left( \oplus_{\ell \mid \nu} (\mathbb Z/\ell^{a_\ell} \mathbb
Z)^{\lfloor \frac{r_\ell}{2} \rfloor}\right)$. 
We get
$\Sha(A)[\nu] \simeq G^2$ if $r_\ell$ is even for all $\ell \mid \nu$. Similarly,
we get 
$\Sha(A)[\nu] \simeq G^2 \oplus \mathbb Z/\nu \mathbb Z$ if $r_\ell$ is odd for
all $\ell \mid \nu$.
\end{proof}

\begin{lemma}
	\label{lemma:distribution-supported-squares}
	Maintain hypotheses from \autoref{notation:quadratic-twist-notation}
	and notation from \autoref{definition:actual-selmer-distribution}.
	Assume $\nu$ is odd, $n >0$, and $B$ is an integral affine scheme 
	with $2\nu$ invertible on $B$. Let
	$b \in B$ be a
	closed point over which $\mathscr F_b \simeq A[\nu]$, for $A \to U_b$ an
	abelian scheme, as in \autoref{example:abelian-special-fiber}.
	The distributions $X_{A[\nu]^n_{\mathbb F_q}}$ are supported on $\mathcal
	N^0 \coprod \mathcal N^1$. Hence, 
	\begin{equation}
	X_{A[\nu]^n_{\mathbb F_q}} = 
	\prob(X_{A[\nu]^n_{\mathbb F_q}}
	\in \mathcal N^0) \cdot X^0_{A[\nu]^n_{\mathbb F_q}}+
	\prob(X_{A[\nu]^n_{\mathbb F_q}}
	\in \mathcal N^1) \cdot X^1_{A[\nu]^n_{\mathbb F_q}}.
	\label{equation:x-conditional-probability}
	\end{equation}
	\end{lemma}
\begin{proof}
	The claim \eqref{equation:x-conditional-probability} follows from the first
	claim about the support of $X_{A[\nu]^n_{\mathbb F_q}}$ by the law of total expectation.
	We now verify 
	$X_{A[\nu]^n_{\mathbb F_q}}$ are supported on $\mathcal
	N^0 \coprod \mathcal N^1$.

	Using notation as in 
	\autoref{notation:x-points}, it is enough to show the Selmer
	group of any quadratic twist $A_x$ of $A$ lies in $\mathcal N^0$ or
	$\mathcal N^1$.
	In general, there is an exact sequence
	\begin{equation}
		\label{equation:selmer-exact}
		\begin{tikzcd}
			0 \ar {r} & A_x(U_x)/\nu A_x(U_x) \ar {r} & \sel_\nu(A_x) \ar
			{r} & \Sha(A_x)[\nu] \ar {r}
			& 0. 
	\end{tikzcd}\end{equation}
	By \autoref{lemma:sha-almost-square},
	$\Sha(A_x)[\nu]$ lies in $\mathcal N^0 \coprod \mathcal N^1$.
	By \autoref{proposition:rank-description}(2'), $A_x(U_x)[\nu] = 0$, which
	implies that $A_x(U_x)/\nu A_x(U_x)$ is a free $\mathbb Z/\nu \mathbb
	Z$-module.
	Hence, since $\mathbb Z/\nu \mathbb Z$ is injective as a $\mathbb Z/\nu
	\mathbb Z$-module, the exact sequence \eqref{equation:selmer-exact}
	splits and we obtain $\sel_\nu(A_x) \simeq A_x(U_x)/\nu A_x(U_x) \oplus
	\Sha(A_x)[\nu]$.
	Now, we see that since $\Sha(A)[\nu] \in \mathcal N^0 \coprod \mathcal
	N^1$ and $A_x(U_x)/\nu A_x(U_x)$ is a free $\mathbb Z/\nu \mathbb Z$-module,
	$\sel_\nu(A_x) \in \mathcal N^0 \coprod \mathcal N^1$.
\end{proof}

Finally, we are prepared to compute the image of the Dickson invariant map.
\begin{lemma}
	\label{lemma:determinant-monodromy}
	Assume $\nu$ is odd, $n >0$, and $B$ is an integral
	affine base
	scheme $B$ 
	with $2\nu$ invertible on $B$.
	Suppose $b \in B$ is a closed point with finite residue field, and keep hypotheses as in 
	\autoref{notation:quadratic-twist-notation} and
	\autoref{hypotheses:big-monodromy-assumptions}.
	Assume there is an abelian scheme $A \to U_b$ so that $\mathscr F_b
	\simeq A[\nu]$, as in 
\autoref{example:abelian-special-fiber}.
	The Dickson invariant map $D_{Q_{\mathscr F^n_b}} : \on{O}(Q_{\mathscr F^n_b}) \to \prod_{\ell \mid \nu}
	\mathbb Z/2 \mathbb Z$ sends the arithmetic monodromy group
	$\im(\rho_{\mathscr F^n_b})$ surjectively to the diagonal copy of
	$\Delta_{\mathbb Z/2 \mathbb Z} : \mathbb
	Z/2 \mathbb Z \subset \prod_{\ell \mid \nu}
	\mathbb Z/2 \mathbb Z$.
	The same holds for the geometric monodromy group at a geometric point
	$\overline b$ over $b$.
\end{lemma}
\begin{proof}
	First, we argue it suffices to show the Dickson invariant of the 
	arithmetic monodromy group satisfies
	$\im(D_{Q_{\mathscr F^n_b}}  \circ\rho_{\mathscr F^n_b}) \subset \im \Delta_{\mathbb Z/2 \mathbb Z}.$ 
	Indeed,
	for $\overline{b}$ a geometric point over $b$,
	the image of the arithmetic monodromy group $\im(D_{Q_{\mathscr F^n_b}})$
	contains the image of the geometric monodromy group $\im (D_{Q_{\mathscr
	F^n_{\overline b}}})$. 
	Assuming we have shown the arithmetic monodromy has
	image
	the diagonal $\mathbb Z/2 \mathbb Z$ under the Dickson invariant map, 
	to show they are equal,
	it is enough to show the geometric monodromy has nontrivial image under
	the Dickson invariant map. Equivalently, we wish to show the geometric
	monodromy is not contained in the special orthogonal group, which
	follows from
	\autoref{theorem:big-monodromy-mod-ell}.

	We now verify the arithmetic monodromy group has Dickson invariant contained in
$\im \Delta_{\mathbb Z/2 \mathbb Z}.$ 
	The strategy will be to use
	\autoref{lemma:distribution-supported-squares} to determine the
	arithmetic monodromy by relating the Dickson invariant map to the parity of
	the rank of Selmer groups modulo different primes, using
	equidistribution of Frobenius.

	Choose $x \in \qtwist n {U_b} b(\mathbb F_q)$.
	As a first step, we identify $\sel_\nu(A_x)$ with the $1$-eigenspace of
	$\rho_{\mathscr F^n_b}(\frob_x)$, for $\frob_x$ the geometric Frobenius
	at $x$.
	With notation as in 
	\autoref{lemma:selmer-identification}, we can identify
	$\pi^{-1}(x)(\mathbb F_q) \simeq \sel_\nu(A_x)$.
	Since $\pi^{-1}(x)(\mathbb F_q)$ can be identified with the $\frob_x$ 
invariants of $\pi^{-1}(x)(\overline{\mathbb F}_q)$,
	if $g_x := \rho_{\mathscr F^n_b}(\frob_x)$, we also have $\pi^{-1}(x)(\mathbb
	F_q) \simeq \ker(g_x - \id)$.
	Combining these two isomorphisms, we obtain $\ker(g_x - \id)
	\simeq \sel_\nu(A_x)$.
	For $\ell \mid \nu$, we use $g_{x,\ell}$ to denote the image of $g_x$
	under the map $\o(Q_{\mathscr F^n_b}) \to \o(Q_{\mathscr F^n_b[\ell]})$.
	We similarly obtain $\ker(g_{x,\ell} - \id) \simeq \sel_\ell(A_x)$.

	We next constrain the image of the Dickson invariant map applied to
	$\rho_{\mathscr F^n_b}(\frob_x)$.
	From \autoref{lemma:distribution-supported-squares}, we have seen that 
	$\ker(g_{x} - \id) \simeq \sel_\nu(A_x)
	\in \mathcal N^0 \coprod \mathcal N^1$, for $\mathcal N^i$ defined in
	\autoref{definition:actual-selmer-distribution}.
	Since the parity of the rank of $H/\ell H$ of any group $H$ in $\mathcal N^0 \coprod
	\mathcal N^1$ is independent of the prime $\ell \mid \nu$, 
	it follows that $\dim \ker (g_{x,\ell}- \id)$ has parity of rank
	independent of $\ell$, for $\ell \mid \nu$.
	By \autoref{lemma:1-eigenvalue-parity},
	for any $\ell \mid \nu$,
	\begin{align*}
	\dim \ker(g_{x, \ell} -\id) \bmod 2 \equiv \rk V_{\mathscr F^n_b[\ell]} -
	D_{Q_{\mathscr F^n_b}}(g_{x,\ell}).
	\end{align*}
	Since $\rk V_{\mathscr F^n_b[\ell]}$ is
	independent of $\ell \mid \nu$, as $V_{\mathscr F^n_b}$ is a free
	$\mathbb Z/\nu \mathbb Z$-module, we also obtain 
	$D_{Q_{\mathscr F^n_b}}(g_{x,\ell})$ is independent of $\ell \mid \nu$. In other
	words, the Dickson invariant map factors through the diagonal copy $\mathbb Z/2
	\mathbb Z$ for each Frobenius element associated to $x \in \qtwist n
	{U_b}
	b(\mathbb F_q)$.

	The lemma will now follow from equidistribution of
	Frobenius elements in the arithmetic fundamental group, as we next
	explain.
	At this point, we employ a result on equidistribution of Frobenius,
	whose precise form we could not find directly in the literature.
	The result is essentially \cite[Theorem
	4.1]{chavdarov:the-generic-irreducibility} (see also \cite[Theorem
		1]{kowalski:on-the-rank-of-quadratic-twists}
		and \cite[Theorem
	3.9]{fengLR:geometric-distribution-of-selmer-groups})
	except that we need a slightly more general statement which also applies
	to Deligne-Mumford stacks in place of only schemes. 
	The only part of the proof of 
	\cite[Theorem 4.1]{chavdarov:the-generic-irreducibility}
	which does not directly apply to stacks is its use of the
	Grothendieck-Lefschetz
	trace formula with twisted coefficients, but this has been generalized to hold in the context of
	stacks, see \cite[Theorem 4.2]{sun:l-series-of-artin-stacks}. 
	Using this, we can find a sufficiently large $q$ and
	$x \in \qtwist n {U_b} b(\mathbb F_q)$
	with the following property:
	the generator $\frob_x$ of $\pi_1(x)$ is sent to
	any particular element of $\im( D_{Q_{\mathscr F^n_b}}\circ \rho_{\mathscr F^n_b})$ under the composition
	$\pi_1(x) \to \pi_1(\qtwist n {U_b} b)
	\xrightarrow{D_{Q_{\mathscr F^n_b}}\circ\rho_{\mathscr F^n_b}} \prod_{\ell \mid \nu}
	\mathbb Z/2 \mathbb Z$.
	For our choice of $q$ above, note that we may need to take $q$ to be suitably large, and also if $q = p^j$ for
	$p = \chr \mathbb F_q$ we may need to impose a congruence condition on
	$j$.
	Therefore, since every $\frob_x$ has image contained in the diagonal
	$\mathbb Z/2 \mathbb Z$, the same must be true of $\im(D_{Q_{\mathscr F^n_b}} \circ
	\rho_{\mathscr F^n_b})$.
	\end{proof}

\section{The rank double cover}
\label{section:rank-cover}

Perhaps surprisingly, the distribution of Selmer groups of abelian varieties is not determined
by its moments. 
As mentioned in the introduction, if one fixes the parity of the rank of
$\sel_\ell$, this does
not change the distribution of Selmer groups.
Even more surprisingly, once one does condition on the parity of the
rank of $\sel_\ell$, the
BKLPR distribution {\em is} determined by its moments.
In this section, we investigate the geometry associated to a certain double
cover of $\qtwist n U B$, which we define in
\autoref{subsection:describing-rank-double-cover}.
In \autoref{subsection:homological-stability-rank-double-cover},
we will use our homological stability machinery to bound the dimensions of the
cohomology of this
double cover.
In \autoref{subsection:rank-cover-and-parity}, we relate this double cover to
the parity of the dimension of $\sel_\ell$ of an abelian variety.
Specifically, suppose we are given a symplectically self-dual sheaf $\mathscr F$ on $U$, and
a point $b \in B$ with $\mathscr F_b \simeq A[\nu]$, for $A \to U_b$ an abelian
scheme. We will define a particular double
cover $\rankcover n {\mathscr F}$ of $\qtwist n U B$ so
that the images $\rankcover n {\mathscr F_b}(\mathbb F_q) \to \qtwist n {U_b}
b(\mathbb F_q)$ corresponds precisely to abelian varieties whose $\ell^\infty$-Selmer rank has parity
equal to $\rk V_{\mathscr F^n_B} \bmod 2$.

\subsection{The rank double cover and its coefficient system}
\label{subsection:describing-rank-double-cover}

We now define the rank double cover, and subsequently proceed to show the
sequence of rank double covers forms a coefficient system.


\begin{definition}
	\label{definition:rank-cover}
	With notation as in \autoref{definition:monodromy} and
	\autoref{notation:omega},
	suppose the composition
$D_{Q_{\mathscr F^n_B}} \circ \rho_{\mathscr F^n_B}: \pi_1(\qtwist
	n U B) \to \prod_{\ell \mid \nu} \mathbb Z/2 \mathbb Z$ 
	factors through the diagonally embedded $\Delta_{\mathbb Z/2 \mathbb Z}
	: \mathbb Z/2 \mathbb Z \to \prod_{\ell \mid \nu} \mathbb Z/2
	\mathbb Z$.
	We define $\rankcover n {\mathscr F} \to \qtwist n U B$ to be the finite
	\'etale double cover corresponding to the composition
	$D_{Q_{\mathscr F^n_B}} \circ \rho_{\mathscr F^n_B}$,
	viewed as a map $\pi_1(\qtwist
	n U B) \to \mathbb Z/2 \mathbb Z$.
\end{definition}

\begin{remark}
	\label{remark:}
	When we are in the situation of 
	\autoref{lemma:determinant-monodromy}, 
	it follows from 
	\autoref{lemma:determinant-monodromy}, 
	that the map $D_{Q_{\mathscr F^n_B}} \circ \rho_{\mathscr F^n_B}: \pi_1(\qtwist
	n U B) \to \prod_{\ell \mid \nu} \mathbb Z/2 \mathbb Z$ takes image in
	the diagonally embedded copy of $\mathbb Z/2 \mathbb Z$,
	so the hypothesis in \autoref{definition:rank-cover} applies.
\end{remark}

Since the rank double cover is a cover of $\qtwist n U B$, which is in turn a cover of $\conf n U B$, we can ask whether the composition $\rankcover n {\mathscr F} \to \conf n U B$ is associated to a coefficient system. 
There is a technical issue with this question, in that $\qtwist n U B$ is not a
scheme, so the
above cover is not representable. However, after suitably rigidifying this
cover, we shall see that it indeed is associated to a coefficient system.  In order to describe that coefficient system, we will first need to describe the coefficient systems associated to Selmer spaces and to
their $H$-moments.
\begin{example}
	\label{example:specified-hurwitz-coefficient-system}
	Let $B = \spec \mathbb C$ and let $\mathscr F$ be a symplectically self-dual sheaf over $U$ as in \autoref{notation:quadratic-twist-notation}.
	Fix a nontrivial finite $\mathbb Z/\nu \mathbb Z$-module $H$.
	We now define a coefficient system of the type described in  \autoref{example:hurwitz-coefficient-system}, which we will denote $H_{\sallinertia {\mathscr F}{H} g f}$.  
	Recall that here we do not quotient by the conjugation action of the
	relevant group, see \autoref{warning:no-quotient}.
	The $n$th part of $H_{\sallinertia {\mathscr F}{H} g f}$ is the free
	vector space generated by a finite set $\sinertiaindex n {\mathscr F}{H}
	g f$ which we now define.
	Take $G_H :=
	\ahsp_{2r}(\mathbb Z/\nu \mathbb Z)$, as in
	\eqref{equation:asp-m-sequence},
	and, with notation as in 
	\eqref{equation:asp-m-sequence},
	take $c_H := \Pi^{-1}(-\id)$, which is a conjugacy class in $G_H$.
	Take $\sinertiaindex n {\mathscr F}{H}g f \subset \hom(\pi_1(X^{\oplus n} \oplus A_{g,f} - x^{\oplus n}, p_{g,f}), G_H)$
	to be the subset $\mathcal S$ described in
	\autoref{notation:sel-hur}.
	(So, we
		are calling
		$\sinertiaindex n {\mathscr F}{H}g f $ what we called
		$\inertiaindex n {G_H} {c_H} g
f$ in \autoref{example:hurwitz-coefficient-system}.)
	More precisely,
	$\sinertiaindex n {\mathscr F}{H}g f\subset
	\hom(\pi_1(Y^{\oplus n} \oplus A_{g,f} - x^{\oplus n}, p_{g,f}), G_H)$
	is the subset consisting of those homomorphisms which send the loops around the $n$ punctures to  $c_H$, which send local inertia around the $f+1$ punctures to the conjugacy class described in \autoref{notation:sel-hur}, and which have the same projection to $\sp_{2r}(\mathbb
	Z/\nu \mathbb Z)/\{\pm 1\}$ as does $\rho_{\mathscr F}$.
 
 
	So long as we choose the basepoint $p_{g,f}$ to lie on the boundary of
	$A_{g,f}$,
	we can also restrict any homomorphism
	$\hom(\pi_1(Y^{\oplus n} \oplus A_{g,f} - x^{\oplus n}, p_{g,f}), G_H)$
	to a homomorphism
	$\hom(\pi_1(Y^{\oplus n} - x^{\oplus n}, p_{g,f}), G_H)$.
	We denote by $\sinertiaindex n {\underline{\mathscr F} }{H} 0 0 \subset
	\hom(\pi_1(Y^{\oplus n} - x^{\oplus n}, p_{g,f}), G_H)$
	the restriction of $\sinertiaindex n {\mathscr F}{H} g f$ to
	$\hom(\pi_1(Y^{\oplus n} - x^{\oplus n}, p_{g,f}), G_H)$.
Define $H_{\sallinertia { \underline{\mathscr F}}{H} 0 0}$
	to be the associated coefficient system, whose $n$th piece is
	$H_{\sinertiaindex n {\underline{\mathscr F}}{H} 0 0}$, the free vector space
	generated by $\sinertiaindex n {\underline{\mathscr F}}{H} 0 0$.

	Take $V := H_{\sallinertia { \underline{\mathscr F}}{H} 0 0}$ and take $F :=
	H_{\sallinertia {\mathscr F}{H} g f}$.
We claim that $V$ forms a coefficient system for $\Sigma^1_{0,0}$ and $F$ forms a coefficient system 
	for $\Sigma^1_{g,f}$ over $V$.
	Indeed, these sets $\sinertiaindex n {\mathscr F}{H}g f$ are fixed under the
	action of $B^n_{g,f}$ by
	\autoref{lemma:hurwitz-space-invariant-under-braid-group}.
	Hence, they form a coefficient system by
	\autoref{example:hurwitz-coefficient-system}.
	We can identify $\inertiaindex {n+1} {G_H} {c_H} g f \simeq c_H \times \inertiaindex n {G_H} {c_H} g f$, where the map to
	$c_H$ is given by
	the local inertia around the added puncture.
	It follows that $F_{n+1} = k\{T^{n+1}_{g,f}\} \simeq k\{c_H\} \otimes
	k\{\inertiaindex n {G_H} {c_H} g f\}
	= V_1 \otimes F_n$.
	In the case $g = f = 0$, we similarly obtain that $V$ is a 
	coefficient system.

	We also define $\on{Hur}_{\sallinertia {\mathscr F}{H} g f}$ to be the finite covering space of $\Conf^n_{X^{\oplus
	n} \oplus A_{g,f}}$ associated to the set $\sinertiaindex n {\mathscr F}
	H g f$.
\end{example}

Building on \autoref{example:specified-hurwitz-coefficient-system}, we
next describe the coefficient system corresponding to the rank double cover.

\begin{example}
	\label{example:rank-coefficient-system}
	Take $G = \mathbb Z/2 \mathbb Z$, $c = \{1\} \in \mathbb Z/2 \mathbb Z$,
	corresponding to the nontrivial element, and consider the Hurwitz space coefficient system
	$H_{\allinertia {\mathbb Z/2 \mathbb Z} { \{1\}} g f}$.
	We recall that in the definition of these coefficient systems, we do not
	quotient by the conjugation action of the relevant group, see
	\autoref{warning:no-quotient}.
	By \autoref{example:hurwitz-coefficient-system}, this is a coefficient
	system for $\Sigma^1_{g,f}$ which we claim lies over the trivial coefficient system for
	$\Sigma^1_{0,0}$.
	Indeed, 
	$H_{\inertiaindex n {\mathbb Z/2 \mathbb Z} {\{1\}} 0 0}$ is
	$1$-dimensional because $c = \{1\}$ has size $1$ and moreover the
	coefficient system is trivial because 
	$\mathbb Z/2 \mathbb Z$ is commutative.

	We assume $\mathscr F$ is a symplectically self dual sheaf as in
	\autoref{definition:rank-cover}.
	We use notation as in
	\autoref{example:specified-hurwitz-coefficient-system}, and take the
	group $H$
	there to be $\mathbb Z/\nu \mathbb Z$.
	For every $n$,
	there is a map of finite sets
	$\phi_{\mathscr F^n_B} : \sinertiaindex n {\mathscr F} {\mathbb Z/\nu
	\mathbb Z} g f \to \inertiaindex n {\mathbb Z/2
	\mathbb Z} {\{1\}} g f$
	which induces a map of $B^n_{g,f}$ representations.
	Moreover, the fiber of 
	$\phi_{\mathscr F^n_B}$, which we call $W_{\mathscr F^n_B}$, over any
	fixed point of the target can be identified
	with a free $\mathbb
	Z/\nu \mathbb Z$-module which has rank $\dim V_{\mathscr F^n_B} + 2r$.
	There is an action of a fiber of $\mathscr F$ 
	on the finite cover of $Q^n_{g,f}$ corresponding to
	$W_{\mathscr F^n_B}$
	by conjugation,
	and we let $\overline{W}_{\mathscr F^n_B}$ denote the
	quotient of $W_{\mathscr F^n_B}$ by this conjugation action.
	Since $\mathscr F$ acts by conjugation on the corresponding cover, it
	follows that $\overline{W}_{\mathscr F^n_B}$ inherits the structure of a
	$B^n_{g,f}$ representation.
	The $B^n_{g,f}$ representations corresponding to the sets $\inertiaindex
	n {\mathbb Z/2 \mathbb Z} {\{1\}} g f$ 
	yield a finite covering space of $\Conf^n_{X^{\oplus
	n} \oplus A_{g,f}}$ of degree $2^{2g}$, which we call $Q^n_{g,f}$.

	In \autoref{lemma:topological-rank-as-dickson}, we will construct a
	finite \'etale double cover $R_{\mathscr F^n_{\mathbb C}}$ of $Q^n_{g,f}$.
	The fiber of $R_{\mathscr F^n_{\mathbb C}}$ over a point of $\Conf^n_{X^{\oplus
	n} \oplus A_{g,f}}$ corresponds to a finite set $S^{n,\on{rk}}_{\mathscr
	F, g, f}$ 
	of order $2^{2g+1}$
	and yields a $B^n_{g,f}$ representation which we call
	$(H^{\rk}_{\mathscr F, g, f})_n$.
	We will show $(H^{\rk}_{\mathscr F, g, f})_n$ forms a coefficient system
	for $\Sigma^1_{g,f}$ over the trivial coefficient system for
	$\Sigma^1_{0,0}$ in \autoref{lemma:action-on-rank-cover} and
	\autoref{lemma:rank-coefficient-system}.
\end{example}

\begin{lemma}
	\label{lemma:topological-rank-as-dickson}
	Continuing with notation as in
	\autoref{example:rank-coefficient-system},
	the action of $\pi_1(Q^n_{g,f})$ on
	$\overline{W}_{\mathscr F^n_B}$
	is $\mathbb Z/\nu \mathbb Z$ linear and factors
	through an orthogonal group $\on{O}_{\mathscr F^n_B}(\mathbb Z/\nu
	\mathbb Z)$.
	Moreover, the action factors through the preimage of the diagonal $\mathbb Z/2
	\mathbb Z \subset \prod_{\ell \mid \nu} \mathbb Z/2 \mathbb Z$ under the Dickson invariant
	and hence composition with the Dickson invariant defines a map 
	$\pi_1(Q^n_{g,f}) \to \on{O}_{\mathscr F^n_B}(\mathbb Z/\nu
	\mathbb Z) \to \mathbb Z/2 \mathbb Z$, corresponding to a finite \'etale double cover
	$R_{\mathscr F^n_{\mathbb C}} \to Q^n_{g,f}$.
\end{lemma}
\begin{proof}
	We may identify the Selmer stack $\selspace {\mathscr F^n_B}$ with a
	Hurwitz stack 
$\selhur {\mathscr F^n_B} {\mathbb Z/\nu \mathbb Z}$
	via \autoref{proposition:selmer-to-hurwitz}. 
	This Hurwitz stack $\selhur {\mathscr F^n_B} {\mathbb Z/\nu \mathbb Z}$
 has a further cover
	given by a pointed Hurwitz space as in \autoref{definition:pointed-hurwitz-space},
	which is identified with the cover corresponding to the coefficient
	system whose $n$th part is
	$\sinertiaindex n {\mathscr F} {\mathbb Z/\nu \mathbb Z} g f$,
	via \autoref{warning:no-quotient}. (Recall we are using 
$\sinertiaindex n {\mathscr F} {\mathbb Z/\nu \mathbb Z} g f$,
for what was called
$\inertiaindex n {G_H} {c_H} g f$
there, as in \autoref{example:specified-hurwitz-coefficient-system}.)
	Quotienting 
	$W_{\mathscr F^n_B}$ by the conjugation action of a fiber
	of $\mathscr F$, we obtain
	$\overline{W}_{\mathscr F^n_B}$. 
	This corresponds to quotienting the pointed Hurwitz space
	by the conjugation action, which is the Hurwitz space $\selhur {\mathscr F^n_B} {\mathbb Z/\nu \mathbb Z}$, so we obtain an
	identification of 
	$\overline{W}_{\mathscr F^n_B}$ as a $\pi_1(Q^n_{g,f})$ representation with a
	geometric fiber of $\selhur {\mathscr F^n_B} {\mathbb Z/\nu \mathbb Z}
	\times_{\qtwist n U B} Q^n_{g,f}$.
	over $Q^n_{g,f}$.
	Hence, we may identify $\overline{W}_{\mathscr F^n_B}$ as a
	$\pi_1(Q^n_{g,f})$ representation with
	$V_{\mathscr F^n_B}$ as a $\pi_1(Q^n_{g,f})$ set,
	viewing 
	$\pi_1(Q^n_{g,f}) \subset \pi_1(\qtwist n U B)$, as $Q^n_{g,f}$ is a
	finite \'etale double cover of $\qtwist n U B$.
	Hence, we obtain the factorization through the claimed group by our
	assumption on $\mathscr F$ from
	\autoref{definition:rank-cover}.
\end{proof}

Gearing up to explicitly describe the rank double cover as a coefficient system, we next record, in terms of generators, the action of 
$\pi_1(\Conf^n_{\Sigma^1_{g,f}},[S])$ on $\pi_1(\Sigma^1_{g,f+n}, p)$ for $[S]$
and $p$ basepoints.
One can prove the description in \autoref{remark:explicit-mcg-action} by computing where the loops as in
\autoref{subsection:elliptic-torsion-monodromy} are sent under the appropriate
Dehn twists or half twists.
Also related to this is the explicit presentation for
$\pi_1(\Conf^n_{\Sigma^1_{g,f}})$
given in
\cite[Theorem 1.1]{bellingeri:on-presentations-of-surface-braid-groups}, which
shows that the four types of loops described in 
\autoref{remark:explicit-mcg-action} generate 
$\pi_1(\Conf^n_{\Sigma^1_{g,f}},[S])$.

\begin{remark}
	\label{remark:explicit-mcg-action}
	Use notation as in
	\autoref{subsection:elliptic-torsion-monodromy} for loops on
	$\Sigma^1_{g,f+n}$, where the $n$ punctures correspond to a set $S
	\subset \Sigma^1_{g,f}$.
	We use $\delta_{f+1}$ for a loop around the boundary component and
	$p := s_{f+1}$.
	For $n$ even, the action of certain generators of
	$\pi_1(\Conf^n_{\Sigma^1_{g,f}},[S])$ on $\pi_1(\Sigma^1_{g,f+n},
	s_{f+1})$ act
	on the data
\begin{align}
	\label{equation:loop-tuple}
	(\alpha_1, \beta_1, \ldots, \alpha_g, \beta_g, \gamma_1, \ldots,
	\gamma_n, \delta_1, \ldots, \delta_{f+1}) 
\end{align}
	in the following way:
	\begin{enumerate}
		\item The full twist of $p_n$ around a loop surrounding $s_1,
			s_2, \ldots, s_i$
			sends \eqref{equation:loop-tuple} to
		\begin{align*}
			(\alpha_1, \ldots, \gamma_{n-1}, \gamma_n \delta_1
				\cdots \delta_i \gamma_n (\delta_1
				\cdots \delta_i)^{-1} \gamma_n^{-1}, \gamma_n
				\delta_1 \gamma_n^{-1},
			\ldots, \gamma_n \delta_i \gamma_n^{-1},\delta_{i+1},
		\ldots, \delta_{f+1}). 	
		\end{align*}
	\item For $1 \leq i \leq n-1$, the half-twist of $p_i$ around $p_{i+1}$ 
(moving $p_i$ counterclockwise toward
point $p_{i+1}$ and $p_{i+1}$ counterclockwise toward $p_i$)
		sends 
\eqref{equation:loop-tuple}
		to
		\begin{align*}
			(\alpha_1, \ldots, \gamma_{i-1}, \gamma_i \gamma_{i+1}
				\gamma_i^{-1}, \gamma_i, \gamma_{i+2},
				\gamma_{i+3}, \ldots, \gamma_n, \delta_1,
			\ldots, \delta_{f+1}).
		\end{align*}
		\item Moving $p_1$ across $\alpha_i$ and then in a loop around
		$s_{f+1}, \ldots, s_1, p_n, \ldots, p_2$ sends 
\eqref{equation:loop-tuple}
		to

\begin{align*}
	&\left( \alpha_1, \ldots, \beta_{i-1}, \left(\prod_{j=1}^{i-1}
		[\alpha_j, \beta_j] \right)^{-1} \gamma_1 \left(\prod_{j=1}^{i-1}
		[\alpha_j, \beta_j] \right) \alpha_i, \beta_i,
		\alpha_{i+1}, \ldots, \beta_g,
	\right. \\
	&\qquad \left. 
		\left(
\left(\prod_{j=1}^{i-1}
		[\alpha_j, \beta_j] \right)
			\beta_i^{-1}  \left(\prod_{j=i+1}^{g}
		[\alpha_j, \beta_j] \right) \right)^{-1}  \gamma_1 \left(\left(\prod_{j=1}^{i-1}
		[\alpha_j, \beta_j] \right) \beta_i^{-1}
		\left( \prod_{j=i+1}^{g}
	[\alpha_j, \beta_j] \right) \right) , 
	\gamma_1 \gamma_2 \gamma_1^{-1}, \ldots, \gamma_1 \delta_{f+1}
\gamma_1^{-1} \right).
\end{align*}
\item Moving $p_1$ across $\beta_i$ 
and then in a loop around
		$s_{f+1}, \ldots, s_1, p_n, \ldots, p_2$
	sends 
\eqref{equation:loop-tuple}
		to
\begin{align*}
	&\left( \alpha_1, \ldots, \beta_{i-1}, \alpha_i, \alpha_i^{-1}\left(\prod_{j=1}^{i-1}
		[\alpha_j, \beta_j] \right)^{-1}
		\gamma_1
		\left(\prod_{j=1}^{i-1} [\alpha_j, \beta_j] \right) \alpha_i \beta_i, \alpha_{i+1},
	\ldots, \alpha_g, \beta_g, \right. \\
	&\qquad \left. \left(
\left(\prod_{j=1}^{i-1}
		[\alpha_j, \beta_j] \right)
			\alpha_i  \left( \prod_{j=i+1}^{g}
	[\alpha_j, \beta_j] \right) \right)^{-1}  \gamma_1 \left(\left(\prod_{j=1}^{i-1}
		[\alpha_j, \beta_j] \right) \alpha_i \left( \prod_{j=i+1}^{g}
[\alpha_j, \beta_j] \right) \right) , \gamma_1 \gamma_2 \gamma_1^{-1}, \ldots,
\gamma_1 \delta_{f+1} \gamma_1^{-1} \right).
\end{align*}
	\end{enumerate}
\end{remark}

\begin{lemma}
	\label{lemma:action-on-rank-cover}
	We use the notation
	introduced in \autoref{example:rank-coefficient-system}.
	For $n$ even, and $B = \spec \mathbb C$,
	any element of $\pi_1(Q^n_{g,f})$ mapping to
	one of the following elements 
	of $\pi_1(\Conf^n_{X^{\oplus n} \oplus A_{g,f}})$ 
	acts on
	$\overline{W}_{\mathscr F^n_B}$
	with trivial Dickson invariant:
	\begin{enumerate}
		\item moving $p_i$ in a half-twist about $p_{i+1}$, which is conjugate to 
			the transformation in 
	\autoref{remark:explicit-mcg-action}(2),
\item 	moving $p_i$ twice across $\alpha_i$ or $\beta_i$, corresponding to a
	conjugate of the square of the transformation from
	\autoref{remark:explicit-mcg-action}(3) or (4).
	\end{enumerate}
Elements of $\pi_1(\Conf^n_{X^{\oplus n} \oplus A_{g,f}})$
sending $p_j$ around $s_i$ as in \autoref{remark:explicit-mcg-action}(1) may act
either with trivial or nontrivial Dickson invariant, where the triviality of the
Dickson invariant is a function of $i$
but not $j$.
\end{lemma}

\begin{proof}
We now explain how the claims
can be deduced from the explicit formula for the
action of 
$\pi_1(\Conf^n_{\Sigma^1_{g,f}})$ on 
$\pi_1(\Sigma^1_{g,f})$ from \autoref{remark:explicit-mcg-action}. 
We will use this in conjunction with 
the description of $V_{\mathscr F^n_B}$ in
\autoref{lemma:torsor-description} to verify
\autoref{lemma:action-on-rank-cover} with a modicum of computation.

We now describe coordinates for a certain free
$\mathbb Z/\nu \mathbb Z$-module $P_n$ of rank
$2r(2g + n + f+1)$
of which
$\overline{W}_{\mathscr F^n_B}$ is a subquotient.
Consider the free $\mathbb Z/\nu \mathbb Z$-module with the coordinates
$x^i_j$ for $1 \leq j \leq 2r$, $1 \leq i \leq 2g + n + (f+1)$.
Here, the $i$ indexes the 
$2g + n + (f+1)$
different entries in \eqref{equation:loop-tuple}
while the $j$ indexes the coordinate in the vector $v$ upon plugging in a matrix
of the form 
\eqref{equation:mtor-matrix} for each such entry.

Let $p_0$ be a point on $U$ disjoint from $D$ and $p$.    The group
$\pi_1(U-D-p_0, p)$ is free on the generators $(\alpha_1, \beta_1, \ldots, \alpha_g, \beta_g, \gamma_1, \ldots,
	\gamma_n, \delta_1, \ldots, \delta_{f+1})$, and also contains an element $\delta_{f+2}$ which satisfies the relation
\begin{align}
	\label{equation:fundamental-group-relation-one-more-point}
	(\alpha_1 \beta_1
	 \alpha_1^{-1} \beta_1^{-1}) \cdots (\alpha_g \beta_g \alpha_g^{-1} \beta_g^{-1})\gamma_1 \cdots \gamma_n\delta_1 \cdots \delta_{f+1}= \delta_{f+2}^{-1}.
\end{align}

We can also think of $P_n$ as the set of group homomorphisms $\phi: \pi_1(U-D-p_0, p) \to
	\asp_{2r}(\mathbb Z/\nu\mathbb Z)$ whose projection to $\sp_{2r}(\mathbb Z/\nu\mathbb Z)$ is $\rho_{\mathscr F'}$.  (In particular, this implies that $\phi(\delta_{f+2})$ lies in $(\mathbb Z/\nu\mathbb Z)^{2r}$, since $\rho_{\mathscr F'}(\delta_{f+2}) = \id.)$ 
 The section $s:\sp_{2r}(\mathbb Z/\nu\mathbb Z) \to \asp_{2r}(\mathbb
 Z/\nu\mathbb Z)$ affords one such homomorphism, namely $\phi_0 := s \circ
 \rho_{\mathscr F'}$.  (In the explicit form of \eqref{equation:mtor-matrix},
 $\phi_0$ sends each generator to a pair $(M,v)$ with $v=0$.)  
 
 Given such a $\phi$, we can attach to each of the $2g+n+f+1$ free generators of
 $\pi_1(U-D-p_0, p)$ an element of $(\mathbb Z/\nu\mathbb Z)^{2r}$; namely, we
 send a generator $g$ to $\phi(g) \phi_0^{-1}(g)$.  This gives the desired
 $(\mathbb Z/\nu\mathbb Z)$-module of rank $2r(2g+n+f+1)$.  Equivalently, we can
 think of $P_n$ as the space of $1$-cocycles from $\pi_1(U-D-p_0, p)$ to
 $(\mathbb Z/\nu\mathbb Z)^{2r}$, with respect to the group action of 
$\pi_1(U-D-p_0, p)$ on $(\mathbb Z/\nu\mathbb Z)^{2r}$ given by $\rho_{\mathscr F'}$.  This description makes it clear that the braids, which are automorphisms of $\pi_1(U-D-p_0, p)$ fixing $\rho_{\mathscr F'}$, act linearly on $P_n$.  (Of course, this can also be derived from the explicit description of the braid group action.)

Note that 
$\overline{W}_{\mathscr F^n_B}$
can be identified as a subquotient of $P_n$ via
\autoref{lemma:torsor-description}.  The reason for working with $P_n$ rather than with $\overline{W}_{\mathscr F^n_B}$ directly is that the explicit description of $P_n$ makes it easier to work out the action of a braid in concrete enough terms to easily compute Dickson invariants.

We first address $(2)$ in the statement.
To do so, let $T_{\alpha_i}$ denote the transformation described in 
	\autoref{remark:explicit-mcg-action}(3),
	which moves $p_1$ across $\alpha_i$.
	We wish to show the Dickson invariant of the transformation
	induced by 
	$T_{\alpha_i}^2$ on $\overline{W}_{\mathscr F^n_B}$
	is trivial.
	Letting 
	\begin{align*}
\eta&:= \left(\prod_{j=1}^{i-1}
		[\alpha_j, \beta_j] \right)
			\beta_i^{-1}  \left(\prod_{j=i+1}^{g}
		[\alpha_j, \beta_j] \right) \\
		\varepsilon &:= \eta^{-1} \gamma_1 \eta\\
\lambda &:= 
\left(\prod_{j=1}^{i-1}
		[\alpha_j, \beta_j] \right)^{-1} \gamma_1 \left(\prod_{j=1}^{i-1}
		[\alpha_j, \beta_j] \right) \\
\lambda' &:= 
\left(\prod_{j=1}^{i-1}
		[\alpha_j, \beta_j] \right)^{-1}\varepsilon \left(\prod_{j=1}^{i-1}
		[\alpha_j, \beta_j] \right)
	\end{align*}
	and using the formula from \autoref{remark:explicit-mcg-action}(3),
	the transformation $T_{\alpha_i}^2$ sends
	\eqref{equation:loop-tuple} to
\begin{align}
	\label{equation:past-alpha-squared-result}
	\left( \alpha_1, \ldots, \beta_{i-1}, \lambda' \lambda \alpha_i, \beta_i,
		\alpha_{i+1}, \ldots, \beta_g,
	\eta^{-2} \gamma_1 \eta^2, 
	\varepsilon \gamma_1 \gamma_2 \gamma_1^{-1}\varepsilon^{-1}, \ldots,
	\varepsilon \gamma_1 \delta_{f+1}
\gamma_1^{-1} \varepsilon^{-1} \right).
\end{align}

First, we will show the action of $T_{\alpha_i}^2$ on the free module $P_n$ has
square determinant.
The key calculation which we will use repeatedly is the following. 
Given a matrix 
$M \in \sp_{2r}(\mathbb Z/\nu \mathbb Z)$ and $v \in \left( \mathbb Z/\nu \mathbb Z \right)^{2r}$ 
we use $(M,v)$ to denote the element of $\asp_{2r}(\mathbb Z/\nu \mathbb Z)$ as
in
\eqref{equation:mtor-matrix}.
Then, 
\begin{equation}
	\label{equation:conjugation-computation}
	\begin{aligned}
	\left( N,w \right)^{-1} \cdot (M,v) \cdot (N,w)
	&=
	(N^{-1}, -N^{-1}w) \cdot (MN, Mw+v) \\
	&=
	(N^{-1} MN, -N^{-1} w + N^{-1}v + N^{-1} Mw).
	\end{aligned}
\end{equation}
Now, $T_{\alpha_i}^2$ acting on $P_n$ can be expressed as the composite of
several transformations. It is the composite of the maps 
induced by sending $\gamma_2, \ldots, \delta_{f+1}$ to
their conjugate by $\varepsilon \gamma_1$,
followed by the map
induced by
$\alpha_i \mapsto \lambda' \lambda \alpha_i$ followed by the map induced by
$\gamma_1 \mapsto \eta^{-2} \gamma_1 \eta^2$.

We claim that each of these transformations have square determinant $1$, and hence the
composite will have square $1$.
First, we show the determinant is a square for each transformation induced by sending a matrix $(M,v)$ associated to one of $\gamma_2, \ldots, \delta_{f+1}$ to
its conjugate by a matrix $(N,w)$, associated to $\varepsilon \gamma_1$.
In this case, $N
= \id$ and so the output of conjugation is $(M, -w+v+Mw)$
by \eqref{equation:conjugation-computation}.
Since the $x^i_j$ entries appearing in $w$ are disjoint from those
appearing in $v$, this transformation is unipotent, so has determinant $1$.

We next consider the transformation coming from $\alpha_i \mapsto \lambda'
\lambda \alpha_i$. Since $\lambda'$ and $\lambda$ are both conjugate to
$\gamma_1$, the element $\lambda'
\lambda$ corresponds to a pair of the form 
$(\id,w)$.  So if $(M,v)$ corresponds to $\alpha_i$, then the transformation in
question sends $(M,v)$ to $(\id,w)\cdot (M,v) = (M,v + w)$.  Once again, the $x^i_j$ entries appearing in $w$ are disjoint from those appearing in $v$, so this transformation is unipotent and has determinant $1$.

Third, to calculate the determinant of the map 
induced by $\gamma_1 \mapsto \eta^{-2} \gamma_1 \eta^2$
we use \eqref{equation:conjugation-computation} with $(N,w)$ corresponding to
$\eta$ and $(M,v)$ corresponding to $\gamma_1$. 
Since $M=-1$, the output of the transformation is $(-\id,N^{-1}v -2N^{-1}w)$. 
Since $w$ only involves $x^i_j$ which are disjoint from those associated to
$\gamma_1$, appearing in $v$, the determinant of this transformation agrees with
that of $N$. Since we are conjugating by $\eta^2$, the resulting $N$ has square
determinant.

In order to conclude the Dickson invariant associated to the action of $T_{\alpha_i}^2$ 
on $\overline{W}_{\mathscr F^n_B}$
is
trivial, it remains to check that this operator still has square determinant
upon passing to the subquotient
$\overline{W}_{\mathscr F^n_B}$ of $P_n$.
To do so, note that $\overline{W}_{\mathscr F^n_B}$ of $P_n$ is obtained from
$P_n$ in three steps:
\begin{enumerate}[(A)]
	\item We first take the subspace dictated by the drop condition
\autoref{lemma:torsor-description}(2), 
\item we then take
the subspace where the product
\eqref{equation:fundamental-group-relation} is satisfied, upon plugging in
matrices for the generators of the fundamental group,
\item
and finally we pass to the quotient by the conjugation action as described at
the end of \autoref{lemma:torsor-description}.
\end{enumerate}
To prove $T_{\alpha_i}^2$ has square determinant (and hence trivial Dickson
invariant) on $\overline{W}_{\mathscr F^n_B}$, we will show that for each of
these steps, the induced operator on the associated subspace or quotient space
has trivial determinant.

To simplify our calculations, we may and shall assume for the remainder of this
proof that $\nu$ is prime; this does not restrict the conclusion of the theorem,
because it follows from \autoref{lemma:determinant-monodromy} that, for any
prime $\ell \mid \nu$, the Dickson invariant of the action of an automorphism of
$\overline{W}_{\mathscr F^n_B}$ can be computed on $\overline{W}_{\mathscr
F^n_B} \otimes_{\mathbb Z / \nu \mathbb Z} \mathbb Z / \ell \mathbb Z.$  Under
this hypothesis, all our $\Z /\nu \Z$-modules are now vector spaces over a field and the Dickson invariant is additive in exact sequences, a fact we will use repeatedly in the argument that follows.

For step $(A)$, we note that there is a homomorphism
\beq
L: P_n \ra \bigoplus_{i=1}^{f+1} \left[ \left(\mathbb Z/\nu \mathbb Z
\right)^{2r} / \im  (\rho_{\mathscr F}(\delta_i) - \id) \right] 
\eeq
whose kernel is the subspace specified by condition
\autoref{lemma:torsor-description}(2). The map $L$ is defined by projection onto
the $\delta_1, \ldots, \delta_{f+1}$ coordinates of $P_n$ followed by projection
of the $\delta_i$ coordinate onto the quotient by $\im  (\rho_{\mathscr
F}(\delta_i) - \id)$.

We also know that $T_{\alpha_i}^2$ acts on the $\delta_i$ coordinate by
conjugation by $\varepsilon \gamma_1 \in \asp_{2r}(\mathbb Z/\nu \mathbb Z)$;
since $\varepsilon \gamma_1$ is an element of the form $(\id,w)$, conjugation by
$\varepsilon \gamma_1$ modifies the $\delta_i$ coordinate by adding an element
of the form $(\rho_{\mathscr F}(\delta_i)w - w)$.
Via a computation analogous to \eqref{equation:conjugation-computation},
for any $x \in P_n$, we have $L(x) = L(T_{\alpha_i}^2 x)$. 
We conclude that $T_{\alpha_i}^2$ acts trivially (and a fortiori with trivial Dickson invariant) on $P_n / L(P_n)$.

For step $(B)$, we observe that $T_{\alpha_i}^2$, considered as an automorphism of $\pi_1(U-D-p_0,p)$, preserves the element $\delta_{f+2}$, and so it preserves the left-hand side of \eqref{equation:fundamental-group-relation-one-more-point}.  In particular, if
\beq
M: P_n \to \left(\mathbb Z/\nu \mathbb Z \right)^{2r} 
\eeq
is the map provided by the left-hand side of \eqref{equation:fundamental-group-relation-one-more-point}, whose kernel is the subspace of $P_n$ obeying \eqref{equation:fundamental-group-relation}, then the induced action of $T_{\alpha_i}^2$ on $P_n / M(P_n)$ is trivial. 

Finally, for step $(C)$, we wish to show $T_{\alpha_i}^2$ acts trivially on the
subspace generated by coboundaries, corresponding to changing the basepoint of the
original matrix.
It is possible to compute this directly using the formula
\eqref{equation:past-alpha-squared-result}, but we will provide a more conceptual explanation.
First, $T_{\alpha_i}^2$ acts linearly on $P_n$ and thus fixes the zero element.  But $T_{\alpha_i}^2$ commutes with the operation of conjugating all $2g+n+f+1$ coordinates by a matrix of the form $(\id,
v)$, for any $v \in (\mathbb Z/\nu \mathbb Z)^{2r}$.  This operation is not linear on $P_n$, but it is {\em affine-linear}, acting by translation by a coboundary $c_v$.  Since $T_{\alpha_i}^2$ commutes with translations by all coboundaries in $P_n$ and fixes $0$, it also fixes all coboundaries in $P_n$.

The combination of the three steps above allows us to conclude that $T_{\alpha_i}^2$ acts with square determinant on 
$\overline{W}_{\mathscr F^n_B}$, and hence has trivial Dickson invariant.
Similar reasoning shows the action in 
\autoref{remark:explicit-mcg-action}(4) is trivial, concluding the verification
of $(2)$.

Next, we will show the elements from $(1)$ of the statement act
trivially on the module $P_n$.
Using the description of the
half-twist from
\autoref{remark:explicit-mcg-action}(2), we can study the resulting 
$2r(2g+n+(f+1)) \times 2r(2g+n+(f+1))$ matrix associated to how 
the transformation in \autoref{remark:explicit-mcg-action}(2) acts on $P_n$,
upon plugging in matrices of the form \eqref{equation:mtor-matrix} for each 
entry in
\autoref{remark:explicit-mcg-action}(2).
Because the $M$ from \eqref{equation:mtor-matrix} associated to each $\gamma_i$ is
$-\id$, we find the above
$2r(2g+n+(f+1)) \times 2r(2g+n+(f+1))$ matrix 
is a block diagonal matrix, consisting of 
$2r$ blocks of size $2g+n+(f+1)$.
In particular, the determinant of this matrix acting on $P_n$ is a $2r$th power,
so it is a square.
To conclude the Dickson invariant is trivial, it remains to justify why passing
$\overline{W}_{\mathscr F^n_B}$ of $P_n$, preserves the condition that the
determinant is a square.
The action on the quotient space to the drop condition from $(A)$ above is
trivial, because the $\delta_i$ are preserved by this transformation.
Triviality of the action on the quotient associated to $(B)$ and subspace
associated to $(C)$ follow in the same way as for $T_{\alpha_i}^2$ in the proof
of $(2)$ from the statement of the lemma above.

The final part of the statement of this lemma, regarding the action in
	\autoref{remark:explicit-mcg-action}(1) holds because one can express
	the generator sending $p_j$ around $s_i$ as a product of half twists
	permuting the $p_j$ with the loop sending $p_n$ around $s_i$, and all
	these half twists have trivial image under the Dickson invariant map, as
	we have shown above.
	\end{proof}
\begin{lemma}
	\label{lemma:rank-coefficient-system}
With notation as in
\autoref{example:rank-coefficient-system},
	the sequence $(H^{\rk}_{\mathscr F, g, f})_n$ of $B^n_{g,f}$ representations
	defines a coefficient system for $\Sigma^1_{g,f}$ over the trivial
	coefficient system $V$ for $\Sigma^1_{0,0}$.
\end{lemma}
\begin{proof}
	Using the explicit description of the double cover $R_{\mathscr
	F^n_{\mathbb C}} \to Q^n_{g,f}$ from
	\autoref{lemma:action-on-rank-cover}(2)
	we first claim the double cover is in fact the base change of a double
	cover of $\conf n U B$.
	Indeed, to show this, it is equivalent to show we can extend the homomorphism
	$\pi_1(Q^n_{g,f}) \to \mathbb Z/2 \mathbb Z$ to a homomorphism
	$\pi_1(\conf n U B) \to \mathbb Z/ 2\mathbb Z$. We can extend this
	homomorphism, for example, by sending loops corresponding to
	moving $p_j$ across $\alpha_i$ or $\beta_i$,
	as in \autoref{remark:explicit-mcg-action}(3) and (4),
	to the trivial element of $\mathbb Z/2 \mathbb Z$. 

	This shows the cover $S^{n,\on{rk}}_{\mathscr F, g,f}$ is then the product of a two
	element set $\{a_n, b_n\}$ corresponding to the above double cover of $\conf n U B$
	with the set 
	$\inertiaindex n {\mathbb Z/2 \mathbb Z} {\{1\}} g f$.
	Therefore, it suffices to show the free vector space on both of these
	sets form coefficient systems.
	First,
	$\inertiaindex n {\mathbb Z/2 \mathbb Z} {\{1\}} g f$
	forms a coefficient system over the trivial coefficient
	system for $\Sigma^1_{0,0}$ by
	\autoref{example:hurwitz-coefficient-system}.

	Second, we explain why the explicit description of the action of $\pi_1(\conf n U B)$
	and $\{a_n, b_n\}$ obtained from \autoref{lemma:action-on-rank-cover}
	shows the free vector space on this collection of these two element
	sets,
	$\{a_n,b_n\}_{n \geq 1}$,
	forms a coefficient system over the trivial coefficient system for
	$\Sigma^1_{0,0}$.
	Indeed, the description of half-twists from
	\autoref{lemma:action-on-rank-cover}(1) shows this lies over the trivial
	coefficient system for $\Sigma^1_{0,0}$.
	The condition to be a coefficient system over the trivial coefficient
	system amounts to checking that 
the action of $\id \times B^{n-i}_{g,f}\subset
B^i_{0,0} \times B^{n-i}_{g,f} \subset B^n_{g,f}$,
on $\{a_n, b_n\}$
can be identified with the action of
$B^{n-i}_{g,f}$ on $\{a_{n-i}, b_{n-i}\}$ via the map of sets 
$\{a_{n-i}, b_{n-i}\} \to \{a_n, b_n\}$ given by $a_{n-i} \mapsto a_n, b_{n-i} \mapsto b_n$.
This indeed holds since 
$B^{n-i}_{g,f} \subset B^n_{g,f}$,
is generated by a subset of the transformations described in
\autoref{lemma:action-on-rank-cover}, and the action of each generator of $B^{n-i}_{g,f} \subset
B^n_{g,f}$ on $\{a_n, b_n\}$ acts in the same way on $\{a_{n-i}, b_{n-i}\}$, via 
\autoref{lemma:action-on-rank-cover}.
\end{proof}

\begin{remark}
	\label{remark:explicit-rank-cover-description}
	In fact, the proof of \autoref{lemma:rank-coefficient-system}
	shows that the $n$th graded part of the rank double
	cover can be identified with a free vector space on a set $V$ of size
	$2^{2g+1}$, for $V$ an explicit quotient 
	$H^1(\Sigma^1_{g,f}, \mathbb Z/2 \mathbb Z) \to V$.
The $B^n_{g,f}$ action is obtained via
	a surjection
	$B^n_{g,f} \to (B^n_{g,f})^{\on{ab}} \simeq (B^1_{g,f})^{\on{ab}} \simeq
H^1(\Sigma^1_{g,f}, \mathbb Z)  \to
	H^1(\Sigma^1_{g,f}, \mathbb Z/2 \mathbb Z) \to V$.
\end{remark}

\begin{remark} 
	\label{remark:root-numbers}
	It should come as no surprise that the result of the explicit
	topological computation of \autoref{lemma:action-on-rank-cover} ends
	up having a rather simple form, as described in
	\autoref{remark:explicit-rank-cover-description}.
	When $\mathscr F = A[\nu]$ is the $\nu$-torsion of an abelian variety
	$A$, it is possible to show that the Dickson invariant is determined by
	the root number of the quadratic twist $A_\chi$ of $A$. 

	Moreover, the root number of $A_\chi$ can be computed via a relatively
	simple formula, as we now explain.
	Let $C_{K(C)} := \mathbb A^\times_{K(C)}/K(C)^\times$ denote the idele class group (the
	quotient of the ideles of $K(C)$ by the units of $K(C)$).
	The global Artin homomorphism $\theta : C_{K(C)} \to
	\gal(\overline{K(C)}/K(C))$ from global class field theory allows us to
	pull back representations of $\gal(\overline{K(C)}/K(C))$ to $C_{K(C)}$.
	Now, if we consider $\chi$ as a quadratic extension of $K(C)$, the field over
	which the generic fiber of $A$ is defined, it corresponds to a
	surjection
	$\chi': \gal(\overline{K(C)}/K(C)) \to \{\pm 1\}$. Hence we can
	pull it back
	along the global Artin homomorphism to obtain a map 
	$\widetilde{\chi}: C_{K(C)} \to \{\pm 1\}$, with
	$\widetilde{\chi} = \chi'\circ \theta$.
	If $B$ is an abelian variety, we use $W(B)$ to denote its root number.
	We can compute $W(A_\chi)$ in terms of $W(A)$ using the explicit formula
  \begin{equation*}
	  W(A_\chi) = W(A) \widetilde{\chi}(N_A),
    \end{equation*}
    where $N_A$ is the conductor of $A$,
see \cite[Corollary 6.12]{bisatt:explicit-root-numbers} and \cite[Proposition 1]{sabitova:twisted-root-numbers}.
\end{remark}

Finally, we construct coefficient systems associated to the
 fiber product of covers associated to $H$ moments and the rank
double cover.

\begin{example}
	\label{example:fiber-product-with-rank-coefficients}
	Continuing with notation as in
	\autoref{example:rank-coefficient-system},
	the coefficient systems
	$H^{\rk}_{\mathscr
	F, g,f}, H_{\allinertia {\mathbb Z/2 \mathbb Z} {1} g
	f}, H_{\sallinertia {\mathscr F}{H} g f}$
	all correspond respectively to the finite covers of $\on{Conf}^n_{X^{\oplus n} \oplus
	A_{g,f}}$ for varying $n$:
	$R_{\mathscr F^n_B}, Q^n_{g,f}$, and $\on{Hur}_{\sinertiaindex n {\mathscr F}{H} g f}$.
	Define $\on{Hur}^{\rk}_{\sinertiaindex n {\mathscr F}{H} g f} :=
	R_{\mathscr F^n_B} \times_{Q^n_{g,f}} \on{Hur}_{\sinertiaindex n {\mathscr F}{H} g f}$ and let
	$H^{\rk}_{\sallinertia {\mathscr F}{H} g f}$
	be the corresponding coefficient system.
	Take 
	$V := H_{\sallinertia {  \underline{\mathscr F}}{H} 0 0}$ and $F :=
	H^{\rk}_{\sallinertia {\mathscr F}{H} g f}$. 
	Then, $F$ is a coefficient system over $V$ because 
	$H_{\sallinertia {\mathscr F}{H} g f}$ is a coefficient system over $V$
	and
	both $H_{\allinertia {\mathbb Z/2 \mathbb Z} {\{1\}} g f}$ and
	$H^{\rk}_{\mathscr F, g,f}$ are coefficient systems over the trivial coefficient
	system, the latter by \autoref{lemma:rank-coefficient-system} and the
	former as explained in \autoref{example:rank-coefficient-system}.
\end{example}

\subsection{Homological stability of the rank double cover}
\label{subsection:homological-stability-rank-double-cover}

We next set out to prove the main homological stability properties for the
spaces related to Selmer groups we are interested in. Namely, in
\autoref{lemma:moments-cohomology} we will prove
these results
for the Selmer stacks, the rank double cover, and moments associated to both of
these.

\begin{notation}
	\label{notation:selmer-space-moment-notation}
	Let $H$ be a finite $\mathbb Z/\nu \mathbb Z$-module of the form $H \simeq
\prod_{i=1}^m \mathbb Z/\nu_i \mathbb Z$.
Let $\mathscr F$ be a lcc symplectically self-dual sheaf of free $\mathbb Z/\nu \mathbb Z$-modules, and maintain hypotheses as in
\autoref{notation:quadratic-twist-notation} and
\autoref{definition:selmer-sheaf}. 
We then have Selmer spaces $\selspace {\mathscr F[\nu_i]^n_B}$ obtained from
the $\nu_i$-torsion sheaves $\mathscr F[\nu_i]$ in place of $\mathscr F$.
Define
\begin{align*}
	\selspacemoments {\mathscr F^n_B} H := \selspace {\mathscr F[\nu_1]^n_B} \times_{\qtwist n
		U B} \selspace {\mathscr F[\nu_2]^n_B} \times_{\qtwist n U B} \cdots
	\selspace {\mathscr F[\nu_m]^n_B}.
\end{align*}
Also define $\rankselspacemoments {\mathscr F^n_B} H := \selspacemoments {\mathscr
F^n_B} H \times_{\qtwist n U B} \rankcover n {\mathscr F}$
and
define $\rankselhur {\mathscr F^n_B} H := \selhur {\mathscr F^n_B} H
\times_{\qtwist n U B} \rankcover n {\mathscr F}$.
\end{notation}

\begin{lemma}
\label{lemma:hypothesis-cohomology-bound}
The hypotheses of
\autoref{corollary:cohomology-bound}
are satisfied if
$V = H_{\sallinertia { \underline{\mathscr F}}{H} 0 0}$
and $F$ is either
$H_{\sallinertia {\mathscr F} H g f}$ or
$F = H^{\rk}_{\sallinertia {\mathscr F}{H} g f}$.
\end{lemma}
\begin{proof}
We consider two cases:
\begin{enumerate}
	\item  $V = H_{\sallinertia { \underline{\mathscr F}}{H} 0 0}$ and $F =
		H_{\sallinertia {\mathscr F}{H} g f}$,
	\item $V = H_{\sallinertia { \underline{\mathscr F}}{H} 0 0}$ and $F =
		H^{\rk}_{\sallinertia {\mathscr F}{H} g f}$.
\end{enumerate}
Note that by 
\autoref{example:specified-hurwitz-coefficient-system}
and \autoref{example:fiber-product-with-rank-coefficients},
$V$ and $F$ are indeed coefficient systems.
We will first consider case $(1)$ and show the existence of a homogeneous
central non-invertible $\mathbb U$ in $R^V$ with
kernel and cokernel of finite degree.
Note that $c_H$ does not generate $G_H = \ahsp_{2r}(\mathbb Z/\nu
\mathbb Z)$ but instead generates the preimage of $\{\pm 1\} \subset
\sp_{2r}(\mathbb Z/\nu \mathbb Z)$ in $G_H$.
Let $S_H \subset G_H$ denote the subgroup generated by $c_H$.
Note that $S_H$ has order $2 \bmod 4$ because $\nu$ is odd.
Then, $(S_H, c_H)$ is non-splitting in the sense of
\cite[Definition 3.1]{EllenbergVW:cohenLenstra}
by
\cite[Lemma 3.2]{EllenbergVW:cohenLenstra}.
It then follows from \cite[Lemma 3.5]{EllenbergVW:cohenLenstra}
that there is a homogeneous non-invertible central $\mathbb U$ with finite degree
kernel and cokernel.

We can deduce 
case (2) from case (1).
Namely, taking the same operator $\mathbb U$ as in part $(1)$,
we can view
$F = H^{\rk}_{\sallinertia {\mathscr F}{H} g f}$
as two copies of
$H_{\sallinertia {\mathscr F}{H} g f}$
Since we have already shown in the first case that the action of $\mathbb U$ on
$H_{\sallinertia {\mathscr F}{H} g f}$ has kernel and cokernel of finite degree,
the same holds for the action of $\mathbb U$ on 
$F = H^{\rk}_{\sallinertia {\mathscr F}{H} g f}$.
\end{proof}

\begin{lemma}
	\label{lemma:moments-cohomology}
	Let $H$ be a finite $\mathbb Z/\nu \mathbb Z$-module and $B = \spec
	\mathbb C$.
	We work with coefficient systems over the field $\mathbb Z/\ell' \mathbb
	Z$, for
	$\ell'$ relatively prime to $2, q$, and $\# \asp_{2r}(\mathbb Z/\nu
	\mathbb Z)$.
	Suppose $\mathscr F$ is as in \autoref{hypotheses:theta-map}.
		There is a constant
	$K$ depending on $\mathscr F$ and $H$ but not on $n$, for $n$ even,
	so that 
	\begin{equation}
	\label{equation:h-cohomology-bound}
\begin{aligned}
	\dim H^i(\pi_1(\on{Conf}^n_{X^{\oplus n} \oplus A_{g,f}}, x^{\oplus
	n}),H_{\sinertiaindex n {\mathscr F}{H} g f}) &< K^{i+1}
\text{	and} \\
\dim H^i(\pi_1(\on{Conf}^n_{X^{\oplus n} \oplus A_{g,f}},
x^{\oplus n}),H^{\rk}_{\sinertiaindex n {\mathscr F}{H} g f}) &< K^{i+1}.
\end{aligned}
\end{equation}
	Then, 
	\begin{equation}
	\label{equation:selmer-cohomology-bound}
	\begin{aligned}
\dim H^i(\selspacemoments {\mathscr F^n_{\mathbb C}} H, \mathbb Z/\ell'
	\mathbb Z) & < K^{i+1} \text{ and } \\
\dim H^i(\rankselspacemoments {\mathscr F^n_{\mathbb C}} H, \mathbb Z/\ell'
	\mathbb Z) &< K^{i+1}.
\end{aligned}
\end{equation}
\end{lemma}
\begin{proof}
	First, the bound \eqref{equation:h-cohomology-bound}
follows
from \autoref{corollary:cohomology-bound}
whose hypotheses are verified by \autoref{lemma:hypothesis-cohomology-bound}.

For \eqref{equation:selmer-cohomology-bound},
note that in order to bound the homology of 
$\selspacemoments {\mathscr F^n_{\mathbb C}} H$, by transfer and the assumption
that $\ell' \neq 2$,
it suffices to bound the
homology of its finite \'etale double cover
$\rankselspacemoments {\mathscr F^n_{\mathbb C}} H$.

Recall that 
we use the notation
$\on{Hur}_{\sinertiaindex n {\mathscr F}{H}g f}$ and
$\on{Hur}^{\rk}_{\sinertiaindex n {\mathscr F}{H}g f}$ for the
finite unramified covering
space over
$\on{Conf}^n_{X^{\oplus n} \oplus A_{g,f}}$ corresponding to the action of
of $\pi_1(\on{Conf}^n_{X^{\oplus n} \oplus A_{g,f}},
	x^{\oplus n})$ on $H_{\sinertiaindex n {\mathscr F}{H}g f}$
and
$\pi_1(\on{Conf}^n_{X^{\oplus n} \oplus A_{g,f}}, x^{\oplus n})$
on $H^{\rk}_{\sinertiaindex n {\mathscr F}{H}g f}$.
It follows from these definitions that
\begin{align*}
	H^i( \on{Hur}^{\rk}_{\sinertiaindex n {\mathscr F}{H}g f}, \mathbb Z/\ell' \mathbb Z) \simeq
	H^i(\pi_1(\on{Conf}^n_{X^{\oplus n} \oplus A_{g,f}}, x^{\oplus
	n}),H^{\on{\rk}}_{\sinertiaindex n {\mathscr F}{H} g f}).
\end{align*}

To conclude the final statement for bounding the homology of 
$\rankselspacemoments {\mathscr F^n_{\mathbb C}} H$, by transfer, it suffices to show
	$\on{Hur}^{\rk}_{\sinertiaindex n {\mathscr F}{H}g f}$
	defines a finite \'etale cover of $\rankselspacemoments {\mathscr
	F^n_{\mathbb C}} H$.
	We next use the isomorphism
	$\selspacemoments {\mathscr F^n_{\mathbb C}}H
	\to \selhur {\mathscr F^n_{\mathbb C}} H$ from
	\autoref{corollary:selmer-to-hurwitz-moments} over $\qtwist n U B$, 
	which also yields the identification 
	$\rankselspacemoments  {\mathscr F^n_{\mathbb C}}H
	\simeq \rankselhur {\mathscr F^n_{\mathbb C}} H$.
	It therefore suffices to show
	$\rankselhur {\mathscr F^n_{\mathbb C}} H$ 
	has a finite covering space by 
	$\on{Hur}^{\rk}_{\sinertiaindex n{\mathscr F}{H}g f}$.
	There is an action of the group $G_H$ as in
	\autoref{example:specified-hurwitz-coefficient-system} on the latter
	(acting via conjugation on 
	$\on{Hur}_{\sinertiaindex n {\mathscr F}{H}g f}$ and trivially on
$R_{\mathscr F^n_{\mathbb C}}$).
	The quotient by this $G_H$ action is precisely
	$\rankselhur {\mathscr F^n_{\mathbb C}} H$,
	as follows from \autoref{remark:hurwitz-as-quotient}, since 
$\on{Hur}_{\sinertiaindex n {\mathscr F}{H}g f}$ is the fiber product of the
pointed Hurwitz space with the rank double cover, while 
$\rankselhur {\mathscr F^n_{\mathbb C}} H$ is the fiber product of the usual
Hurwitz space with the rank double cover.
\end{proof}

\begin{remark}
	\label{remark:mapping-class-group-action-on-cohomology}
	The Hurwitz stacks and Selmer stacks, whose cohomology we analyze in
	\autoref{lemma:moments-cohomology} have (up to finite index issues) an
	action of $\on{Mod}_{g,f}$ the mapping class group of a genus $g$,
	$f$-punctured surface. Hence, their stable cohomology groups are virtual
	$\on{Mod}_{g,f}$ representations.
	It would be extremely interesting to determine
	which representations these are. A precursor to doing so would be to
	compute the dimension of these representations. We also cannot rule out
	the possibility these dimensions are $0$, and so the representations are
	not particularly interesting. See also \autoref{remark:0-cohomology}
\end{remark}

\subsection{Relation between the rank double cover and parity of rank}
\label{subsection:rank-cover-and-parity}

Our main reason for introducing the rank double cover is that it tells us about
the parity of the rank of $\sel_\ell$, as we next explain.
For the next statement, recall the definition of
$\mathcal N^i$ from 
\autoref{definition:actual-selmer-distribution}.
\begin{lemma}
	\label{lemma:image-of-rank-cover-has-selmer-parity}
	Assume $\nu$ is odd, $n >0$ is even, and $B$ is an integral affine scheme 
	with $2\nu$ invertible on $B$. Let $b\in B$ a closed
	point with residue field $\mathbb F_{q_0}$. Let $\mathbb F_q$ be a
	finite extension of $\mathbb F_{q_0}$. Use hypotheses as in 
	\autoref{notation:quadratic-twist-notation},
	\autoref{hypotheses:big-monodromy-assumptions},
	and
\autoref{example:abelian-special-fiber}, so $\mathscr F_b \simeq  A[\nu]$.	
Let $\ell \mid \nu$ and $i := \rk V_{\mathscr F^n_b[\ell]} \bmod 2 \in \{0,1\}$.
	Then, for $x \in \qtwist
	n {U_b} b(\mathbb F_q)$, $\sel_\nu(A_x) \in \mathcal N^i$
	if and only if $x$ lies in the image of $\rankcover n {\mathscr F_b}(\mathbb F_q) \to \qtwist n
	{U_b} b(\mathbb F_q)$.
\end{lemma}
\begin{proof}
	Let $g_x := \rho_{\mathscr F^n_b}(\frob_x)$. For
	$\ell \mid \nu$, we use $g_{x,\ell}$ to denote the image of $g_x$
	under the map $\o(Q_{\mathscr F^n_b}) \to \o(Q_{\mathscr F^n_b[\ell]})$.
	First, \eqref{equation:kernel-to-determinant} yields
	\begin{align*}
	\dim \ker(g_{x, \ell} -\id) \bmod 2 \equiv\rk V_{\mathscr F^n_b[\ell]}-
D_{Q_{\mathscr F^n_b}} (g_{x, \ell}) \bmod 2.
	\end{align*}
Next,
\autoref{lemma:selmer-identification}
	gives $\ker(g_x - \id)
	\simeq \sel_\nu(A_x)$. Combining these, we find
	\begin{align*}
	D_{Q_{\mathscr F^n_b}}(g_{x,\ell}) \equiv \rk
	V_{\mathscr F^n_b[\ell]} - \dim \ker(g_{x, \ell} -\id) \equiv
	\rk V_{\mathscr F^n_b[\ell]} - \dim \sel_\ell(A_x) \bmod 2.
	\end{align*}
	Since this holds for every $\ell \mid \nu$, we find that
	$D_{Q_{\mathscr F^n_b}}(g_{x,\ell})$ takes the value $0$ if and only if 
	$\rk V_{\mathscr F^n_b[\ell]} \equiv \dim \sel_\ell(A_x) \bmod 2$.
	Since the finite \'etale double cover $\rankcover n {\mathscr F} \to
	\qtwist n {U_b}
	b$ is trivial over each $\mathbb F_q$-point with trivial Dickson
	invariant,
	$D_{Q_{\mathscr F^n_b}}(g_{x,\ell})$ takes the value $0$ if and only if $x \in \qtwist n
	{U_b} b (\mathbb F_q)$ is in the image of $\rankcover n {\mathscr F} (\mathbb F_q)$.
	We conclude the result because 
	$\rk V_{\mathscr F^n_b[\ell]} \equiv \dim \sel_\ell(A_x)$
	can be restated as 
	$\sel_\nu(A_x) \in \mathcal N^i$, with $i =\rk
	V_{\mathscr F^n_b[\ell]}\bmod 2$.
\end{proof}

We now use the previous lemma to show that the distribution of Selmer elements
on the double cover controlling the
parity of the rank agrees with the locus of points on the base where the rank of
$\sel_\ell$
has a specified parity.
This is a fairly trivial observation, but allows us to connect moments of the
rank double cover to moments of the space of quadratic twists with specified
parity of rank of $\sel_\ell$. 
This plays
a key role in proving our main theorem,
\autoref{theorem:main-finite-field}.
For this, recall the definition of $X^i_{A[\nu]^n_{\mathbb F_q}}$ from
\autoref{definition:actual-selmer-distribution}.
\begin{lemma}
	\label{lemma:distribution-equals-rank-cover-distribution}
	With assumptions and notation as in
\autoref{lemma:image-of-rank-cover-has-selmer-parity},
so, in particular, 
$i := \rk V_{\mathscr F^n_b[\ell]} \bmod 2 \in \{0,1\}$
for every $\ell \mid \nu$,
	we have
	\begin{align}
	\label{equation:expectation-rank-cover}
		\frac{\# \rankselspacemoments {\mathscr F^n_b}H (\mathbb
		F_q)}{\#\rankcover n {\mathscr F_b} (\mathbb F_q)} = \mathbb E(\# \hom(X^i_{A[\nu]^n_{\mathbb F_q}},H)).
	\end{align}
\end{lemma}
\begin{proof}
Using \autoref{lemma:image-of-rank-cover-has-selmer-parity},
	the distribution $X^i_{A[\nu]^n_{\mathbb F_q}}$ agrees with the
	distribution of Selmer groups at points $x \in \qtwist n {U_b} b(\mathbb
	F_q)$ in the
	image of $\rankcover n {\mathscr F_b}(\mathbb F_q) \to \qtwist n {U_b} b(\mathbb
	F_q)$.
	Since $\rankcover n {\mathscr F_b} \to \qtwist n {U_b} b$ is a finite \'etale double cover,
	each $\mathbb F_q$-point of $\qtwist n {U_b} b$ in the image of a $\mathbb
	F_q$-point of $\rankcover n {\mathscr F_b}$ has exactly two $\mathbb F_q$-points in its
	preimage. This means that, for $y$ varying over points of $\rankcover n
	{\mathscr F_b}(\mathbb F_q)$ and $K \in \mathcal N$ a finite $\mathbb Z/\nu \mathbb
	Z$-module,
	\begin{align*}
		\prob(X^i_{A[\nu]^n_{\mathbb F_q}} \simeq K) &=
		\prob\left(\sel_\nu(A_x) \simeq K \vert x \in \im (\rankcover n
		{\mathscr F_b} (\mathbb F_q) \to \qtwist n {U_b} b(\mathbb F_q)) \right)
	\\
	&= \prob(\sel_\nu(A_y) \simeq K).
	\end{align*}
	Taking the expectation of the number of maps to $H$, which is the same
	as the number of maps from $H$, it is enough to
	show the left hand side of
	\eqref{equation:expectation-rank-cover} is the expected number of maps
	from $H$ to $\sel_\nu(A_y)$.
	This follows from \autoref{lemma:selmer-identification} and the
	definition of $\rankselspacemoments {\mathscr F^n_b}H$ as a fiber
	product.
\end{proof}

\section{Computing the moments}
\label{section:moments}

The purpose of this section is to combine our homological stability results with
our
big monodromy results to
determine the
moments of Selmer groups in quadratic twist families.
The analogous problem of determining the moments in the context of Cohen-Lenstra was approached in
\cite{EllenbergVW:cohenLenstra}, where the problem was much easier as the relevant big monodromy result
was already available in the literature.
In \autoref{subsection:orthogonal-moments}, we compute various statistics
associated to kernels of random elements of orthogonal groups.
Via equidistribution of Frobenius elements we then relate this to components of
Selmer stacks in \autoref{subsection:connected-components}.

\subsection{Moments related to random elements of orthogonal groups}
\label{subsection:orthogonal-moments}

We next compute statistics associated to random elements of orthogonal groups.
In \autoref{lemma:kernel-and-bklpr-distributions}, we compute the distributions of $1$-eigenspaces of random elements of
orthogonal group, and show that these limit to the BKLPR distribution as the
size of the matrix grows. Moreover, we show this in a strong enough sense so
that the limit of the moments is the moment of the limit.

Our next computation
is quite analogous to that of \cite[Proposition
4.13]{fengLR:geometric-distribution-of-selmer-groups}, except that here we work
over $\mathbb Z/\nu \mathbb Z$ for general $\nu$, instead of the case that $\nu$
is prime covered in \cite{fengLR:geometric-distribution-of-selmer-groups}.

	For what follows, we use the notation of
	\cite[\S4.2.1]{fengLR:geometric-distribution-of-selmer-groups}. 
	In the case $\nu$ is prime, 
	we let $A,
	B, C$ be the three nontrivial cosets of $\Omega(Q)$ in $\o(Q)$ so that $\on{sp}^-_Q$ is nontrivial
	on $A$ and $C$, while $D_Q$ is nontrivial on $B$ and $C$.
	For $Z$ a nonnegative integer-valued random variable, we let $G_Z(t) = \sum_{i \in \mathbb N}
	\on{Prob}(Z = i) t^i$.
	As in \cite[\S4.2.1]{fengLR:geometric-distribution-of-selmer-groups},
	for $\bullet \in \{\Omega, A, B, C\}$, we use $\on{RSel}_V^\bullet$ to
	denote the random variable given as $\dim \ker(g - \id)$ for $g$ a uniform random
	element of the coset $\bullet$.

\begin{lemma}
	\label{lemma:coset-generating-functions}
	Let $(Q, V)$ be a quadratic space over $\mathbb Z/\ell \mathbb Z$, with
	$\ell$ an odd prime.
When $\dim V = 2s$ is even, 
\begin{align*}
	G_{\on{RSel}^B_V} &= G_{\on{RSel}^C_V}, \\
G_{\on{RSel}^\Omega_V} &= G_{\on{RSel}^A_V} + \frac{1}{\#
	\Omega(Q)}\prod_{i=0}^{s-1} (t^2 -\ell^{2i}).
\end{align*}
When $\dim V = 2s + 1$ is odd, 
\begin{align*}
	G_{\on{RSel}^B_V} &= G_{\on{RSel}^C_V} + \frac{2 (-1)^{\frac{\ell-1}{2}}\ell^s}{\#
	\Omega(Q)} \prod_{i=1}^{s-1} (t^2-
	\ell^{2i}), \\
	G_{\on{RSel}^\Omega_V} &= G_{\on{RSel}^A_V} + \frac{t}{\#
	\Omega(Q)}\prod_{i=0}^{s-1} (t^2 -\ell^{2i}).
\end{align*}
\end{lemma}
\begin{proof}
	For the proof when $\dim V = 2s$, note that \cite[Lemma
	4.7]{fengLR:geometric-distribution-of-selmer-groups} easily generalizes
	to show that for any coset $H$ of $\Omega(Q)$ in $\o(Q)$,
	$G_{\on{RSel}^H_V}(\ell^i) = G_{\on{RSel}^\Omega_V}(\ell^i)$ whenever
	$2i +2 \leq \dim V$.
	When $\dim V$ is even, the proof proceeds mutatis mutandis as in 
	\cite[Theorem 4.4]{fengLR:geometric-distribution-of-selmer-groups}.

	Therefore, it remains to prove the case that $\dim V = 2s + 1$ is odd.
	We again proceed following the proof strategy of 
	\cite[Theorem 4.4]{fengLR:geometric-distribution-of-selmer-groups}.
	By \autoref{lemma:1-eigenvalue-parity},
	only even powers of $t$ can appear in $G_{\on{RSel}^B_V}(t)$ and
	$G_{\on{RSel}^C_V}(t)$.
	These are therefore even polynomials of degree at most $\dim V$ and agree
	at the $\dim V-1$ values $\pm 1, \pm \ell, \ldots, \pm \ell^{\frac{\dim V -
	3}{2}}$ by \cite[Lemma
	4.5]{fengLR:geometric-distribution-of-selmer-groups}. 
	Since $\dim V$ is odd and the polynomials are even, the polynomials in fact have degree at most $\dim
	V - 1$, and hence are determined up to a scalar.
	That is, $G_{\on{RSel}^B_V}(t) - G_{\on{RSel}^C_V}(t)$ is a
	scalar multiple of $\prod_{i=1}^{\frac{\dim V - 3}{2}} (t^2-
	\ell^{2i})$.
	To pin that scalar multiple down, we can examine the coefficient of
	$t^{\dim V - 1}$ in $G_{\on{RSel}^\bullet_V}(t)$, for $\bullet \in \{B,
	C\}$. This coefficient is
	$\frac{\# R_\bullet(Q)}{\#\Omega(Q)}$,
	where $R_\bullet(Q)$ is the set of reflections in $\bullet$, since any non-identity element of the orthogonal group fixing a
	codimension $1$ plane is a reflection.
	Since there are $\ell^{2s}+\ell^s$ reflections with value $\alpha$ for any
	square $\alpha \in \mathbb F_\ell^\times$, and 
	$\ell^{2s}-\ell^s$ reflections with value $\beta$ for any
	for any nonsquare $\beta \in \mathbb F_\ell^\times$, the definition of
	$\on{sp}^-_Q$ yields
	that 
	\begin{align*}
	G_{\on{RSel}^B_V} - G_{\on{RSel}^C_V} = \frac{2(-1)^{\frac{\ell-1}{2}}\ell^s}{\#
	\Omega(Q)} \prod_{i=1}^{\frac{\dim V - 3}{2}} (t^2-
	\ell^{2i}) =
\frac{2(-1)^{\frac{\ell-1}{2}}\ell^s}{\#
	\Omega(Q)} \prod_{i=1}^{s-1} (t^2-
	\ell^{2i}).
	\end{align*}

	Finally, the remaining two cosets satisfy the relation
	$G_{\on{RSel}^\Omega_V} = G_{\on{RSel}^A_V} + \frac{1}{\#
	\Omega(Q)}\prod_{i=0}^{s-1} (t^2 -\ell^{2i})$ by an argument analogous
	to the last paragraph of the proof of 
	\cite[Theorem 4.4]{fengLR:geometric-distribution-of-selmer-groups}:
	Indeed, 
	$G_{\on{RSel}^B_V}(t)$ and $G_{\on{RSel}^C_V}(t)$ are two odd degree
	$\dim V$ polynomials agreeing on the $\dim V$ values $0, \pm 1, \pm \ell,
	\ldots, \pm \ell^{\frac{\dim V - 3}{2}}$, so their difference is
	divisible by $t\prod_{i=1}^{\frac{\dim V - 3}{2}} (t^2- \ell^{2i})$, and
	the constant of proportionality can be determined using that the
	identity is the only element with a $\dim V$ dimensional fixed space.
\end{proof}

We next define a notion of $m$-total variation distance, which will be useful
for proving moments of two distributions converge, see
\autoref{remark:tv-convergence}. 
\begin{definition}
	\label{definition:}
	Let $\mathcal N$ denote the set of isomorphism classes of finite
	$\mathbb Z/\nu \mathbb Z$-modules.
	Let $X, Y$ be two $\mathcal N$ valued random variables. For $m \in \mathbb Z_{\geq 0}$,
we define the {\em $m$-total variation distance}
or $d^m_{\on{TV}}(X,Y)$
\begin{align*}
d^m_{\on{TV}}(X, Y) := \sum_{H \in \mathcal N} (\#H)^m \left | \on{Prob}(X=H)
	- \on{Prob}(Y=H) \right |.
\end{align*}
\end{definition}
\begin{remark}
	\label{remark:}
	When $m = 0$, and the random variable is real valued instead of valued in
$\mathcal N$, this is twice the usual notion of total variation distance, 
see \cite[\S4.1 and Proposition 4.2]{levinPW:markov-chains-and-mixing-times}.
We claim that a sequence of random variables $(X_n)_{n \geq 0}$
converges to $Y$ in distribution if the total variation distance between $X_n$
and $Y$ tends to $0$ in $n$: Indeed, convergence in distribution simply means
pointwise convergence for distributions on a discrete probability space.
\end{remark}

\begin{remark}
	\label{remark:tv-convergence}
	The point of the
definition of $m$-total variation distance is that if a sequence of random variables $X_n$ converges to $Y$ in
$m$-total variation distance then the $m$th moment of $X_n$ converges to the
$m$th moment of $Y$.
This follows directly from the definition of $m$-total variation distance.
\end{remark}

With the above definition in hand, we are prepared to show the distribution of $1$-eigenspaces of random orthogonal group
matrices converges in a strong sense to the BKLPR distribution, as the size of
the matrix grows.

\begin{lemma}
	\label{lemma:kernel-and-bklpr-distributions}
	Let $(V_{\nu,n}, Q_{\nu,n})_{n \in \mathbb Z_{> 0}}$ be a sequence of nondegenerate
	quadratic spaces 
	over $\mathbb Z/\nu \mathbb Z$, for $\nu$ odd.
	Suppose $\rk V_{\nu,n} \geq n$.
\begin{enumerate}
	\item Suppose $G_{\nu,n} \subset \o(Q_{\nu,n})$ is a subgroup containing
	$\Omega(Q_{\nu,n})$ and not contained in $\so(Q_{\nu,n})$.
	Let $R_{\nu,n}$ denote the distribution of $\ker(g - \id)$ for $g
	\in G_{\nu,n}$ a uniform random element.

	For any $m \in \mathbb Z_{\geq 0}$, the limit $\lim_{n \to \infty} R_{\nu,n}$ converges in $m$-total variation
	distance 
	to a distribution
	which agrees with $\bklpr \nu$.
\item Suppose $G_{\nu,n} \subset \so(Q_{\nu,n})$ is a subgroup containing
	$\Omega(Q_{\nu,n})$.
	Let $R_{\nu,n}^{\rk}$ denote the distribution of $\ker(g - \id)$ for $g
	\in G_{\nu,n}$ a uniform random element.
	Suppose that $\rk V_{\nu,n} \bmod 2$ is independent of $n$.
	For any $m \in \mathbb Z_{\geq 0}$, the limit $\lim_{n \to \infty}
	R_{\nu,n}^{\rk}$ converges in $m$-total variation
	distance 
	to a distribution
	which agrees with $\paritybklpr \nu {\rk V_{\nu,n} \bmod 2}$, which is
	independent of $n$ by assumption.
\end{enumerate}
\end{lemma}
\begin{proof}[Proof sketch]
	We start by verifying $(1)$.
	The argument closely follows
	\cite[Theorem 6.4]{fengLR:geometric-distribution-of-selmer-groups}.
	We now provide some more details on the changes one must make.

	We first claim the result holds when $\nu = \ell$ is an odd prime. 
	For $s \mid \nu$, we use $Q_{s,n}$ and $R_{s,n}$ for the reduction mod $s$
	of $Q_{\nu,n}$ and $R_{\nu,n}$.
	For $f(n)$ and $g(n)$ two functions of $n$, we write $f(n) \ll g(n)$ if
	there are constants $C_1, C_2$ so that $f(n) < C_1 g(n)$ for all $n >
	C_2$.
	As an initial step in our argument, we next verify in 
	\autoref{lemma:tv-bound-random-to-bklpr} that
	when $\nu = \ell$ is prime, $
	d_{\on{TV}}^m(R_{\ell,n}, \bklpr \ell) \ll
	\ell^{-((n/2)^2-\varepsilon)}$.
	\begin{lemma}
		\label{lemma:tv-bound-random-to-bklpr}
		With notation as in
		\autoref{lemma:kernel-and-bklpr-distributions},
		for $\ell$ an odd prime, 
		\begin{align*}
	d_{\on{TV}}^m(R_{\ell,n}, \bklpr \ell) \ll
	\ell^{-((n/2)^2-\varepsilon)}.
		\end{align*}
		\end{lemma}
	\begin{proof}
	For $G$ a finite group, we use $R^G$ to denote the distribution the
	dimension of the $1$-eigenspace of a uniformly random element of $G$.
	We can first bound
	$d^m_{\on{TV}}(R_{\ell,n},R^{\o_{\dim V}})$, where we use $\o_{\dim
	V}$ to denote the orthogonal group over a finite field $\mathbb Z/\ell
	\mathbb Z$ of
	dimension $\dim V$, which $G_{\ell,n}$ is a subset of. 
	Note, by 
	convention, $n \leq \dim V$.
	The proof of this bound on $m$-total variation distance is quite similar to that of
	\cite[Theorem 4.23]{fengLR:geometric-distribution-of-selmer-groups},
	except that we replace the input of \cite[Theorem
	4.4]{fengLR:geometric-distribution-of-selmer-groups}
	with that of
	\autoref{lemma:coset-generating-functions}, and note that since these
	probability distributions are both supported on $\{0, \ldots, \dim V\}$,
	$d^m_{\on{TV}}(R_{\ell,n}, R^{\o_{\dim V}}) \leq (\dim V)^m \cdot
	d^0_{\on{TV}}(R_{\ell,n}, R^{\o_{\dim V}})$. Now,
$d^0_{\on{TV}}(R_{\ell,n}, R^{\o_{\dim V}})$
	was
	shown to be $\ll \ell^{-(\frac{\dim V}{2})^2}$ in 
	\cite[Theorem  4.23]{fengLR:geometric-distribution-of-selmer-groups}
	when $\dim V$ is even dimensional with discriminant $1$, and, as
	mentioned, an analogous proof applies here.
	We conclude that $d^m_{\on{TV}}(R_{\ell,n}, R^{\o_{\dim V}}) \ll
	\ell^{-((n/2)^2-\varepsilon)}$.

	Hence, to show $d^m_{\on{TV}}(R_{\ell,n}, \bklpr \ell) \ll
	\ell^{-((n/2)^2-\varepsilon)}$,
	it suffices to bound $d^m_{\on{TV}}(R^{\o_{\dim V}},\bklpr \ell) \ll\ell^{-(\lfloor \dim V/2 \rfloor)^2}$.
	To this end,
	let $2s$ denote the smallest even integer with $2s \leq \dim V$.
	We use 
	$\o^+(2s,\ell)$ to denote the discriminant $1$ orthogonal
	group over $\mathbb F_\ell$ of rank $2s$.
	The formulas in \cite[Theorem 2.7 and
	2.9]{fulmanS:distribution-number-fixed},
	which give the dimension of fixed spaces of elements of orthogonal
	groups,
	show 
	\begin{align}
&d^m_{\on{TV}}(R^{\o_{\dim V}},R^{\o^+(2s,\ell)}) \\
\label{equation:two-sums}
&\leq
	\sum_{k=0}^s (2k)^m \ell^{-(2ks+s^2- k^2 + (s-k))} + \sum_{k=0}^s
	(2k+1)^m\ell^{-(2ks+s^2 - k^2 +(s-k))}
	\\
	&\ll \ell^{-s^2}.
	\end{align}
	The first sum in \eqref{equation:two-sums} is accounted for by the second line of 
	\cite[Theorem 2.9(1)]{fulmanS:distribution-number-fixed}
	(and this is the only one that appears in the case $\dim V$ is even)
	and the second sum is accounted for by the $i=n-k$ term in the sum
	appearing in 
	\cite[Theorem 2.7(2)]{fulmanS:distribution-number-fixed}.

	To conclude the bound 
	$
	d_{\on{TV}}^m(R_{\ell,n}, \bklpr \ell) \ll
	\ell^{-((n/2)^2-\varepsilon)}$, it remains to bound
	\begin{align*}
	d^m_{\on{TV}}(R^{\o^+(2s,\ell)}, \bklpr \ell) \ll \ell^{-s^2}.
	\end{align*}
	This was essentially done in the last paragraph of the proof of 
	\cite[Theorem 4.23]{fengLR:geometric-distribution-of-selmer-groups}
	combined with 
	\cite[Corollary 4.24]{fengLR:geometric-distribution-of-selmer-groups},
	and we now give a slightly more direct argument.
First, $d^m_{\on{TV}}(R^{\o^+(2s,\ell)}, R^{\o^+(2s+2,\ell)}) \ll \ell^{-s^2}$, using the formulas in \cite[Theorem
	2.9]{fulmanS:distribution-number-fixed}, similarly to the preceding
	paragraph.
	This implies that
	$d^m_{\on{TV}}(R^{\o^+(2s,\ell)}, \lim_{s \to \infty} R^{\o^+(2s,\ell)})
	\ll \ell^{-s^2}$.
	An explicit formula for this limiting distribution is given in
	\cite[Theorem 2.9(3)]{fulmanS:distribution-number-fixed}.
	Note that in the case where $\ell$ is prime, which we are currently considering, the ``BKLPR heuristic'' first appeared as the ``Poonen-Rains heuristic'' \cite{poonenR:random-maximal-isotropic-subspaces-and-selmer-groups}, whose explicit formula is given by \cite[Conjecture 1.1(a)]{poonenR:random-maximal-isotropic-subspaces-and-selmer-groups}. By inspection, this agrees with the distribution 
	appearing in 
	\cite[Theorem 2.9(3)]{fulmanS:distribution-number-fixed}, yielding our
	claim that
	$d^m_{\on{TV}}(R_{\ell,n}, \bklpr \ell) \ll \ell^{-((n/2)^2-\varepsilon)}$.
	\end{proof}
	
	Proceeding with the proof of
	\autoref{lemma:kernel-and-bklpr-distributions},
	we next explain why the Markov properties established in \cite[Theorem 5.1 and Theorem
	5.13]{fengLR:geometric-distribution-of-selmer-groups} for the
	$R_{\ell^j,n}$ and $\bklpr {\ell^j}$ imply that we also obtain
	convergence in $m$-total variation distance
	$\lim_{n \to \infty} R_{\ell^j,n} \to \bklpr {\ell^j}$.
	Technically, \cite[Theorem 5.1]{fengLR:geometric-distribution-of-selmer-groups}
	is only stated in the case the quadratic space has even rank. However,
	the proof for $\ell$ odd does not use the assumption that the rank is
	even.
	Although the BKLPR distribution only varies over even dimensional vector
	spaces, we have showed above that 
	$\lim_{n \to \infty} R_{\ell,n} = \bklpr {\ell}$.
	Since both distributions satisfy the same Markov
	property relating the $\mod \ell^j$ and the $\mod \ell^{j-1}$ versions, 
	the $m$-total variation distance also tends to $0$ between the
	$\bmod \ell^j$ distributions, and so
	$\lim_{n \to \infty} R_{\ell^j,n} \to \bklpr {\ell^j}$
	in $m$-total variation distance.

	To obtain the case of general $\nu$, write $\nu = \prod_{\ell}
	\ell^{a_\ell}$.
	The various distributions $\bklpr {\ell^{a_\ell}}$ are not in general
	independent, but they are independent after conditioning on the parity
	of the rank of their reduction $\bmod \ell$.
	Similarly, the distributions $R_{\ell^{a_\ell},n}$ are not independent,
	but they are independent after conditioning on the value of the
	coset of $\Omega(Q_{\nu,n})$ in $G_{\nu,n}$,
	as
	$\Omega(Q_{\nu, n}) = \prod_{\text{prime }
	\ell \mid \nu} \Omega(Q_{\ell^{a_\ell},n})$.
	We therefore obtain that the distribution
	of any specified coset of
	$\Omega(Q_{\nu,n})$ with specified
	value of $D_{Q_{\nu,n}}$ approaches the distribution
	$\bklpr {\ell^{a_\ell}}$, conditioned on
	the parity of the rank as $n \to \infty$,
	in $m$-total variation distance.
	Summing over different cosets on both sides gives the claimed
	convergence in $m$-total variation distance
	$\lim_{n \to \infty} R_{\nu,n}\to\bklpr \nu$.

	To conclude, it remains to deal with $(2)$.
	This is completely analogous to the proof of $(1)$, but where one
	compares distributions to random kernels of special orthogonal groups at
	each step.
	The distribution of $\dim \ker(g - \id)$ for $g \in \so(Q)$, for $(V,Q)$
	over $\mathbb F_\ell$ can be deduced from the distribution over $g \in
	\o(Q)$ using \autoref{lemma:1-eigenvalue-parity}.
	Namely, \autoref{lemma:1-eigenvalue-parity} shows that $\dim \ker(g -
	\id) \equiv \dim V \bmod 2$ for $g \in \so(Q)$. Since of $\o(Q)$ elements are equally likely to lie
	in $\so(Q)$ and $\o(Q) - \so(Q)$, we find 
	\begin{align*}
		\prob(\dim \ker(g - \id) = s | g \in
	\o(Q)) = \frac{1}{2} \prob(\dim \ker(g - \id) = s | g \in
	\so(Q))	
\end{align*}
when $s \equiv \dim V \bmod 2.$

	One can then obtain analogous asymptotic bounds on
	$d_{\on{TV}}^m(R^{\rk}_{\ell,n},\paritybklpr {\ell} {\dim V
	\bmod 2})$ to those proven in \autoref{lemma:tv-bound-random-to-bklpr},
	using
	these explicit formulas. Next one can use the Markov property to obtain
	analogous bounds on $d_{\on{TV}}^m(R^{\rk}_{\ell^j,n},\paritybklpr
		{\ell^j} {\dim V_{\nu,n}
	\bmod 2})$.
	Finally, one can use the Chinese remainder theorem to obtain analogous
	bounds on $d_{\on{TV}}^m(R^{\rk}_{\nu,n},\paritybklpr {\nu} {\dim
			V_{\nu,n}
	\bmod 2})$.
\end{proof}

\subsection{Connected components of Selmer stacks}
\label{subsection:connected-components}

We are now ready to prove the key input to a ``$q \to \infty$ first, then $n \to \infty$`` version
of our main result, which amounts to counting connected components of Selmer
stacks. 

In \autoref{proposition:geometric-moments}, we
combine the above to compute the number of components of Selmer stacks.
To compute this number of connected components, we will combine our big
monodromy result from \autoref{proposition:big-monodromy-composite} with the
convergence result of \autoref{lemma:kernel-and-bklpr-distributions} to deduce
that the number of components agrees with moments of the BKLPR distribution.
Following this, in \autoref{theorem:point-counting-computation} we combine the above
with our main homological stability theorem to compute the moments of Selmer
groups in quadratic twist families.

\begin{proposition}
	\label{proposition:geometric-moments}
	Maintain hypotheses as in \autoref{notation:quadratic-twist-notation},
	\autoref{hypotheses:big-monodromy-assumptions}, 
	\autoref{notation:selmer-space-moment-notation},
and \autoref{example:abelian-special-fiber},
	so that $\mathscr F_b \simeq A[\nu]$.
	Take $b = \on{Spec} \mathbb F_q \in B$ a closed point, and suppose the
	bound on $n$ from \eqref{equation:bound-on-n} is satisfied.
	\begin{enumerate}
		\item 
Every connected component of $\selspacemoments {\mathscr F^n_b} H$ is geometrically connected
	and the number of such connected components is equal to $\mathbb E(\#
	\hom(\bklpr \nu, H))$
	for $n$ sufficiently large, depending on $H$. 
\item 
Every connected component of $\rankselspacemoments {\mathscr F^n_b} H$ is geometrically connected
	and the number of such connected components is equal to $\mathbb E(\#
	\hom(\paritybklpr \nu {\rk V_{\mathscr F^n_b} \bmod 2}, H))$
	for $n$ sufficiently large, depending on $H$. 
	\end{enumerate}
\end{proposition}

\begin{remark}
	\label{remark:}
	There has been much recent work,
	notably
	\cite{lipnowskiST:cohen-lenstra-roots-of-unity} and
	\cite{sawinW:conjectures-for-distributions-containing-roots-of-unity},
	studying versions of the Cohen-Lenstra heuristics in the presence of
	roots of unity.
	When working over function fields, the difference in behavior of the Cohen-Lenstra heuristics
	when the base field has certain roots of unity, 
	can be traced back to 
	a certain moduli space whose connected components are not all
	geometrically connected.
	However, in the context of the BKLPR heuristics,
	\autoref{proposition:geometric-moments} shows the connected components
	are always geometrically connected. This explains why the BKLPR
	heuristics are not sensitive to roots of unity in the base field.
\end{remark}
\begin{remark}
	\label{remark:}
	We note that \autoref{proposition:geometric-moments} is quite closely related to the main results of
\cite{parkW:average-selmer-rank-in-quadratic-twist-families}. 
Although it is not exactly stated in this language, it follows from the
Lang-Weil bounds that they prove
a version of \autoref{proposition:geometric-moments}
in the special case that $H$ is of the form $\mathbb Z/\ell \mathbb Z$ for $\ell
\geq 5$ a prime, and $A$ an elliptic curve.
Both of our proofs follow a similar approach, and their proof is essentially a
special case of ours.
\end{remark}

\begin{proof}
	As a first step, note that the monodromy representation $D_{Q_{\mathscr
	F^n_b}}\circ \rho_{\mathscr F^n_b}$ surjects onto the diagonal copy of $\mathbb
	Z/2 \mathbb Z$ by \autoref{lemma:determinant-monodromy}.
	We first deal with case $(1)$.
	Let $\overline b$ denote a geometric point over $b$.
	Take $G_{\nu,n}$ to be the arithmetic monodromy group at $b$,
		$\im \rho_{\mathscr F^n_b}$.

	This is a union of cosets of the geometric monodromy 
	$\im \rho_{\mathscr F^n_{\overline b}}$ in the orthogonal group,
	so is not contained in the special orthogonal group by
	\autoref{proposition:big-monodromy-composite}, as we are assuming $n$
	satisfies the bound of \eqref{equation:bound-on-n}.
	Therefore, $G_{\nu,n}$
	satisfies the hypotheses of
	\autoref{lemma:kernel-and-bklpr-distributions}(1).
 Let $R_{\nu,n}$ denote the distribution of $\ker(g - \id)$ for $g
	\in G_{\nu,n}$ a uniform random element.
	In what follows,
	we will show 
	$\mathbb E(\# \hom(R_{\nu,n}, H))$ agrees with the number of connected
	components of $\selspacemoments {\mathscr F^n_b} H$.
	Granting this, and
	using \autoref{lemma:kernel-and-bklpr-distributions}, which shows that
	the
	$R_{\nu,n}$ converge in $m$-total variation distance to $\bklpr \nu$,
	we find 
	$\lim_{n \to \infty} \mathbb E(\# \hom(R_{\nu,n}, H))$ converges to
	$\mathbb E(\# \hom(\bklpr \nu, H))$, whenever $H$ is a free $\mathbb
	Z/\nu \mathbb Z$-module of rank $m$.

	Having shown the desired convergence for free $H$, we claim that the general case that $H$ is a $\mathbb Z/\nu \mathbb Z$-module with $m$
	generators follows from the case that $H$ is a free module with $m$
	generators. 
	Indeed, it suffices to show the 
	postulation that
	homomorphisms to such $H$ form a subset of homomorphisms to $(\mathbb
	Z/\nu \mathbb Z)^m$. For this choose an injection $H \to \left( \mathbb
		Z/\nu \mathbb Z
	\right)^m$. For any finite group $K$, $\hom(K, H) \hookrightarrow
	\hom(K, \left( \mathbb Z/\nu \mathbb Z \right)^m)$ is injective. Hence
	we obtain the postulation, and therefore the claim.

	It remains to show $\mathbb E(\# \hom(R_{\nu,n}, H))$ agrees with the number of connected
	components of $\selspacemoments {\mathscr F^n_b} H$, all of which are
	geometrically connected.
	This follows from a standard monodromy argument and Burnside's lemma, as we now explain.
	The action of $G_{\nu,n}$ on $V_{\mathscr F^n_b}$ is via the
	standard representation of the orthogonal group on its underlying vector
	space.
	Let $H = \prod_{i=1}^m \mathbb Z/\nu_i \mathbb Z$.
	Then, the action
	$\phi_{\mathscr F^n_b, H} : G_{\nu,n} \to \aut\left(\prod_{i=1}^m V_{\mathscr
	F[\nu_i]^n_b}\right)$
	is via the diagonal action of the orthogonal group on $\prod_{i=1}^m
	V_{\mathscr F[\nu_i]^n_b}$:
	$\phi_{\mathscr F^n_b,H}(g)(v_1, \ldots, v_m) = (gv_1, \ldots,
	gv_m)$, where $g \in G_{\nu,n}$, $v_i \in V_{\mathscr F[\nu_i]^n_b}$, and $gv_i$ denotes the
	standard action of an element of an orthogonal group on its
	underlying free module.
	Hence, the number of connected components of $\selspacemoments
	{\mathscr F^n_b} H$
	is equal to the number of orbits of $G_{\nu,n}$ on $\prod_{i=1}^m
	V_{\mathscr F[\nu_i]^n_b}$
	under the above diagonal action $\phi_{\mathscr F^n_b,H}$.
	Now, using Burnside's lemma, this number of orbits is equal to
	$\frac{1}{\# G_{\nu,n}} \sum_{g \in G_{\nu,n}} \#
	\ker(\phi_{\mathscr F^n_b, H} (g) - \id)$. 
	Noting that an element in $\ker(\phi_{\mathscr F^n_b, H} (g) - \id)$
	is a tuple $(v_1, \ldots, v_m)$ so that $g v_i = v_i$ and $\nu_i v_i =
	0$,
	we can identify $\ker(\phi_{\mathscr
			F^n_b,
	H} (g) - \id) \simeq \hom(H,\ker\phi_{\mathscr F^n_b, \mathbb Z/\nu
	\mathbb Z}(g) - \id)$. Hence, 
	\begin{align*}
	\frac{1}{\# G_{\nu,n}} \sum_{g \in G_{\nu,n}} \#
	\ker(\phi_{\mathscr F^n_b, H} (g) - \id) 
	&=
	\frac{1}{\# G_{\nu,n}} \sum_{g \in G_{\nu,n}} \#\hom(H,\ker\phi_{\mathscr
			F^n_b, \mathbb Z/\nu
	\mathbb Z}(g) - \id) \\
	&=
	\frac{1}{\# G_{\nu,n}} \sum_{g \in G_{\nu,n}} \#\hom(\ker\phi_{\mathscr
			F^n_b, \mathbb Z/\nu
	\mathbb Z}(g) - \id, H) \\
	&= \mathbb E(\#
	\hom(R_{\nu,n}, H)).
	\end{align*}
	
	The same argument as above goes through if one replaces $G_{\ell,n}$ with the
	geometric monodromy group. This shows the number of
	components over $\overline{\mathbb F}_q$ is also
	$\mathbb E(\# \hom(\bklpr \nu, H))$ for $n$ sufficiently large,
	and so the number of components over $\overline{\mathbb F}_q$
	agrees with the number of
	connected components over $\mathbb F_q$. Therefore, every connected component is geometrically connected.

	To conclude, it remains to deal with case $(2)$.
	We note that 
	$\rk V_{\mathscr F^n_b} \bmod 2$ is independent of $n$, as follows from
	the formula in \autoref{proposition:rank-description}.
	This is completely analogous to $(1)$, but one uses
	\autoref{lemma:kernel-and-bklpr-distributions}(2) in place of
	\autoref{lemma:kernel-and-bklpr-distributions}(1),
	and therefore as output obtains 
	the number of components agrees with
	$\mathbb E(\# \hom(\paritybklpr \nu {\rk V_{\mathscr F^n_b}}, H))$
	instead of
	$\mathbb E(\# \hom(\bklpr \nu, H))$.
\end{proof}

Using the above computation of the connected components of our space, we are
able to combine it with our topological tools, the Grothendieck-Lefschetz
trace formula, and Deligne's bounds to deduce the $H$-moments of the distribution of Selmer groups in
quadratic twist families.

\begin{theorem}
	\label{theorem:point-counting-computation}
		Suppose $B = \spec R$ for $R$ a DVR of generic characteristic $0$
		with
	closed point $b$ with residue field $\mathbb F_{q_0}$ and geometric
	point $\overline b$ over $b$.
	Keep hypotheses as in \autoref{hypotheses:big-monodromy-assumptions}:
	Namely, 
suppose $\nu$ is an odd integer and $r \in \mathbb Z_{>0}$ so that every prime $\ell \mid \nu$
	satisfies $\ell > 2r + 1$.
	Let $C$ a smooth proper curve with geometrically connected fibers over $B$,
	$Z \subset C$ finite \'etale nonempty over $B$, and $U := C - Z$.
	Let $\mathscr F$ be a rank $2r$, tame, locally constant constructible, symplectically self-dual sheaf
	of free $\mathbb Z/\nu \mathbb Z$-modules over $U$.
	We assume there is some point $x \in C_{\overline b}$ at which
	$\drop_x(\mathscr F_{\overline b}[\ell]) = 1$ for every prime $\ell \mid \nu$.
	Also suppose
	$\mathscr F_{\overline b}[\ell]$ is irreducible
	for each $\ell \mid \nu$, and that 
	the map $j_* \mathscr F_{\overline b}[\ell^w] \to j_* \mathscr F_{\overline b}[\ell^{w-t}]$ is
	surjective for each prime $\ell \mid \nu$ such that $\ell^w \mid \nu$, and
	$w \geq t$.
Fix $A \to U_b$ as in \autoref{example:abelian-special-fiber}
and suppose the tame irreducible locally constant constructible symplectically self-dual sheaf $\mathscr F$
satisfies
$\mathscr F_b \simeq A[\nu]$.
	For any finite $\mathbb Z/\nu \mathbb Z$-module $H$, 
	and any finite field extension $\mathbb F_{q_0} \subset \mathbb F_q$,
	there are constants $C(H, \mathscr F)$ depending on $H$ and $\mathscr F$, but not on $q$
	or $n$, so that
	\begin{align}
		\label{equation:full-moments}
	\left|\frac{\# \selspacemoments {\mathscr F^n_B} H (\mathbb
	F_q)}{q^{n}} - \# \hom(\bklpr
	\nu ,H) \right| &\leq \frac{C(H, \mathscr F)}{\sqrt{q}} \\
	\label{equation:parity-moments}
		\left |\frac{\# \rankselspacemoments {\mathscr F^n_B}H (\mathbb F_q)}{
	q^{n}} - \# \hom(\paritybklpr
	\nu {\rk V_{\mathscr F^n_B} \bmod 2},H)\right| &\leq \frac{C(H, \mathscr F)}{\sqrt{q}}
	\end{align}
	for all even $n > C(H, \mathscr F)$, and all $q$ with $\sqrt{q} > C(H, \mathscr F)$.

	Moreover, suppose there is a point $\sigma \in Z(B)$ over which
	$\mathscr F$ has trivial inertia.
	There are 
	functions $f_{H,\mathscr F}(q)$,
	and
	positive constants 
	$I(H)$,
	$C(H, \mathscr F)$, and
	$J(\mathscr F, H)$ so that
	\begin{align}
		\label{equation:full-moments-stable}
	\left|\frac{\# \selspacemoments {\mathscr F^n_B} H (\mathbb
	F_q)}{q^{n}} - f_{H,\mathscr F}(q) \right| &\leq  \left(\frac{C(H,
\mathscr F)}{\sqrt{q}}\right)^{\frac{n-J(\mathscr F, H)}{2I(H)}} 
	\end{align}
	for all even $n > C(H, \mathscr F)$, and all $q$ with $\sqrt{q} > 2 C(H, \mathscr F)$.
\end{theorem}
\begin{proof}
	This follows
	from preceding results in our paper, together with the 
	Grothendieck-Lefschetz trace formula and Deligne's bounds, much in
	the same way that \cite[Theorem 8.8]{EllenbergVW:cohenLenstra} follows
	from \cite[Proposition 7.8]{EllenbergVW:cohenLenstra}. 
	The remainder of the proof is somewhat standard, but we spell out the
	details for completeness.
	
	We first explain \eqref{equation:full-moments} and
	\eqref{equation:parity-moments}.
Fix a point $b \in B$ with residue field $\mathbb F_q$ with geometric point
$\overline b$ over $b$.
Let $(Y_n)_{n \geq 1}$ be a sequence of stacks over $B$ which is either
either a sequence of the form $(\selspacemoments {\mathscr F^n_{B}} H)_{n \geq 1}$
	or $(\rankselspacemoments {\mathscr F^n_{B}} H)_{n \geq 1}$.
	Define the sequence $(W_n)_{n \geq 1}$ to be
	$W_n := (Y_n)_{\mathbb C}$, for some map $\spec \mathbb C \to B$.

	We next bound the cohomology groups of the geometric fiber of $Y_n$ over
	$\overline{b}$, via
	comparison to the cohomology of $W_n$.
Note that the $Y_n$ have coarse spaces which are finite \'etale covers of $\conf n U B$.
Note that there is a normal crossings
compactification of 
$\conf n U B$ by \autoref{corollary:hurwitz-compactification}.
It follows from \cite[Proposition 7.7]{EllenbergVW:cohenLenstra}
that the geometric generic fiber of $Y_n$ over $B$ has isomorphic cohomology to the
geometric special fiber of $Y_n$ over $B$.
Now, we will choose $\ell'$ to be a sufficiently large prime, which may even depend on
$n$. We will see in the course of the proof how large $\ell'$ needs to be. 
(It is enough to take $\ell'$ to be prime to 
$q, n!, \# \asp_{2r}(\mathbb Z/\nu \mathbb Z),$ and $2$.)
In other words,
if we use $X_n := (Y_n)_{\overline{\mathbb F}_q}$ for the geometric special
fiber, we obtain
$H^i(X_n, \mathbb Z/\ell' \mathbb Z) \simeq H^i(W_n, \mathbb Z/\ell' \mathbb
Z)$.
By \autoref{lemma:moments-cohomology},
the latter has dimension bounded by $K^{i+1}$, for some constant $K$ independent
of $n$.
Note that
$\dim H^i(X_n, \mathbb Z/\ell' \mathbb Z) \geq \rk H^i(X_n, \mathbb Z_{\ell'}) \geq
\dim H^i(X_n, \mathbb Q_{\ell'})$, so we also have that 
$H^i(X_n, \mathbb Q_{\ell'})$ is bounded by $K^{i+1}$.

Since $Y_n$ is a finite \'etale cover of
the smooth Deligne-Mumford stack $\qtwist n U b$, every connected component
is smooth and hence irreducible.
Let $Z_n$ denote the number of connected components of $X_n$.
Since all the connected components of $X_n$ are base changed from $\mathbb F_q$, by \autoref{proposition:geometric-moments}, proving
\eqref{equation:full-moments} and \eqref{equation:parity-moments} amounts to
proving
	\begin{align*}
		\left| \frac{\#Y_n(\mathbb F_q)}{q^{\dim X_n}} - Z_n\right| \leq
		\frac{C}{\sqrt{q}}
	\end{align*}
	where $C$ is a constant depending on the sequence $(X_n)_{n \geq 1}$, but not the subscript
	$n$.

	Since $X_n$ is smooth of dimension $n$, using Poincar\'e duality,
	$\dim H_{\on{c}}^{2n-i}(X_n, \mathbb
	Q_{\ell'}) = \dim H^i(X_n, \mathbb Q_{\ell'})$.
	We may then produce a
	a constant $D$, depending on the
	sequence $(X_n)_{n \geq 1}$, but not $n$, such that $\dim H_{\on{c}}^{2n-i}(X_n, \mathbb
	Q_{\ell'}) = \dim H^i(X_n, \mathbb Q_{\ell'}) \leq D^i$.
	For example, we can take $D = K^2$.

	Since every eigenvalue of geometric Frobenius $\frob_q$ acting on the compactly
	supported cohomology group $H_{\on{c}}^j(X_n, \mathbb Q_{\ell'})$ of the
	stack $X_n$ is bounded in absolute value by $q^{j/2}$, using Sun's
	generalization of Deligne's bounds to algebraic stacks
	\cite[Theorem 1.4]{sun:l-series-of-artin-stacks},
	we find
	\begin{equation}
		\label{equation:deligne-bound}
	\begin{aligned}
		&\left | q^{-\dim X_n} \sum_{j < 2 \dim X_n} (-1)^j \tr \left( \frob_q |
			H^j_{\on{c}}(X_n, \mathbb Q_{\ell'})\right) \right|
			\\
			&\leq q^{-\dim X_n} \sum_{j=0}^{2\dim X_n-1} q^{j/2} \dim H^j_c(X_n,
		\mathbb Q_{\ell'}) 
		\\
		&\leq q^{-\dim X_n} \sum_{j=0}^{2\dim X_n-1}q^{j/2} D^{2\dim X_n-j}
\\
		&\leq \sum_{k=1}^\infty \left( \frac{D}{\sqrt{q}}
	\right)^k.
	\end{aligned}
\end{equation}
	This is bounded by $2D/\sqrt{q}$ whenever $D/\sqrt{q} \leq 1/2$.
	Hence, taking $C :=2D$, we obtain
	\begin{align*}
		\left| q^{-\dim X_n} \sum_{j < 2\dim X_n} (-1)^j \tr \left( \frob_q |
			H^j_{\on{c}}(X_n, \mathbb Q_{\ell'}) 
	\right) \right| \leq \frac{C}{\sqrt{q}}
	\end{align*}
	whenever $C \leq \sqrt{q}$.
	Therefore, using the Grothendieck-Lefschetz trace formula, it is enough
	to show $\tr \left( \frob_q | H^{2 \dim X_n}_{\on{c}}(X_n, \mathbb
	Q_{\ell'})\right) = Z_n q^{\dim X_n}$
	for $n$ sufficiently large, say larger than some constant $C_1$. 	
	By Poincar\'e duality, this is equivalent to showing that there are $Z_n$
	connected components of $X_n$, all of which are defined over $\mathbb
	F_q$.
	Indeed, this was shown in 
	\autoref{proposition:geometric-moments}.
	Finally, we then take $C(H, \mathscr F)$ in the statement to be $\max(C,
	C_1)$, which proves \eqref{equation:full-moments} and
	\eqref{equation:parity-moments}.

	We conclude by briefly outlining how one may similarly obtain
	\eqref{equation:full-moments-stable} 
	by additionally using
	\autoref{theorem:frobenius-equivariance}.
	We assume $q$ is sufficiently large so that the hypotheses of
	\autoref{theorem:frobenius-equivariance} are satisfied; namely, that $\selspace {\mathscr G^{d,0,\infty}_{ B}}(\mathbb F_q)\neq
	\emptyset$. This will happen for all sufficiently large $q$ by the
	Lefschetz trace formula.
	We maintain the notation set up earlier in the proof.
	We use $J$ to denote $J(\mathscr F, H)$ from the theorem statement and
	$I$ to denote $I(H)$.
	By 
	\autoref{theorem:frobenius-equivariance},
	when $n > Ip + J$ is even,
	$\tr(\frob_q^{-1} | H^p(X_n, \mathbb Q_{\ell'}))$ takes on a
	value independent of $n$.
	Fixing $n$ even with $n > Ip + J$,
	let $t_{p}(q):=\tr(\frob_q^{-1} | H^p(X_n,
	\mathbb Q_{\ell'}))$,
	where as usual $\frob_q$ denotes geometric frobenius.
		Then define $f(q) := \sum_{j=0}^\infty (-1)^j
		t_{j}(q)$.
		(We use $f$ in place of the function 
			$f_{H,\mathscr F}$ 
			as in the theorem statement.)
	We claim that $f$ converges as a function in $q$ for $q$ sufficiently
	large.
	Indeed, \autoref{lemma:moments-cohomology} and Deligne's bounds on
	the eigenvalues of Frobenius acting on cohomology yield $|t_{p}(q)| < \frac{K^{p+1}}{q^{p/2}}$, and so $f(q)$ is bounded by a geometric
	series; see \eqref{equation:deligne-bound} and the surrounding
	paragraphs
	for a similar bounding argument which
	is spelled out in more detail.
	
	We now conclude \eqref{equation:full-moments-stable} 
	by applying the
	Grothendieck-Lefschetz trace formula.
	Note that the condition $n > Ip + J$ is equivalent to the condition $p <
	\frac{n-J}{I}$.
	Note here we are using that the Galois representation
	$H^{2n-i}_{\on{c}}(X_n, \mathbb Q_{\ell'})$ is identified with the
	Galois representation
	$H^{i}(X_n, \mathbb Q_{\ell'})^\vee(-n)$ via Poincar\'e duality, and so $q^n \cdot
	\tr(\frob_q^{-1} | H^{i}(X_n, \mathbb Q_{\ell'})) = \tr\left(
	\frob_q | H^{2n-i}_{\on{c}}(X_n, \mathbb Q_{\ell'}) \right)$.
	From this, it follows that
	$\frac{1}{q^n}\tr\left( \frob_q | H^{2n-j}_{\on{c}}(X_n, \mathbb Q_{\ell'})
	\right) = t_j(q)$.
	Using the above observation combined with the Grothendieck Lefschetz trace formula,
	the difference 
	$\left| \frac{\# X_n(\mathbb F_q)}{q^{n}} - f(q)
	\right|$ can be bounded by the sum of $\sum_{j=\frac{n-J}{I}}^\infty (-1)^j t_{j}(q)$
	and 
	\begin{align}
		\label{equation:unstable-small-cohomology-bound}
		\frac{1}{q^n} \sum_{j \leq 2n -\frac{n-J}{I}} (-1)^j \tr\left( \frob_q |
	H^j_{\on{c}}(X_n, \mathbb Q_{\ell'}) \right).
	\end{align}
	By a computation analogous to \eqref{equation:deligne-bound}, we can
	bound \eqref{equation:unstable-small-cohomology-bound} in absolute value by
	$\frac{C}{q^{\frac{n-J}{2I}}}$, for an appropriate constant $C$ not
	depending on $q$ or $n$, once $n$ is sufficiently large and $\sqrt{q} >
	2C$.
\end{proof}
\begin{remark}
	\label{remark:number-field}
Suppose one started with a setup as in
\autoref{theorem:point-counting-computation},
but where $B$ is a nonempty open in $\spec \mathscr O_K$, for $K$ a number field.
(Note that if one starts with this setup over $\spec K$, one can spread it out
to such a $B$.)
For any geometric point $\spec \overline{\mathbb F}_q \to B$, we
can identify the cohomology groups of the relevant moduli spaces (labeled $X_n$
in the proof of \autoref{theorem:point-counting-computation})
over $\spec \overline{\mathbb F}_q$
with the corresponding cohomology groups over the geometric generic point $\spec
\mathbb C \to B$,
(which are the cohomology of $W_n$ in the proof of
\autoref{theorem:point-counting-computation},)
independently of the choice of geometric point above.
Then, one could prove a result as in \autoref{theorem:point-counting-computation}, but with the limit in $q$ ranging
over primes of all but finitely many characteristics, instead of only
powers of a given prime power $q_0$.

\end{remark}

\begin{remark}
	\label{remark:independence-of-constants}
	Although the constants $C(H, \mathscr F)$ in
	\autoref{theorem:point-counting-computation}
	depend on $\mathscr F$ and $H$ as stated, they can in fact be chosen to be
	functions of $\nu$, the rank $2r$ of $\mathscr F$ and the degree $f+1$
	of $Z$, and the genus $g$ of $C$, as we next explain.

	One way to see this is via comparison to the complex numbers.
	Then, over the complex numbers, 
	the constants only depend on the topological type of the finite covering space
	associated to $\mathscr F$ over $U$.
	There are only finitely many such topological types once we fix
	$r, \nu$, and $f$,
	since the number of these types is bounded by the number of homomorphisms
	$\pi_1(\Sigma_{g,f+1}) \to \asp_{2r}(\mathbb Z/\nu \mathbb Z)$, of which
	there are only finitely many.
	Hence, the relevant constants $C(H, \mathscr F)$ can be taken to only depend on $r, \nu,
	f,g$, and $H$.
\end{remark}

\begin{remark}
	\label{remark:0-cohomology}
	Suppose the stable cohomology groups of spaces appearing in the proof of
	\autoref{theorem:point-counting-computation},
	which are not in the top degree,
	vanish.
	Then, via the Grothendieck-Lefschetz trace formula, one could deduce
	that the constants $C(H, \mathscr F)$ actually vanish.
	This would imply some of our main results, such as
\autoref{theorem:main-moments},
hold on the nose for fixed, sufficiently large $q$, depending on $H$, without the need for taking a large $q$
limit.
\end{remark}

\begin{remark}
	\label{remark:}
	It seems likely one could additionally find a function 
	$f_{H,\mathscr F}^{\on{rk}}(q)$
	as in the statement of \autoref{theorem:point-counting-computation}
	and
	positive constants 
	$I(H)$,
	$C(H, \mathscr F)$, and
	$J(\mathscr F, H)$ so that
	\begin{align}
		\label{equation:parity-moments-stable}
		\left |\frac{\# \rankselspacemoments {\mathscr F^n_B}H (\mathbb F_q)}{
	q^{n}} - f_{H,\mathscr F}^{\on{rk}}(q) \right| &\leq \left(\frac{C(H,
	\mathscr F)}{\sqrt{q}}\right)^{\frac{n-J(\mathscr F, H)}{2I(H)}}
	\end{align}
	for all even $n > C(H, \mathscr F)$, and all $q$ with $\sqrt{q} > 2 C(H, \mathscr F)$.
	For this, one would only need to generalize
	\autoref{theorem:frobenius-equivariance} to also work for the rank double
	cover.
	This seems quite doable, but we have opted not to carry it out as it was
	not required for our main theorems. We do, however, believe it would be
	quite interesting to work out.
\end{remark}

We conclude with a variant of
\autoref{theorem:point-counting-computation}, where the powers of $q$ appearing
in the denominators of
\eqref{equation:full-moments} and \eqref{equation:parity-moments}
are replaced by the number of points of the stack of quadratic twists.

\begin{corollary}
	\label{corollary:}
	With notation and hypotheses as in \autoref{theorem:point-counting-computation},
	after suitably changing the constants $C(H, \mathscr F)$,
	we also have
	\begin{align}
		\label{equation:full-moments-base}
	\left|\frac{\# \selspacemoments {\mathscr F^n_B}H (\mathbb
	F_q)}{ \# \qtwist n U B(\mathbb F_q)} - \# \hom(\bklpr
	\nu,H) \right| &\leq \frac{C(H, \mathscr F)}{\sqrt{q}} \\
	\label{equation:parity-moments-base}
		\left |\frac{\# \rankselspacemoments {\mathscr F^n_B}H (\mathbb F_q)}{
	 \# \rankcover n {\mathscr F}(\mathbb F_q)} - \#
	\hom(\paritybklpr
	\nu {\rk V_{\mathscr F^n_B} \bmod 2},H)\right| &\leq \frac{C(H, \mathscr F)}{\sqrt{q}}.
	\end{align}
	for all even $n> C(H, \mathscr F)$, and all $q$ with $\sqrt{q} > C(H, \mathscr F)$, and
 	$\gcd(q, 2\nu) = 1$.
\end{corollary}
\begin{proof}
	First, applying \autoref{theorem:point-counting-computation} 
	in the case $H$ is the trivial group gives that both
	$\#\qtwist n U B(\mathbb F_q)$ and
	$\#\rankcover n {\mathscr F}(\mathbb F_q)$ have
	$q^{n}$-points, up to an error
	of $C(\id)/\sqrt{q}$.

	Hence, in \autoref{theorem:point-counting-computation},
	after adjusting the constant $C(H, \mathscr F)$, we can freely replace
	$q^{n}$ appearing in the denominator in
	\eqref{equation:full-moments} and \eqref{equation:parity-moments}
	with $\#\qtwist n U B(\mathbb F_q)$ and $\#\rankcover n {\mathscr F}(\mathbb
	F_q)$.
\end{proof}

\section{Determining the distribution from the moments}
\label{section:determining-distribution}

In this section, we complete the proof of our main result.
In \autoref{subsection:approximating-distributions-by-moments}
we prove a probabilistic result, which we use to show that the distributions we are studying
are determined by their moments, conditioned on the parity of the $\ell^\infty$-Selmer
rank.
Then, in \autoref{subsection:proving-main-result}, we put everything together,
proving our main results in \autoref{subsubsection:proof-main-finite-field},
\autoref{subsubsection:proof-main-moments}, and
\autoref{subsubsection:proof-main-minimalist}.

\subsection{Approximating distributions by approximating moments}
\label{subsection:approximating-distributions-by-moments}

In \autoref{theorem:point-counting-computation}, we determined the moments of
distributions relating to Selmer groups, after taking appropriate limits.
We would like to show these moments determine the distribution.
If we knew the moments exactly, without taking a $q \to \infty$ limit,
we could appeal to 
\cite[Theorem 4.1]{nguyenW:local-and-global-universality}
to show the distribution is also determined.
The next general result will allow us to deal with this issue of taking the $q
\to \infty$ limit.
We thank Melanie Wood for pointing out the following argument, which simplifies
our previous approach.
\begin{proposition}
	\label{proposition:general-moments-to-distribution}
	Let $\mathcal N$ denote the set of isomorphism classes of finite abelian
	$\mathbb Z/\nu \mathbb Z$-modules and let $\mathcal S \subset \mathcal
	N$ denote a subset.
	Suppose $(X^i_j)_{i \in I, j \in J}$ form a set of $\mathcal
	S$-valued random variables, for $I, J$ two infinite subsets of the
	positive integers.
	Suppose there is some $\mathcal S$-valued random variable $Y$ so that
	\begin{enumerate}
		\item for every $H \in \mathcal N$ and for any fixed sufficiently large value of $i$ depending on
	$H$, 
	\begin{align*}
	\lim_{j
	\to \infty} \mathbb E \left( \# \surj(X^i_j, H)\right) = \mathbb E\left(
	\# \surj(Y, H)\right),
	\end{align*}
	and
\item for any sequence $(Y_s)_{s \geq 1}$ of $\mathcal S$-valued random
	variables such that 
	\begin{align*}
	\lim_{s \to \infty} \mathbb E\left(\# \surj(Y_s,
	H)\right) =
	\mathbb E \left(\# \surj(Y, H)\right),
	\end{align*}
we have $\lim_{s \to \infty} \prob(Y_s \simeq
	A) = \prob(Y \simeq A)$ for every $A \in \mathcal N$.
	\end{enumerate}
	Then, both
	\begin{align*}
		\lim_{j \to \infty} \limsup_{i \to \infty} \prob(X^i_j \simeq A)
		\text{ and }
		\lim_{j \to \infty} \liminf_{i \to \infty} \prob(X^i_j \simeq A)
	\end{align*}
	exist, and are equal to $\prob(Y\simeq A)$.
\end{proposition}
\begin{proof}
Enumerate the countable set $\mathcal N$, so that $H_t$ is the $t$th
	element of $\mathcal N$.
	By our first assumption, for fixed sufficiently large $i$ depending on $H$,
	$\lim_{j \to \infty} \mathbb E(\# \surj(X^i_j, H)) = \mathbb E(\#
	\surj(Y, H))$.
	This implies we can find a sequence of pairs $(i_s, j_s)_{s \geq 1}$
	so that for every $s \geq 1$ and every $t \leq s$, 
	\begin{align*}
		\left |\mathbb E (\#\surj(X^{i_s}_{j_s}, H_t)) - \mathbb E(\#
		\surj(Y, H_t)) \right | < 2^{-s}.
	\end{align*}
	This implies that $\lim_{s \to \infty} \mathbb E (\#\surj(X^{i_s}_{j_s},
	H)) = \mathbb E\left( \# \surj\left( Y, H \right) \right)$ for every $H
	\in \mathcal N$.
	Hence, by our second assumption, applied to the sequence $(Y_s)_{s \geq
1}$ defined by $Y_s :=
	X^{i_s}_{j_s}$, we find 
	$\lim_{s \to \infty} \prob(X^{i_s}_{j_s} \simeq
	A) = \prob(Y \simeq A)$.
	Using \cite[Lemma 2.22]{sawin:identifying-measures}, 
	we find 
	\begin{align*}
	\limsup_{j \to \infty} \limsup_{i \to \infty} \prob(X^i_j \simeq
	A) =
	\liminf_{j \to \infty} \liminf_{i \to \infty} \prob(X^i_j \simeq
	A)
	= \prob(Y \simeq A).
	\end{align*}

	To conclude, 
	note that 
	\begin{align*}
	\limsup_{j \to \infty} \limsup_{i \to \infty} \prob(X^i_j \simeq
	A) \geq \liminf_{j \to \infty} \limsup_{i \to \infty} \prob(X^i_j \simeq
	A) \geq \liminf_{j \to \infty} \liminf_{i \to \infty} \prob(X^i_j \simeq
	A),
	\end{align*}
and since the outer two limits are equal, they also agree with the
	middle one. This implies $\lim_{j \to \infty} \limsup_{i \to \infty} \prob(X^i_j \simeq
	A)$ exists and agrees with $\prob(Y \simeq A)$. Analogously, we also
	find $\lim_{j \to \infty} \liminf_{i \to \infty} \prob(X^i_j \simeq
	A)$ exists and agrees with $\prob(Y \simeq A)$.
\end{proof}

\subsection{Proving the main result}
\label{subsection:proving-main-result}

We can now prove our main result.
To set up notation,
suppose we are in the setting of \autoref{example:abelian-special-fiber}, so
that $A
\to U_b$ is an abelian scheme with $\mathscr F_b \simeq A[\nu]$.
For $x \in \qtwist n {U_b} b$, and $A_x \to U_x$ the corresponding abelian scheme over a curve, we use
$\sel_\nu(A_x)$ to denote the $\nu$-Selmer group of the generic fiber of $A_x$
over $U_x$.
In the following theorem, we use the standard convention that the $\mathbb F_q$-points of a stack, such as $\qtwist n U B$,
are counted weighted by the inverse of the size of the
automorphism group of that point.
Also recall the notation introduced in
\autoref{definition:actual-selmer-distribution} for the distributions of Selmer
groups.
The following statement is nearly our main result, but here we start out over a DVR,
instead of a finite field.
Following the proof of this, we will need to lift all our data from a finite field
to a DVR in order to deduce \autoref{theorem:main-finite-field}.
\begin{theorem}
	\label{theorem:main-dvr}
	Suppose $B = \spec R$ for $R$ a DVR of generic characteristic $0$ with
	closed point $b$ with residue field $\mathbb F_{q_0}$ and geometric
	closed point $\overline b$ over $b$.
	Keep hypotheses as in \autoref{hypotheses:big-monodromy-assumptions}:
	Namely, 
suppose $\nu$ is an odd integer and $r \in \mathbb Z_{>0}$ so that every prime $\ell \mid \nu$
	satisfies $\ell > 2r + 1$.
	Let
	$C$ be a smooth proper curve with geometrically connected fibers over $B$,
	$Z \subset C$ finite \'etale nonempty over $B$, and $U := C - Z$.
	Let $\mathscr F$ be a rank $2r$, tame, locally constant constructible, symplectically self-dual sheaf
	of free $\mathbb Z/\nu \mathbb Z$-modules over $U$.
	We assume there is some point $x \in C_{\overline b}$ at which
	$\drop_x(\mathscr F_{\overline b}[\ell]) = 1$ for every prime $\ell \mid \nu$.
	Also suppose
	$\mathscr F_{\overline b}[\ell]$ is irreducible
	for each $\ell \mid \nu$, and that 
	the map $j_* \mathscr F_{\overline b}[\ell^w] \to j_* \mathscr F_{\overline b}[\ell^{w-t}]$ is
	surjective for each prime $\ell \mid \nu$ such that $\ell^w \mid \nu$, and
	$w \geq t$.
Fix $A \to U_b$ as in \autoref{example:abelian-special-fiber}
and suppose the tame irreducible locally constant constructible symplectically
self-dual sheaf of free $\mathbb Z/\nu \mathbb Z$-modules $\mathscr F$
satisfies
$\mathscr F_b \simeq A[\nu]$.
With notation as in \autoref{definition:actual-selmer-distribution},
we have that, for each $\mathbb Z/\nu \mathbb Z$-module $H$,
		\begin{equation}
	\begin{aligned}
		\label{equation:total-distribution-limit}
		&\lim_{\substack{q \to \infty \\ \mathbb F_{q_0} \subset \mathbb
		F_q}} \limsup_{\substack{n \to \infty \\ n
		\hspace{.1cm} \mathrm{even}}} \prob(X_{A[\nu]^n_{\mathbb
		F_q}} \simeq H) \\
		&\lim_{\substack{q \to \infty \\ \mathbb F_{q_0} \subset \mathbb F_q}} \liminf_{\substack{n \to \infty \\ n
		\hspace{.1cm} \mathrm{even}}} \prob(X_{A[\nu]^n_{\mathbb
		F_q}} \simeq H) \\
	\end{aligned}
\end{equation}
exist and agree with $\prob(\bklpr \nu\simeq H)$.
	Similarly, for $i \in \{0,1\}$, 
	\begin{equation}
	\begin{aligned}
		\label{equation:parity-limit}
		&\lim_{\substack{q \to \infty \\ \mathbb F_{q_0} \subset \mathbb F_q}} \limsup_{\substack{n \to \infty \\ n
		\hspace{.1cm} \mathrm{even}}} \prob( X^i_{A[\nu]^n_{\mathbb
		F_q}}\simeq H) \\
		&\lim_{\substack{q \to \infty \\ \mathbb F_{q_0} \subset \mathbb F_q}} \liminf_{\substack{n \to \infty \\ n
		\hspace{.1cm} \mathrm{even}}} \prob(X^i_{A[\nu]^n_{\mathbb
		F_q}} \simeq H) \\
	\end{aligned}
\end{equation}
exist and agree with $\prob(\paritybklpr \nu i \simeq H)$.
\end{theorem}
\begin{proof}
	First, take $i_0 := \rk V_{A^n_B} \bmod 2 \in \{0,1\}$.
	We will apply \autoref{proposition:general-moments-to-distribution}
	with $\mathcal S = \mathcal N^{i_0}, Y = \paritybklpr
	\nu {i_0}, X^n_q =  X^{i_0}_{A[\nu]^n_{\mathbb
	F_q}}$ 
	to prove \eqref{equation:parity-limit} for $i = i_0$.
	(Here, we use $X^n_q$ in place of the notation $X^i_j$ from
	\autoref{proposition:general-moments-to-distribution}.)

	We will now check the hypotheses of 
	\autoref{proposition:general-moments-to-distribution}.
	We need to check the $X^n_q$ and $Y$ are both supported on $\mathcal S$, as well
	as the two enumerated hypotheses of
	\autoref{proposition:general-moments-to-distribution}.
	The $X^n_q$ are supported on $\mathcal S$ by
	\autoref{lemma:distribution-supported-squares}.
	To show $Y$ is supported on $\mathcal S$, from the definition in
	\autoref{subsubsection:bklpr-n-selmer},
	it is enough to show the
	distribution $\mathcal T_{r, \mathbb Z/\nu \mathbb Z}$ defined there is
	supported on abelian groups which are squares, i.e., abelian groups of
	the form $K^2$ for $K$ an abelian group. For this, it is enough to show
	that for any prime $\ell \mid \nu$,
	$\mathcal T_{r, \mathbb Z/\ell \mathbb Z}$ is supported on squares.
	This follows because it is supported on groups with
	a nondegenerate alternating pairing by \cite[Proposition
	5.5]{bhargavaKLPR:modeling-the-distribution-of-ranks-selmer-groups},
	using that
	abelian groups of odd order with a nondegenerate alternating pairing
	have square order.

	We next check the enumerated hypotheses of 
	\autoref{proposition:general-moments-to-distribution}.
	The first enumerated hypothesis of
	\autoref{proposition:general-moments-to-distribution}
	follows from combining
	\autoref{lemma:distribution-equals-rank-cover-distribution} and
	\eqref{equation:parity-moments-base},
	together with an inclusion exclusion argument allows us to replace the
	$\hom$ appearing in these results with 
	$\surj$.
	In order to verify the second enumerated hypothesis of 
	\autoref{proposition:general-moments-to-distribution}, we use
	\autoref{proposition:bklpr-moments}, which bounds the moments of $Y = \paritybklpr
	\nu {i_0}$. The second hypothesis then follows from 
	\cite[Theorem 4.1]{nguyenW:local-and-global-universality}.
	This verifies the hypotheses of 
	\autoref{proposition:general-moments-to-distribution}, and its
	conclusion implies 
	\eqref{equation:parity-limit} for $i = i_0$.

	Having proven \eqref{equation:parity-limit} for $i = i_0$, we next
	aim to prove it for $i = 1-i_0$.
	In this case, note that for any $H \in \mathcal N$,
	$\#\surj(X_{A[\nu]^n_{\mathbb F_q}},H)$ and
	$\# \surj(X^{i_0}_{A[\nu]^n_{\mathbb F_q}}, H)$ take on the same value,
	up to an error of $C(H, \mathscr F)/\sqrt{q}$, by combining
	\autoref{lemma:distribution-equals-rank-cover-distribution},
	\autoref{theorem:point-counting-computation},
	and \autoref{proposition:bklpr-moments}.
	It follows that $\# \surj(X^{1-i_0}_{A[\nu]^n_{\mathbb F_q}}, H)$ also
	takes on this same value, up to an error of $2C(H, \mathscr F)/\sqrt{q}$.
	Hence, an analogous argument to the one above for the case $i = i_0$, 
	this time applying
	\autoref{proposition:general-moments-to-distribution}
	with 
$\mathcal S = \mathcal N^{1-i_0}, Y = \paritybklpr
\nu {1-i_0}, X^n_q =  X^{1-i_0}_{A[\nu]^n_{\mathbb
F_q}}$
proves \eqref{equation:parity-limit} for $i = 1-i_0$.

	Finally, it remains to prove \eqref{equation:total-distribution-limit}.
	By \autoref{lemma:distribution-supported-squares}, the distribution
	$X_{A[\nu]^n_{\mathbb F_q}}$ is supported on $\mathcal N^0 \coprod
	\mathcal N^1$, and so both limits in
	\eqref{equation:total-distribution-limit} exist by summing the limits in
	\eqref{equation:parity-limit} in the cases $i =0$ and $i = 1$.
	Since
	\begin{align*}
		X_{A[\nu]^n_{\mathbb F_q}} = X^0_{A[\nu]^n_{\mathbb F_q}} \cdot
		\prob(X_{A[\nu]^n_{\mathbb F_q}} \in \mathcal N^0) + 
		X^1_{A[\nu]^n_{\mathbb F_q}} \cdot
		\prob(X_{A[\nu]^n_{\mathbb F_q}} \in \mathcal N^1),
	\end{align*}
	it is enough to show
	\begin{align}
		\label{equation:half-points-in-image-of-rank-cover}
1/2 = \lim_{\substack{q \to \infty \\ \mathbb F_{q_0} \subset \mathbb F_q}} \limsup_{\substack{n \to \infty \\ n
		\hspace{.1cm} \mathrm{even}}}\prob(X_{A[\nu]^n_{\mathbb F_q}} \in \mathcal N^{i_0}),
	\end{align}
	and the analogous statement for $\liminf$ in place of $\limsup$.
	Indeed, by \autoref{lemma:image-of-rank-cover-has-selmer-parity}, the probability
$\prob(X_{A[\nu]^n_{\mathbb F_q}} \in \mathcal N^{i_0})$
	is exactly the probability that an $\mathbb F_q$-point of $\qtwist n U
	B$ is in the image of an $\mathbb F_q$-point of $\rankcover n {\mathscr F}$.
	Note that for $n > 0$, $\qtwist n U B$ and $\rankcover n {\mathscr F}$ are both geometrically
irreducible; the latter uses \autoref{theorem:big-monodromy-mod-ell}, which
implies that the geometric monodromy is nontrivial under the Dickson invariant map.
	Using
	\eqref{equation:full-moments} for the trivial group $H = \id$ and
	\eqref{equation:parity-moments} for $H = \id$
	we find both $\qtwist n U B$ and $\rankcover n {\mathscr F}$
	have $q^{\dim \rankcover n {\mathscr F}} + O(1/\sqrt{q})$ points, where the
	implicit constant is independent of $n$.
	This implies \eqref{equation:half-points-in-image-of-rank-cover}
	because the number of $\mathbb F_q$-points in the image of 
	$\rankcover n {\mathscr F}(\mathbb F_q) \to \qtwist n U B (\mathbb F_q)$ is half
	the number of $\mathbb F_q$-points of $\rankcover n {\mathscr F}(\mathbb F_q)$, since this map
	is a finite \'etale double cover.
	\end{proof}

	We have nearly proven our main result, \autoref{theorem:main-finite-field},
	except that \autoref{theorem:main-dvr} begins over a base $B$ of generic
	characteristic $0$, while \autoref{theorem:main-finite-field} begins
	over a finite field.
	It remains to show that if one starts over a finite field, one can lift
	the relevant data to a DVR with generic characteristic $0$.
	This is essentially the content
	of the next lemma, for which we use the following definition.

\begin{definition}
       \label{definition:symplectic-sheaf-data}
       Given a base scheme $B$, {\em a symplectic
sheaf data over $B$} 
is a quadruple $(C, U, Z, \mathscr F)$ over $B$,
where $C$ is a relative smooth
proper curve with geometrically connected fibers over $B$, $U \subset C$ is a nonempty
open,
$Z = C - U$ is a nonempty divisor which is finite \'etale over $B$, and
$\mathscr F$ is a tame symplectically self-dual sheaf of $\mathbb Z/\nu \mathbb
Z$-modules on $U$.
\end{definition}

\begin{lemma}
        \label{lemma:lifting-to-char-0}
        Suppose we are given a symplectic sheaf data $(C_0, U_0, Z_0, \mathscr
        F_0)$ over $\spec \mathbb F_q$.
        If $B$ is the spectrum of a complete DVR with residue field 
        $\mathbb F_q$,
        there exists a symplectic sheaf data $(C, U, Z, \mathscr F)$ over $B$
        whose restriction to $b$, $(C_b, U_b, Z_b, \mathscr F_b)$, is isomorphic to
        $(C_0, U_0, Z_0, \mathscr F_0)$. 
\end{lemma}
\begin{proof}
        The general strategy of the proof will be to show we can lift           
	$(C_0, U_0, Z_0,\mathscr F_0)$ to arbitrary neighborhoods of $b \in B$ and then
algebraize this data.
If $B = \spec S$, with $S$ a complete DVR and uniformizer $\pi$, let $b_n := \spec S/\pi^{n+1}$.
If $(C_i, Z_i)$ is some lifting of $(C_0, Z_0)$ to $b_i$, then the obstruction
to further lifting it to $b_{i+1}$ vanishes because it lies in the coherent cohomology
group
$H^2(C_0, \Omega_{C_0/b}(\on{log} Z_0)) = 0$.
By \cite[Theorem 8.4.10]{FantechiGIK:fundamentalAlgebraicGeometry},
we can lift $C_i$ to $C_0$ over $B$ using the ample line
bundle $\mathscr O_{C_i}(Z_i)$ on $C_i$.
Using \cite[Corollary 8.4.5]{FantechiGIK:fundamentalAlgebraicGeometry}, we
obtain a closed subscheme $Z \subset C$ restricting to $Z_i
\subset C_i$ over $b_i$.
Note that $Z$ is finite \'etale over $B$ because it dominates $B$ and $Z_b =
Z_0$ is geometrically reduced, (as the residue field is assumed to be perfect,) hence smooth over $b$.

Next, we wish to show $\mathscr F_0$ over $U_0$ lifts to $\mathscr F$ over $U$.
In fact, $\mathscr F_0$ has a unique lift by \cite[Corollary
3.1.3]{wewers:deformations-of-tame-admissible-covers-of-curves},
which we note uses our tameness assumption on $\mathscr F_0$.
Note there that $\mathscr F_0$ is a locally constant constructible sheaf with finite coefficients,
and when applying the above, we are viewing it as a finite \'etale cover of
$U_0$. The lift $\mathscr F$ corresponds to a locally constant
constructible sheaf, using the
uniqueness of the lift.
Moreover, by uniqueness of the lift above, the isomorphism $\mathscr F_0 \simeq
\mathscr F_0^\vee(1)$ giving $\mathscr F_0$ its symplectic self-dual structure lifts to
an isomorphism $\mathscr F \simeq \mathscr F^\vee(1)$, giving $\mathscr F$ a
self-dual structure.
Since $\mathscr F_0 \otimes \mathscr F_0 \to \mu_\nu$ factors through $\wedge^2
\mathscr F_0$, we also obtain that $\mathscr F \otimes \mathscr F \to \mu_\nu$
factors through $\wedge^2 \mathscr F$, implying $\mathscr F$ is symplectically
self-dual.
\end{proof}

\subsubsection{Proof of \autoref{theorem:main-finite-field}}
\label{subsubsection:proof-main-finite-field}

We first explain the proof of \autoref{theorem:main-finite-field}.
Let $b = \spec \mathbb F_q$,
and $(C, U, Z, A[\nu])$ be our given symplectic sheaf data over $b$ as in
\autoref{theorem:main-finite-field}. 
Let $B$ be a complete DVR with closed point
$b$ and generic characteristic $0$.
By \autoref{lemma:lifting-to-char-0}, we can realize 
$(C, U, Z, A[\nu])$
as the restriction along $b \to B$ of some symplectic sheaf data $(C_B, U_B,
Z_B, \mathscr F_B)$ on $B$.
Note that the hypotheses of \autoref{theorem:main-dvr}
imply those of \autoref{theorem:main-finite-field} as mentioned in the last
paragraph of
\autoref{hypotheses:big-monodromy-assumptions}.
Hence, 
\autoref{theorem:main-finite-field} follows from \autoref{theorem:main-dvr}.
\qed

\subsubsection{Proof of \autoref{theorem:main-moments}}
\label{subsubsection:proof-main-moments}

As in the proof of
\autoref{theorem:main-finite-field}
in \autoref{subsubsection:proof-main-finite-field} above,
we may lift all our symplectic sheaf data over $\mathbb F_q$ to symplectic sheaf
data over the spectrum of a complete DVR $B$, with residue field $\mathbb F_q$ and generic
characteristic $0$, using
\autoref{lemma:lifting-to-char-0}.
To obtain \eqref{equation:moment-limsup} of \autoref{theorem:main-moments}, we note that $\sym^2 H$ is the
$H$-surjection moment of the BKLPR distribution by
\autoref{proposition:bklpr-moments}. 
Hence, 
\eqref{equation:moment-limsup}
follows from
\autoref{theorem:point-counting-computation}, together with an
inclusion-exclusion to show that points on a certain subset of the components of 
$\selspacemoments {\mathscr F^n_b}{H}$
correspond to {\em surjections} onto $H$, in place of all homomorphisms.

For establishing \eqref{equation:moment-limit-residue},
we only need show that the limit
\begin{align*}
	\lim_{\substack{n \to \infty \\ n \hspace{.1cm} \mathrm{ even}}}
				\frac{\sum_{x \in \qtwist n U
			{\mathbb F_q}(\mathbb F_{q^j})}  \# \surj
		(\sel_\nu(A_x), H)}{\sum_{x \in \qtwist n U {\mathbb F_q}(\mathbb F_{q^j})}
		1}
\end{align*}
exists, as then \eqref{equation:moment-limsup} yields what the limit as $j \to \infty$ of this
value must be.
The limit exists by
\eqref{equation:full-moments-stable}, together with an
inclusion-exclusion to show that points on a certain subset of the components of 
$\selspacemoments {\mathscr F^n_b}{H}$
correspond to surjections onto $H$, in place of all homomorphisms.
\qed

\subsubsection{Proof of \autoref{theorem:main-minimalist}}
\label{subsubsection:proof-main-minimalist}

We next explain the proof of \autoref{theorem:main-minimalist}.
Choose $\nu=\ell$ a prime as in \autoref{theorem:main-finite-field}.
Note that this only excludes finitely many possibilities for $\ell$, so any
sufficiently large $\ell$ works.
By \autoref{theorem:main-dvr},
together with 
\autoref{lemma:lifting-to-char-0} as in
\autoref{subsubsection:proof-main-finite-field}
above,
we obtain equidistribution of the parity of the dimension of the $\ell$-Selmer
group in the quadratic twist family,
since the BKLPR distribution predicts the parity of the rank of the $\ell$-Selmer group of
the abelian variety is even half the time and odd half the time. 
It follows from \autoref{lemma:sha-almost-square}
that the parity of $\rk_{\ell^\infty} A$
agrees with the parity of the rank of $\sel_\ell(A)$.
Therefore, 
the
parity of $\rk_{\ell^\infty}$ is also equidistributed.

To conclude the result, we only need to prove that the probability that
$\ell^\infty$-Selmer rank
is $\geq 2$ is $0$.
It follows from \autoref{theorem:main-moments} (and an inclusion
exclusion to relate surjections to homomorphisms) that the average size
of the $\nu$-Selmer group is $\sum_{\sigma \mid \nu} \sigma$.
Therefore, the same argument as in
\cite[Proposition 5]{bhargavaS:average-4-selmer} 
(see also
\cite[p.246-247]{poonenR:random-maximal-isotropic-subspaces-and-selmer-groups})
implies that the probability that the $\ell^\infty$-Selmer rank is $\geq 2$ is $0$.
\qed

\subsubsection{Proof of 
\autoref{theorem:main-finite-field}
in the special case $\nu = \ell$}
\label{subsubsection:proof-special-case-prime}

We will now give a somewhat shorter proof of
\autoref{theorem:main-finite-field}
in the special case that  $\nu$ is a prime $\ell$.
In particular, we need only the first lines of
\eqref{equation:h-cohomology-bound} and
\eqref{equation:selmer-cohomology-bound}, and not the second lines of these
equations.  As explained in \autoref{subsubsection:proof-main-finite-field}, the case $\nu = \ell$ is all that is necessary for the application to \autoref{theorem:main-minimalist}.  We sketch below how to handle this case for the convenience of those readers who are not in need of the full generality of \autoref{theorem:main-finite-field}. 

The main difference when $\nu = \ell$ is that in this case it is easier to recover the distribution from the moments.  In the proof of \autoref{theorem:main-finite-field} we need to compute the moments of the variables $X^i_{A[\nu]^n_{\mathbb F_q}}$ for $i=0,1$; in other words, we need the moments of the mod $\ell$-Selmer rank conditional on parity.  But when $\nu = \ell$, we can get by with less; by \cite[Thm 2.10, Cor 2.12]{wood:probability-theory-for-random-groups}, the distribution of Selmer ranks converges to the BKLPR distribution if the moments converge to the BKLPR moments, and if the parity of the mod $\ell$-Selmer rank is equidistributed between odd and even.  The former statement is what we have proved in \eqref{equation:full-moments}.  

It remains to show that the parity of the mod $\ell$-Selmer rank is equidistributed. First of all, it follows from the discussion in \autoref{lemma:sha-almost-square} that the mod $\ell$-Selmer rank of a quadratic twist $A_\chi$ has the same parity as the $\ell^\infty$-Selmer rank of $A_\chi$.  By  \cite[Theorem 1.1]{trihanY:the-ell-parity-conjecture}, this parity is determined by the root number $W(A_\chi)$.  
Let $N_A$ denote the conductor of $A$.
Using notation from \autoref{remark:root-numbers}, we use $\widetilde{\chi}:
C_{K(C)} \to
\{\pm 1\}$ to denote the
map corresponding to the quadratic extension $\chi$, and by \cite[Cor. 6.12]{bisatt:explicit-root-numbers}, 
we can relate the root number of $A$ and $A_\chi$ by
\begin{equation*}
	W(A_\chi) = W(A)\widetilde{\chi}(N_A).
\end{equation*}
So what remains is to show that the average of $\widetilde{\chi}(N_A)$ as $\chi$ ranges over quadratic characters of discriminant $B$ 
approaches $0$ as $B$ goes to infinity.  
This follows from \cite[Theorem
2]{bhargavaSW:geometry-of-numbers-over-global-fields-i} upon noting that the contribution to the local term $m_{\mathfrak p}(\Sigma_{\mathfrak p})$
in \cite[Theorem 2]{bhargavaSW:geometry-of-numbers-over-global-fields-i}, for
$\mathfrak p$ the specified point from
\eqref{equation:multiplicative-hypothesis},
over which the cover is trivial,
is equal to the local contribution from which the cover is \'etale but
nontrivial, and $\widetilde{\chi}(N_A)$ has opposite signs in these two cases.
\qed

\appendix
\section[Frobenius equivariance]{Frobenius equivariance\\ By Aaron Landesman }
\label{section:frobenius-equivariance}

Throughout this section, we will use notation as in
\autoref{notation:quadratic-twist-notation} and
\autoref{definition:selmer-sheaf}.
In particular, we work over a base scheme $B$ and $\nu$ is an odd integer with $2
\nu$ invertible on $B$.
Additionally, we will assume there exists some section $\sigma \in Z(B)$ over
which our symplectically self-dual sheaf $\mathscr F$ of $\mathbb Z/\nu \mathbb
Z$-modules on $U = C - Z$ has trivial inertia at
$\sigma$, and hence its pushforward along $U \to C$ is lcc in a neighborhood of
$\sigma$.
The main result is \autoref{theorem:frobenius-equivariance}, which shows that
the stabilization isomorphisms on the cohomology of Selmer spaces are equivariant for the
action of Frobenius.
The only part of our paper this appendix comes into play is to prove
\eqref{equation:moment-limit-residue}
(and \eqref{equation:full-moments-stable} along the way).
That is, we can show the limit in
\eqref{equation:moment-limit-residue} exists,
instead of only showing that the $\liminf$ and
$\limsup$ exist as in \eqref{equation:moment-limsup}.

In order to prove \autoref{theorem:frobenius-equivariance}, we first set up
notation to describe a partially compactified version of Selmer
spaces in \autoref{subsection:compactified-notation}.
Then, we introduce log structures and a logarithmic version of the stabilization map
in \autoref{subsection:gluing-map-with-log}.
In \autoref{subsection:u-degree}, we show that we may take the topological stabilization map
to have degree $2$.
Next, we show this logarithmic stabilization map agrees with the topological
stabilization map in \autoref{subsection:gluing-to-stabilization}.
Finally, in \autoref{subsection:main-stabilization-proof}, we prove
\autoref{theorem:frobenius-equivariance}.

We thank Dori Bejleri for suggesting that the general strategy taken here could
work.
We would also like to mention that the idea of viewing these sort of
stabilization maps in algebraic geometry as coming from log geometry is not
new. Variants have been studied in, for example
\cite{abramovichCGB:punctured},
\cite{gross:remarks},
\cite{holmesS:logarithmic-cohomological-field-theories},
\cite{parker:log-geometry-and-exploded},
and \cite{bergstromDPW:hyperelliptic-curves-the-scanning-map}.

\subsection{Notation for the partially compactified Selmer space}
\label{subsection:compactified-notation}

We next set up notation for a partially compactified version of Selmer spaces.
First we define a partially compactified version of configuration space in
\autoref{subsubsection:compactified-configuration-space},
then we define a partially compactified version of the space of quadratic twists
in \autoref{subsubsection:compactified-quadratic-twist}, and finally we define a
partially  compactified
version of the selmer space in \autoref{subsubsection:compactified-selmer}.

\subsubsection{Defining a partially compactified configuration space with sections}
\label{subsubsection:compactified-configuration-space}

Let $\mathcal K_{n+f+1,g}(C,1)$ denote the moduli stack (which is in fact a
scheme) of $n+f+1$-pointed stable
maps of degree $1$ over $B$ to our given curve $C$.
\begin{remark}
	\label{remark:}
	The above stack of stable maps parameterizes curves, one of whose
	components is $C$, and all other components have genus $0$, and are
	contracted under the map to $C$.
	We next construct a locally closed substack of a quotient stack of this
	which corresponds to
	only allowing $d$ of the $n$ points to simultaneously collide with
	$\sigma$.
\end{remark}

We suppose that $\sigma \subset Z$ and let $U = C - Z$ as usual.
Suppose that $Z$ has connected components of degrees $1, f_1, \ldots,
f_k$, the
first $1$ corresponding to $\sigma$, so that
$1+ (\sum_{i=1}^k f_i) = 1+ f  = \deg Z$.
There is an action of $S_n \times S_{f_1} \times \cdots \times S_{f_k}$
on $\mathcal K_{n+f+1,g}(C,1)$ by permuting the $n+f+1$ points.
There is an evaluation map
$\ev : [\mathcal K_{n+f+1,g}(C,1)/S_n \times S_{f_1} \times \cdots \times S_{f_k}] \to
C \times [C^{f_1}/S_{f_1}] \times \cdots \times [C^{f_k}/S_{f_k}]$.
Define
$\overline{\st n {\sigma} U B}$ to be the fiber of the map $\ev$
over the point of $C \times [C^{f_1}/S_{f_1}] \times \cdots \times [C^{f_k}/S_{f_k}]$ 
corresponding to the divisor $Z$.


We further fix an integer $d$.
There is an open substack
$\st n {\sigma} U B \subset \overline{\st n {\sigma} U B}$
which is set theoretically supported on the locus where the universal curve is
either irreducible (hence isomorphic to $C$) or a union of $C$ and $\mathbb P^1$
where a subset of $d$ of the $n$ points collide into the point $\sigma$.
More precisely, this open substack can be described as the complement of the following
divisors:
first, the divisor parameterizing two points, neither of which is $\sigma$,
colliding,
and, second,
the divisor where $d'$ points collide with $\sigma$ for $d'
\neq d$.

\begin{remark}
	\label{remark:nc-divisor}
	We observe that
$\st n {\sigma} U B$ is smooth and the complement of $\conf n U B \subset \st n
{\sigma} U B$ is a smooth divisor.
This follows from \autoref{theorem:nc} since one can use this to realize $\st n
{\sigma} U B$ as an open in a compactification of $\conf n U B$ with normal
crossings boundary.
\end{remark}

\subsubsection{Defining a partially compactified space of quadratic twists with
sections}
\label{subsubsection:compactified-quadratic-twist}
We next define a space of quadratic twists over $\st n {\sigma} U B$ which we
can think of as partially
compactifying $\qtwist n U B$ (though in actuality it will partially compactify a double
cover of $\qtwist n U B$; the double cover corresponding to specifying a
point in the universal double cover over $\sigma$).
To generalize the Selmer space to stable curves, we use similar notation to our
definition of Selmer space, but include the subscript $\on{St}$ throughout.

Assume that $Z$ is the disjoint union of multisections of degrees $f_0 := 1, f_1,
\ldots, f_k$ the first multi-section corresponding to a section $\sigma \in C(B)$.
There is a universal schematic proper curve $\mathscr C^{n,\sigma}_{\on{St},B} \to C
\times_B \st n{\sigma} U B
\to \st n{\sigma} U B$.
There is a universal degree $n$ divisor $\mathscr D^{n,\sigma}_{\on{St},B}
\subset \mathscr
C^{n,\sigma}_{\on{St},B}$. 

\begin{remark}
	\label{remark:}
	The first map 
$\mathscr C^{n,\sigma}_{\on{St},B} \to C \times_B \st n{\sigma} U B$
is an isomorphism over $\conf n U B \subset \st n \sigma U B$, but in general may
contain additional genus $0$ fibers corresponding to locations where $d$ of the $n$ points
collide with $\sigma$.
\end{remark}

We next define an extension of a variant of the Selmer sheaf over $\st n{\sigma} U B$.
The informal idea is that $\stqtwist n {\sigma}U B$ parameterizes double covers of
curves with a degree $1$ stable map to $C$, branched over a degree $n$ divisor
in the smooth locus of the nodal curve, together with a trivialization of this
double cover at $\sigma$.
We now give a more formal definition.

Let $[C/(\mathbb Z/2 \mathbb Z)] = C \times B(\mathbb Z/2 \mathbb Z)$ denote the
stack quotient of $C$ by the trivial $\mathbb Z/2\mathbb Z$ action.
(Recall we are assuming $2$ is invertible on $B$.)
Next,
$\stqtwist n{\sigma} U B$ can be constructed from 
$\mathcal K_{n+f+1,g}([C/(\mathbb Z/2 \mathbb Z)],1)$
much the same way 
$\st n{\sigma} U B$ was constructed from 
$\mathcal K_{g,n+f+1}(C,1)$.
Here, 
$\mathcal K_{n+f+1,g}([C/(\mathbb Z/2 \mathbb Z)],1)$
denotes twisted stable maps from a genus $g$ twisted curve $\mathcal X$ with $n+f+1$ marked points to 
$[C/(\mathbb Z/2 \mathbb Z)]$ such that the composition to the coarse space
$\mathcal X \to [C/(\mathbb Z/2 \mathbb Z)] \to C$ has degree $1$, in the sense
that the line bundle $\mathscr O_C(\sigma)$ on $C$ pulls back to a degree $1$ line
bundle on $\mathcal X$.
Namely, we first form the quotient of 
$\mathcal K_{g,n+f+1}([C/(\mathbb Z/2 \mathbb Z)],1)$
by the action of
$S_n \times S_{f_1} \times \cdots \times S_{f_k}$.
We next construct
the fiber of the evaluation map
$\on{ev}: [\mathcal K_{g,n+f+1}([C/(\mathbb Z/2 \mathbb Z)],1)/S_n \times S_{f_1} \times
\cdots \times S_{f_k}] \to
C \times [C^{f_1}/S_{f_1}] \times \cdots \times [C^{f_k}/S_{f_k}]$ over the
point corresponding to the divisor $Z$.
We let
$\overline{\stqtwist n{\sigma} U B}$
be the double cover of this fiber, obtained by specifying a point in the fiber of
the double cover over the pullback of $\sigma$.
We let
$\stqtwist n{\sigma} U B \subset\overline{\stqtwist n{\sigma} U B}$
denote the open substack parameterizing double
covers which are balanced in the sense of
\cite[\S2.1.3]{abramovichCV:twisted-bundles},
which also map to $\st n {\sigma} U B \subset \overline{\st n {\sigma} U B}$.
We define 
$\openstqtwist n {\sigma} U B := \stqtwist n {\sigma} U B \times_{\st n \sigma U B} \conf n U B$.

It will be useful to additionally specify a slight variant of the above
construction, where one marks $2$ (or more) sections, instead of just a single
section. Namely, if $\sigma_1, \ldots, \sigma_t$ are $t$ sections in $Z(B) \subset C(B)$, we use
$\stqtwist n{\sigma_1, \ldots, \sigma_t} U B$ to denote the analogous construction, but
where one additionally marks a point of the double cover over each of $\sigma_1,
\ldots, \sigma_t$.
In particular, $\stqtwist n{\sigma_1, \ldots, \sigma_t} U B$ is a finite \'etale cover of degree
$2^{t-1}$ over $\stqtwist n{\sigma_1} U B$.
\begin{remark}
	\label{remark:}
	In what follows, we will only apply this construction with multiple sections in the case $C =
	\mathbb P^1$, $t = 2$, and $\{\sigma_1, \sigma_2\} = \{0,\infty\}$.
\end{remark}

\subsubsection{Defining a partially compactified Selmer sheaf}
\label{subsubsection:compactified-selmer}
Using the description of twisted stable maps, there is a universal schematic
curve $\mathscr R^{n,\sigma}_B :=\mathscr C^{n,\sigma}_{\on{St},B} \times_{\st n{\sigma} U B} \stqtwist n{\sigma} U B$ over
$\stqtwist n {\sigma} U B$ 
with a 
finite degree $2$ cover
branched over
the universal degree $n$ divisor $\mathscr
D^n_{\on{St}, B} \times_{\st n \sigma U B} \stqtwist n \sigma U B$.
There is also an universal evaluation map
$h_{\on{St},B}^{n,\sigma}: \mathscr R^{n,\sigma}_B \to C$ coming from the definition of
stable maps to $C$.
Let $j: U \to C$ denote the inclusion.
We next construct a ``universal quadratic twist'' of the sheaf
$(h_{\on{St},B}^{n,\sigma})^* j_*\mathscr F$.

Recall that, by construction, there is a double cover of $\mathscr R^{n,\sigma}_B$
branched over 
the universal degree $n$ divisor $\mathscr
D^n_{\on{St}, B} \times_{\st n \sigma U B} \stqtwist n \sigma U B$.
Let $\widetilde{j} : \mathscr R^{n,\sigma}_B - \mathscr
D^n_{\on{St}, B} \times_{\st n \sigma U B} \stqtwist n \sigma U B \to \mathscr
R^{n,\sigma}_B$ denote the open immersion and let $\chi^n_B$ denote the
nontrivial rank $1$ local system of $\mathbb Z/\nu \mathbb Z$-modules
on
$\mathscr R^{n,\sigma}_B - \mathscr
D^n_{\on{St}, B} \times_{\st n \sigma U B} \stqtwist n \sigma U B$
which is trivialized on the above
mentioned
double cover.
Define
$\mathscr F^{n,\sigma}_{\on{St},B} := \widetilde{j}_*( \chi^n_B \otimes
\widetilde{j}^*((h_{\on{St},B}^{n,\sigma})^* j_*\mathscr F))$,
which we view as the universal quadratic twist of 
$(h_{\on{St},B}^{n,\sigma})^* j_*\mathscr F)$
on
$\mathscr R^{n,\sigma}_B$.

Define $\selspace {\mathscr F^{n,\sigma}_{\on{St},B}}$ to be the algebraic space represented by the
\'etale sheaf
parameterizing torsors for 
$\mathscr F^{n,\sigma}_{\on{St},B}$ over $\mathscr R^{n,\sigma}_B$ together
with a trivialization of the torsor over
$\mathscr R^{n,\sigma}_B$ at the section $\sigma$.
We also use $\selspace {\mathscr F^{n,\sigma}_{B}}$
for the restriction of $\selspace {\mathscr F^{n,\sigma}_{\on{St}, B}}$ along the open
immersion $\conf n U B \subset \st n \sigma U B$.
Note that the fiber $(j_*\mathscr F)_{\sigma}$ is finite \'etale of degree $\nu^{\rk
\mathscr F}$ over $\sigma$ by the assumption that the inertia of $\mathscr F$ at $\sigma$
is trivial.

We also define 
$\mathscr F^{n,\sigma_1, \ldots, \sigma_t}_{\on{St},B}$ and
$\selspace {\mathscr F^{n,\sigma_1, \ldots, \sigma_t}_{\on{St}, B}}$
over
$\stqtwist n{\sigma_1, \ldots, \sigma_t} U B$
as the pullbacks of
$\mathscr F^{n,\sigma_1}_{\on{St},B}$ and
$\selspace {\mathscr F^{n,\sigma_1}_{\on{St}, B}}$
along
$\stqtwist n{\sigma_1, \ldots, \sigma_t} U B \to \stqtwist n{\sigma_1} U B$.
We define
$\mathscr F^{n,\sigma_1, \ldots, \sigma_t}_{B}$ and
$\selspace {\mathscr F^{n,\sigma_1, \ldots, \sigma_t}_{B}}$
as the further restrictions to 
$\openstqtwist n{\sigma_1, \ldots, \sigma_t} U B \subset
\stqtwist n{\sigma_1, \ldots, \sigma_t} U B$.
\begin{remark}
	\label{remark:}
	We will only need this variant with $t > 1$ in the case $t = 2, C =
	\mathbb P^1, \{\sigma_1, \sigma_2\} = \{0,\infty\}$ and $\mathscr F = \mathscr G$ is a trivial sheaf on
	$\mathbb G_m$.
	Note that the universal torsor for $\mathscr F_B^{n,\sigma_1, \ldots,
	\sigma_t}$ over 
$\selspace {\mathscr F^{n,\sigma_1, \ldots, \sigma_t}_{B}}$
is only trivialized at the first marked section. 
So, in particular,
	the universal 
	$\mathscr F_B^{n,0,\infty}$
	torsor
	over
	$\selspace {\mathscr F^{n,0,\infty}_{B}}$ is trivialized at $0$ but not
	$\infty$, 
	while
	the universal $\mathscr F_B^{n,\infty,0}$ torsor over
	$\selspace {\mathscr F^{n,\infty,0}_{B}}$ is trivialized at
	$\infty$ but not at $0$.
\end{remark}



\subsection{The gluing map with log structures}
\label{subsection:gluing-map-with-log}

In this subsection, we define the gluing map, which joins the Selmer space of
degree $n - d$ over
$C$ with a trivialization at $p$ to the Selmer space of degree $d$ over $\mathbb P^1$ with a
trivialization at $0$ and $\infty$, and sends it to the partially compactified Selmer
space of degree $n$ with a trivialization at $p$.
We first define the gluing map in \autoref{subsubsection:gluing-map} and
\autoref{lemma:coefficient-system-algebraic-restriction}.
We then briefly review relevant parts of logarithmic algebraic geometry in
\autoref{subsubsection:log-background}.
In \autoref{lemma:log-gluing-map}, we define version of the gluing map for
logarithmic stacks.

\subsubsection{Defining the gluing map}
\label{subsubsection:gluing-map}

Fix an even positive integer $d$ and let $\mathbb G_m \subset \mathbb P^1$ over our base
$B$ denote the
complement of the sections $\infty : B \to \mathbb P^1$
and $0 : B \to \mathbb P^1$.
For even $n > d$,
there is a gluing map
$\Delta : \conf d {\mathbb G_m} B \times \conf {n-d} U B \to
\st n{\sigma} U B$
which send a degree $d$ reduced subscheme of $\mathbb G_m$ and $n-d$ subscheme
of $U$ to the corresponding degree $n$ subscheme on the glued curve $C
\coprod_{\sigma \sim 0} \mathbb P^1$,
obtained by gluing the section $\sigma \in C(B)$ to the $0$ section on $\mathbb P^1$.

Over $\Delta$, there is another gluing map
\begin{align}
	\label{equation:gamma-map}
	\Gamma : \openstqtwist d {0,\infty}{\mathbb G_m} B \times \openstqtwist {n-d}{\sigma} U B \to
\stqtwist n{\sigma} U B
\end{align}
 which we define next.
A point of $\openstqtwist {n-d}{\sigma} U B$ can be described as $[X \to C,z]$, for $X
\to C$ a double cover unramified over $\sigma$, and $z$ a choice of point in the
preimage of $\sigma$.
We use $z'$ to denote the remaining section of $X$ in the preimage of $\sigma$.
A point of $\openstqtwist d {0,\infty}{\mathbb G_m} B$ can be described as
$[Y \to \mathbb P^1,v, s]$ where $Y \to \mathbb P^1$ is a double cover, $v$ is a
choice of point in the preimage of $0$, and $s$ is a choice of point over $\infty$.
We also use $v'$ to denote the remaining section of $Y$ over $0$.
The map is then given by gluing $v$ to $z$ and gluing $v'$ to $z'$ to obtain a
double cover $[X \coprod_{v\sim z, v' \sim z'} Y \to C \coprod_{\sigma\sim 0} \mathbb
P^1,s]$, which we view as a point of 
$\stqtwist n{\sigma} U B$, for $s$ the choice of section over the marked section
$\infty \in \mathbb P^1 \subset C \coprod_{\sigma\sim 0} \mathbb P^1$.
(We note here that the universal section $\sigma$ on $\mathscr R^{n,\sigma}_B$
restricts to $\infty \in \mathbb P^1$.)
Via the description above, the map $\Gamma$ from \eqref{equation:gamma-map} is induced by a gluing map on the universal curves
\begin{align*}
	\Theta: 
	(C \times_B \openstqtwist {n-d}{\sigma} U B) \times (\mathbb P^1_B \times_B
	\openstqtwist d {0,\infty}{\mathbb G_m} B) \to
	\mathscr R^{n,\sigma}_B
\end{align*}
together with a gluing map on double covers of these universal curves (which we
will not need to distinguish with further notation).
We also define $\mathscr G$ to be the constant sheaf of $\mathbb Z/\nu \mathbb Z$-modules of rank $2r = \rk \mathscr F$ on $\mathbb G_m \subset \mathbb P^1$.
There are maps 
\begin{align*}
	\Theta' &: 
	(\mathbb P^1_B \times_B
	\selspace {\mathscr G^{d,0,\infty}_{B}}) \times_B
(C \times_B \openstqtwist {n-d}{\sigma} U B) 
	\to
	\mathscr R^{n,\sigma}_B \\
\Gamma' &: \selspace {\mathscr G^{d,0,\infty}_{B}}\times_B \openstqtwist {n-d}{\sigma} U B \to
\stqtwist n{\sigma} U B
\end{align*}
where $\Theta'$ and $\Gamma'$ are obtained from $\Theta$ and $\Gamma$ by
precomposing these with the projection
$\selspace {\mathscr G^{d,0,\infty}_{ B}} \to 
\openstqtwist d {0,\infty}{\mathbb G_m} B$.
Define the projections
\begin{align*}
	\pi_C &: \selspace {\mathscr
	G^{d,0,\infty}_{ B}} \times_B
\openstqtwist {n-d}{\sigma} U B 
	\to \openstqtwist {n-d}{\sigma} U B \\
\pi_{\mathbb P^1}&: \selspace {\mathscr G^{d,0,\infty}_{ B}}
\times_B \openstqtwist {n-d}{\sigma} U B
\to \selspace {\mathscr G^{d,0,\infty}_{ B}}.
\end{align*}
\begin{lemma}
	\label{lemma:coefficient-system-algebraic-restriction}
	There is an isomorphism of lcc \'etale sheaves on 
$\selspace {\mathscr G^{d,0,\infty}_{ B}} \times
\openstqtwist {n-d}{\sigma} U B$,
\begin{align}
	\label{equation:selmer-pullback}
\pi_{C}^*\selsheaf
{\mathscr F^{n-d,\sigma}_{\on{St},B}} \oplus \pi_{\mathbb P^1}^* \selsheaf {\mathscr
	G^{d,\infty,0}_{B}} \simeq
	(\Gamma')^* \selsheaf{\mathscr F^{n,\sigma}_{\on{St},B}}.
\end{align}
\end{lemma}
\begin{proof}
	By construction,
	$\Gamma^* \selsheaf{\mathscr F^{n,\sigma}_{\on{St},B}}$
	is a universal torsor
	for quadratic twists of
	$\Theta^* (h_{\on{St},B}^{n,\sigma})^* (j_*\mathscr F)$, together with a trivialization over
	$\infty \in \mathbb P^1 \subset C \coprod_{\sigma \sim 0} \mathbb P^1_B$ (since $\infty$ is the pullback of the section
	corresponding to $\sigma$ over $\st n \sigma U B$).
	It follows that $(\Gamma')^* \selsheaf{\mathscr
	F^{n,\sigma}_{\on{St},B}}$,
	which is the pullback of $(\Gamma)^* \selsheaf{\mathscr F^{n,\sigma}_{\on{St},B}}$
	along 
	$\selspace {\mathscr G^{d,0,\infty}_{ B}} \to \openstqtwist d
	{0,\infty}{\mathbb G_m} B$,
	is a universal torsor for quadratic twists of 
	$(\Theta')^* (h_{\on{St},B}^{n,\sigma})^* (j_*\mathscr F)$, together
	with a specified trivialization over
	$\infty \in \mathbb P^1 \subset C \coprod_{\sigma \sim 0}
	\mathbb P^1_B$. 
	This sheaf is also trivializable at $0$ because we pulled it back to 
	$\selspace {\mathscr
	G^{d,0,\infty}_{ B}}$,
	although no trivialization is specified at $0$.
	We can choose an \'etale path so as to make an identification of the
	fiber of $\mathscr F$ over $0$ and the fiber over $\infty$, and thereby
	transfer the trivialization at $\infty$ to a trivialization at $0$.
	Specifying a torsor for a quadratic twist of $(\Theta')^* (h_{\on{St},B}^{n,\sigma})^* (j_*\mathscr F)$
	which is trivialized in this way at $\sigma\sim 0$ is equivalent to specifying
	a torsor for the restriction of the quadratic twist to $\mathbb P^1$
	with a trivialization at $\infty$, together with a torsor for the
	restriction to $C$ with a trivialization at $\sigma$. In other words, this gives
the desired isomorphism \eqref{equation:selmer-pullback}.
\end{proof}


\subsubsection{Background on logarithmic geometry}
\label{subsubsection:log-background}
For a general reference on cohomology of log schemes, we recommend
\cite{illusie:an-overview}. We will work with a subcategory of log stacks
called Deligne-Faltings log stacks. 
We recommend
\cite[\S8]{bergstromDPW:hyperelliptic-curves-the-scanning-map}, especially
\cite[\S8.2]{bergstromDPW:hyperelliptic-curves-the-scanning-map},
for an introduction to the basic properties of Deligne-Faltings log stacks we will use.
We next review notation for Deligne-Faltings log stacks.
It is shown in \cite[Theorem 3.6]{borneV:parabolic-sheaves-on-logarithmic-schemes} that
Deligne-Faltings log schemes
correspond to a full subcategory of log structures on schemes.
(We will not need this, but it is shown in
\cite[Theorem 3.6]{borneV:parabolic-sheaves-on-logarithmic-schemes}
that they correspond to the full subcategory of quasi-integral log
structures.)
Moreover the proof that these Deligne-Faltings log structures form a full
subcategory of usual log structures generalizes straightforwardly from the case
that we work over a scheme to the case that we work over a Deligne-Mumford stack.

Let $X$ be a Deligne-Mumford stack. Let $\mathfrak{Div}_X$ denote the \'etale
sheaf of symmetric monoidal groupoids on $X$ which associates to an \'etale map
$U \to X$ the groupoid of pairs $(L,s)$ for $L$ an invertible sheaf on $U$ and
$s : \mathscr O_U \to L$ a section. The symmetric monoidal structure is given by tensor
product of line bundles and multiplication of the corresponding sections.
A {\em pre-Deligne-Faltings log structure on $X$} is an \'etale presheaf of
commutative monoids $A$ on $X$ and a symmetric monoidal functor $\phi: A \to
\mathfrak{Div}_X$.
We call a pre-Deligne-Faltings log structure on $X$ a {\em Deligne-Faltings log
structure on $X$}
if $A$ is a sheaf and $\phi$ has trivial kernel.

A {\em Deligne-Faltings log stack} is a Deligne-Mumford stack $X$ equipped with
a Deligne-Faltings log structure $(A,\phi)$.
Suppose we have two Deligne-Mumford stacks $X$ and $Y$ with log structures
$(A,\phi)$ and $(B, \psi)$.
A {\em morphism of Deligne-Faltings log structures} is the data of a map $f: X \to Y$, a map of \'etale sheaves $h: A
\to f^* B$ on $X$, and a natural transformation of symmetric monoidal functors from $f^* \psi \circ h$ to $\phi$.
A morphism as above is {\em strict} if $h: A \to f^* B$ is an isomorphism.

We now mention a few types of log structures we will encounter in what follows.

The {\em trivial log structure} denotes the unique Deligne-Faltings log structure associated to the
sheaf $A = 0$. The unique map $A \to \mathfrak{Div}_X$ necessarily
sends the unit to the unit, i.e., $0 \mapsto (\mathscr O_X,1)$.

We also use the {\em standard log
structure} to denote the Deligne-Faltings log structure associated to the
constant sheaf $A = \underline{\mathbb N}$ on the natural numbers
and the function $A \to \mathfrak{Div}_X$ sending a generator to $(\mathscr O_X,0)$.
(Note that this has trivial kernel because the unit in the monoidal structure is 
$(\mathscr O_X,1)$ and not
$(\mathscr O_X,0)$.)

Suppose we are given a line bundle $\mathscr L$ on $X$ and a section $s: \mathscr O_X \to
\mathscr L$. So long as $(\mathscr L,s) \neq (\mathscr O_X,0)$, we can associate a Deligne-Faltings log structure as follows.
The sheaf $A$ is the constant sheaf $\underline{\mathbb N}$ and the map
$\underline{\mathbb N} \to
\mathfrak{Div}_X$ induced by sending $1$ to $(\mathscr L,s)$.
We call this the {\em log structure associated to $(\mathscr
L,s)$}.
If $D \subset X$ is a smooth divisor, the 
{\em log structure associated to $D$}
is the Deligne-Faltings log structures
associated to $(\mathscr O_X(D), s: \mathscr O_X \to \mathscr O_X(D))$.

\subsubsection{Defining our log stacks}
\label{subsubsection:defining-our-log-stacks}
We now define the Deligne-Faltings log stacks we will work with.
Let 
$\selspace {\mathscr H}$ denote the stack associated to the
sheaf $\selsheaf{\mathscr H}$.
Let $E \subset \selspace {\mathscr F^{n,\sigma}_{\on{St},B}}$ denote the divisor
which is the preimage of $\st n \sigma U B  - \conf n U B \subset \st n \sigma U B$.
Let $\logselspace {\mathscr F^{n,\sigma}_{\on{St},B}}$
denote the Deligne-Faltings log stack with underlying space 
$\selspace {\mathscr F^{n,\sigma}_{\on{St},B}}$ with the Deligne-Faltings log
structure given by the log structure associated to $E$, as in
\autoref{subsubsection:log-background}.
Define a Deligne-Faltings log scheme
$\left( \left(\selspace {\mathscr G^{d,\infty,0}_{B}} \times_{\openstqtwist d
		{0,\infty} {\mathbb G_m}B} \selspace {\mathscr G^{d,0,
\infty}_{B}}\right)
			\times_B \selspace {\mathscr
		F^{n-d,\sigma}_{B}} \right)^{\on{log}}$
with underlying scheme 
\begin{align}
	\label{equation:log-space}
	\left(\selspace {\mathscr G^{d,\infty,0}_{B}} \times_{\openstqtwist d
		{0,\infty} {\mathbb G_m}B} \selspace {\mathscr G^{d,0,\infty}_{B}}\right)
			\times_B \selspace {\mathscr
			F^{n-d,\sigma}_{B}}
\end{align}
with the standard log structure, as defined in
\autoref{subsubsection:log-background}.
%
%

\begin{lemma}
	\label{lemma:log-gluing-map}
	Suppose $B$ is the spectrum of a complete DVR (or, more generally, has trivial
	Picard group).
	The isomorphism of
\autoref{lemma:coefficient-system-algebraic-restriction}
yields a strict map of Deligne-Faltings log stacks
\begin{align}
	\label{equation:log-scheme-map}
	\alpha: \left(\left(\selspace {\mathscr G^{d,\infty,0}_{B}} \times_{\openstqtwist d
	{0,\infty} {\mathbb G_m}B} \selspace {\mathscr
G^{d,0,\infty}_{B}}\right) \times_B \selspace {\mathscr
F^{n-d,\sigma}_{B}} \right)^{\on{log}} \to \logselspace {\mathscr F^{n,\sigma}_{\on{St},B}}.
\end{align}
\end{lemma}
\begin{proof}
	The map of underlying stacks of these Deligne-Faltings log stacks is obtained from 
	\autoref{lemma:coefficient-system-algebraic-restriction}.
	Note that $\alpha^* \tau$ is the
	zero section because $\tau$ vanishes on the image of $\alpha$.
	It remains to show that the line bundle
	$\mathscr O_{\selspace {\mathscr F^{n,\sigma}_{\on{St},B}}}(E)$
	pulls back to 
	the trivial bundle.
	
	First, we define another line bundle $\mathscr L$ on
	\eqref{equation:log-space}. Then, we show
	$\mathscr L$ is isomorphic to the trivial bundle on
	\eqref{equation:log-space}.
	Finally, we will show 
	$\mathscr O_{\selspace {\mathscr F^{n,\sigma}_{\on{St},B}}}(E)$ pulls back to
	$\mathscr L$.

Let $\pi_1$ and $\pi_2$ denote the two projections from 
\eqref{equation:log-space}
onto its two factors.
Let $\mathbb T_{\sigma}$ denote
the line bundle on
$\selspace {\mathscr F^{n-d,\sigma}_{B}}$
which is the  restriction to $\sigma$ of the relative tangent bundle 
for the universal curve over 
$\selspace {\mathscr F^{n-d,\sigma}_{B}}$.
Similarly, let $\mathbb T_0$ denote the line bundle on
$\left(\selspace {\mathscr G^{d,\infty,0}_{B}} \times_{\openstqtwist d
		{0,\infty} {\mathbb G_m}B} \selspace {\mathscr G^{d,0,\infty}_{B}}\right)$ 
which is the  restriction to $0$ of the relative tangent bundle 
for the universal curve over 
$\left(\selspace {\mathscr G^{d,\infty,0}_{B}} \times_{\openstqtwist d
		{0,\infty} {\mathbb G_m}B} \selspace {\mathscr G^{d,0,\infty}_{B}}\right)$
	which is pulled back from the universal curve over 
$\selspace {\mathscr G^{d,0,\infty}_{B}}$.
Then, define $\mathscr L := \pi_1^* \mathbb T_0 \otimes \pi_2^* \mathbb T_{\sigma}$ and
take $\tau$ to be the zero
section of this line bundle.

	We claim that $\mathscr L$ is isomorphic to the trivial bundle. 
	It suffices to show both
$\mathbb T_0$ and $\mathbb T_{\sigma}$ are trivial. 
These are pulled back from a line bundle on $B$ since the section $\sigma$ is
pulled back from $C \to B$ and the section $0$ is pulled back from $\mathbb
P^1_B
\to B$. Hence, these bundles are trivial because $B$ has trivial Picard group.

It remains to show 
	$\mathscr O_{\selspace {\mathscr F^{n,\sigma}_{\on{St},B}}}\to\mathscr O_{\selspace {\mathscr F^{n,\sigma}_{\on{St},B}}}(E)$
	pulls back to the zero section of
	$\pi_1^* \mathbb T_0 \otimes \pi_2^*
	\mathbb T_{\sigma}$.
	This can be proven via an argument analogous to
	\cite[p. 346, line 2]{arbarelloCG:geomtry-of-algebraic-curves-ii}, as we
	next further expand on.
	A minor technicality is that the gluing map $\alpha$ joining $C$ and
	$\mathbb P^1$ factors as the composition of a positive dimensional
	smooth map of relative dimension $1$ and an \'etale map. 
	Namely, it factors through a gluing map joining $C$ and $P$, where
	$P$ is a genus $0$
	curve with $0$ and $\infty$ marked. Because a third point is not
	marked on $P$, $P$ may not be isomorphic to $\mathbb P^1$.
	The latter gluing map, described by gluing $C$ to $P$, does define an
	\'etale map to $\logselspace {\mathscr F^{n,\sigma}_{\on{St},B}}$.
	Since the normal bundle of an \'etale morphism is identified with the
	pullback of the ideal sheaf of its image, an argument similar to 
	\cite[p. 346, line 2]{arbarelloCG:geomtry-of-algebraic-curves-ii}
	shows that $\mathscr O_{\selspace {\mathscr F^{n,\sigma}_{\on{St},B}}}(E)$ pulls
	back to the tensor product of the tangent line bundles on $C$ and the
	genus $0$ curve. When one further pulls this line bundle back to
	$\left(\selspace {\mathscr G^{d,\infty,0}_{B}} \times_{\openstqtwist d
		{0,\infty} {\mathbb G_m}B} \selspace {\mathscr G^{d,0,\infty}_{B}}\right)
			\times_B \selspace {\mathscr
			F^{n-d,\sigma}_{B}}$, we obtain the claim.
\end{proof}

\subsection{A fun, combinatorial group theory interlude}
\label{subsection:u-degree}

Taking a break from the heavy machinery of log geometry, we will need a result
from combinatorial group theory which strengthens \cite[Lemma
3.5]{EllenbergVW:cohenLenstra}. This will show the degree of $\mathbb U$ may be
taken to be $2$ in our setting above.

Let $G$ be a group and $c \subset G$ be a conjugacy class. Recall that $(G,c)$
is {\em non-splitting} if $c$ generates $G$ and, for any subgroup $K \subset G$, $c \cap K$ is either empty
or a single conjugacy class. I.e., $c$ does not split into multiple conjugacy
classes upon intersecting with $K$.
Consider the coefficient system $V_n$ for $\Sigma^1_{0,0}$, as in 
\autoref{example:hurwitz-coefficient-system}, associated to the group $G$,
and $c \subset G$ a specified conjugacy class.
We use $R^V$ to denote $\oplus_{n \geq 0} H_0(B^n_{0,0}, V_n)$,
$r_h$ denote right multiplication by $h$ on $R^V$, and 
$\on{ord}(h)$ to denote the order of $h$.

\begin{proposition}
	\label{proposition:degree-order}
	Let $(G,c)$ be non-splitting and $D$ any positive integer.
Then $\mathbb U := \sum_{h \in c} r_h^{D \on{ord}(h)} \in R^V$
	is a homogeneous central element with finite degree kernel and
	cokernel. Hence, $\mathbb U$ satisfies the hypotheses of
\autoref{theorem:central-u-implies-cohomology-bound}.
\end{proposition}
\begin{proof}
	In \cite[Lemma 3.5]{EllenbergVW:cohenLenstra}
it was shown that there exists some integer $D$ so that 
$\mathbb U_D := \sum_{h \in c} r_h^{D\on{ord}(h)} \in R^V$
satisfies the conclusion of the theorem statement.
We want to show $D$ can be taken to be any positive integer.
Let $S_t(K)$ denote the subset of the quotient set $c^t/B_t$, via the standard braid group
action of $B_t$ on $t$-tuples of elements in $c$, consisting of those $t$-tuples
of elements which generate $K$.
The proof of \cite[Lemma 3.5]{EllenbergVW:cohenLenstra}
shows that we may take $D$ to be any positive integer so that for every subgroup
$K \subset G$,
and for $t$ sufficiently large,
the map $r_{h}^{\on{ord}(h)D} : S_t(K) \to S_{t + \on{ord}(h)D}(K)$ is a bijection which is
independent of choice of $h \in c \cap K$.
Therefore, by possibly replacing $(G,c)$ with $(K, c \cap K)$,
to complete the proof, it suffices to show that for any non-splitting $(G,c)$
for $t$ sufficiently large, and for any $h, k \in c$,
$r_h^{\on{ord}(h)}, r_{k}^{\on{ord}(k)} : S_t(G) \to S_{t + 2}(G)$ induce the same bijection.

Let $t$ be sufficiently large and $x \in S_t(G)$.
We wish to show the class of $r_h^{\on{ord}(h)}(x)$ agrees with the class of $r_k^{\on{ord}(k)}(x)$.
Let $\widehat{G}$ denote the group
$S_c \times_{G^{\on{ab}}} \mathbb Z$,
where $G^{\on{ab}}$ denotes the abelianization of $G$
and
$S_c$ is a reduced Schur cover for $(G,c)$, as defined in
\cite[Definition, p. 21]{wood:an-algebraic-lifting-invariant};
for the reader's benefit, we next review this notation.
A Schur cover $S \to G$ is a central extension of $G$ by some group $K$ so
that the class of the extension in $H^2(G,K)$ maps to an isomorphism in
$\hom(H_2(G, \mathbb Z),K)$ under the map from the universal coefficients exact
sequence. A reduced Schur cover $S_c \subset S$ is a particular subgroup
which surjects onto $G$.
Following the notation of \cite{wood:an-algebraic-lifting-invariant},
we notate the element of $\hat{G}$ corresponding to $h\in c$ as $(\hat{h}, e_h)
\in S_c \times_{G^{\on{ab}}} \mathbb Z = \hat{G}$,
for 
$e_h = 1 \in \mathbb Z$
and
$\hat{h} \in S_c$ projecting to $h \in G$ under the map $S_c \to G$ coming from
the definition of a reduced Schur cover.
For all other $k \in c$, if $k = shs^{-1}$, we can choose any lift $\tilde{s}\in
S_c$ of $s$ and take $\hat{k} := \tilde{s} \hat{h} \tilde{s}^{-1}$. This is
independent of the choice of $s$ and $\tilde{s}$ by \cite[Lemma
2.3]{wood:an-algebraic-lifting-invariant}.
Write $x = (h_1, \ldots, h_t)$.
It follows from
\cite[Theorem 3.1 and Theorem 2.5]{wood:an-algebraic-lifting-invariant}
that showing $r_h^{\on{ord}(h)}x$ lies in the same orbit as $r_k^{\on{ord}(k)} x$
is equivalent to showing
$(\hat{h}, e_h)^{\on{ord}(h)}\cdot (\hat{h}_1, e_{h_1}) \cdots (\hat{h}_t, e_{h_t})
=
(\hat{k}, e_k)^{\on{ord(k)}} \cdot
(\hat{h}_1, e_{h_1}) \cdots (\hat{h}_t, e_{h_t}).$
Equivalently,
we wish to show 
$(\hat{h}, e_h)^{\on{ord}(h)} =
(\hat{k}, e_k)^{\on{ord(k)}}$.
We are assuming $h$ and $k$ lie in the same conjugacy class, and hence have the
same order.
Thus, the second coordinates of the above products agree, and it is enough to
show
$\hat{h}^{\on{ord}(h)} \hat{k}^{-\on{ord(k)}} = \id$.
Writing $k = shs^{-1}$, and
using the relation from \cite[Lemma 2.3]{wood:an-algebraic-lifting-invariant} that
for $\tilde{s}$ any lift of $s$,
$\tilde{s}\hat{h}\tilde{s}^{-1} = \widehat{shs^{-1}}$,
we find
\begin{align*}
\hat{h}^{\on{ord}(h)} \hat{k}^{-\on{ord}(h)} = \hat{h}^{\on{ord}(h)}
	\tilde{s} \hat{h}^{-\on{ord}(h)} \tilde{s}^{-1} =
	[\hat{h}^{\on{ord}(h)}, \tilde{s}].
\end{align*}
However, $\hat{h}^{\on{ord}(h)} $ lies in the center of $S_c$, since its image
in $G$ is $h^{\on{ord}(h)}
= \id$ and
$S_c$ is a central
extension of $G$ by $\ker(S_c \to G)$.
Therefore, $\hat{h}^{\on{ord}(h)}$ commutes with
$\tilde{s}$, so 
$[\hat{h}^{\on{ord}(h)}, \tilde{s}] = \id$, as desired.
\end{proof}

\subsection{Relating the gluing map to the stabilization map}
\label{subsection:gluing-to-stabilization}

In this subsection, we compare the logarithmic gluing map constructed in
\eqref{equation:log-scheme-map} to the stabilization map on cohomology.
The main result is \autoref{proposition:u-map-identification}, which shows they
can be identified in a suitable sense.

In the case $B =\spec \mathbb C,$
we obtain a stabilization map $\mathbb U$ on the cohomology over $\spec \mathbb C$ as
follows:
Following \cite[Lemma 3.5]{EllenbergVW:cohenLenstra},
define $\mathbb U := \sum_{h \in c} r_h^{D\on{ord}(h)} \in R^V$ for $r_h$ right
multiplication by $h$, $\on{ord}(h)$
the order of $h$. Let $d = D \cdot \ord(h)$ denote the degree of $\mathbb U$.
(In our case, $h$ will always have order $2$, so $d = 2D$, and ultimately we
	will take $D = 1$, but we will continue
to use $d$ as we believe it is somewhat clarifying.)
Then, $\mathbb U$ is a homogeneous central element with finite degree kernel and
cokernel by \autoref{proposition:degree-order},
and hence satisfies the hypotheses of
\autoref{theorem:central-u-implies-cohomology-bound}.
The map $\mathbb U$ on homology can be reexpressed in terms of a map on compactly
supported cohomology which we continue to call $\mathbb U$.
We take $\ell'$ to be a prime invertible on $B$.
We may identify this $\mathbb U$ operator over the complex numbers with an operator
$\mathbb U_{\overline{\mathbb F}_q}$ on the $\overline{\mathbb F}_q$ cohomology via the
following commutative diagram
\begin{equation}
	\label{equation:u-transfer}
	\begin{tikzcd} 
		H_{\on{c}}^i(\selspace {\mathscr F^{n-d,\sigma}_{\mathbb C}},
		\mathbb Q_{\ell'}(n-d))\ar {r}{\mathbb U} \ar {d} & 
		H_{\on{c}}^{i+2d}(\selspace {\mathscr F^{n,\sigma}_{\mathbb C}},
		\mathbb Q_{\ell'}(n))  \ar {d} \\
H_{\on{c}}^i(\selspace {\mathscr F^{n-d,\sigma}_{\overline{\mathbb F}_q}}, \mathbb
Q_{\ell'}(n-d)) \ar {r}{\mathbb U_{\overline{\mathbb F}_q}} & 
H_{\on{c}}^{i+2d}(\selspace {\mathscr F^{n,\sigma}_{\overline{\mathbb F}_q }}, \mathbb
Q_{\ell'}(n)),
\end{tikzcd}\end{equation}
where the vertical isomorphisms are obtained via the specialization maps (there are
isomorphisms by
\cite[Proposition 7.7]{EllenbergVW:cohenLenstra})
and the map
$\mathbb U_{\overline{\mathbb F}_q}$ is the unique map making the diagram
commute.

One can show in a fashion analogous to
\autoref{example:specified-hurwitz-coefficient-system}
that the sequence of spaces $\selspace {\mathscr F^{n,\sigma}_{\mathbb C}}$
correspond to a coefficient system $F_n$ for $\Sigma_{g,f}^1$ over the same
coefficient system $V_n$ for $\Sigma^1_{0,0}$ as in 
\autoref{example:specified-hurwitz-coefficient-system}.
Now, let $c \subset \asp_{2g}(\mathbb Z/\nu \mathbb Z)$ denote the conjugacy
class of elements projecting to $-\id$ in $\sp_{2g}(\mathbb Z/\nu \mathbb Z)$
and let $G \subset \asp_{2g}(\mathbb Z/\nu \mathbb Z)$ denote the subgroup
generated by $c$.

We will next define another map coming from logarithmic geometry. In order to
define that map, we need the following result.

\begin{lemma}
	\label{lemma:scheme-over-fq}
	If we are given $x \in 	\selspace {\mathscr G^{d,0,\infty}_{
	B}}(\mathbb F_q)$,
we may identify the $\overline{\mathbb F}_q$-points of 
\begin{align*}
	\selspace {\mathscr G^{d,\infty,0}_{B}} \times_{\openstqtwist d
		{0,\infty} {\mathbb G_m}B} \selspace {\mathscr G^{d,0,\infty}_{B}}
\end{align*}
over $x$ with tuples in $G^d$ whose
product is $\id \in G$, with $G$ as defined above. This fiber has an action of Frobenius,
	$\frob_q$.
	For $d$ even, under this
	bijection, the set of elements $\{(h, \ldots, h) : h \in c\}$ in the
	fiber over $x$
	constitutes a union of $\frob_q$ orbits.
\end{lemma}
\begin{proof}
	First, we can identify the $\overline{\mathbb F}_q$-points over $x$ with
	the set claimed above using the
	description of the Selmer space as a Hurwitz space from
	\autoref{corollary:selmer-to-hurwitz-moments}. Although this isomorphism
	is only given over $\mathbb C$, one can also deduce this isomorphism
	over $\overline{\mathbb F}_q$ using that the Selmer space is a quotient
	of a finite \'etale cover of a configuration space by a finite group,
	and said configuration space has a normal 
	crossings compactification using
	\autoref{corollary:hurwitz-compactification}.

	Since Frobenius must preserve the conjugacy class of an element in
	$G$, as the conjugacy class can be read from from the inertia of the
	corresponding cover,
	the orbit of $(h, \ldots, h)$ must consist of elements of the form $(h_1, \ldots, h_n)$
	where each $h_i \in c$. 

	It remains to show that any such element in this
	orbit satisfies $h_i = h_j$ for $1 \leq i \leq n$.
	We have that $\selspace {\mathscr G^{d,\infty,0}_{B}} \times_{\openstqtwist d
		{0,\infty} {\mathbb G_m}B} \selspace {\mathscr G^{d,0,\infty}_{B}}$ is a finite \'etale cover
	of $\selspace {\mathscr G^{d,0,\infty}_{
	B}}$, and hence we obtain an action of
	the fundamental group of $\selspace {\mathscr G^{d,0,\infty}_{
	B}}$ on
	the geometric fiber of 
$\selspace {\mathscr G^{d,\infty,0}_{B}} \times_{\openstqtwist d
		{0,\infty} {\mathbb G_m}B} \selspace {\mathscr G^{d,0,\infty}_{B}}$
	over a given point $x \in \selspace {\mathscr G^{d,0,\infty}_{B}}$.
	The above sheaf over $\selspace {\mathscr G^{d,0,\infty}_{x}}$ is obtained as the base
	change of a sheaf over $\conf d {\mathbb P^1} x$
	using that $\mathscr G$ is trivial, so extends over $\infty$ and
	$0$.

	We next conclude the proof by showing the set $\{(h, \ldots, h) : h \in G\}$ over the
	image of $x$ form a union of $\frob_q$ orbits.
	Let $\overline{x}$ denote a geometric point over $x$.
	Note that this set 
	$\{(h, \ldots, h) : h \in G\}$
	now inherits an action of the fundamental group of 
	$\conf d {\mathbb P^1} x$,
	which is a semidirect product
	of its geometric fundamental group,
	$\pi_1(\conf d {\mathbb P^1} {\overline{x}})$, the
profinite completion of the braid group, and $\pi_1(x) \simeq \widehat{\mathbb Z}$,
	generated by Frobenius.
Hence, for any $\eta \in \pi_1(\conf d {\mathbb P^1} {\overline{x}})$
there is some $\eta' \in \pi_1(\conf d {\mathbb P^1} {\overline{x}})$
	with
	$\eta \frob_q (h, \ldots, h) = \frob_q \eta'(h,\ldots, h)$.
	Since the
	braid group fixes $(h, \ldots, h)$, we find
	$\eta \frob_q (h, \ldots, h) = \frob_q(h, \ldots, h)$, and hence
	$\frob_q(h, \ldots, h)$
	is fixed by the action of the profinite completion of the Braid group.
	Since elements of the form $(k, \ldots, k) \in G^n$ are the only
	elements fixed by the profinite completion of the Braid group, we must
	have $\frob_q(h, \ldots, h) = (k, \ldots, k)$ for some $k \in G$.
\end{proof}

\subsubsection{Defining a stabilization map from logarithmic geometry}
\label{subsubsection:stabilization-log-map}

Let $B$ be the spectrum of a complete DVR with residue field $\mathbb F_q$ and
generic characteristic $0$.
Suppose there exists 
$x \in \selspace {\mathscr G^{d,0,\infty}_{ B}}(\mathbb F_q)$.
One can lift this to a section of
$\selspace {\mathscr G^{d,0,\infty}_{B}}(B)$ over $B$
using smoothness of $\selspace {\mathscr G^{d,0,\infty}_{B}}$
to lift the point over any power of the maximal ideal of the DVR corresponding
to $B$, which then algebraizes to a $B$-point by \cite[Corollary
8.4.6]{FantechiGIK:fundamentalAlgebraicGeometry}.
Let $\iota: S_{\mathbb F_q} \subset \selspace {\mathscr G^{d,\infty,0}_{B}} \times_{\openstqtwist d
		{0,\infty} {\mathbb G_m}B} \selspace {\mathscr G^{d,0,\infty}_{B}}
$
denote the reduced closed subscheme over $\mathbb F_q$ whose base change
to $\overline{\mathbb F}_q$ corresponds to the set of $\overline{\mathbb F}_q$-points $\cup_{h \in c} (h, \ldots, h)$.
This is a well defined subscheme by \autoref{lemma:scheme-over-fq}.
Let $S_B$ denote a lift of $S_{\mathbb F_q}$ over the given lift of $x$, which
exists and is unique because the cover 
$\selspace {\mathscr G^{d,\infty,0}_{B}} \times_{\openstqtwist d
		{0,\infty} {\mathbb G_m}B} \selspace {\mathscr G^{d,0,\infty}_{B}}
	\to
	\selspace {\mathscr G^{d,0,\infty}_{B}}$
is finite
\'etale.
One may verify the complement of
$\selspace {\mathscr F^{n,\sigma}_{T}} \subset \selspace {\mathscr F^{n,\sigma}_{\on{St},T}}$
is a smooth divisor, using its description as a finite cover of
$\conf n U B \subset \st n {\sigma} U B$, which has complement a smooth divisor
by
\autoref{remark:nc-divisor}.
(The cover is not \'etale over the boundary, but it is branched over the boundary
of a fixed degree, which is enough to guarantee the smoothness above.)
For $T \to B$ the spectrum of a field,
using smoothness of
the complement
$\selspace {\mathscr F^{n,\sigma}_{T}} \subset \selspace {\mathscr F^{n,\sigma}_{\on{St},T}}$
mentioned above,
we also obtain the identification
\begin{align*}
\delta: 
H^i \left(\selspace {\mathscr F^{n,\sigma}_{T}}, \mathbb Q_{\ell'}(n)	\right)
\simeq
H^i \left(\logselspace {\mathscr F^{n,\sigma}_{\on{St},T}}, \mathbb
	Q_{\ell'}(n)
\right).
\end{align*}
by 
\cite[\S8.4.3]{bergstromDPW:hyperelliptic-curves-the-scanning-map}.

Next, the inclusion
\begin{align}
	\label{equation:s-inclusion}
	\beta: S_B \simeq S_B \times_B B \xrightarrow{\iota \times \{1\}}
 \selspace {\mathscr G^{d,\infty,0}_{B}} \times_{\openstqtwist d
		{0,\infty} {\mathbb G_m}B} \selspace {\mathscr G^{d,0,\infty}_{B}}
\times_B \mathbb G_m,
\end{align}
induces a strict map of Deligne-Faltings log stacks
\begin{align*}
	\left( S_B \times_B \selspace {\mathscr F^{n-d,\sigma}_{B}}
	\right)^{\on{log}} \to
\left(\left(\selspace {\mathscr G^{d,\infty,0}_{B}} \times_{\openstqtwist d
		{0,\infty} {\mathbb G_m}B} \selspace {\mathscr G^{d,0,\infty}_{B}}\right)
			\times_B \selspace {\mathscr F^{n-d,\sigma}_{B}}\right)^{\on{log}}
\end{align*}
where we endow $S_B \times_B \selspace {\mathscr F^{n-d,\sigma}_{B}}$
with the standard log structure,
corresponding to the sheaf $\underline{\mathbb N}$ and the map to $\mathfrak{Div}_{S_B
\times_B \selspace {\mathscr F^{n-d,\sigma}_{B}}}$ sending
$1$ to $(\mathscr O_{S_B \times_B \selspace {\mathscr F^{n-d,\sigma}_{B}}}, 0)$.

Using the above described maps along with the map $\alpha$ from
\eqref{equation:log-scheme-map},
and base changing along some spectrum of a field $T \to B$,
we obtain a map on cohomology
\begin{equation}
\begin{aligned}
	\label{equation:composite-log-scheme-map}
H^{i} \left(\selspace {\mathscr F^{n,\sigma}_{T}}, \mathbb
	Q_{\ell'}	\right)
	&\xrightarrow{\delta}
H^{i} \left(\logselspace {\mathscr F^{n,\sigma}_{\on{St},T}}, \mathbb
Q_{\ell'}	\right) \\
	&\xrightarrow{\alpha^*}
H^{i} \left(\left(	\left(\selspace {\mathscr G^{d,\infty,0}_{T}} \times_{\openstqtwist d
		{0,\infty} {\mathbb G_m}T} \selspace {\mathscr G^{d,0,\infty}_{T}}\right)
			\times_T \selspace {\mathscr
	F^{n-d,\sigma}_{T}}\right)^{\on{log}}, \mathbb Q_{\ell'} \right) \\
&\xrightarrow{\beta^*}
H^i\left( \left(S_T \times_T \selspace {\mathscr F^{n-d,\sigma}_{T}}
\right)^{\on{log}}, \mathbb Q_{\ell'} \right).
\end{aligned}
\end{equation}
\begin{lemma}
	\label{lemma:factor-through-trivial-log}
		The map $\alpha^* \circ \delta$ in \eqref{equation:composite-log-scheme-map}
	over $T = \spec \mathbb C$ can be identified with the map induced on cohomology of
	the gluing map described as follows.
	The map takes in the following data:
	\begin{enumerate}
		\item a direction $\tau$ on the unit circle,
		\item an $\asp_{2g}(\mathbb Z/\nu \mathbb Z)$-cover
	of $\Sigma^1_{g,f}$, 
\item an
	$\asp_{2g}(\mathbb Z/\nu \mathbb Z)$-cover of $\Sigma^2_{0,0}$, 
\item a specified identification of the boundary of
	$\Sigma^1_{g,f}$ with $S^1$,
\item a specified identification of one of the boundary components of
	$\Sigma^2_{0,0}$, corresponding to the point $0 \in \mathbb P^1$, with $S^1$.
	\end{enumerate}
	The gluing map then glues the two copies of $S^1$ in $(4)$ and $(5)$ via a
	rotation by $\tau$, and glues the boundary components of the covers by
	the pullback of this identification.
\end{lemma}
\begin{proof}
	We let $T = \spec \mathbb C$ and verify the explicit description
	of the map. 
Recall the map $\alpha^*$ from
\eqref{equation:log-scheme-map} is obtained from the map of underlying stacks from
\autoref{lemma:coefficient-system-algebraic-restriction}
together with the map of logarithmic structures from
\autoref{lemma:log-gluing-map}.
	We next identify the Deligne-Faltings log schemes in the source and target of the map
$\alpha^*$ in \eqref{equation:composite-log-scheme-map}
with their corresponding Kato-Nakayama spaces as in 
\cite[Examples 8.3.8
and 8.3.9]{bergstromDPW:hyperelliptic-curves-the-scanning-map}, to obtain a map
of topological stacks.
We then have a commutative square of Deligne-Faltings log stacks
\begin{equation}
	\label{equation:selmer-to-conf-square}
	\begin{tikzcd} 
	\left(\left(\selspace {\mathscr G^{d,\infty,0}_{T}} \times_{\openstqtwist d {0,\infty} {\mathbb G_m}T} \selspace {\mathscr G^{d,0,\infty}_{T}}\right)
			\times_T \selspace {\mathscr F^{n-d,\sigma}_{T}}\right)^{\on{log}}	 \ar {r} \ar {d} & \logselspace {\mathscr F^{n,\sigma}_{\on{St},T}} \ar {d} \\
		\left(\conf d {\mathbb G_m} B \times \conf {n-d} U B \right)^{\on{log}} \ar {r} & \left( \st n {\sigma}
		U B \right)^{\on{log}},
\end{tikzcd}\end{equation}
where $\left(\conf d {\mathbb G_m} B \times \conf {n-d} U B \right)^{\on{log}}$
has
the standard log structure (associated to the sheaf $\underline{\mathbb N}$ with the
	generator mapping to the
trivial line bundle with the $0$ section), and 
$\left( \st n {\sigma} U B \right)^{\on{log}}$
has the log structure defined by the boundary divisor $\st n {\sigma} U B -
\conf n U B$.
The pull back of the line bundle giving the Deligne-Faltings log structure
associated to the
boundary divisor
$\st n {\sigma} U B$
can be identified with
the trivial bundle
on
$\conf d {\mathbb G_m} B \times \conf {n-d} U B$, which is more canonically the tensor product of the tangent bundles at $0$ and $\sigma$, 
by a proof analogous to the proof of \autoref{lemma:log-gluing-map}.
Hence, the gluing map associated to the bottom map in 
\eqref{equation:selmer-to-conf-square}
can be described as choosing a unit tangent vector over $0$, choosing a
unit tangent vector over $\sigma$, and then gluing the unit tangent spaces so as to
identify those unit tangent vectors. This results in a point of the
Kato-Nakayama space of 
$\left( \st n {\sigma}U B \right)^{\on{log}}$.
Choose an identification of $U$ with
the interior of
$\Sigma^1_{g,f}$, with one boundary component corresponding to $\sigma$,
and an identification of $\mathbb G_m$ with the interior of $\Sigma^1_{0,1}$,
with a boundary component at $0$ and a puncture at $\infty$.
Topologically, we can further identify the above map with a map gluing 
$\Sigma^1_{g,f}$ to $\Sigma^2_{0,0}$ via altering
$\Sigma^1_{0,1}$ to $\Sigma^2_{0,0}$ by replacing the puncture at $\infty$ with a
boundary component.
The above yields a description of the gluing map on configuration spaces
analogous to that in the statement.
Using the commutative square \eqref{equation:selmer-to-conf-square},
the map of line bundles associated to the map of Selmer spaces is
pulled back from the corresponding map of line bundles on configuration spaces,
yielding the identification we wished to show.
\end{proof}

Additionally, there is a map of Deligne-Faltings log schemes
$\left(S_T \times_T \selspace {\mathscr F^{n-d,\sigma}_{T}}
\right)^{\on{log}}\xrightarrow{\widetilde{\gamma}} S_T \times_T \selspace
{\mathscr F^{n-d,\sigma}_{T}}$,
where we use $S_T \times_T \selspace {\mathscr F^{n-d,\sigma}_{T}}$
to denote the Deligne-Faltings log scheme with trivial log structure
(corresponding to the $0$ monoid).
This induces a map on cohomology
\begin{align}
	\label{equation:cohomology-map-to-trivial-log-structure}
	H^i( S_T \times_T \selspace {\mathscr F^{n-d,\sigma}_{T}}, \mathbb Q_{\ell'})
	\xrightarrow{\gamma}
H^i\left( \left(S_T \times_T \selspace {\mathscr F^{n-d,\sigma}_{T}}
\right)^{\on{log}}, \mathbb Q_{\ell'}\right).
\end{align}
\begin{proposition}
	\label{proposition:u-map-identification}
	Assume $B$ is the spectrum of a complete DVR with residue field $\mathbb F_q$
	and generic characteristic $0$.
	Suppose $\selspace {\mathscr G^{d,0,\infty}_{ B}}(\mathbb F_q)\neq
	\emptyset$.
	If $T$ is either $\spec \overline{\mathbb F}_q$ or $\spec \mathbb C$,
	there is a canonical splitting 
\begin{align}
	\label{equation:splitting-cohomology-map-to-trivial-log-structure}
H^i\left( \left(S_T \times_T \selspace {\mathscr F^{n-d,\sigma}_{T}}
\right)^{\on{log}}, \mathbb Q_{\ell'}\right)
	\xrightarrow{\varepsilon}
	H^i( S_T \times_T \selspace {\mathscr F^{n-d,\sigma}_{T}}, \mathbb Q_{\ell'})
\end{align}
of $\gamma$, i.e., $\varepsilon \circ \gamma= \id$. Additionally, 
if $\eta : H^i( S_T \times_T \selspace {\mathscr F^{n-d,\sigma}_{T}}, \mathbb
Q_{\ell'}) \to H^i( \selspace {\mathscr F^{n-d,\sigma}_{T}}, \mathbb
Q_{\ell'})$ is the summation map obtained by identifying $S_T$ with a disjoint
union of points and summing the resulting cohomology elements,
the composition of \eqref{equation:composite-log-scheme-map} 
with $\eta \circ \varepsilon$
is Poincar\'e dual to a map
\begin{align}
	\label{equation:T-cohomology-map}
H^i_{\on{c}}(\selspace {\mathscr
F^{n-d,\sigma}_{T}}, \mathbb Q_{\ell'}(n-d))
\to
	H^{i+2d}_{\on{c}} \left(\selspace {\mathscr F^{n,\sigma}_{T}}, \mathbb Q_{\ell'}(n)\right).
\end{align}
which agrees with $\mathbb U$ when $T = \spec \mathbb C$ and agrees with
$\mathbb U_{\overline{\mathbb F}_q}$ with $T = \spec
\overline{\mathbb F}_q.$
\end{proposition}
\begin{proof}
	First, we explain how to deduce the final statement
	when $T = \spec
	\overline{\mathbb F}_q$ from the case that $T = \spec \mathbb C$ using the
	specialization map.
	For the final statement with $T = \spec
	\overline{\mathbb F}_q$, we wish to prove
	the surjective specialization map is an isomorphism, and so we wish to
	prove the constructible cohomology sheaves on $B$ corresponding to each
	of the terms in \eqref{equation:composite-log-scheme-map} and
	\eqref{equation:cohomology-map-to-trivial-log-structure} are locally
	constant on $B$.
	Local constancy of the cohomology of
	$\selspace {\mathscr F^{n-d,\sigma}_{B}}$, and hence also of its finite
	cover,
	$S_B \times_B \selspace {\mathscr F^{n-d,\sigma}_{B}}$, follows from \cite[Proposition 7.7]{EllenbergVW:cohenLenstra}.
Hence, by functoriality of the specialization map, it is enough to 
verify
local constancy of the cohomology for the projection $\left(S_B \times_B
\selspace {\mathscr F^{n-d,\sigma}_{B}} \right)^{\on{log}} \to B$
and that the splitting $\varepsilon$ from
\eqref{equation:splitting-cohomology-map-to-trivial-log-structure}
is compatible with the specialization map.

We first verify local constancy of the cohomology.
Observe that we can write $\left(S_B \times_B
\selspace {\mathscr F^{n-d,\sigma}_{B}} \right)^{\on{log}}$ as the fiber product
$\left( S_B \times_B
\selspace {\mathscr F^{n-d,\sigma}_{B}} \right) \times_B B^{\on{log}}$
where here we give $\left( S_B \times_B \selspace {\mathscr F^{n-d,\sigma}_{B}} \right)$
and $B$ the trivial log structures, corresponding to the $0$ monoid, and $B^{\on{log}}$ the standard log
structure.
By the K\"unneth theorem, whose log cohomology version in our setting follows from proper base
change \cite[Proposition 6.3]{illusie:an-overview}
and the projection formula, it is enough to show the cohomology sheaves
associated to both 
$\left( S_B \times_B \selspace {\mathscr F^{n-d,\sigma}_{B}} \right)$ and
$B^{\on{log}}$ are locally constant.
We have already verified the former above, while the latter follows from
\cite[Theorem 5.2]{illusie:an-overview}, assuming $\ell'$ is invertible on $B$.
Moreover the cohomology $B^{\on{log}}$ is isomorphic to
that of $\mathbb G_m$ over $B$.

We next define the splitting $\varepsilon$ in
\eqref{equation:splitting-cohomology-map-to-trivial-log-structure}
and verify it
is compatible with the specialization map.
Notice that the above description of the cohomology of
$\left(S_T \times_T
\selspace {\mathscr F^{n-d,\sigma}_{T}} \right)^{\on{log}}$
gives an isomorphism of cohomology rings with Frobenius action
\begin{align*}
	H^\bullet\left( \left(S_T \times_T \selspace {\mathscr F^{n-d,\sigma}_{T}}
\right)^{\on{log}}, \mathbb Q_{\ell'}\right)
&\simeq
H^\bullet \left( (S_T \times_T \selspace {\mathscr F^{n-d,\sigma}_{T}}) \times_T
\mathbb G_m, \mathbb Q_{\ell'}\right)
\\
&\simeq
H^\bullet \left( S_T \times_T \selspace {\mathscr F^{n-d,\sigma}_{T}}, \mathbb Q_{\ell'}\right)
\otimes H^\bullet(\mathbb G_m, \mathbb Q_{\ell'}),
\end{align*}
the latter isomorphism via the K\"unneth isomorphism.
Hence, we obtain a map
\begin{align*}
	H^i \left( S_T \times_T \selspace {\mathscr F^{n-d,\sigma}_{T}}, \mathbb Q_{\ell'}\right)
\otimes H^0(\mathbb G_m, \mathbb Q_{\ell'})
\to
H^i\left( \left(S_T \times_T \selspace {\mathscr F^{n-d,\sigma}_{T}}
\right)^{\on{log}}, \mathbb Q_{\ell'}\right),
\end{align*}
where the isomorphism is equivariant for the Frobenius action when $T = \spec \overline{\mathbb
F}_q$.
This gives the desired splitting $\varepsilon$ from 
\eqref{equation:splitting-cohomology-map-to-trivial-log-structure}.
Moreover, the above subspace is compatible with the specialization map, as we wished
to show.
Overall, this reduces us to verifying the final claim when $T =\spec \mathbb C$.

We conclude by verifying the final statement when $T = \spec \mathbb C$.
Let $S^1$ denote the circle.
On the level of Kato-Nakayama spaces, the splitting $\varepsilon$ defined above can be
	obtained from the inclusion 
	$S_T \times_T \selspace {\mathscr F^{n-d,\sigma}_{T}} \to S^1 \times (S_T \times_T \selspace
	{\mathscr F^{n-d,\sigma}_{T}})$, coming from choosing a fixed direction $\tau \in S^1$.
		If we compose with the projection 
$\iota_h: S_T \times_T \selspace {\mathscr F^{n-d,\sigma}_{T}} \to \selspace {\mathscr
F^{n-d,\sigma}_{T}}$
associated to a particular tuple
$(h,\ldots, h)$ over $x \in\selspace {\mathscr G^{d,0,\infty}_{ B}}(\mathbb
F_q)$,
the description from \autoref{lemma:factor-through-trivial-log}
implies that the map
$\iota_h^* \circ \varepsilon\circ \beta^* \circ \alpha^* \circ \delta $
on cohomology is induced by the map of Kato-Nakayama spaces described as
follows:
start with an 
$\asp_{2r}(\mathbb Z/\nu \mathbb Z)$-cover of
$\Sigma^1_{g,f}$ and glue on 
a disc
with $d$ punctures having monodromy around each such puncture given by $h$.
The map $\mathbb U$ is the sum over $h \in c$ of the Poincar\'e duals of these maps on
cohomology, and hence
the composite of
$\eta \circ \varepsilon$
with \eqref{equation:composite-log-scheme-map} 
is Poincar\'e dual to $\mathbb U$.
\end{proof}

By \autoref{proposition:u-map-identification}, the map $\mathbb U_{\overline{\mathbb F}_q}$ is
identified with a map
\begin{align}
	\label{equation:finite-field-cohomology-map}
H^i_{\on{c}}(\selspace {\mathscr
F^{n-d,\sigma}_{{\overline{\mathbb F}_q}}}, \mathbb Q_{\ell'}(n-d))
\to
	H^{i+2d}_{\on{c}} \left(\selspace {\mathscr F^{n,\sigma}_{{\overline{\mathbb
	F}_q}}}, \mathbb Q_{\ell'}(n)\right).
\end{align}

\subsection{Proving the main Frobenius equivariance result}
\label{subsection:main-stabilization-proof}

In this subsection, we prove our main result,
\autoref{theorem:frobenius-equivariance},
that the stabilization map is equivariant for the Frobenius action.

As a preliminary step to connect the version of Selmer spaces where we mark
extra data over $\sigma$ to the version without such marked data, we need to
understand the group action relating these two spaces.
Note that there is an action of
$\mathbb Z/2
\mathbb Z \ltimes (j_*\mathscr F)_{\sigma}$
on $\selspace {\mathscr F^{n-d,\sigma}_{ B}}$
where the $\mathbb Z/2 \mathbb Z$ acts by negation on the fiber
$(j_*\mathscr F)_{\sigma}$
and the copy of $(j_*\mathscr F)_{\sigma}$ acts by translation.
The quotient of $\selspace {\mathscr F^{n-d,\sigma}_{ B}}$ by this action is
$\selspace {\mathscr F^{n-d}_{B}}$, where we no longer include the marked point
$\sigma$.

\begin{lemma}
	\label{lemma:frobenius-equivariance-on-cover}
	Assume $B$ is the spectrum of a complete DVR with residue field $\mathbb F_q$
	and generic characteristic $0$.
	Suppose $\selspace {\mathscr G^{d,0,\infty}_{ B}}(\mathbb F_q)\neq
	\emptyset$.
	The map \eqref{equation:finite-field-cohomology-map} is equivariant for
	the actions of Frobenius and the actions
	of $\mathbb Z/2
\mathbb Z \ltimes (j_*\mathscr F)_{\sigma}$
on both sides.
\end{lemma}
\begin{proof}
The map \eqref{equation:finite-field-cohomology-map}
	is equivariant for the action of Frobenius since it is the composite of
	the dual map \eqref{equation:composite-log-scheme-map} with the
	Frobenius equivariant maps $\varepsilon$ and $\eta$ in
	\eqref{proposition:u-map-identification}.
	Note here we are using that maps of Deligne-Faltings log schemes induce functorial maps
		on their cohomology,
		as follows from functoriality
	of the Kummer \'etale topology,	see \cite[\S 2.1]{illusie:an-overview}.
	The composite map \eqref{equation:finite-field-cohomology-map} is then
	also equivariant for the action of Frobenius because
	$S_{\overline{\mathbb F}_q}$ is defined over $\mathbb F_q$ by
	\autoref{lemma:scheme-over-fq}.

	We conclude by arguing that the action of $\mathbb Z/2 \mathbb Z \ltimes (j_*\mathscr F)_{\sigma}$ is also
	equivariant for the map \eqref{equation:finite-field-cohomology-map}.
	One can identify the action of this group with the action on the fiber
	over $\sigma$.
	The gluing
	map $\mathbb U$ in topology induces an equivariant map on cohomology for this
	group action, and the algebraic map
	\eqref{equation:finite-field-cohomology-map} is identified with the map
	$\mathbb U$ via \autoref{proposition:u-map-identification} and
	\eqref{equation:u-transfer}.
\end{proof}

We are now ready to deduce our main result of this section.
\begin{theorem}
	\label{theorem:frobenius-equivariance}
	Assume $B$ is a complete DVR with residue field $\mathbb F_q$
	and generic characteristic $0$.
	Suppose $\selspace {\mathscr G^{d,0,\infty}_{ B}}(\mathbb F_q)\neq
	\emptyset$.
	Suppose $\nu$ is an odd integer with $2 \nu$ invertible on $B$.
	Suppose $Z$ as in \autoref{notation:curve-notation} has a section
	$\sigma: B \to Z$ and $\mathscr F$, a sheaf of $\mathbb Z/\nu \mathbb
	Z$ modules as in
	\autoref{notation:quadratic-twist-notation}, has trivial inertia along
	$\sigma$.
	Suppose $H$ is a finite $\mathbb Z/\nu \mathbb Z$-module.
For any positive even integer $n$,
	there is a map
	\begin{align}
		\label{equation:equivariance-no-rank}
		H^{2n-p}_{\on{c}}(\selspacemoments  {\mathscr F^n_{\overline{\mathbb
			F}_q}} H, \mathbb
		Q_{\ell'}(n)) &\to H^{2n-p+4}_{\on{c}}(\selspacemoments {\mathscr
			F^{n+2}_{\overline{\mathbb
			F}_q}} H, \mathbb
		Q_{\ell'}(n+2))
	\end{align}
	which is equivariant for the action of Frobenius. 
	There is a positive integer constant $I(H)$, depending on $H$, as well as a positive
	integer constant $J(\mathscr F,H)$, depending on $\mathscr F$ and $H$,
	so that \eqref{equation:equivariance-no-rank} is
	an isomorphism whenever $n
	> I(H)p + J(\mathscr F,H)$.
\end{theorem}
\begin{remark}
	\label{remark:}
	The map \eqref{equation:equivariance-no-rank} is induced from the map
	\eqref{equation:finite-field-cohomology-map}, with $d = 2$ and $n$
	replaced by $n + 2$, via transfer.
\end{remark}
\begin{proof}
First, by \autoref{proposition:degree-order}, since we are working with
	$c \subset G = H \rtimes (\mathbb Z/2 \mathbb Z)$ corresponding to the elements
	of order $2$, we may take the operator $\mathbb U$ to have degree $2$.

	We will only explain
	the proof in the case that $H = \mathbb Z/\nu \mathbb Z$.
	The case of 
	general $H$, where one takes iterated fiber products of the Selmer space
	over the space of quadratic twists, is quite analogous. 
	However, we opt to just explain the case that $H = \mathbb Z/\nu \mathbb
	Z$ to avoid introducing an onslaught of additional notation that does
	not require any new ideas.

	First, we note that the map $\mathbb U_{\overline{\mathbb F}_q}$ is
	equivariant for the action of Frobenius by 
\autoref{lemma:frobenius-equivariance-on-cover}.
When $n > I(H)p +J(\mathscr F, H)$, 
	commutativity of \eqref{equation:u-transfer} implies that 
	$\mathbb U_{\overline{\mathbb F}_q}$ is an isomorphism 
	\begin{align*}
	H^{2n-p}_{\on{c}}(\selspace {\mathscr F^{n,\sigma}_{\overline{\mathbb
			F}_q}}, \mathbb
		Q_{\ell'}(n)) \to H^{2n-p+4}_{\on{c}}(\selspace {\mathscr
				F^{n+2,\sigma}_{\overline{\mathbb
			F}_q}}, \mathbb
		Q_{\ell'}(n+2)),
	\end{align*}
	which uses that the corresponding map $\mathbb U$ over
	$\mathbb C$ is an isomorphism by
	Poincar\'e duality and 
	\autoref{theorem:central-u-implies-cohomology-bound}.
	See
	\autoref{example:specified-hurwitz-coefficient-system} and
	\autoref{example:fiber-product-with-rank-coefficients} for why the
	relevant representations of $B^n_{g,f}$ form coefficient systems.

	Finally, using that the map 
	\eqref{equation:finite-field-cohomology-map}
	is equivariant for the
	action of
	$\mathbb Z/2
	\mathbb Z \ltimes (j_*\mathscr F)_{\sigma}$
	by \autoref{lemma:frobenius-equivariance-on-cover},
	we obtain an induced map on the cohomology of the quotient space by this action
	of $\mathbb Z/2	\mathbb Z \ltimes (j_*\mathscr F)_{\sigma}$.
	By transfer, since we are assuming $2 \nu$ is invertible on $B$, the cohomology of the quotient is
	also equivariant for the action of Frobenius.
	Since the quotient of $\selspace {\mathscr F^{n,\sigma}_{ \overline{\mathbb
	F}_q}}$ by this 
$\mathbb Z/2	\mathbb Z \ltimes (j_*\mathscr F)_{\sigma}$ action is
	$\selspace {\mathscr F^{n}_{\overline{\mathbb F}_q}}$,
	(without trivializations over $\sigma$,)
	we obtain the maps 
	\begin{align*}
		H^{2n-p}_{\on{c}}(\selspace {\mathscr
F^{n}_{\overline{\mathbb F}_q}}, \mathbb Q_{\ell'}(n))
\to
H^{2n-p+4}_{\on{c}} \left(\selspace {\mathscr F^{n+2}_{\overline{\mathbb
	F}_q}}, \mathbb Q_{\ell'}(n+2)\right)
	\end{align*}
	are Frobenius equivariant and moreover are isomorphisms when 
	$n > I(H)p+J(\mathscr F,H)$.
\end{proof}

\section[A normal crossings compactification of Hurwitz spaces]{A normal crossings
compactification of Hurwitz spaces\\ By Dori Bejleri and Aaron Landesman}
\label{section:nc}

The main consequence of this appendix,
\autoref{corollary:hurwitz-compactification}, proves that configuration spaces
of points on a pointed smooth curve, considered earlier in this
paper, have normal crossing compactifications. 
This was crucially used to compare the cohomology of Hurwitz spaces over $\overline{\mathbb F}_q$
with the cohomology of Hurwitz spaces over $\mathbb C$.
Because it is little extra work, we also show that Hurwitz spaces, which are
finite \'etale covers of these configurations spaces, have normal
crossing compactifications. 
In order to achieve this comparison between $\mathbb C$ and $\overline{\mathbb
F}_q$, when dealing with a Hurwitz space for a finite group
$G$, we work over the base $\mathbb Z[1/|G|]$. In particular, our results hold over mixed
characteristic bases.
Additionally, we allow the base curve to be
semistable, and do not require that it is smooth.
We begin by constructing the normal crossings compactifications of configuration spaces and Hurwitz
spaces in
\autoref{subsection:describing-compactification}.
We next introduce various notation for log covers in
\autoref{subsection:cover-notation}.
We then reduce our task to proving a certain map is log smooth in
\autoref{subsection:reducing-to-log-smoothness}.
Finally, we verify the above mentioned map is 
log smooth in
\autoref{subsection:verifying-log-smooth}.

\subsection{The normal crossings compactification via twisted stable
maps}
\label{subsection:describing-compactification}

In order to prove the Hurwitz spaces we consider have a normal crossings
compactification, we first define the relevant compactification, in terms
of twisted stable maps.

\begin{notation}
	\label{notation:stable-maps}
	Let $B$ be a Deligne-Mumford stack and let $\pi: C \to B$ be
	a projective family of nodal
	curves with geometrically connected fibers of genus $g$.
	For each geometric point $b \to B$,
	let $[C_b]$ denote the fundamental class of a fiber of $\pi$ viewed as a $1$-cycle.
	
Fix a divisor $Z \subset C$ which is finite \'etale of degree $d$ over $B$ and
contained in the smooth locus of $C \to B$.
Fix a finite group $G$ whose order is invertible on $B$ and let $[C/G]$ denote the stack quotient of $C$ by the
trivial $G$ action.
The reader may wish to recall the notion of a twisted stable map being balanced
as defined in \cite[Definition 3.2.4]{abramovichV:compactifying-the-space-of-stable-maps}; colloquially this means the stabilizer action on smoothing parameters on
each side of a twisted node are inverse to each other.
Let ${\mathcal K}_{g,n+d}([C/G],1)$ denote the moduli stack of
balanced twisted stable
maps whose $S$-points described as follows:
Given a map $S \to B$ for $S$ a scheme,
${\mathcal K}_{g,n+d}([C/G],1)(S)$
is the groupoid of representable maps $h : \mathcal X \to [C/G]$ 
from an $(n+d)$-pointed balanced twisted curve $\mathcal X$
such that
\begin{enumerate}
	\item $X \to S$ is the coarse space of $\mathcal X$ with map $f: X \to
		C$ induced by $h$,
	\item the fibers of $X \to S$ have genus $g$, and
	\item $(f_s)_*[X_s] = [C_s]$ for each geometric point $s \to S$,
		where $[X_s]$ is the fundamental class of the fiber over $s \to S$.
\end{enumerate}
We note that ${\mathcal K}_{g,n+d}([C/G],1)$
is an algebraic stack proper over $B$ by
\cite[\S8.3 and \S8.4]{abramovichV:compactifying-the-space-of-stable-maps}.

There is an action of $S_d$ permuting the final $d$ marked points of the curve
$\mathcal X$.
The
quotient stack $[{\mathcal K}_{g,n+d}([C/G],1)/S_d]$
parameterizes stable maps with $n$ marked sections as well
as an \'etale degree $d$ divisor contained in the smooth locus and disjoint from the $n$ marked sections.
There is an evaluation map ${\mathcal K}_{g,n+d}([C/G],1) \to C^d_B$ 
to the $d$-fold fiber product over $B$
sending an $(n + d)$-pointed map to the image of the final $d$
sections, and hence we obtain a map
$\pi: [{\mathcal K}_{g,n+d}([C/G],1)/S_d] \to [C^d_B/S_d]$.
If $[Z] : \spec B \to [C^d_B/S_d]$ denotes the $B$-point of $\Conf_{C/B}^d
\subset [C^d_B/S_d]$
corresponding to the finite \'etale degree $d$ divisor $Z$,
we then define
\begin{align*}
	{\mathcal K}_{g,n}([C/G],Z, 1) := [{\mathcal
	K}_{g,n+d}([C/G],1)/S_d] \times_{\pi, [C^d_B/S_d],[Z] } B.
\end{align*}
In other words, 
${\mathcal K}_{g,n}([C/G],Z, 1)$ is the closed substack of 
$[{\mathcal K}_{g,n+d}([C/G],1)/S_d]$ so that the degree $d$ marked
divisor maps to $Z \subset C$.
\end{notation}

The following is the main result of this section, which will lead to a normal
crossing compactification of Hurwitz space in
\autoref{corollary:hurwitz-compactification}. We will later generalize
\autoref{theorem:nc} to nodal curves in \autoref{theorem:nc_nodal}.

\begin{theorem}
	\label{theorem:nc}
	Let $B$ be a regular locally noetherian scheme, $C \to B$ a smooth projective curve with geometrically connected fibers.
	Let $Z \subset C$ be a degree $d$ divisor which is finite \'etale over
	$B$.
	The Deligne-Mumford stack ${\mathcal
	K}_{g,n}([C/G],Z, 1)$ is smooth and proper
	over $B$.
	Moreover, the locus of points in 
${\mathcal K}_{g,n}([C/G],Z,1)$ corresponding to stable maps with smooth source forms a dense open substack of 
${\mathcal K}_{g,n}([C/G],Z, 1)$ with complement a normal crossings divisor.
\end{theorem}
We will prove this in \autoref{subsubsection:nc-proof}.
\begin{remark}
	\label{remark:}
	In the case $Z$ is a disjoint union of sections, (which always holds if
	$Z$ has degree $0$ or $1$,) $B$ is a scheme, and $G$ is trivial, one can verify that 
	${\mathcal K}_{g,n}([C/G],Z,1)$
is in fact a projective scheme, and not just an algebraic stack.
This amounts to verifying that the inertia stack is trivial, which then implies
it is projective because it is known the coarse moduli space is projective 
\cite[Theorem 1.4.1]{abramovichV:compactifying-the-space-of-stable-maps}.
\end{remark}

One of our main motivations for proving \autoref{theorem:nc} is that it
provides a normal crossings compactification of Hurwitz spaces of $G$-covers of $C$. In particular, if we take $G$ to be trivial, it provides
a normal crossings compactification of a configuration space of points in $C - Z$. A normal crossings compactification in the case $C = \mathbb P^1$ and $Z =
\infty$ and $G$ is trivial
was given in an ad hoc fashion in \cite[Lemma 7.6]{EllenbergVW:cohenLenstra}.
When $Z$ is empty and $G$ is trivial,
this normal crossings compactification was given in
\cite{fultonM:a-compactification}.
However, even in
the case $G$ is trivial, $C$ is arbitrary, and $Z$ is nonempty, which is the most important case for the present paper, we do not know of a reference. 
A normal crossings compactification of a variant of our 
Hurwitz spaces was
constructed in \cite[Corollary p.
390-391]{mochizuki:the-geometry-of-the-compactification},
also using log geometry.

\begin{corollary}
	\label{corollary:hurwitz-compactification}
	With notation as in \autoref{definition:fixed-hur} and
	\autoref{definition:pointed-hurwitz-space},
	both the Hurwitz stack
$\hur G n Z {\mathcal S} C B$
and the pointed Hurwitz scheme $\hur G n {\sigma \subset Z} {\mathcal S} C B$
are dense opens inside a Deligne Mumford stack which is smooth and proper over
$B$, such that the complementary divisor is a normal crossings divisor.
In particular, taking $U := C - Z$, the scheme $\conf n U B$
as defined in \autoref{notation:curve-notation}
	is a dense open subscheme of 
	a smooth proper Deligne-Mumford stack,
	such that the complement is a normal crossings divisor.
\end{corollary}
\begin{proof}
	There is an action of $S_n$ on the 
	stack ${\mathcal K}_{g,n}([C/G],Z,1)$ which permutes the $n$ marked points.
	Consider the quotient stack 
	$[{\mathcal K}_{g,n}([C/G],Z,1)/S_n].$
	An appropriate union of components of this quotient stack contains a
	dense open substack parameterizing those smooth
	covers of $C$, which precisely correspond to points of
	$\hur G n Z {\mathcal S} C B$.
	The complement is a normal crossings divisor by \autoref{theorem:nc}.
	In the case we mark a section $\sigma \subset Z$ and mark a point of the
	cover over
	$\sigma$, we can form an appropriate finite \'etale cover of 
	$[{\mathcal K}_{g,n}([C/G],Z,1)/S_n]$ corresponding to marking a
	section over $\sigma$, (similar to the construction in 
	\autoref{definition:pointed-hurwitz-space}),
	and a union of components of this cover contains 
	$\hur G n {\sigma \subset Z} {\mathcal S} C B$
	as a dense open substack with complement a normal crossings divisor.

	As a special case, taking $G = \id$, we obtain that 
$\conf n U B$ forms a dense open subscheme of $[{\mathcal K}_{g,n}(C,Z,1)/S_n]$,
	whose complement is a normal crossings divisor.
\end{proof}

\subsection{Notation for log covers}
\label{subsection:cover-notation}

In order to prove \autoref{theorem:nc}, we use log deformation theory. 
The starting point is the observation that every twisted stable map as in
\autoref{notation:stable-maps}
	can be endowed with the structure of a map of log stacks and this induces a log structure on the space of twisted stable maps itself.
To carefully describe these log structures, we require a hefty amount of
notation.
We begin by describing a log structure on the moduli stack of curves, which
corresponds to the divisor parameterizing singular curves.
Throughout this section, we will assume all log structures appearing are fine.

\begin{notation}
	\label{notation:log-curves}
	Let $\overline{\mathscr M}_{g,n+\underline{d}}^{\on{log}}$ denote the
	log stack whose underlying stack is $[\overline{\mathscr M}_{g,n +
	d}/S_d]$, where $S_d$ acts on the final $d$ marked points, over $\spec
	\mathbb Z[1/|G|]$; the log structure on $\overline{\mathscr M}_{g,n+\underline{d}}^{\on{log}}$
	is given by the
reduced divisor parameterizing singular curves. We note that the points of 
the underlying stack $\overline{\mathscr M}_{g,n+\underline{d}}$
of $\overline{\mathscr M}_{g,n+\underline{d}}^{\on{log}}$
parameterize
tuples $(C, p_1, \ldots, p_n, Z)$, where $C$ is a nodal curve, $p_i$ are marked
smooth points, and $Z$ is a degree $d$ \'etale divisor contained in the smooth
locus such that $K_C + Z + \sum p_i$ is ample.
When $n = 0$, 
we let $\mathscr C$ denote the universal curve over $\overline{\mathscr
M}_{g,\underline{d}}$, and let
$\mathscr Z \subset \mathscr C$ denote the distinguished degree $d$ divisor.
We let the finite group $G$ act trivially on $\mathscr C$ and $[\mathscr C/G]$
denote the quotient stack.
\end{notation}

We next introduce notation to describe various aspects of the geometric points of the stack
${\mathcal K}_{g,n}([\mathscr C/G], \mathscr Z,1)$.
See \autoref{figure:deformation-components} for a picture depicting some of this
notation.

\begin{notation}
	\label{notation:divisors-and-coarse-space}
	Let $S$ be a scheme.
	Let $[h: \mathcal X \to [C/G],
	\mathcal{D}+\mathcal{E}] \in
{\mathcal K}_{g,n}([\mathscr C/G],\mathscr  Z,1)(S)$ be a point; here we use 
$C$ and $Z$ denote the pullbacks of $\mathscr C$ and $\mathscr Z$ to $S$, 
$\mathcal X$ to
denote the twisted curve, $\mathcal{D} \subset \mathcal X$ is a closed
substack which is a gerbe over the $n$
sections in the smooth locus of $\mathcal X$, and $\mathcal{E} \subset \mathcal
X$ a substack in the smooth locus of $\mathcal X$ which is a gerbe over
the degree $d$ divisor mapping to $Z \subset C$.
	We also use $X$ to denote the coarse space of $\mathcal X$, and we will 
	write $E \subset X$ for the degree $d$ subscheme of $X$ corresponding to
	$\mathcal{E}\subset \mathcal X$ and $D \subset X$ the subscheme corresponding to
	$\mathcal{D} \subset \mathcal X$. 
	These both lie in the smooth locus of $X$ and $E$ maps isomorphically to
	$Z$.

	We use $\pi: \mathcal X \to X$ and $\psi: [C/G] \to C$ to denote the
	coarse space maps, and
	$f : X \to C$
	to denote the map on coarse spaces induced by $h$.
\end{notation}
\begin{remark}
	\label{remark:}
	From now on, following \autoref{notation:divisors-and-coarse-space}, we
	will use the notation $C \to S$ for the
	target of an $S$-point of our stable maps.
	(In particular, this is not to be confused with $C \to B$, which we are
		replacing by $\mathscr C \to \overline{\mathscr
	M}_{g,\underline{d}}$ and $C$ is the pullback of $\mathscr C$ to $S$.) 
	We note that this is a slight conflict of notation with
	\autoref{notation:stable-maps}, but the notation $C \to B$ there will not come
	up for us again in the remainder of this section.
\end{remark}

\begin{notation}
	\label{notation:stable-compactification}
	Continuing to use notation as in
	\autoref{notation:divisors-and-coarse-space},
	we suppose $S$ is of the form
	$V = \spec k$, for $k$ an algebraically closed field. 
	By \autoref{lemma:base-description}, 
	we can write
	$X$ in the form $X = P \cup \widetilde{C}$ satisfying the conditions
	from \autoref{lemma:base-description}. In particular, $P$ is the
	union of the irreducible components contracted under $f$.

	We also let $W \subset P$ denote the union of irreducible components of
	$P$, whose connected components consist of $W_j \subset P_j$ defined as
	follows:
	Let $P_j \subset P$ denote a connected component of $P$ which joins $u,v
	\in \widetilde{C}$ mapping to a node in $C$. 
	We take $W_j \subset P_j$ to be the union of
	irreducible components of $P_j$ which are not directly between $u$ and $v$; more
	formally, we can say these irreducible components of $P_j$ in $W_j$
	do not
	correspond to the vertices of the dual graph of $P_j$ which lie in a
	minimal path joining the
	irreducible
	component meeting $u$ to the irreducible component meeting $v$.
	For each $P_j \subset P$ a connected component of $P$ mapping to a
	smooth point of $C$, we take $W_j
	\subset P_j$ to be the union of the irreducible components of $P_j$
	which are not directly between the irreducible component on which $E$ lies and the
	irreducible component meeting $\widetilde{C}$; more formally 
	the components of $W_j$ do not correspond to the vertices of the dual graph of $P_j$ which
	lie in a minimal path joining the component on which $E$ lies to the
	component meeting $\widetilde{C}$.

	We define $Y \subset X$ to denote the union of irreducible
	components of $X$ which are not contained in $W$.
	Define $\rho: X \to Y$ and $t: Y \to C$ so that $f = t \circ \rho$ and let
	$i : W \to X$ denote the inclusion
	For $p \in X$, let $k_p$ denote the skyscraper sheaf at a point $p$.
\end{notation}
\begin{figure}
	\includegraphics[width=5cm, height=6cm]{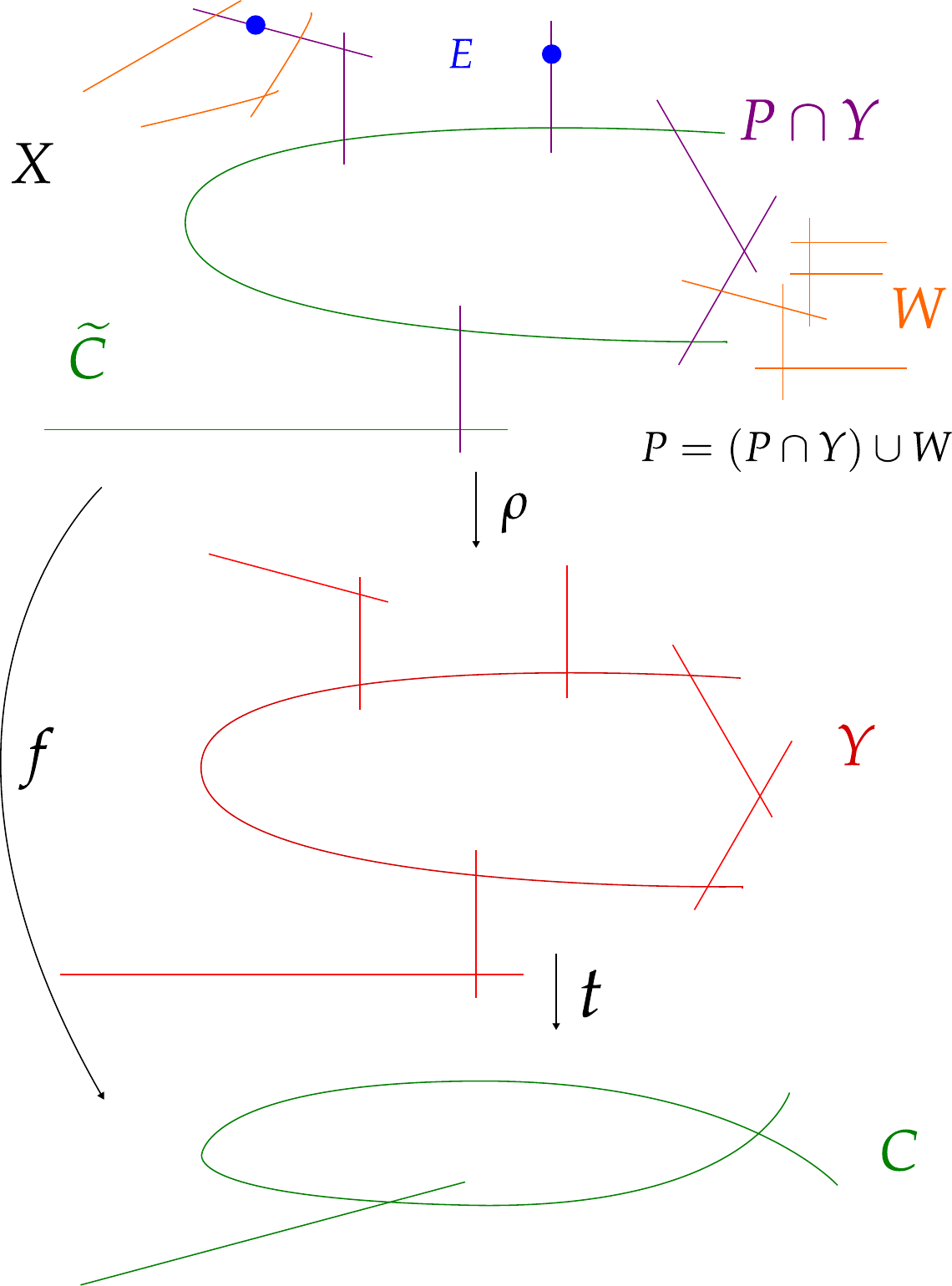}
\caption{
	A picture depicting the names we have given to various parts of the
	coarse space $X$ of $\mathcal X$.}
\label{figure:deformation-components}
\end{figure}

The following lemma was to make sense of
\autoref{notation:stable-compactification} above.

\begin{lemma}
	\label{lemma:base-description}
	Using notation for $h: \mathcal X \to [C/G], \psi: [C/G] \to C, \pi:
	\mathcal X \to X$ and $f: X \to C$ as in \autoref{notation:stable-compactification},
	we have $f \circ \pi = \psi \circ h$. Moreover,
	$X$ is of the form $X = P \cup \widetilde{C}$, where $P$ and
	$\widetilde{C}$ satisfy the following conditions:
	\begin{enumerate}
		\item $\widetilde{C}$ is a partial normalization of $C$ at a finite set $N$ of its
			nodes,
		\item $P$ is a genus $0$ semistable curve,
		\item $P$ is contracted under the map $X \to C$, and
		\item any connected component of $P$ is either contracted to a
	smooth point of $C$, in which case it meets $\widetilde{C}$ at a single
	smooth point, or the component is contracted to a node of $C$, in which case it meets
	$\widetilde{C}$ at both preimages of the node.
	\end{enumerate}
\end{lemma}
\begin{proof}
	We have that $f \circ \pi = \psi \circ h$ by the universal property of
	the coarse space $X$.
We now show $X = \widetilde{C} \cup P$ satisfies the conditions as in the statement.
	Since $X \to C$ has degree $1$ on each component of $C$, it must be the
	union of a birational map with several contracted components, and hence
	must be of the form $\widetilde{C} \cup P$ for $\widetilde{C}$ a partial
	normalization of $C$ and $P$ the components which are contracted under
	the map.
	To conclude, we wish to show $P$ has properties $(2)$ and $(4)$.
	First, since the genus of $X$ agrees with the genus of $C$, each
	connected component of $P$ must have genus $0$.
	Continuing to use that the genus of $X$ agrees with the genus of $C$,
	if a connected component of $P$ is contracted to a node in $C$, it must meet the two preimages of the
	node in $\widetilde{C}$ nodally. Similarly, if a connected component $P$ is contracted to a smooth point, the
	only way $X$ has the same genus as $C$ is if that component of $P$ meets the
	preimage of that point in a single node, as claimed.
\end{proof}

Using the preceding notation, we are now ready to describe the relevant log
structures on our twisted curves.

\begin{notation}
	\label{notation:log-stable-maps}
	Using \autoref{notation:stable-maps}, \autoref{notation:log-curves}, and
\autoref{notation:divisors-and-coarse-space},
let
${\mathcal K}_{g,n}([\mathscr C/G],\mathscr Z, 1)^{\on{log}}$, 
denote the log stack
whose underlying stack is 
${\mathcal K}_{g,n}([\mathscr C/G],\mathscr Z, 1)$ with the log structure we
describe next.
For a scheme $S$, suppose we have an $S$-point of ${\mathcal K}_{g,n}([\mathscr C/G],\mathscr Z, 1)$,
corresponding to a twisted stable map
$\mathcal X \to [C/G]$ over $S$.
We endow $(\mathcal X, \mathcal M_{\mathcal X}) \to (S, \mathcal M_S)$ with the log
structure described in
\cite[\S3.10]{Olsson:log-twisted-curves}
obtained by viewing $\mathcal X$ as an $n$ pointed twisted curve together with a
degree $d$ divisor (so that, in particular, there is a copy of
	$\underline{\mathbb N}$ in
$M_{\mathcal X}$ over the $n$ marked gerbes
and the degree $d$ marked gerbe on $\mathcal X$).
Similarly, $C \to S$ has a canonical log structure from
\cite[\S3.10]{Olsson:log-twisted-curves}, and we endow $[C/G]$ with the pullback of this log structure along $[C/G] \to
C$ amalgamated with the log structure induced by the Cartier divisor $Z$ (so in
particular there is a copy of $\mathbb{N}$ along the preimage of $Z$ in
$[C/G]$).
We denote this log structure by $([C/G], \mathcal M'_{[C/G]}) \to (S, \mathcal M'_S)$.

In general $\mathcal M'_S$ may be different from $\mathcal M_S$ when $\mathcal X$ has
more nodes than $C$ or has twisted nodes lying over the nodes of $C$. 
If $X$ denotes the coarse space of $\mathcal X$ with its log structure
$\mathcal M_X$ (including the $n$ points and degree $d$ divisor), and $C$ has
log structure $\mathcal M'_C$ (including the degree $d$ divisor), then $f$ has the structure of a log map $(f,f^\flat):(X, \mathcal M_X) \to (C,
\mathcal M'_C)$. 
We now describe the structure of this log map, see also \cite[Theorem B.6]{AMW_log_relative}.
First, after replacing
$S$ with an \'etale cover,
so that $f$ factors as $X \to Y \to C$ where $Y \to
C$ is a
composition of log blowups of $C$ and $X \to Y$ is a contraction of trees of rational curves lying over smooth unmarked points of $Y$. 
Using notation which restricts on geometric fibers to that in \autoref{notation:stable-compactification} and
\autoref{figure:deformation-components}, $Y \to C$ is a sequence of log blowups
of nodes and expansions of marked sections which contracts the chains of rational
curves denoted $P \cap Y$ and $X \to Y$ contracts
the trees of rational curves
denoted $W$. Now $Y \to C$ is a morphism of log schemes by
construction and $X \to Y$ is a morphism of log schemes since $W$ lies over the
strict locus of $Y$ and $X \to Y$ is an isomorphism away from $W$. Thus the
composition is a morphism of log schemes.
Then, by composing the coarse space map $(\mathcal X, \mathcal M_{\mathcal X})
\to (X, \mathcal M_X)$ with the above maps, we have a map of log stacks
$(\mathcal X, \mathcal M_{\mathcal X}) \to (C, \mathcal M'_C)$
over $(S, \mathcal M_{S}) \to (S, \mathcal M'_{S})$. 

Since the log
structure on $[C/G]$ is pulled back from the 
log structure on $C$,
we obtain a corresponding commutative diagram
\begin{equation}
	\label{equation:}
	\begin{tikzcd} 
		(\mathcal X, \mathcal M_{\mathcal X}) \ar {r} \ar {d} & ([C/G], \mathcal M'_{[C/G]}) \ar {d} \\
		(S, \mathcal M_S) \ar {r} & (S, \mathcal M'_S).
\end{tikzcd}\end{equation}
The $S$-points of 
${\mathcal K}_{g,n}([\mathscr C/G],\mathscr Z, 1)^{\on{log}}$, 
comprise all of the above data, with the log structure on $S$ for such an $S$-point given by $(S,\mathcal M_S)$. 

The injective map $\mathcal M_S' \to \mathcal M_S$ of locally free log structures is not necessarily saturated due to the
presence of twisted nodes in $\mathcal X$ lying over nodes of $C$. We let
$\mathcal M_S' \rightarrow \mathcal M_S'' \to \mathcal M_S$ denote its
saturation. Then $\mathcal M_S' \hookrightarrow \mathcal M_S''$ is a simple
extension 
(see, for example, \cite[Definition 1.5]{Olsson:log-twisted-curves})
of locally free log structures. 
\end{notation}
\subsection{Reducing to log smoothness}
\label{subsection:reducing-to-log-smoothness}

In this subsection we will show how log smoothness of the map
${\mathcal K}_{g,n}([\mathscr C/G],\mathscr Z, 1)^{\on{log}} \to \overline{\mathscr
M}_{g, \underline{d}}^{\on{log}}$ implies
our main result, \autoref{theorem:nc}.
We also deduce a generalization of 
\autoref{theorem:nc} where we allow the curve there to be nodal.

\begin{proposition}
	\label{proposition:log-smooth}
	With notation as in \autoref{notation:log-stable-maps},
	the log algebraic stack ${\mathcal K}_{g,n}([\mathscr C/G],\mathscr Z, 1)^{\on{log}}$, 
	is log smooth over $\overline{\mathscr
M}_{g, \underline{d}}^{\on{log}}$.
\end{proposition}

We will return to the proof of \autoref{proposition:log-smooth} in
\autoref{subsubsection:proof-log-smooth}. 

We next record a version of \autoref{theorem:nc} for nodal curves. 
The reader may refer to
\autoref{notation:log-curves} and \autoref{notation:log-stable-maps}
for notation used in the next statement. We say that a log smooth morphism is
{\em semistable} if it is saturated and the source and target are regular with
log structure given by a normal crossings divisors, see \cite[Remark
3.6.6]{Illusie_temkin_logsmooth}.
The reader may also wish to consult \cite[Definition 0.1]{abramovich_karu} and
\cite[Subsection 4.2.1]{alt_semistable}.

\begin{theorem}\label{theorem:nc_nodal} ${\mathcal K}_{g,n}([\mathscr
	C/G],\mathscr Z, 1)^{\on{log}}$ is a normal crossings compactification
	of the locus of points corresponding to stable maps with smooth source and the log structure induced by the complementary divisor.
	Moreover, there is a factorization ${\mathcal K}_{g,n}([\mathscr
	C/G],\mathscr Z, 1)^{\on{log}} \xrightarrow{\alpha} \widetilde{\mathscr M}
	\xrightarrow{\beta} \overline{\mathscr M}_{g,
	\underline{d}}^{\on{log}}$,
	where $\alpha$ is semistable and $\beta$ is proper, quasifinite, log
	\'etale, and birational on each component of the source. 
\end{theorem}

\begin{remark}
	In the statement of \autoref{theorem:nc_nodal}, $\widetilde{\mathscr M}$ is a union of components of the stack of simple extensions of log structures over $\overline{\mathscr M}_{g,\underline{d}}$ as in \cite[Section 5.2]{Olsson:log-twisted-curves} and it parameterizes certain twisted curves by the proof of \cite[Theorem 1.10]{Olsson:log-twisted-curves}. 
\end{remark}

\begin{remark}
	We note that when $|G|=1$,
	\autoref{theorem:nc_nodal} reduces to the well known statement that the
	forgetful map $\overline{\mathscr M}_{g,n+d}^{\on{log}} \to
	\overline{\mathscr M}_{g,d}^{\on{log}}$ is semistable, where the moduli
	spaces of curves are equipped with their boundary log structures,
	parameterizing singular curves. The fiber over a geometric point
	representing a curve $(C, p_{n+1}, \ldots, p_{n+d})$ is a log smooth
	compactification of the configuration space of $n$ points on $C^{sm}
	\setminus \{p_{n+1}, \ldots, p_{n+d}\}$. This agrees with the Fulton-MacPherson compactification 
	given in
	\cite{fultonM:a-compactification}
	when $C$ is smooth and $d = 0$. 
\end{remark}

\subsubsection{Proof of \autoref{theorem:nc}, \autoref{theorem:nc_nodal}}
\label{subsubsection:nc-proof}
We begin by explaining why \autoref{theorem:nc} and the first part of \autoref{theorem:nc_nodal} follow from 
\autoref{proposition:log-smooth}. Let $C \to S$ denote either 
\begin{enumerate}
	\item the family from
\autoref{theorem:nc} where $S = B$ is regular and has the trivial log structure or 
\item
the pullback of
the universal family over 
$\overline{\mathscr M}_{g,\underline{d}}^{\on{log}}$ along some strict map $S \to
\overline{\mathscr M}_{g,\underline{d}}^{\on{log}}$
from a log scheme $S$
whose map of underlying stacks is \'etale.
\end{enumerate}

We now verify that $S$ is log regular in the above two cases.
Using 
\cite[7.3(b)]{illusie:an-overview}, the log scheme $S$ log
regular in case $(1)$.
In case $(2)$, note that $\spec \mathbb Z[1/|G|]$ with the trivial log
structure is log regular by 
\cite[7.3(b)]{illusie:an-overview}.
Since 
$\overline{\mathscr M}_{g,\underline{d}}$
is log smooth over 
$\spec \mathbb Z[1/|G|]$, we obtain $S$ is also log smooth over 
$\spec \mathbb Z[1/|G|]$. Hence $S$ is log regular by 
\cite[7.3(c)]{illusie:an-overview}.

We next show stable maps to such $C
\to S$ as above form a normal
crossings compactification of the locus of such maps with smooth source.
Using 
\autoref{proposition:log-smooth},
we find that 
${\mathcal K}_{g,n}([C/G],Z, 1)^{\on{log}}$ is log smooth over $S$.
By \cite[7.3(c)]{illusie:an-overview},
${\mathcal K}_{g,n}([C/G],Z, 1)^{\on{log}}$ is log regular.
Note also that the log structure defined on
${\mathcal K}_{g,n}([C/G],Z, 1)^{\on{log}}$ coming from the divisor
parameterizing singular covers is pulled back from that on $\overline{\mathscr
M}_{g, n+\underline{d}}$, as follows from \cite[Theorem 1.10]{Olsson:log-twisted-curves}
and the proof of 
\cite[Lemma 5.1]{Olsson:log-twisted-curves}.
At a geometric point $x$ of ${\mathcal K}_{g,n}([C/G],Z, 1)^{\on{log}}$, the characteristic monoid of the 
log structure described in \cite[\S3.10]{Olsson:log-twisted-curves}
and the log structure on $S$ is described two lines before
\cite[(3.6.6)]{Olsson:log-twisted-curves}. This log structure is identified with
$\underline{\mathbb N}^{n(x)}$, where $n(x)$ is the number of nodes of
the twisted curve $\mathcal X$ corresponding to the point $x$, so this log structure is locally free.
Then, by \cite[7.3(b)]{illusie:an-overview},
we obtain that the log structure on 
${\mathcal K}_{g,n}([C/G],Z, 1)^{\on{log}}$ is defined by a normal crossings
divisor whose complement is the locus of triviality of the log structure and
${\mathcal K}_{g,n}([C/G],Z, 1)$ is regular.
Since the open subset of triviality of the log structure on 
${\mathcal K}_{g,n}([C/G],Z, 1)^{\on{log}}$
is precisely the locus
of covers of curves with smooth source, we find that
the above normal
crossings divisor is that parameterizing the locus of covers where the source is
singular.
Finally, the fact that 
the locus of points in 
${\mathcal K}_{g,n}([C/G],Z, 1)$ corresponding to stable maps with
smooth source forms a dense open of 
${\mathcal K}_{g,n}([C/G],Z,1)$ follows from 
\cite[7.3(d)]{illusie:an-overview}. This completes the proof of \autoref{theorem:nc} and the first part of \autoref{theorem:nc_nodal}. 

We conclude by now proving the second part of \autoref{theorem:nc_nodal}. By the last
paragraph of \autoref{notation:log-stable-maps}, any $S$-point of ${\mathcal
K}_{g,n}([C/G],Z,1)$ induces a simple extension $\mathcal M_S' \hookrightarrow
\mathcal M_S''$ where $\mathcal M_S'$ is the pullback of the log structure of
$\overline{\mathscr M}_{g,\underline{d}}^{\on{log}}$ along $S \to
\overline{\mathscr M}_{g,\underline{d}}^{\on{log}}$. Thus there is a map
${\mathcal K}_{g,n}([C/G],Z,1)^{\on{log}}$ to a union of connected components of the stack of simple extensions of
log structures (\cite[Section 5.2]{Olsson:log-twisted-curves}) over
$\overline{\mathscr M}_{g,\underline{d}}^{\on{log}}$, which we will denote by
$\widetilde{\mathscr M}$. Note that by representability of the map $\mathcal X
\to [C/G]$, the order of the simple extension is bounded by $|G|$. Then
$\widetilde{\mathscr M} \to \overline{\mathscr M}_{g,\underline{d}}^{\on{log}}$
is proper, quasi-finite, log \'etale, and birational on each component of the source by \cite[Lemma 5.3(ii)]{Olsson:log-twisted-curves} and \cite[Proposition 3.4]{Kato1989}. 
It follows that ${\mathcal K}_{g,n}([C/G],Z,1)^{\on{log}} \to \widetilde{\mathscr
M}$ is a log smooth and saturated morphism where the source and target are regular with normal crossings log structures by the previous paragraph. 
This completes the proof. 
\qed

\vspace{.2cm}

\subsection{Verifying log smoothness}
\label{subsection:verifying-log-smooth}

In the remainder of this section, we prove \autoref{proposition:log-smooth},
stating that ${\mathcal K}_{g,n}([\mathscr C/G], \mathscr Z,1)^{\on{log}} \to
\overline{\mathscr M}_{g,\underline{d}}^{\on{log}}$ is log smooth,
which will also complete the proof of \autoref{theorem:nc}.

We will approach 
\autoref{proposition:log-smooth} via deformation
theory.
To begin understanding the deformation theory,
we next describe the map on log cotangent sheaves associated to
$X \to C$. 
\begin{remark}
	\label{remark:cotangent-map}
	Given a geometric point $V$ of
	$\mathcal K_{g,n}([C/G],Z,1)^{\on{log}}$,
	using notation from \autoref{notation:stable-compactification} and \autoref{notation:log-stable-maps},
	(where we use $V$ for what is called $S$ there,)
	there is an associated map $(f, f^\flat): (X, \mathcal M_{X}) \to (C, \mathcal M'_{C})$.
This induces a map $f^* \Omega^{\on{log}}_{C/V} \to \Omega^{\on{log}}_{X/V}$ which we
now describe as a composition of three maps.
First, recall that $(C, \mathcal M'_C)$ is a log curve over $(V, \mathcal M'_V)$.
Let $\mathcal M_C$ denote the pullback of $\mathcal M'_C$ along $(V, \mathcal
M_V) \to (V, \mathcal M'_V)$.
Then, the map $(X, \mathcal M_{X}) \to (C, \mathcal M'_{C})$ factors through
$(X, \mathcal M_{X}) \to (C, \mathcal M_{C})$ and the two versions of the
relative logarithmic sheaf of differentials
$\Omega^{\on{log}}_{C/V}$ for these two log structures $\mathcal M_C$ and $\mathcal
M'_{C}$ are isomorphic.
Hence, to describe
$f^* \Omega^{\on{log}}_{C/V} \to \Omega^{\on{log}}_{X/V}$
we can endow $C$ with the log structure 
$\mathcal M_{C}$ so that the map of log schemes $(X, \mathcal M_X) \to (C,
\mathcal M_C)$ is over the fixed base $(V,
\mathcal M_V)$.
First, there is a map 
$(a, a^\flat) : (X, \mathcal M_{X}) \to (X, \mathcal M_{X}^d)$, where $\mathcal M_X^d$ is the
log structure on $X$ from
\cite[\S3.10]{Olsson:log-twisted-curves}
obtained by forgetting the $n$ marked points and
only remembering the degree $d$ divisor.
Next, there is a map
$(\rho, \rho^\flat) : (X, \mathcal M_{X}^d) \to (Y, \mathcal M_{Y})$, for $Y$ as in
\autoref{notation:stable-compactification} and $\mathcal M_Y$ the log
structure on $Y$, including the degree $d$ divisor, but not including any of the $n$
marked points.
And finally, we have a map $(t, t^\flat): (Y, \mathcal M_Y) \to (C, \mathcal
M_C)$.
Note that $f = t \circ \rho \circ a$.
We can now identify the map 
$f^* \Omega^{\on{log}}_{C/V} \to \Omega^{\on{log}}_{X/V}$ as the composite map
\begin{align}
	\label{equation:composite-log-differentials}
	a^* \rho^* t^* \Omega^{\on{log}}_{C/V} \to a^* \rho^*
	\Omega^{\on{log}}_{Y/V} \to a^*
	\Omega^{\on{log},d}_{X/V} \to \Omega^{\on{log}}_{X/V}
\end{align}
where $\Omega^{\on{log},d}_{X/V}$ denotes the relative sheaf of logarithmic
differentials associated to the log
structure $\mathcal M_X^d$.
Using \cite[Proposition 1.13]{kato:log-smooth-deformation-theory-and-moduli},
and the fact that the identification there is functorial for maps of log
schemes,
we can identify \eqref{equation:composite-log-differentials} with the sequence of maps
\begin{equation}
	\label{equation:composite-differentials}
	f^* \omega_{C/V}(Z) = \rho^* t^* \omega_{C/V}(Z) \xrightarrow{\alpha} \rho^*
	\omega_{Y/V}(E) \xrightarrow{\varepsilon} \omega_{X/V}(E)
	\xrightarrow{\delta} \omega_{X/V}(D+E).
\end{equation}
We denote the composite map in \eqref{equation:composite-differentials} by $\phi$.
\end{remark}

To better understand the deformation theory associated to a stable map, our
first step will be to understand the map $t^*\omega_{C/V}(Z) \to
\omega_{Y/V}(E)$, whose pullback under $\rho$ is the first map in \eqref{equation:composite-differentials}. 

\begin{lemma}
	\label{lemma:differential-computation}
	For $t: Y \to C$ as in \autoref{notation:stable-compactification}, there is
	an isomorphism $\omega_{C/V}(Z) \simeq t_* \omega_{Y/V}(E)$
	as well as an isomorphism $t^*\omega_{C/V}(Z) \simeq \omega_{Y/V}(E)$.
\end{lemma}
\begin{proof}
	Write $Y = \widetilde{C} \cup Q$, for $Q$ the union of components of $Y$
	not contained in $\widetilde{C}$, for $\widetilde{C}$ as in
	\autoref{notation:stable-compactification}. 
	For $s: \widetilde{C} \coprod Q \to Y$, we have an exact sequence
	\begin{equation}
		\label{equation:differential-restriction}
		\begin{tikzcd}
			0 \ar {r} & \omega_{Y/V}(E) \ar {r} &
			s_* (\omega_{Y/V}(E)|_{\widetilde{C}})
			\oplus s_*(\omega_{Y/V}(E)|_{Q})   \ar {r} & \mathscr O|_{\widetilde{C} \cap Q}  \ar {r}
			& 0.
	\end{tikzcd}\end{equation}
	We can think of this sequence as expressing 
	a local section of $\omega_{Y/V}(E)$ as a log differential on the normalization of $Y$ with poles along $E$ and
	poles along the preimages of the nodes whose corresponding residues
	sum to zero.
	Observe that $\omega_{Y/V}|_{\widetilde{C}} \simeq
	\omega_{\widetilde{C}/V}(\widetilde{C} \cap Q)$  
	and $\omega_{Y/V}|_{Q} \simeq
	\omega_{Q/V}(\widetilde{C} \cap Q)$.
	Hence, pushing forward \eqref{equation:differential-restriction} along $t$, we get an exact sequence
	\begin{equation}
		\label{equation:pushforward-restriction}
		\begin{tikzcd}
			0 \ar {r} & t_*\omega_{Y/V}(E) \ar {r} &
		t_* s_* (\omega_{\widetilde{C}/V}(\widetilde{C} \cap Q +
	E|_{\widetilde{C}} ))
			\oplus t_* s_*\omega_{Q/V}(\widetilde{C} \cap
		Q+E|_{Q})   \ar {r} & t_* \mathscr O|_{\widetilde{C} \cap
		Q}.
			& 
	\end{tikzcd}\end{equation}

	Now, let $M := t(Q)$. Note that 
	$t_* s_*(\omega_{Q/V}(\widetilde{C} \cap Q + E|_Q))$ is supported on
	$M$, which is a disjoint union of
points. By construction of $Q$, using
\autoref{lemma:base-description},
each connected
component of $Q$ is a chain of $\mathbb P^1$'s and $\widetilde{C} \cap Q + E|_Q$
consists of a degree two subscheme on each such connected component, with a
degree $1$ point on each component on either end of the chain.
Since the dualizing sheaf of $\mathbb P^1$ has degree $-2$, this allows us to identify $\omega_{Q/V}(\widetilde{C} \cap Q) \simeq \mathscr
O_Q$ and hence
$t_* s_*(\omega_{Q/V}(\widetilde{C} \cap Q))$ is identified with $s_* t_* \mathscr
O_{Q}$. This is a skyscraper sheaf supported on $M$, which we
denote $k_{M}$.
Hence, the above sequence \eqref{equation:pushforward-restriction} becomes 
	\begin{equation}
		\label{equation:}
		\begin{tikzcd}
			0 \ar {r} & t_*\omega_{Y/V}(E) \ar {r} &
			t_* s_* \omega_{\widetilde{C}/V}(\widetilde{C} \cap Q+
			E|_{\widetilde{C}} )
			\oplus k_{M}  \ar {r}{\mu} & t_* \mathscr
			O_{\widetilde{C} \cap Q} & 
	\end{tikzcd}\end{equation}

	We claim that 
	this sequence expresses the condition that $t_* \omega_{Y/V}(E)$ is the
	subsheaf of 
	$t_* s_* \omega_{\widetilde{C}/V}(\widetilde{C} \cap Q +
	E|_{\widetilde{C}} )$ whose poles at
	preimages of a given node along the normalization map $ t \circ
	s|_{\widetilde{C}}: \widetilde{C}\to C$ agree. Since
	$\omega_{C/V}(Z)$ also has this description, this will yield an
	identification
	$t_* \omega_{Y/V}(E) \simeq \omega_{C/V}(Z)$.
	To verify our claim above, there are two cases. 
	The easier case occurs in the neighborhood of a point of
	$\widetilde{C}  \cap Q$ mapping to a smooth point of $C$.
	Then, the map locally in a small neighborhood $U$ of such a point $p$ is identified with
$t_* s_* \omega_{\widetilde{C}/V}(\widetilde{C} \cap Q +
	E|_{\widetilde{C}} )|_U \oplus k_p \to k_p, (a,b) \mapsto a -
	b$, and the kernel is
$t_* s_* \omega_{\widetilde{C}/V}(\widetilde{C} \cap Q +
E|_{\widetilde{C}} )|_U$, as claimed.
The more difficult case is to compute the kernel at a nodal point of $C$.
Here, the fiber of $\mu$ is identified with a map $k^{\oplus 2} \oplus k \to
	k^{\oplus 2}$ given by $(a,c,b) \mapsto (a-b, b-c)$. The first two
	copies of $k$ on the source correspond to the residues of the sheaf on the two preimages
	of the node in $\widetilde{C}$ and the third copy of $k$ corresponds to
	the section on the contracted component of $Q$. Lying in the kernel of
	this map expresses the condition that the residues on each side of
	$\widetilde{C}$ agree with the value on the contracted component of $Q$.
	Said another way, the values of the residues on each side of
	$\widetilde{C}$ agree.
	This verifies our claim.

	Finally,
	since we showed above the restriction of
	$\omega_{Y/V}(E)$ to any fiber of $Y \to C$ is the structure sheaf, the
	adjoint $t^*\omega_{C/V}(Z) \to \omega_{Y/V}(E)$
	to our isomorphism $\omega_{C/V}(Z) \simeq t_*\omega_{Y/V}(E)$
	restricts to an isomorphism on each contracted fiber of $Y \to C$. 
	Since 
$t^*\omega_{C/V}(Z) \to \omega_{Y/V}(E)$
	also restricts to an isomorphism on $\widetilde{C}$, it is an
	isomorphism.
\end{proof}

We will see later that the log cotangent complex associated to a geometric point of 
${\mathcal K}_{g,n}([\mathscr C/G], \mathscr Z,1)^{\on{log}}$ as in
\autoref{notation:stable-compactification} can be identified with
the two-term complex
$f^* (\omega_{C/V}(Z)) \xrightarrow{\phi} \omega_{X/V}(D+E)$.
The following lemma will therefore help us analyze the deformation theory of
${\mathcal K}_{g,n}([\mathscr C/G],\mathscr Z, 1)^{\on{log}}$.
\begin{lemma}
	\label{lemma:cotangent-kernel}
We use notation as in \autoref{notation:stable-compactification}, 
where $i : W \to C$ is the inclusion.
With $\phi$ as defined in \autoref{remark:cotangent-map},
		$\ker \phi \simeq i_*\mathscr O_W(-(Y \cap W))$.
	\end{lemma}
	\begin{proof}
		We first describe the map $\phi$, which was defined as a
		composition $\delta \circ \varepsilon \circ \alpha$ in
		\autoref{remark:cotangent-map}, in a more concrete fashion.
Let $j: Y \to X$ denote the inclusion and $\rho: X \to Y$ the map
	contracting $W$. 
		The following statements can be obtained by unwinding the definitions of
		the maps induced on log differentials, used to define
		\eqref{equation:composite-differentials}.
		The map $\alpha: \rho^* t^* \omega_{C/V}(Z) \to \rho^*
		\omega_{Y/V}(E)$ 
		in \eqref{equation:composite-differentials}
		is
	obtained as the pullback under $\rho$ of the isomorphism $t^*\omega_{C/V}(Z) \to \omega_{Y/V}(E)$ from
	\autoref{lemma:differential-computation}.
	The map $\delta: \omega_{X/V}(E) \to \omega_{X/V}(E+D)$
	in \eqref{equation:composite-differentials}
	is obtained from twisting the inclusion $\mathscr
	O_X \to \mathscr O_X(D)$ by $\omega_{X/V}(E)$.
Finally, it remains to describe the map $\varepsilon: \rho^* \omega_{Y/V}(E) \to
\omega_{X/V}(E)$
in \eqref{equation:composite-differentials}.
Since $\rho \circ j = \id_Y$,
	There is an isomorphism
	$\omega_{Y/V} \simeq \rho_* j_* \omega_{Y/V}$ which yields by adjunction a
	map $\beta: \rho^* \omega_{Y/V}(E) \to j_* \omega_{Y/V}(E)$. 
	Define the map $\gamma: j_* \omega_{Y/V}(E) \to \omega_{X/V}(E)$ as that
	obtained via the inclusion $j_*(\omega_{Y/V}(E)) \simeq \omega_{X/V}(E - (Y \cap W))
	\hookrightarrow \omega_{X/V}(E)$.
	Then, $\varepsilon = \gamma \circ \beta$ and so 
the map $\phi$ is the composite of the maps
$\rho^* t^* \omega_{C/V}(Z) \xrightarrow{\alpha} \rho^* \omega_{Y/V}(E)
\xrightarrow{\beta} j_* \omega_{Y/V}(E) \xrightarrow{\gamma} \omega_{X/V}(E)
\xrightarrow{\delta} \omega_{X/V}(E+D).$

We now wish to identify the kernel of $\phi$.
First, the map $\alpha$ is an isomorphism
by \autoref{lemma:differential-computation}.
The maps $\gamma$ and $\delta$ are both injective maps of locally free sheaves
by construction. Therefore, we can identify the kernel of $\phi$ with the kernel
of $\beta: \rho^* \omega_{Y/V}(E) \to j_* \omega_{Y/V}(E)$.
This map $\beta$ is an isomorphism away from $W$, so we only need compute the
kernel restricted to $W$.
On $W$ the map $\beta$ restricts to the map $\mathscr O_W \to \mathscr O_{W \cap
Y}|_W$ and so the kernel is indeed $\mathscr O_W(-W \cap Y)$. Hence, the kernel
of $\phi$ is $i_* \mathscr O_W(-W \cap Y)$, as claimed.
	\end{proof}

We will see that the obstructions to deforming a point 
of ${\mathcal K}_{g,n}([\mathscr C/G], \mathscr Z,1)^{\on{log}}$ as in
\autoref{notation:stable-compactification} lie in 
$\on{Ext}^2( \mathbb L_h^{\on{log}}, \mathscr O_{\mathcal X})$,
for $\mathbb L_h^{\on{log}} = [h^*(\omega_{[C/G]/V}([Z/G])) \to \omega_{\mathcal
	X/V}(\mathcal D+
\mathcal E)]$.
Therefore, the next lemma will verify that deformations are unobstructed and
hence be used to show ${\mathcal K}_{g,n}([\mathscr C/G], \mathscr Z,1)$ is log smooth over
$(\overline{\mathscr M}_{g,n+\underline{d}})^{\on{log}}$.

\begin{lemma}
	\label{lemma:vanishing-ext}
	With notation as in \autoref{notation:stable-compactification},
	let $\mathbb L_h^{\on{log}} = [h^*
	(\omega_{[C/G]/V}([Z/G])) \to \omega_{\mathcal X/V}(\mathcal
D+\mathcal E)]$
	denote the two term complex on $\mathcal X$
	where the first term lies in degree $-1$ and the second in degree $0$.
	Then $\on{Ext}^2( \mathbb L_h^{\on{log}}, \mathscr O_{\mathcal X})= 0$.
\end{lemma}
\begin{proof}
	First, we identify 
	$\on{Ext}^2( \mathbb L_h^{\on{log}}, \mathscr O_{\mathcal X})
\simeq \on{Ext}^2( \mathbb L_f^{\on{log}}, \mathscr O_{X})$
where $\mathbb L_f^{\on{log}} = [f^* (\omega_{C/V}(Z)) \to \omega_{X/V}(D+E)]$,
also 
in degrees $[-1,0]$.
For $\pi: \mathcal X \to X$ the coarse space, and any line bundle $\mathscr L$ on $X$ the adjunction map $\mathscr L \to \pi_*
\pi^*\mathscr L$ is an isomorphism, as can be verified locally using that
$\mathscr O_X \to \pi_* \pi^* \mathscr O_X$ is an isomorphism. 
Hence,
because $\pi^* (\omega_{X/V}(D + E)) \simeq \omega_{\mathcal X/V}(\mathcal D + \mathcal E)
$ by \cite[Proposition 3.11]{ascherB:smoothability-of-relative-stable-maps}, we find
$\omega_{X/V}^\vee(-D - E)
\simeq \pi_* (\omega_{\mathcal X/V}^\vee(-\mathcal D - \mathcal E))$.
We also have 
\begin{align*}
	\pi_* h^* \omega_{[C/G]/V}^\vee \simeq \pi_* h^* \psi^* \omega_{C/V}^\vee \simeq
	\pi_* \pi^* f^* \omega_{C/V}^\vee \simeq f^* \omega_{C/V}^\vee,
\end{align*}
using that $C \to [C/G]$ is \'etale, that $f \circ \pi = \psi\circ h$
by \autoref{lemma:base-description}, and that $\pi_* \pi^* \mathscr O_X \simeq
\mathscr O_X$.
The above observations yield the third isomorphism in the below chain of
isomorphisms:
\begin{align*}
	\on{Ext}^2( \mathbb L_h^{\on{log}}, \mathscr O_{\mathcal X}) 
	& \simeq H^2(\mathcal X, \omega_{\mathcal X/V}^\vee (-\mathcal D -
	\mathcal E) \to h^*( \omega_{[C/G]/V}^\vee(-[Z/G]))) \\
	&\simeq H^2(X, \pi_* (\omega_{\mathcal X/V}^\vee (-\mathcal D - \mathcal
		E))
	\to \pi_* h^* (\omega_{[C/G]/V}^\vee(-[Z/G]))) \\
	&\simeq H^2(X, \omega_{X/V}^\vee (-D -E)
		\to f^* \omega_{C/V}^\vee(-Z)) \\
		&\simeq \on{Ext}^2( \mathbb L_f^{\on{log}}, \mathscr O_{X}).
\end{align*}
	Therefore, it suffices to prove
$\on{Ext}^2( \mathbb L_f^{\on{log}}, \mathscr O_X) = 0$.

	It follows from \autoref{lemma:cotangent-kernel} that $\mathbb
	L_f^{\on{log}} = [f^*
	(\omega_{C/V}(Z)) \xrightarrow{\phi} \omega_{X/V}(D+E)]$ sits in the following
	exact triangle
	\begin{align}
		\label{equation:distinguished-triangle}
		i_* \mathscr O_W(-(Y \cap W))[1]
		\to \mathbb L_f^{\on{log}} \to 
		\mathscr Q[0]
		 \to \qquad
	\end{align}
	where $\mathscr Q$ the cokernel of the map 
	$f^*(\omega_{C/V}(Z)) \xrightarrow{\phi} \omega_{X/V}(D+E)$.
	Applying $\hom(\bullet, \mathscr O_X)$ to
	\eqref{equation:distinguished-triangle} and taking the long exact
	sequence yields the exact sequence
	\begin{equation}
		\label{equation:ext-sequence}
		\begin{tikzcd}[column sep = small]
			\qquad & \on{Ext}^2(\mathscr Q, \mathscr O_X)
			 \ar {r} &
			 \on{Ext}^2(\mathbb L_f^{\on{log}}, \mathscr O_X) \ar {r} & \on{Ext}^2(i_* \mathscr O_W(-(Y \cap W))[1], \mathscr
			O_X).
	\end{tikzcd}\end{equation}
	It is therefore enough to show that the first and third terms of 
	\eqref{equation:ext-sequence}
	vanish.
	In general, by Serre duality, for $\mathscr F$ a coherent sheaf on a
	Gorenstein curve $X$, 
	$\on{Ext}^i(\mathscr F, \mathscr O_X)$ is dual to 
	$\on{Ext}^{1-i}(\mathscr O, \mathscr F \otimes \omega_{X/V}) \simeq H^{1-i}(\mathscr F \otimes
	\omega_{X/V})$.
	From this, it follows that 
$\on{Ext}^2(\mathscr Q, \mathscr O_X) = 0$, as the $-1$st cohomology of any
	coherent sheaf vanishes.

	To complete the proof, it remains only to show 
	$\on{Ext}^2(i_* \mathscr O_W(-(Y \cap W))[1], \mathscr O_X)$ vanishes.
	Using Serre duality,
	\begin{align*}
\on{Ext}^2(i_* \mathscr O_W(-(Y \cap W))[1], \mathscr O_X)
		&\simeq 
		\on{Ext}^1(i_* \mathscr O_W(-(Y \cap W)), \mathscr O_X)
		\\
		& \simeq H^0(X, i_* \mathscr O_W(-(Y \cap W)) \otimes
		\omega_{X/V})^\vee
		\\
		&\simeq
		H^0(W, \mathscr O_W(-(Y \cap W)) \otimes \omega_{X/V}|_W)^\vee
		\\
		&\simeq
		H^0(W, \mathscr O_W(-(Y \cap W)) \otimes \omega_{W/V}(Y \cap
		W))^\vee
		\\
		&\simeq
		H^1(W, \mathscr O_W) \\
		&= 0.
	\end{align*}
	In the final step, we are using that each connected component of $W$ has
arithmetic genus $0$, since it is a union of irreducible components of the arithmetic genus
$0$ curve $P$,
	so $H^1(W, \mathscr O_W)= 0$.
\end{proof}

To prove \autoref{proposition:log-smooth}, we will discuss the deformation theory needed to deduce log
smoothness of
${\mathcal K}_{g,n}([\mathscr C/G],\mathscr Z,1)^{\on{log}} \to (\overline{\mathscr
M}_{g, n+\underline{d}})^{\on{log}}$ from the vanishing demonstrated in
\autoref{lemma:vanishing-ext}. 
Note that ${\mathcal K}_{g,n}([\mathscr C/G],\mathscr Z,1)^{\on{log}}$
parameterizes certain log structures on covers of curves, and we next
introduce a stack
$\mathcal L{\mathcal K}_{g,n}([\mathscr C/G],\mathscr Z,1)$
parameterizing all fine
log structures.
\begin{notation}
	\label{notation:stack-of-log-maps}
	Using notation as in \autoref{notation:log-stable-maps} let 
$\mathcal L{\mathcal K}_{g,n}([\mathscr C/G],\mathscr Z,1)$ denote the stack
whose $S$-points are tuples $(\mathcal M_S, (\pi, \pi^\flat) : (\mathcal X, \mathcal
M_{\mathcal X}) \to (S, \mathcal M_S), (h, h^{\flat}) : (\mathcal X, \mathcal
M_{\mathcal X}) \to ([C/G], \mathcal M'_{[C/G]}))$ where $\mathcal M_S$ is a
fine log
structure on $S$, 
$\pi$ is a family of log twisted curves of type $(g,n+d)$ and
$(h, h^{\flat})$ is a log map such that $h$ is as in
\autoref{notation:stable-maps}. There is a map $\iota: {\mathcal K}_{g,n}([\mathscr C/G],\mathscr Z,1)^{\log} \to
\mathcal L{\mathcal K}_{g,n}([\mathscr C/G],\mathscr Z,1)$
which sends an $S$-point of the source, thought of as a map $\mathcal X \to
[C/G]$ with their log structures, as described in
\autoref{notation:log-stable-maps},
to the corresponding point of 
$\mathcal L{\mathcal K}_{g,n}([\mathscr C/G],\mathscr Z,1)$. 
\end{notation}

Combining the above lemmas with some deformation theory, we
deduce
\autoref{proposition:log-smooth}.

\subsubsection{Proof of \autoref{proposition:log-smooth}}
\label{subsubsection:proof-log-smooth}
We note that 
${\mathcal K}_{g,n+d}([\mathscr C/G],1)$ is a proper algebraic stack by
\cite[Theorem 1.4.1]{abramovichV:compactifying-the-space-of-stable-maps},
and hence 
${\mathcal K}_{g,n}([\mathscr C/G],\mathscr Z,1)$ is also a proper algebraic stack.
To show ${\mathcal K}_{g,n}([\mathscr C/G],\mathscr Z,1)$ is Deligne-Mumford, it
suffices to show ${\mathcal K}_{g,n+d}([\mathscr C/G],1)$ is Deligne-Mumford, which
follows from \cite[Theorem 1.16]{Olsson:log-twisted-curves}.

To conclude the proof, we only need to verify that 
${\mathcal K}_{g,n}([\mathscr C/G],\mathscr Z,1)^{\on{log}}$ is log smooth over
$(\overline{\mathscr M}_{g,\underline{d}})^{\on{log}}$.
Let $S = \spec A$ denote a local Artin scheme over $\mathbb Z[1/|G|]$.
Fix a point $[h: \mathcal X \to [C/G], \mathcal{D}+ \mathcal E] \in {\mathcal
K}_{g,n}([\mathscr C/G],\mathscr Z, 1)^{\on{log}}(S)$.
Note that by \autoref{notation:log-stable-maps},
$S$ has an induced log structure coming from pulling back the log
structure from the associated map
$S \to (\overline{\mathscr M}_{g,n+\underline{d}})^{\on{log}}$ classifying $X$,
the coarse space of $\mathcal X$.
Let $T' = \spec A'$ denote a
thickening of $S$ with $I := \ker A' \to A$.
Suppose $A'$ has residue field $\kappa$, maximal ideal $\mathfrak m$,
and assume $\mathfrak m I = 0$.
In order to verify formal smoothness,
we wish to extend the above $S$-point to a $S'$-point compatible with the above
extension of log structure.
First, we claim the obstruction to deforming our $S$-point above, viewed as
a map
of log stacks, lies in the
hypercohomology group
$\on{Ext}^2( \mathbb L_h^{\on{log}}, I \otimes_{A'} \mathscr O_{\mathcal X})$
for $\mathbb L_h^{\on{log}} := [h^* \Omega_{[C/G]/S}^{\on{log}} \to
\Omega^{\on{log}}_{\mathcal X/S}]$ in degrees $[-1,0]$. 
(We will soon show this is isomorphic to
	the complex
$\mathbb L_h^{\on{log}}$ as defined in \autoref{lemma:vanishing-ext}.)
Indeed by \cite[Theorem
8.36(i)]{Olsson:theLogarithmicCotangent}, there is a canonical obstruction in
$\on{Ext}^2( \mathbb L_h^{G}, I \otimes_{A'} \mathscr O_{\mathcal X})$ where
$\mathbb L_h^G$ is Gabber's cotangent complex, as defined in 
\cite[Definition 8.5]{Olsson:theLogarithmicCotangent}. 
By \cite[Section 8.29]{Olsson:theLogarithmicCotangent} there is a transitivity triangle 
$$
Lh^*\mathbb{L}_{[C/G]/S}^{G} \to \mathbb{L}_{\mathcal X/S}^G \to \mathbb{L}_h^G
$$
and by \cite[Corollary 8.34 and Theorem 1.1(iii)]{Olsson:theLogarithmicCotangent}, we can identify the map $Lh^*\mathbb{L}_{[C/G]/S}^{G} \to \mathbb{L}_{\mathcal X/S}^G$ with 
$\mathbb L_h^{\on{log}}$;
here we use that log smooth curves are integral and that $Lh^* = h^*$ for a locally free sheaf.
We next wish to show $\on{Ext}^2( \mathbb L_h^{\on{log}}, I \otimes_{A'}
\mathscr O_{\mathcal X})  = 0$.
There is an identification 
$\on{Ext}^2( \mathbb L_h^{\on{log}}, I \otimes_{A'} \mathscr O_{\mathcal X} )
\simeq \on{Ext}^2( \mathbb L_{h_0}^{\on{log}}, I \otimes_{\kappa} \mathscr O_{\mathcal
X_0})$,
where 
${\mathcal X_0}$ is the base change of $\mathcal X$ along $\spec \kappa \to
\spec A'$ and
$h_0$ is the base change of $h$ along $\spec \kappa \to
\spec A$,
since $I$ is killed by $\mathfrak m$.
In order to show this $\kappa$ vector space vanishes, we are free to base change to the algebraic
closure of $\kappa$. Hence, for the remainder of the proof, we can assume $S = V = \spec k$ is a geometric point
as in \autoref{notation:stable-compactification},
and we aim to show
$\on{Ext}^2( \mathbb L_h^{\on{log}}, \mathscr O_X)= 0$.

To verify
$\on{Ext}^2( \mathbb L_h^{\on{log}}, \mathscr O_X)= 0$, we next claim we can identify 
$\Omega^{\on{log}}_{\mathcal X/V}
 \simeq \omega_{\mathcal X/V}(\mathcal D+ \mathcal
E)$ and
$\Omega_{[C/G]/V}^{\on{log}} \simeq \omega_{[C/G]/V}([Z/G])$
so that
$\mathbb L_h^{\on{log}} \simeq [h^*  (\omega_{[C/G]/V}([Z/G])) \to
\omega_{\mathcal X/V}(\mathcal D+ \mathcal E)]$.
By \cite[Proposition 1.13]{kato:log-smooth-deformation-theory-and-moduli}
we can identify 
$\Omega^{\on{log}}_{X/V} \simeq \omega_{X/V}(D+E)$.
Then, by \cite[Proposition 3.11]{ascherB:smoothability-of-relative-stable-maps},
if $\pi: \mathcal X \to X$ denotes the coarse space map,
$\Omega^{\on{log}}_{\mathcal X/V} \simeq \pi^*\Omega^{\on{log}}_{X/V} \simeq \pi^*
\omega_{X/V}(D+E) \simeq \omega_{\mathcal X/V}(\mathcal D + \mathcal E)$.
Arguing similarly, we also obtain
$\Omega_{[C/G]/V}^{\on{log}} \simeq \omega_{[C/G]/V}([Z/G])$.
Therefore, the cotangent complex $\mathbb L_h^{\on{log}}$ is identified with
$[h^* (\omega_{[C/G]/V}([Z/G])) \to \omega_{\mathcal
X/V}(\mathcal D + \mathcal E)]$.
Now, note that
$\on{Ext}^2( \mathbb L_h^{\on{log}}, \mathscr O_{\mathcal X}) = 0$,
by \autoref{lemma:vanishing-ext}.

We are nearly done, and it only remains to explain why the vanishing of 
the obstruction space
$\on{Ext}^2( \mathbb L_h^{\on{log}}, \mathscr O_{\mathcal X})$ actually implies
log smoothness of
${\mathcal K}_{g,n}([\mathscr C/G],\mathscr Z,1)^{\on{log}} \to
(\overline{\mathscr M}_{g,\underline{d}})^{\on{log}}$.
To this end, let $\mathcal{L}og_{\overline{\mathscr M}_{g,\underline{d}}}$ denote the algebraic stack 
classifying fine log schemes over $\overline{\mathscr M}_{g,\underline{d}}$,
as defined in \cite[Section 5]{Olsson:logStacks}, and, in particular,
\cite[Proposition 5.9]{Olsson:logStacks}.
Using \autoref{notation:stack-of-log-maps},
the log structure on 
${\mathcal K}_{g,n}([\mathscr C/G],\mathscr Z,1)^{\on{log}}$ induces maps
${\mathcal K}_{g,n}([\mathscr C/G],\mathscr Z,1) \xrightarrow{\iota}
{\mathcal L}{\mathcal K}_{g,n}([\mathscr C/G],\mathscr Z,1) \xrightarrow{\zeta}
\mathcal{L}og_{\overline{\mathscr M}_{g,\underline{d}}}$.
The vanishing of
$\on{Ext}^2( \mathbb L_h^{\on{log}}, \mathscr O_{\mathcal X})$
implies the map $\zeta$ above is formally smooth.
The log structure
from \autoref{notation:log-stable-maps} is the minimal log structure of the log
map in the sense of \cite[p. 724]{wise:moduli-of-morphisms}
and so
\cite[Theorem B.2]{wise:moduli-of-morphisms}
implies $\iota$ above is an open embedding.
Therefore, the composite $\zeta \circ \iota$
is formally smooth, hence smooth.
It is shown in \cite[Theorem 4.6(ii) and (iii)]{Olsson:logStacks}
that if $(W, \mathcal M_W)$ is a scheme with fine log structure
then $(W, \mathcal M_W) \to (\overline{\mathscr M}_{g,\underline{d}}, \mathcal M_{\overline{\mathscr M}_{g,\underline{d}}})$ is log smooth if 
$W \to \mathcal{L}og_{\overline{\mathscr M}_{g,\underline{d}}}$
is smooth.
From this, one can easily deduce the same holds in the case that
$(W, \mathcal M_W)$ is an algebraic stack with fine log structure by passing to a smooth
cover of $W$ by a scheme. Hence, we obtain that
${\mathcal K}_{g,n}([\mathscr C/G],\mathscr Z,1)^{\on{log}}$ is log smooth over
$(\overline{\mathscr M}_{g,\underline{d}})^{\on{log}}$,
completing
the proof. 
\qed

\bibliographystyle{alpha}
\bibliography{./bibliography, ./log_bibliography}

\newcommand{\etalchar}[1]{$^{#1}$}
\def\cprime{$'$} \providecommand{\noopsort}[1]{}
\begin{thebibliography}{MPPRW24}

\bibitem[AB23]{ascherB:smoothability-of-relative-stable-maps}
Kenneth Ascher and Dori Bejleri.
\newblock Smoothability of relative stable maps to stacky curves.
\newblock {\em \'Epijournal G\'eom. Alg\'ebrique}, 7:Art. 2, 22, 2023.

\bibitem[ACG11]{arbarelloCG:geomtry-of-algebraic-curves-ii}
Enrico Arbarello, Maurizio Cornalba, and Pillip~A. Griffiths.
\newblock {\em Geometry of algebraic curves. {V}olume {II}}, volume 268 of {\em
  Grundlehren der Mathematischen Wissenschaften [Fundamental Principles of
  Mathematical Sciences]}.
\newblock Springer, Heidelberg, 2011.
\newblock With a contribution by Joseph Daniel Harris.

\bibitem[ACGS20]{abramovichCGB:punctured}
Dan Abramovich, Qile Chen, Mark Gross, and Bernd Siebert.
\newblock Punctured logarithmic maps.
\newblock {\em arXiv preprint arXiv:2009.07720v2}, 2020.

\bibitem[Ach08]{Achter:cohenQuadratic}
Jeffrey~D. Achter.
\newblock Results of {C}ohen-{L}enstra type for quadratic function fields.
\newblock In {\em Computational arithmetic geometry}, volume 463 of {\em
  Contemp. Math.}, pages 1--7. Amer. Math. Soc., Providence, RI, 2008.

\bibitem[Ach23]{achenjang:the-average-size-of-2-selmer}
Niven Achenjang.
\newblock The average size of 2-{S}elmer groups of elliptic curves in
  characteristic 2.
\newblock {\em arXiv preprint arXiv:2310.08493v2}, 2023.

\bibitem[ACV03]{abramovichCV:twisted-bundles}
Dan Abramovich, Alessio Corti, and Angelo Vistoli.
\newblock Twisted bundles and admissible covers.
\newblock volume~31, pages 3547--3618. 2003.
\newblock Special issue in honor of Steven L. Kleiman.

\bibitem[AK00]{abramovich_karu}
D.~Abramovich and K.~Karu.
\newblock Weak semistable reduction in characteristic 0.
\newblock {\em Invent. Math.}, 139(2):241--273, 2000.

\bibitem[ALT19]{alt_semistable}
Karim Adiprasito, Gaku Liu, and Michael Temkin.
\newblock Semistable reduction in characteristic 0, 2019.

\bibitem[AMW14]{AMW_log_relative}
Dan Abramovich, Steffen Marcus, and Jonathan Wise.
\newblock Comparison theorems for {G}romov-{W}itten invariants of smooth pairs
  and of degenerations.
\newblock {\em Ann. Inst. Fourier (Grenoble)}, 64(4):1611--1667, 2014.

\bibitem[AV02]{abramovichV:compactifying-the-space-of-stable-maps}
Dan Abramovich and Angelo Vistoli.
\newblock Compactifying the space of stable maps.
\newblock {\em J. Amer. Math. Soc.}, 15(1):27--75, 2002.

\bibitem[BDPW23]{bergstromDPW:hyperelliptic-curves-the-scanning-map}
Jonas Bergstr{\"o}m, Adrian Diaconu, Dan Petersen, and Craig Westerland.
\newblock Hyperelliptic curves, the scanning map, and moments of families of
  quadratic {L}-functions.
\newblock {\em arXiv preprint arXiv:2302.07664v3}, 2023.

\bibitem[Bel04]{bellingeri:on-presentations-of-surface-braid-groups}
Paolo Bellingeri.
\newblock On presentations of surface braid groups.
\newblock {\em J. Algebra}, 274(2):543--563, 2004.

\bibitem[Bis19]{bisatt:explicit-root-numbers}
Matthew Bisatt.
\newblock Explicit root numbers of abelian varieties.
\newblock {\em Trans. Amer. Math. Soc.}, 372(11):7889--7920, 2019.

\bibitem[BKL{\etalchar{+}}15]{bhargavaKLPR:modeling-the-distribution-of-ranks-selmer-groups}
Manjul Bhargava, Daniel~M. Kane, Hendrik~W. Lenstra, Jr., Bjorn Poonen, and
  Eric Rains.
\newblock Modeling the distribution of ranks, {S}elmer groups, and
  {S}hafarevich-{T}ate groups of elliptic curves.
\newblock {\em Camb. J. Math.}, 3(3):275--321, 2015.

\bibitem[BKLOS19]{bhargavaKLS:3-isogeny}
Manjul Bhargava, Zev Klagsbrun, Robert~J. Lemke~Oliver, and Ari Shnidman.
\newblock 3-isogeny {S}elmer groups and ranks of abelian varieties in quadratic
  twist families over a number field.
\newblock {\em Duke Math. J.}, 168(15):2951--2989, 2019.

\bibitem[BLR90]{BoschLR:Neron}
Siegfried Bosch, Werner L{\"u}tkebohmert, and Michel Raynaud.
\newblock {\em N\'eron models}, volume~21 of {\em Ergebnisse der Mathematik und
  ihrer Grenzgebiete (3) [Results in Mathematics and Related Areas (3)]}.
\newblock Springer-Verlag, Berlin, 1990.

\bibitem[BM23]{bianchiM:polynomial-stability}
Andrea Bianchi and Jeremy Miller.
\newblock Polynomial stability of the homology of {H}urwitz spaces.
\newblock {\em arXiv preprint arXiv:2303.11194v1}, 2023.

\bibitem[BS13a]{bhargavaS:average-4-selmer}
Manjul Bhargava and Arul Shankar.
\newblock The average number of elements in the 4-{S}elmer groups of elliptic
  curves is 7.
\newblock {\em arXiv preprint arXiv:1312.7333v1}, 2013.

\bibitem[BS13b]{bhargavaS:average-5-selmer}
Manjul Bhargava and Arul Shankar.
\newblock The average size of the 5-selmer group of elliptic curves is 6, and
  the average rank is less than 1.
\newblock {\em arXiv preprint arXiv:1312.7859v1}, 2013.

\bibitem[BS15a]{bhargava-shankar:binary-quartic-forms-having-bounded-invariants}
Manjul Bhargava and Arul Shankar.
\newblock Binary quartic forms having bounded invariants, and the boundedness
  of the average rank of elliptic curves.
\newblock {\em Ann. of Math. (2)}, 181(1):191--242, 2015.

\bibitem[BS15b]{bhargavaS:ternary}
Manjul Bhargava and Arul Shankar.
\newblock Ternary cubic forms having bounded invariants, and the existence of a
  positive proportion of elliptic curves having rank 0.
\newblock {\em Ann. of Math. (2)}, 181(2):587--621, 2015.

\bibitem[BS23]{bianchiS:homology-of-configuration-spaces}
Andrea Bianchi and Andreas Stavrou.
\newblock Homology of configuration spaces of surfaces modulo an odd prime.
\newblock {\em arXiv preprint arXiv:2307.08664v1}, 2023.

\bibitem[BSS21]{bhargavaSS:the-second-moment}
Manjul Bhargava, Arul Shankar, and Ashvin Swaminathan.
\newblock The second moment of the size of the $2 $-selmer group of elliptic
  curves.
\newblock {\em arXiv preprint arXiv:2110.09063v1}, 2021.

\bibitem[BSW15]{bhargavaSW:geometry-of-numbers-over-global-fields-i}
Manjul Bhargava, Arul Shankar, and Xiaoheng Wang.
\newblock Geometry-of-numbers methods over global fields {I}: Prehomogeneous
  vector spaces.
\newblock {\em arXiv preprint arXiv:1512.03035v1}, 2015.

\bibitem[BV12]{borneV:parabolic-sheaves-on-logarithmic-schemes}
Niels Borne and Angelo Vistoli.
\newblock Parabolic sheaves on logarithmic schemes.
\newblock {\em Adv. Math.}, 231(3-4):1327--1363, 2012.

\bibitem[Cad07]{cadman:using-stacks-to-impose-tangency}
Charles Cadman.
\newblock Using stacks to impose tangency conditions on curves.
\newblock {\em Amer. J. Math.}, 129(2):405--427, 2007.

\bibitem[Ces16]{cesnavicius:selmer-groups-as-flat-cohomology-groups}
Kestutis Cesnavicius.
\newblock Selmer groups as flat cohomology groups.
\newblock {\em J. Ramanujan Math. Soc.}, 31(1):31--61, 2016.

\bibitem[Cha97]{chavdarov:the-generic-irreducibility}
Nick Chavdarov.
\newblock The generic irreducibility of the numerator of the zeta function in a
  family of curves with large monodromy.
\newblock {\em Duke Math. J.}, 87(1):151--180, 1997.

\bibitem[Cha23]{chang:hurwitz-spaces-nichols-algebras}
Kevin Chang.
\newblock Hurwitz spaces, {N}ichols algebras, and {I}gusa zeta functions.
\newblock {\em arXiv preprint arXiv:2306.10446v3}, 2023.

\bibitem[CLQR04]{cremonaLQR:modular-curves-and-abelian-varieties}
John Cremona, Joan-Carles Lario, Jordi Quer, and Kenneth Ribet, editors.
\newblock {\em Modular curves and abelian varieties}, volume 224 of {\em
  Progress in Mathematics}.
\newblock Birkh\"{a}user Verlag, Basel, 2004.
\newblock Papers from the conference held in Bellaterra, July 15--18, 2002.

\bibitem[Con14]{conrad:reductive-group-schemes}
Brian Conrad.
\newblock Reductive group schemes.
\newblock In {\em Autour des sch\'emas en groupes. {V}ol. {I}}, volume 42/43 of
  {\em Panor. Synth\`eses}, pages 93--444. Soc. Math. France, Paris, 2014.

\bibitem[Det08]{dettweiler:on-the-middle-convolution-of-local-systems}
Michael Dettweiler.
\newblock On the middle convolution of local systems. {W}ith an {A}ppendix by
  {M}. {D}ettweiler and {S}. {R}eiter.
\newblock {\em arXiv preprint arXiv:0810.3334v1}, 2008.

\bibitem[dJ02]{de-jong:counting-elliptic-surfaces-over-finite-fields}
A.~J. de~Jong.
\newblock Counting elliptic surfaces over finite fields.
\newblock {\em Mosc. Math. J.}, 2(2):281--311, 2002.
\newblock Dedicated to Yuri I. Manin on the occasion of his 65th birthday.

\bibitem[DS23]{davisS:the-hilbert-polynomial-of-quandles}
Ariel Davis and Tomer~M Schlank.
\newblock The {H}ilbert {P}olynomial of {Q}uandles and {C}olorings of {R}andom
  {L}inks.
\newblock {\em arXiv preprint arXiv:2304.08314v1}, 2023.

\bibitem[Dwy80]{dwyer:twisted-homological-stability}
W.~G. Dwyer.
\newblock Twisted homological stability for general linear groups.
\newblock {\em Ann. of Math. (2)}, 111(2):239--251, 1980.

\bibitem[ELS20]{ellenbergLS:nonvanishing-of-hyperelliptic-zeta-functions}
Jordan~S. Ellenberg, Wanlin Li, and Mark Shusterman.
\newblock Nonvanishing of hyperelliptic zeta functions over finite fields.
\newblock {\em Algebra Number Theory}, 14(7):1895--1909, 2020.

\bibitem[ETW17]{ellenbergTW:fox-neuwirth-fuks-cells}
Jordan~S Ellenberg, TriThang Tran, and Craig Westerland.
\newblock Fox-{N}euwirth-{F}uks cells, quantum shuffle algebras, and {M}alle's
  conjecture for function fields.
\newblock {\em arXiv preprint arXiv:1701.04541v2}, 2017.

\bibitem[EVW16]{EllenbergVW:cohenLenstra}
Jordan~S. Ellenberg, Akshay Venkatesh, and Craig Westerland.
\newblock Homological stability for {H}urwitz spaces and the {C}ohen-{L}enstra
  conjecture over function fields.
\newblock {\em Ann. of Math. (2)}, 183(3):729--786, 2016.

\bibitem[FGI{\etalchar{+}}05]{FantechiGIK:fundamentalAlgebraicGeometry}
Barbara Fantechi, Lothar G{\"o}ttsche, Luc Illusie, Steven~L. Kleiman, Nitin
  Nitsure, and Angelo Vistoli.
\newblock {\em Fundamental algebraic geometry}, volume 123 of {\em Mathematical
  Surveys and Monographs}.
\newblock American Mathematical Society, Providence, RI, 2005.
\newblock Grothendieck's FGA explained.

\bibitem[FK88]{FreitagK:lectures-etale}
Eberhard Freitag and Reinhardt Kiehl.
\newblock {\em \'{E}tale cohomology and the {W}eil conjecture}, volume~13 of
  {\em Ergebnisse der Mathematik und ihrer Grenzgebiete (3) [Results in
  Mathematics and Related Areas (3)]}.
\newblock Springer-Verlag, Berlin, 1988.
\newblock Translated from the German by Betty S. Waterhouse and William C.
  Waterhouse, With an historical introduction by J. A. Dieudonn{\'e}.

\bibitem[Fla90]{flach:a-generalization-of-the-cassels-tate-pairing}
Matthias Flach.
\newblock A generalisation of the {C}assels-{T}ate pairing.
\newblock {\em J. Reine Angew. Math.}, 412:113--127, 1990.

\bibitem[FLR23]{fengLR:geometric-distribution-of-selmer-groups}
Tony Feng, Aaron Landesman, and Eric~M. Rains.
\newblock The geometric distribution of {S}elmer groups of elliptic curves over
  function fields.
\newblock {\em Math. Ann.}, 387(1-2):615--687, 2023.

\bibitem[FM94]{fultonM:a-compactification}
William Fulton and Robert MacPherson.
\newblock A compactification of configuration spaces.
\newblock {\em Ann. of Math. (2)}, 139(1):183--225, 1994.

\bibitem[FS16]{fulmanS:distribution-number-fixed}
Jason Fulman and Dennis Stanton.
\newblock On the distribution of the number of fixed vectors for the finite
  classical groups.
\newblock {\em Ann. Comb.}, 20(4):755--773, 2016.

\bibitem[Gol79]{goldfeld:conjectures-on-elliptic-curves}
Dorian Goldfeld.
\newblock Conjectures on elliptic curves over quadratic fields.
\newblock In {\em Number theory, {C}arbondale 1979 ({P}roc. {S}outhern
  {I}llinois {C}onf., {S}outhern {I}llinois {U}niv., {C}arbondale, {I}ll.,
  1979)}, volume 751 of {\em Lecture Notes in Math}, pages pp 108--118.
  Springer, Berlin, 1979.

\bibitem[Gre10]{Greicius:surjective}
Aaron Greicius.
\newblock Elliptic curves with surjective adelic {G}alois representations.
\newblock {\em Experiment. Math.}, 19(4):495--507, 2010.

\bibitem[Gro68]{grothendieck:brauer-iii}
Alexander Grothendieck.
\newblock Le groupe de {B}rauer. {III}. {E}xemples et compl\'ements.
\newblock In {\em Dix expos\'es sur la cohomologie des sch\'emas}, volume~3 of
  {\em Adv. Stud. Pure Math.}, pages 88--188. North-Holland, Amsterdam, 1968.

\bibitem[Gro23]{gross:remarks}
Mark Gross.
\newblock Remarks on gluing punctured logarithmic maps.
\newblock {\em arXiv preprint arXiv:2306.02661v1}, 2023.

\bibitem[GRW18]{galatiusRW:homological-stability-for-moduli-spaces-i}
S\o~ren Galatius and Oscar Randal-Williams.
\newblock Homological stability for moduli spaces of high dimensional
  manifolds. {I}.
\newblock {\em J. Amer. Math. Soc.}, 31(1):215--264, 2018.

\bibitem[Hal08]{hall:bigMonodromySympletic}
Chris Hall.
\newblock Big symplectic or orthogonal monodromy modulo {$l$}.
\newblock {\em Duke Math. J.}, 141(1):179--203, 2008.

\bibitem[HB93]{Heath:theSizeOf-i}
D.~R. Heath-Brown.
\newblock The size of {S}elmer groups for the congruent number problem.
\newblock {\em Invent. Math.}, 111(1):171--195, 1993.

\bibitem[HB94]{Heath:theSizeOf-ii}
D.~R. Heath-Brown.
\newblock The size of {S}elmer groups for the congruent number problem. {II}.
\newblock {\em Invent. Math.}, 118(2):331--370, 1994.
\newblock With an appendix by P. Monsky.

\bibitem[HLHN14]{ho-lehung-ngo:average-size-of-2-selmber-groups}
Q.~P. H{\`{\^o}}, V.~B. L\^e~H\`ung, and B.~C. Ng\^o.
\newblock Average size of 2-{S}elmer groups of elliptic curves over function
  fields.
\newblock {\em Math. Res. Lett.}, 21(6):1305--1339, 2014.

\bibitem[Hoa23]{hoang:fox-neuwirth-cells-resultant}
Anh Trong~Nam Hoang.
\newblock Fox-{N}euwirth cells, quantum shuffle algebras, and character sums of
  the resultant.
\newblock {\em arXiv preprint arXiv:2308.01410v1}, 2023.

\bibitem[HS23]{holmesS:logarithmic-cohomological-field-theories}
David Holmes and Pim Spelier.
\newblock Logarithmic cohomological field theories.
\newblock {\em arXiv preprint arXiv:2308.01099v2}, 2023.

\bibitem[Ill02]{illusie:an-overview}
Luc Illusie.
\newblock An overview of the work of {K}. {F}ujiwara, {K}. {K}ato, and {C}.
  {N}akayama on logarithmic \'{e}tale cohomology.
\newblock Number 279, pages 271--322. 2002.
\newblock Cohomologies $p$-adiques et applications arithm\'{e}tiques, II.

\bibitem[IT14]{Illusie_temkin_logsmooth}
Luc Illusie and Michael Temkin.
\newblock Expos\'e{} {X}. {G}abber's modification theorem (log smooth case).
\newblock Number 363-364, pages 167--212. 2014.
\newblock Travaux de Gabber sur l'uniformisation locale et la cohomologie
  \'etale des sch\'emas quasi-excellents.

\bibitem[Iva93]{ivanov:on-the-homology-stability-for-teichmuller-modular-forms}
Nikolai~V Ivanov.
\newblock On the homology stability for teichm{\"u}ller modular groups: closed
  surfaces and twisted coefficients.
\newblock {\em Contemporary Mathematics}, pages 149--194, 1993.

\bibitem[Kan13]{kane:on-the-ranks-of-the-2-selmer-groups-of-twists}
Daniel Kane.
\newblock On the ranks of the 2-{S}elmer groups of twists of a given elliptic
  curve.
\newblock {\em Algebra Number Theory}, 7(5):1253--1279, 2013.

\bibitem[Kat89]{Kato1989}
K.~Kato.
\newblock Logarithmic structures of {F}ontaine-{I}llusie.
\newblock In {\em Algebraic analysis, geometry, and number theory ({B}altimore,
  {MD}, 1988)}, pages 191--224. Johns Hopkins Univ. Press, Baltimore, MD, 1989.

\bibitem[Kat96]{katz2016rigid}
Nicholas~M. Katz.
\newblock {\em Rigid local systems}, volume 139 of {\em Annals of Mathematics
  Studies}.
\newblock Princeton University Press, Princeton, NJ, 1996.

\bibitem[Kat00]{kato:log-smooth-deformation-theory-and-moduli}
Fumiharu Kato.
\newblock Log smooth deformation and moduli of log smooth curves.
\newblock {\em Internat. J. Math.}, 11(2):215--232, 2000.

\bibitem[Kat02]{katz:twisted-l-functions-and-monodromy}
Nicholas~M. Katz.
\newblock {\em Twisted {$L$}-functions and monodromy}, volume 150 of {\em
  Annals of Mathematics Studies}.
\newblock Princeton University Press, Princeton, NJ, 2002.

\bibitem[KMR13]{klagsburnMR:disparity-in-selmer-ranks}
Zev Klagsbrun, Barry Mazur, and Karl Rubin.
\newblock Disparity in {S}elmer ranks of quadratic twists of elliptic curves.
\newblock {\em Ann. of Math. (2)}, 178(1):287--320, 2013.

\bibitem[Kow06]{kowalski:on-the-rank-of-quadratic-twists}
E.~Kowalski.
\newblock On the rank of quadratic twists of elliptic curves over function
  fields.
\newblock {\em Int. J. Number Theory}, 2(2):267--288, 2006.

\bibitem[Lan21]{landesman:geometric-average-selmer}
Aaron Landesman.
\newblock The geometric average size of {S}elmer groups over function fields.
\newblock {\em Algebra Number Theory}, 15(3):673--709, 2021.

\bibitem[Lau81]{laumon:semi-coninuity-du-conducteur-de-swan}
G.~Laumon.
\newblock Semi-continuit\'{e} du conducteur de {S}wan (d'apr\`es {P}.
  {D}eligne).
\newblock In {\em The {E}uler-{P}oincar\'{e} characteristic ({F}rench)},
  volume~83 of {\em Ast\'{e}risque}, pages 173--219. Soc. Math. France, Paris,
  1981.

\bibitem[LL25a]{landesmanL:homological-stability-for-hurwitz}
Aaron Landesman and Ishan Levy.
\newblock Homological stability for {H}urwitz spaces and applications.
\newblock {\em arXiv preprint arXiv:2503.03861v2}, 2025.

\bibitem[LL25b]{landesmanL:stable-homology-of-hurwitz-modules}
Aaron Landesman and Ishan Levy.
\newblock The stable homology of hurwitz modules and applications.
\newblock {\em arXiv preprint arXiv:2510.02068v1}, 2025.

\bibitem[LPW09]{levinPW:markov-chains-and-mixing-times}
David~A. Levin, Yuval Peres, and Elizabeth~L. Wilmer.
\newblock {\em Markov chains and mixing times}.
\newblock American Mathematical Society, Providence, RI, 2009.
\newblock With a chapter by James G. Propp and David B. Wilson.

\bibitem[LST20]{lipnowskiST:cohen-lenstra-roots-of-unity}
Michael Lipnowski, Will Sawin, and Jacob Tsimerman.
\newblock Cohen-{L}enstra heuristics and bilinear pairings in the presence of
  roots of unity.
\newblock {\em arXiv preprint arXiv:2007.12533v1}, 2020.

\bibitem[LT19]{lipnowskyT:cohen-lenstra-for-etale-group-schemes}
Michael Lipnowski and Jacob Tsimerman.
\newblock Cohen-{L}enstra heuristics for \'{e}tale group schemes and symplectic
  pairings.
\newblock {\em Compos. Math.}, 155(4):758--775, 2019.

\bibitem[Mil80]{Milne:etaleBook}
James~S. Milne.
\newblock {\em \'{E}tale cohomology}, volume~33 of {\em Princeton Mathematical
  Series}.
\newblock Princeton University Press, Princeton, N.J., 1980.

\bibitem[Moc95]{mochizuki:the-geometry-of-the-compactification}
Shinichi Mochizuki.
\newblock The geometry of the compactification of the {H}urwitz scheme.
\newblock {\em Publ. Res. Inst. Math. Sci.}, 31(3):355--441, 1995.

\bibitem[MPPRW24]{millerPPR:uniform-twisted-homological-stability}
Jeremy Miller, Peter Patzt, Dan Petersen, and Oscar Randal-Williams.
\newblock Uniform twisted homological stability.
\newblock {\em arXiv preprint arXiv:2402.00354v1}, 2024.

\bibitem[NW22]{nguyenW:local-and-global-universality}
Hoi~H Nguyen and Melanie~Matchett Wood.
\newblock Local and global universality of random matrix cokernels.
\newblock {\em arXiv preprint arXiv:2210.08526v1}, 2022.

\bibitem[Ols03]{Olsson:logStacks}
Martin~C. Olsson.
\newblock Logarithmic geometry and algebraic stacks.
\newblock {\em Ann. Sci. \'Ecole Norm. Sup. (4)}, 36(5):747--791, 2003.

\bibitem[Ols05]{Olsson:theLogarithmicCotangent}
Martin~C. Olsson.
\newblock The logarithmic cotangent complex.
\newblock {\em Math. Ann.}, 333(4):859--931, 2005.

\bibitem[Ols07]{Olsson:log-twisted-curves}
Martin~C. Olsson.
\newblock ({L}og) twisted curves.
\newblock {\em Compos. Math.}, 143(2):476--494, 2007.

\bibitem[Par12]{parker:log-geometry-and-exploded}
Brett Parker.
\newblock Log geometry and exploded manifolds.
\newblock {\em Abh. Math. Semin. Univ. Hambg.}, 82(1):43--81, 2012.

\bibitem[PR12]{poonenR:random-maximal-isotropic-subspaces-and-selmer-groups}
Bjorn Poonen and Eric Rains.
\newblock Random maximal isotropic subspaces and {S}elmer groups.
\newblock {\em J. Amer. Math. Soc.}, 25(1):245--269, 2012.

\bibitem[PW23]{parkW:average-selmer-rank-in-quadratic-twist-families}
Sun~Woo Park and Niudun Wang.
\newblock {On the Average of p-Selmer Ranks in Quadratic Twist Families of
  Elliptic Curves Over Global Function Fields}.
\newblock {\em International Mathematics Research Notices}, page rnad095, 05
  2023.

\bibitem[Riz97]{rizzo:on-the-variation}
Ottavio~Giulio Rizzo.
\newblock {\em On the variation of root numbers in families of elliptic
  curves}.
\newblock ProQuest LLC, Ann Arbor, MI, 1997.
\newblock Thesis (Ph.D.)--Brown University.

\bibitem[Riz99]{rizzo:average-root-numbers}
Ottavio~G. Rizzo.
\newblock Average root numbers in families of elliptic curves.
\newblock {\em Proc. Amer. Math. Soc.}, 127(6):1597--1603, 1999.

\bibitem[Riz03]{rizzo:average-root-numbers-nonconstant}
Ottavio~G. Rizzo.
\newblock Average root numbers for a nonconstant family of elliptic curves.
\newblock {\em Compositio Math.}, 136(1):1--23, 2003.

\bibitem[RW20]{randal-williams:homology-of-hurwitz-spaces}
Oscar Randal-Williams.
\newblock Homology of {H}urwitz spaces and the {C}ohen-{L}enstra heuristic for
  function fields [after {E}llenberg, {V}enkatesh, and {W}esterland].
\newblock {\em Ast\'{e}risque}, (422):Exp. No. 1164, 469--497, 2020.

\bibitem[RWW17]{randal-williamsW:homological-stability-for-automorphism-groups}
Oscar Randal-Williams and Nathalie Wahl.
\newblock Homological stability for automorphism groups.
\newblock {\em Adv. Math.}, 318:534--626, 2017.

\bibitem[Sab13]{sabitova:twisted-root-numbers}
Maria Sabitova.
\newblock Twisted root numbers and ranks of abelian varieties.
\newblock {\em J. Comb. Number Theory}, 5(1):25--30, 2013.

\bibitem[Saw20]{sawin:identifying-measures}
Will Sawin.
\newblock Identifying measures on non-abelian groups and modules by their
  moments via reduction to a local problem.
\newblock {\em arXiv preprint arXiv:2006.04934v3}, 2020.

\bibitem[SD08]{swinnerton-dyer:the-effect-of-twisting}
Peter Swinnerton-Dyer.
\newblock The effect of twisting on the 2-{S}elmer group.
\newblock {\em Math. Proc. Cambridge Philos. Soc.}, 145(3):513--526, 2008.

\bibitem[SGA72]{SGA4}
{\em Th\'eorie des topos et cohomologie \'etale des sch\'emas.}
\newblock Lecture Notes in Mathematics, Vol. 269. Springer-Verlag, Berlin,
  1972.
\newblock S{\'e}minaire de G{\'e}om{\'e}trie Alg{\'e}brique du Bois-Marie
  1963--1964 (SGA 4), Dirig{\'e} par M. Artin, A. Grothendieck, et J. L.
  Verdier. Avec la collaboration de N. Bourbaki, P. Deligne et B. Saint-Donat.

\bibitem[Smi22]{smith:the-distribution-of-selmer-groups-1}
Alexander Smith.
\newblock The distribution of $\ell^\infty$-{S}elmer groups in degree $\ell$
  twist families {I}.
\newblock {\em arXiv preprint arXiv:2207.05674v2}, 2022.

\bibitem[{Sta}]{stacks-project}
The {Stacks Project Authors}.
\newblock {\itshape Stacks Project}.
\newblock \url{http://stacks.math.columbia.edu}.

\bibitem[Sun12]{sun:l-series-of-artin-stacks}
Shenghao Sun.
\newblock {$L$}-series of {A}rtin stacks over finite fields.
\newblock {\em Algebra Number Theory}, 6(1):47--122, 2012.

\bibitem[SW23]{sawinW:conjectures-for-distributions-containing-roots-of-unity}
Will Sawin and Melanie~Matchett Wood.
\newblock Conjectures for distributions of class groups of extensions of number
  fields containing roots of unity.
\newblock {\em arXiv preprint arXiv:2301.00791v1}, 2023.

\bibitem[Tat63]{tate:duality-theorems}
John Tate.
\newblock Duality theorems in {G}alois cohomology over number fields.
\newblock In {\em Proc. {I}nternat. {C}ongr. {M}athematicians ({S}tockholm,
  1962)}, pages 288--295. Inst. Mittag-Leffler, Djursholm, 1963.

\bibitem[Tay92]{taylor:the-geometry-of-the-classical-groups}
Donald~E. Taylor.
\newblock {\em The geometry of the classical groups}, volume~9 of {\em Sigma
  Series in Pure Mathematics}.
\newblock Heldermann Verlag, Berlin, 1992.

\bibitem[Tho19]{thorne:on-the-average-number-of-2-selmer}
Jack~A. Thorne.
\newblock On the average number of 2-{S}elmer elements of elliptic curves over
  {$\Bbb F_q(X)$} with two marked points.
\newblock {\em Doc. Math.}, 24:1179--1223, 2019.

\bibitem[TY14]{trihanY:the-ell-parity-conjecture}
Fabien Trihan and Seidai Yasuda.
\newblock The {$\ell$}-parity conjecture for abelian varieties over function
  fields of characteristic {$p>0$}.
\newblock {\em Compos. Math.}, 150(4):507--522, 2014.

\bibitem[Vas03]{vasiu2003surjectivity}
A.~Vasiu.
\newblock Surjectivity criteria for {$p$}-adic representations. {I}.
\newblock {\em Manuscripta Math.}, 112(3):325--355, 2003.

\bibitem[vdK80]{vanderkallen:homology-stability-for-general-linear-groups}
Wilberd van~der Kallen.
\newblock Homology stability for linear groups.
\newblock {\em Invent. Math.}, 60(3):269--295, 1980.

\bibitem[Ver67]{verdier:a-duality-theorem}
J.-L Verdier.
\newblock A duality theorem in the etale cohomology of schemes.
\newblock In {\em Proc. {C}onf. {L}ocal {F}ields ({D}riebergen, 1966)}, pages
  184--198. Springer, Berlin, 1967.

\bibitem[Wew98]{wewers:thesis}
Stefan Wewers.
\newblock {\em Construction of Hurwitz spaces}.
\newblock Institut für Experimentelle Mathematik Essen, Ph.D.\ thesis, 1998.

\bibitem[Wew99]{wewers:deformations-of-tame-admissible-covers-of-curves}
Stefan Wewers.
\newblock Deformation of tame admissible covers of curves.
\newblock In {\em Aspects of {G}alois theory ({G}ainesville, {FL}, 1996)},
  volume 256 of {\em London Math. Soc. Lecture Note Ser.}, pages 239--282.
  Cambridge Univ. Press, Cambridge, 1999.

\bibitem[Wil09]{wilson:the-finite-simple-groups}
Robert~A. Wilson.
\newblock {\em The finite simple groups}, volume 251 of {\em Graduate Texts in
  Mathematics}.
\newblock Springer-Verlag London, Ltd., London, 2009.

\bibitem[Wis16]{wise:moduli-of-morphisms}
Jonathan Wise.
\newblock Moduli of morphisms of logarithmic schemes.
\newblock {\em Algebra Number Theory}, 10(4):695--735, 2016.

\bibitem[Woo17]{wood:the-distribution-of-sandpile-groups}
Melanie~Matchett Wood.
\newblock The distribution of sandpile groups of random graphs.
\newblock {\em J. Amer. Math. Soc.}, 30(4):915--958, 2017.

\bibitem[Woo21]{wood:an-algebraic-lifting-invariant}
Melanie~Matchett Wood.
\newblock An algebraic lifting invariant of {E}llenberg, {V}enkatesh, and
  {W}esterland.
\newblock {\em Res. Math. Sci.}, 8(2):Paper No. 21, 13, 2021.

\bibitem[Woo22]{wood:probability-theory-for-random-groups}
Melanie~Matchett Wood.
\newblock Probability theory for random groups arising in number theory.
\newblock In {\em Proc. Int. Cong. Math}, volume~6, pages 4476--4508, 2022.

\bibitem[Zyw14]{zywina:inverse-orthogonal}
David Zywina.
\newblock The inverse galois problem for orthogonal groups.
\newblock {\em arXiv preprint arXiv:1409.1151v1}, 2014.

\end{thebibliography}

\end{document}